\documentclass[opre,nonblindrev]{informs3}

\DoubleSpacedXI 
\DeclareMathAlphabet{\mathpzc}{OT1}{pzc}{m}{it}
\usepackage{array}
\usepackage{threeparttable}
\newcolumntype{L}{>{\centering\arraybackslash}m{2cm}}
\usepackage{endnotes}
\usepackage{graphicx}
\usepackage{caption}
\usepackage{tablefootnote}
\usepackage{subcaption}
\usepackage{enumitem}
\usepackage{yfonts}
\usepackage{clipboard}
\newclipboard{myclipboard}
\let\footnote=\endnote

\usepackage{clipboard}
\usepackage[resetlabels,labeled]{multibib}

\usepackage{multirow}

\usepackage{natbib}
\usepackage{mathrsfs}
 \bibpunct[, ]{(}{)}{,}{a}{}{,}%
 \def\bibfont{\small}%
\usepackage{fourier} 
\usepackage{array}
\usepackage{makecell}
\usepackage{mathtools}
\definecolor{brightmaroon}{rgb}{0.76, 0.13, 0.28}
\definecolor{brickred}{rgb}{0.8, 0.25, 0.33}
\definecolor{bostonuniversityred}{rgb}{0.8, 0.0, 0.0}
\usepackage{xcolor}

\newcommand\hl[1]{%
  \bgroup
  \hskip0pt\color{red!80!black}%
  #1%
  \egroup
}
\TheoremsNumberedThrough     
\ECRepeatTheorems

\EquationsNumberedThrough    

\newtheorem{newproposition}{Proposition}
\usepackage[left]{lineno}
\usepackage{xcolor}
\usepackage{blindtext}
\newlength\tmplen
\newcommand\lrlineno[2][1]{%
\linenumbers%
\setcounter{linenumber}{#1}%
\leftlinenumbers%
\color{white}#2\color{black}%
\par%
\setcounter{linenumber}{#1}%
\setbox0=\vbox{#2}%
\setlength\tmplen{\dimexpr\dp0+\ht0+\the\dp\strutbox}%
\vspace*{-\tmplen}%
\rightlinenumbers%
#2%
\par%
}
\begin{document}

\RUNAUTHOR{Liu, Ye, and Lee}

\RUNTITLE{HDSL Under Approximate Sparsity with Applications to Nonsmooth Estimation and Regularized Neural Networks}

\TITLE{High-Dimensional Learning under Approximate Sparsity with Applications to   Nonsmooth Estimation and  Regularized  Neural Networks}

\ARTICLEAUTHORS{%
\AUTHOR{Hongcheng Liu}
\AFF{Department of Industrial and Systems Engineering, University of Florida, Gainesville, FL 32611, \EMAIL{liu.h@ufl.edu}} 
\AUTHOR{Yinyu Ye}
\AFF{Department of Management Science and Engineering, Stanford University, Stanford, CA 94305, \EMAIL{yyye@stanford.edu}}
\AUTHOR{Hung Yi Lee}
\AFF{Department of Industrial and Systems Engineering, University of Florida, Gainesville, FL 32611, \EMAIL{hungyilee@ufl.edu}} 
} 

\ABSTRACT{%
High-dimensional statistical learning (HDSL) has wide applications in data analysis, operations research, and decision-making. Despite the availability of multiple theoretical frameworks, most existing HDSL schemes stipulate the following two conditions: (a) the sparsity, and (b) the restricted strong convexity (RSC). This paper generalizes both conditions via the use of the folded concave penalty (FCP). More specifically, we consider an M-estimation problem where (i) the (conventional) sparsity is relaxed into the approximate sparsity and (ii) the RSC is completely absent. We show that the FCP-based regularization leads to poly-logarithmic sample complexity; the training data size is only required to be poly-logarithmic in the problem dimensionality. This finding can facilitate the analysis of two important classes of models that are currently less understood: the high-dimensional nonsmooth learning and the (deep) neural networks (NN).  For both problems, we show that the poly-logarithmic sample complexity can be maintained. In particular, our results indicate that the generalizability of NNs under over-parameterization can be theoretically ensured with the aid of regularization.}
%


 \HISTORY{}

\maketitle

%


\section{Introduction}

This paper is concerned with  {\it high-dimensional statistical learning} (HDSL), which refers to the problems  of estimating  a large number of parameters with few training data. The HDSL problems are found in wide applications ranging from imaging, bioinformatics, and deep learning, etc. A standard setup of the HDSL    is summarized below:  We are given a sequence of $n$-many i.i.d. sample observations, denoted    $Z_i$, $i=1,...,n$. Those observations are copies of a random vector $\mathcal Z$, which has  unknown support $\mathcal W\subseteq\Re^q$ (for some   positive integer $q$) and an unknown probability distribution.  In addition to the sample observations above, we are also given a   function     $L(\boldsymbol\beta,Z_i)$, where $L:\,\Re^p\times\mathcal W\rightarrow\Re$ measures the statistical loss with respect to the data point $Z_i$ and the vector of fitting parameters  $\boldsymbol \beta:=(\beta_j)\in\Re^p$. Here, the positive integer $p$ is  called the problem dimensionality (which is equal to the number of fitting parameters).  Throughout this paper, we assume  that $L$  is measurable and deterministic,     the expectation $\mathbb E[L(\boldsymbol\beta,\, \mathcal Z)]$ over $\mathcal Z$ is  well-defined  for all $\boldsymbol\beta\in\Re^p$, and $\inf_{\boldsymbol\beta}~\mathbb E[L(\boldsymbol\beta,\, \mathcal Z)]>-\infty$.   {\label{added sentence}\Copy{one sentence}{Though no convexity assumption is imposed explicitly,  many of our results are mainly useful when  $L(\,\cdot\,,z)$ is convex.}} Given the above, it is often essential to estimate the solution to the following {\it population-level problem} in many applications: 
\begin{align}\boldsymbol\beta^*\in\underset{{\boldsymbol\beta\in\Re^p}}{\arg\,\inf}~\left\{\mathbb L(\boldsymbol\beta):=\mathbb E[L(\boldsymbol\beta,\, \mathcal Z)]\right\}.\label{population-level model M-estimation}
\end{align} 
Here,  $\boldsymbol\beta^*$  
is intuitively the vector of fitting parameters which yields the smallest population-level statistical loss (a.k.a., population risk). Therefore, $\boldsymbol\beta^*$  is considered the target of estimation and  referred to as the vector of ``true parameters''.   The HDSL problem of interest is then how to estimate (or approximate)  $\boldsymbol\beta^*$, given the {\it a-priori}  knowledge of the samples $\mathbf Z_1^n:=(Z_1,Z_2,...,Z_n)$ and the formulation of   $L$, when $p\geq n$. We are especially interested in the more challenging case where the sample size $n$ is much smaller than the dimensionality $p$ (i.e.,   $p\gg n$). In measuring the approximation quality (a.k.a., recovery quality) of   an estimator $\widehat{\boldsymbol\beta}\in\Re^p$, we consider a metric of generalization error calculated as $\mathbb L(\widehat{\boldsymbol\beta})-\inf_{\boldsymbol\beta}~\mathbb L({\boldsymbol\beta})$.  This metric is the same  as the {\it excess risk}, which is discussed by \cite{bartlett2006a}, \cite{koltchinskii2010a}, and \cite{on2008a}, among others, as an important, if not the primary, measure of generalization performance for their results.

For the HDSL problems above, most traditional schemes are not applicable, because they usually stipulate that $n>p$. For example, one popularly adopted scheme is to construct a surrogate for the population-level formulation in \eqref{population-level model M-estimation} through the sample average approximation (SAA) below: 
\begin{align}\boldsymbol\beta^{SAA}\in\underset{\boldsymbol\beta}{\arg\,\inf}\,\left\{\mathcal L_n(\boldsymbol\beta,\mathbf Z_1^n):=\frac{1}{n}\sum_{i=1}^n L(\boldsymbol\beta,\, Z_i)\right\},\label{test new results}
\end{align}
where the objective function $\mathcal L_n(\boldsymbol\beta,\mathbf Z_1^n)$ is often also called the {\it empirical risk function} in the context of statistical and machine learning.  The SAA entails  desirable computational and statistical properties \citep[many of which are discussed by][and references therein]{shapiro2014a} but is not designed for handling  high dimensionality. Indeed, the best known upper bound on the approximation error of the SAA solution  is of the order $\mathcal O(\sqrt{p/n})$, where $\mathcal O(\cdot)$ hides some quantities independent of, or poly-logarithmic in, ``$\,\cdot\,$''. Consequently, the {\it estimator} of the true parameters generated by  solving the SAA, as well as by most other traditional statistical learning approaches, may incur non-trivial  errors  when $p\gg n$.

To address high dimensionality, several statistical schemes have already been made available. \cite[See][for  excellent reviews.]{buehlmann2011a,fan2014b} Among them, this paper follows  and   generalizes one of the most successful HDSL techniques introduced by \cite{fan2001a} and \cite{zhang2010a} as in the formulation below:
 \begin{align}
\inf_{\boldsymbol\beta\in\Re^p}& \left\{\mathcal L_{n,\lambda}(\boldsymbol\beta,\, \mathbf Z_1^n):=\mathcal L_n(\boldsymbol\beta,\, \mathbf Z_1^n)+\sum_{j=1}^pP_\lambda(\vert\beta_j\vert)\right\},
\label{general formulation}
\end{align}
where   $P_\lambda:\,\Re_+\rightarrow\Re_+$ is a term of sparsity-inducing regularization in the form of a {\it folded concave penalty} (FCP). One  mainstream  special case  of the existing FCPs, called the {\it minimax concave penalty} (MCP) \citep{zhang2010a}, is of our particular consideration.  The  MCP is formulated as  
\begin{align}P_\lambda(\theta)=\int_{0}^{\theta}\frac{[a\lambda-t]_+}{a}dt,\qquad\theta\geq 0,\label{FCP penalty formulation}
\end{align} with $[\cdot]_+:=\max\{0,\,\,\cdot\,\}$ and tuning parameters $a,\,\lambda>0$.  (Hereafter, we use the term ``FCP'' to refer to the MCP exclusively.) Eq.\ \eqref{general formulation} is nonconvex, to which the local and/or global solutions have been shown to  entail desirable statistical performance \citep{loh2015a,wang2013a,wang2014a,zhang2012a,p2017a}. {\label{added discussion on FCP parameters}\Copy{to understand the roles copy}{To understand the roles of the tuning parameters $a$ and $\lambda$ to the FCP, we may observe that its first derivative, $P_\lambda'(\theta)$,  is a non-increasing function with $P_\lambda'(0)=\lambda$ and $P_\lambda'(\theta)=0$ for all $\theta\geq a\lambda$. This means that $\lambda$ determines how intense the penalty is to induce a fitting parameter that is almost zero to  be exactly zero. The intensity of this penalty becomes smaller as the magnitude of the corresponding fitting parameter increases. Once the absolute value of that parameter is beyond the threshold $a\lambda$, the penalty becomes a constant and thus (locally) ineffective.  Furthermore, we also observe that $P_{\lambda}''(\theta)=-\frac{1}{a}$ for all $\theta\in(0,\,a\lambda)$ and $P_{\lambda}''(\theta)=0$ for all $\theta>a\lambda$. Therefore, $a$  determines the curvature of the FCP near the origin.}}

{\label{added discussion reference 2}\Copy{Alternative sparsity-inducing penalties sen}{Alternative sparsity-inducing penalties, such as  the smoothly clipped absolute deviation (SCAD) introduced by \cite{fan2001a}, the least absolute shrinkage and selection operator  (Lasso)    proposed by \cite{tibshirani1994a},  and the bridge penalty (a.k.a., the $\ell_{\mathbf q}$ penalty  with $0<\mathbf q < 1$) as discussed by \cite{frank1993a}, have all been shown to be very effective in HDSL  by many results due to \cite{fan2001a,bickel2009a,fan2011a,fan2014b,loh2015a,raskutti2011a,n2012a,wang2013a,wang2014a,zhang2012a,zou2006a,zou2008a,liu2017a,liu2018a} and \cite{p2017a}, to name only a few.  Many of those results provide {\it oracle inequalities}, which ``relates the performance of a real estimator with that of an ideal estimator'' \citep{candes2006a}.
 \cite{ndiaye2017a}, \cite{l2010a}, \cite{fan2001a}, \cite{chen2010a}, and \cite{liu2017a} have presented   thresholding rules and bounds on the number of nonzero dimensions for a high-dimensional linear regression problem with different penalty functions.

\indent Despite the availability of several analytical frameworks for HDSL in the current literature, most existing HDSL theories require the two assumptions below, which are sometimes overly critical, to guarantee any generalization performance:

\begin{itemize}
\item[{\it (A)}.] The satisfaction of the (conventional) sparsity condition, written as $\Vert \boldsymbol\beta^*\Vert_0\ll p$, where  $\Vert\cdot\Vert_0$ denotes the number of nonzero entries of a vector. 
\item[{\it (B)}.]  The satisfaction of regularity conditions on the eigenvalues of the Hessian matrix of $L(\,\cdot\,,\mathcal Z)$ in the form of the {\it restricted strong convexity} (RSC)  \citep{n2012a}, the {\it restricted isotropic property} (RIP)  \citep{candes2007a}, or the {\it restricted eigenvalue} (RE) condition  \citep{bickel2009a}. 
\end{itemize}
}}
The sparsity assumption essentially means that few  dimensions ``matter'' despite that the total number of dimensions is very high. Meanwhile, the RSC, RIP, and RE can all be  interpretable as the stipulation that $\mathcal L(\,\cdot\,,\mathbf Z_1^n)$ is {\it strongly convex} everywhere in some subset of $\Re^p$. The RSC is implied by the RE and RIP for some choices of parameters \citep{n2012a,de2009a}.  Except for some special cases of the  generalized linear models \cite[as discussed by, e.g.,][]{bickel2009a}, when both   (A) and (B) above are violated, little is known about the performance of \eqref{general formulation} or that of most other HDSL schemes in terms of their generalization performance in general.   \cite{n2012a}  has considered HDSL under weak sparsity, but the RSC is still assumed for establishing the generalization error bounds.
 
 In contrast to the literature, this paper is  concerned with the effectiveness of \eqref{general formulation} in  addressing the HDSL problems when   the RSC is completely absent and the traditional sparsity is relaxed into the approximate sparsity (A-sparsity) as below. 
 \begin{assumption}\label{Assumption A-sparsity original}
 $\mathbb L(\boldsymbol\beta^*_{{\varepsilon_{A}}})-\inf_{\boldsymbol\beta}\mathbb L(\boldsymbol\beta)\leq{\varepsilon_{A}}$ and $s:=\Vert\boldsymbol\beta^*_{\varepsilon_{A}}\Vert_0\ll p$
for some  $\varepsilon_{A}\geq 0$, $\boldsymbol\beta_{\varepsilon_{A}}^*:\,\Vert \boldsymbol\beta_{\varepsilon_{A}}^*\Vert_\infty\leq R$, and $R\geq 1$. 
\end{assumption}
 {\label{corrected statement A-sparsity}\Copy{Intuition Assumption A-sparsity Copy}{Intuitively, Assumption \ref{Assumption A-sparsity original}   means that, although $\boldsymbol\beta^{*}$ can be dense,   replacing most of the nonzero entries of $\boldsymbol\beta^{*}$ by zero does not cause the population risk to increase too much.}}  It is evident that, if $\varepsilon_A=0$, Assumption \ref{Assumption A-sparsity original} is reduced to the (traditional) sparsity.  
 
 In certain applications of HDSL (e.g., the deep neural networks to be discussed subsequently),  it is more convenient to consider a (slight)    generalization to Assumption \ref{Assumption A-sparsity original}  in the following.

\begin{assumption}\label{Assumption A-sparsity}
 $\mathbb L(\boldsymbol\beta^*_{{\varepsilon_{A}}})-L_g^*\leq{\varepsilon_{A}}$ and $s:=\Vert\boldsymbol\beta^*_{\varepsilon_{A}}\Vert_0\ll p$
for some    $\varepsilon_{A}\geq 0$, $\boldsymbol\beta_{\varepsilon_{A}}^*:\,\Vert \boldsymbol\beta_{\varepsilon_{A}}^*\Vert_\infty\leq R$,  $L_g^*\leq \inf_{\boldsymbol\beta}~\mathbb L(\boldsymbol\beta)$, and $R\geq 1$.
\end{assumption}
Apparently,   Assumption \ref{Assumption A-sparsity} is more general than Assumption \ref{Assumption A-sparsity original}, and the two are equivalent when  $L_g^*=\inf_{\boldsymbol\beta}~\mathbb L({\boldsymbol\beta})$.  Hereafter, both Assumptions \ref{Assumption A-sparsity original} and \ref{Assumption A-sparsity} are referred to as  A-sparsity when there is no ambiguity. Without loss of generality, we let $s> 1$ throughout   this paper.

{\Copy{to copy 2}{The assumption of $\Vert \boldsymbol\beta_{\varepsilon_A}^*\Vert_\infty\leq R$ is non-critical.  It is comparable to, if not less restrictive than, some common  assumptions in the literature. For example,  in addressing HDSL under (the conventional) sparsity,  \cite{p2017a}  and \cite{loh2015a} both assume the    estimator and the vector of true parameters to be contained within a convex and bounded  set of $\{\boldsymbol\beta:\,\vert\boldsymbol\beta \vert\leq R_{\ell_1}\}$  for some $R_{\ell_1}>0$. Verifiably, under their assumptions, $ \Vert \boldsymbol\beta_{\varepsilon_A}^*\Vert_\infty\leq R$ holds with some $R\leq R_{\ell_1}$. Furthermore, we later show that our generalization error bounds depend only logarithmically on $R$. Thus, it is flexible to pick the value of $R$ in practice; we only need to have a coarse estimation of an upper bound on $\Vert \boldsymbol\beta_{\varepsilon_A}^*\Vert_\infty$. Even if $R$ overestimates $\Vert \boldsymbol\beta_{\varepsilon_A}^*\Vert_\infty$ too much, the performance of the proposed scheme would probably not be impacted significantly.}}

We believe that the flexibility of A-sparsity and the relaxation of the RSC can allow the HDSL theories to cover a more comprehensive class of applications. Indeed, as we are to articulate  later, our results on HDSL under A-sparsity can facilitate the comprehension of two  important  classes of problems whose theoretical underpinnings are currently lacking from the literature: (i) A high-dimensional nonsmooth learning problem (nonsmooth HDSL), that is, an HDSL problem with a nonsmooth empirical risk function, and (ii)   a (deep and over-parameterized) neural network (NN) model.

{\color{black} {\label{Weak sparsity discussion 1}\Copy{Weak sparsity discussion 1 content}{More general forms of sparsity,   such as the weak sparsity assumption \citep{n2012a}, have been discussed previously.}}} However, the only existing discussions on simultaneously  relaxing both the sparsity   and the RSC assumptions are due to \cite{liu2018a}, to our knowledge.  Their results imply that  the excess risk of an estimator $\widehat{\boldsymbol\beta}\in\Re^p$ generated as a certain stationary point to the  formulation \eqref{general formulation} can be bounded by {${\mathcal O}\left(\frac{\sqrt{\ln p}}{n^{1/4}}\cdot\left(1+\sqrt{\varepsilon_A}\right)+\varepsilon_{A}\right)$}\label{to change formula}.
 This bound is  reduced to ${\mathcal O}\left(\frac{\sqrt{\ln p}}{n^{1/4}}\right)$ when $\varepsilon_{A}=0$.  In contrast, our findings   {in the current paper} can strengthen the previous results. More specifically,  we relax the subgaussian assumption stipulated by  \cite{liu2018a} and impose the weaker, subexponential, condition instead. In addition, the assumption of twice-differentiability made by \cite{liu2018a} is also  weakened. In the more general settings, we further show that sharper error bounds can be achieved at a stationary point that {(a)} satisfies a set of significant subspace second-order necessary conditions (S$^3$ONC) to be formalized subsequently, and {(b)} has an objective function value no worse than that of the   solution  to the Lasso problem, formulated below:
\begin{align}
\min_{\boldsymbol\beta\in\Re^p} \left\{\mathcal L_n(\boldsymbol\beta,\, \mathbf Z_1^n)+\sum_{j=1}^p \lambda \cdot \vert\beta_j\vert \right\}.\label{Lasso problem}
\end{align}
 We are to discuss   some S$^3$ONC-guaranteeing algorithms to meet the first requirement  soon afterwards. To meet the second requirement, we may always  initialize the S$^3$ONC-guaranteeing algorithm with a solution to   \eqref{Lasso problem}, which is often polynomial-time solvable if $\mathcal L_n(\,\cdot\,,\mathbf Z_1^n)$ is convex. 
 
 Our new bounds on those S$^3$ONC solutions are summarized below.  First, in the case where $\varepsilon_{A}=0$, we can bound  the excess risk  by  ${\mathcal O}\left(\frac{{\ln p}}{n^{2/3}}+\frac{\sqrt{\ln p}}{n^{1/3}}\right)$, which is   better than the aforementioned result  by \cite{liu2018a} in terms of the dependance on $n$. Second, when $\varepsilon_{A}$ is nonzero, the excess risk is then bounded by
 \begin{align}
\mathbb L(\widehat{\boldsymbol\beta})-\inf_{\boldsymbol\beta}~\mathbb L(\boldsymbol\beta)\leq{\mathcal O}\left(\frac{\ln p}{n^{2/3}}+{\frac{\sqrt{\ln p}}{n^{1/3}}}+\sqrt{\frac{\varepsilon_{A}}{n^{1/3}}}+\varepsilon_{A}\right).\label{lasso initialized}
\end{align}
Third, if we further relax the requirement  above and consider an arbitrary S$^3$ONC solution, then the excess risk becomes
\begin{align}
\mathbb L(\widehat{\boldsymbol\beta})-\inf_{\boldsymbol\beta}~\mathbb L(\boldsymbol\beta)\leq{\mathcal O}\left(\frac{{\ln p}}{n^{2/3}}+\sqrt{\frac{{\ln p}}{n}}+\frac{1}{n^{1/3}}+\sqrt{\frac{\Gamma+\varepsilon_{A}}{n^{1/3}}}+\Gamma+\varepsilon_{A}\right),\label{all solutions}
\end{align}
where $\Gamma\geq 0$  is (an underestimation of) the suboptimality gap that this  S$^3$ONC solution incurs in minimizing $\mathcal L_{n,\lambda}(\,\cdot\,,\mathbf Z_1^n)$ (as defined in \eqref{general formulation}).

 Admittedly,  our    excess risk bounds are  less  appealing than the generalizability results made available in some important previous works by \cite{p2017a,raskutti2011a}, and \cite{n2012a}, etc., under the assumption of the RSC. In contrast, we argue that  our results are established under  a more general  set of conditions  and can complement the existing results in the HDSL problems beyond the RSC. {\label{added comments suboptimality intro}\Copy{para in Gamma}{It is also worth noting that \eqref{all solutions} is in the parameterization of  $\Gamma$, which can only  be explicitly controlled when $\mathcal L_n(\,\cdot,\,\mathbf Z_1^n)$ is convex  in general. Nonetheless, we argue that, in some interesting special cases, one may still control $\Gamma$ despite the absence of convexity. One of such examples is presented in this paper  as we discuss the theoretical applications of HDSL under A-sparsity to the NNs in Sections \ref{sec theoretical app} and \ref{additional theoretical results}.}}

The S$^3$ONC   is a necessary condition for local minimality. {Compared to the second-order KKT conditions, the S$^3$ONC is weaker and potentially easier computable.}
To generate a solution that satisfies the S$^3$ONC  admits pseudo-polynomial-time algorithms, such as the variants of Newton's method proposed  by \cite{haeser2017a,bian2015a,Ye1992a,Ye1998a} and \cite{nesterov2006a}. All those algorithms provably ensure a   $\gamma_{opt}$-approximation (with a user-specified error tolerance $\gamma_{opt}>0$) to the second-order KKT conditions at  the best-known iteration complexity of the rate ${\mathcal O}(1/\gamma_{opt}^3)$. The second-order KKT conditions then  imply the S$^3$ONC.  To add to the current solution schemes, we derive a  new  gradient-based method that provably guarantees the S$^3$ONC. In contrast to the literature, the iteration complexity of this new algorithm is   ${\mathcal O}(1/\gamma_{opt}^2)$, which improves upon the existing alternatives. Due to  the gradient-based nature of the proposed algorithm, it does not access the Hessian matrix or its inverse. {Therefore, we think that this gradient-based algorithm may be of some independent interest.}

\subsection{Some theoretical applications}
 
 As mentioned, our results on HDSL under A-sparsity can be employed in the analysis of two important classes of statistical and machine learning models: (a)  nonsmooth HDSL, and (b) deep NNs. Some additional details are  provided below.

\subsubsection{Nonsmooth HDSL.}\label{HDNL setup}
Although several special cases of HDSL with nonsmoothness, such as high-dimensional least absolute regression, high-dimensional quantile regression, and  high-dimensional support vector machine (SVM) have been discussed by \cite{wang2013b}, \cite{belloni2011a}, \cite{zhang2016b,zhang2016c} and  \cite{peng2016a}, there exist few theories that apply to scenarios without  an everywhere differentiable loss function in general, especially when non-differentiability may occur  at, or in a near neighborhood of, the vector of true parameters. 

 In contrast, our theories on HDSL under A-sparsity can be utilized to understand the generalization performance of a flexible set of nonsmooth HDSL problems. Indeed, their nonsmooth statistical loss functions can be approximated by another formulation that preserves the continuous differentiability, and the resulting approximation error can then be handled through the notion of A-sparsity. Analyzing this approximation leads to the following bound on the excess risk at an S$^3$ONC solution when the vector of true parameters is A-sparse in the sense of Definition \ref{Assumption A-sparsity original}:
\begin{align}
{\mathcal O}\left(\frac{{\ln p}}{n^{3/4}}+\frac{\sqrt{\ln p}}{n^{1/4}}+\sqrt{\frac{\varepsilon_{A}}{n^{1/4}}}+\varepsilon_{A}\right).\label{Nonsmooth rate}
\end{align}
In particular, under the conventional sparsity assumption (that is, when $\varepsilon_{A}=0$), the rate above   becomes ${\mathcal O}\left(\frac{{\ln p}}{n^{3/4}}+\frac{\sqrt{\ln p}}{n^{1/4}}\right)$.
To our knowledge, this is perhaps the first generic theory for the high-dimensional M-estimation problems in which the empirical risk function may not be everywhere differentiable.

\subsubsection{Regularized neural network.}\label{RNN setup}
 
The NNs have been frequently discussed and widely applied in recent literature \citep{schmidhuber2015a,lecun2015a,yarotsky2017a}. Despite  the  frequent and exciting advancements in the NN-related algorithms, models, and applications, the development of their theoretical underpinnings is seemingly lagging behind.   \cite{devore1989a}, \cite{yarotsky2017a}, \cite{mhaskar2016a}, and \cite{mhaskar1996a}, etc., have explicated the expressive power of the NNs in the  approximation of different types of functions. As for the generalizability of NNs, one of the focuses of this paper, effective theoretical frameworks have been discussed by \cite{cao2019generalization,li2018learning,brutzkus2017sgd,allen2019learning,wang2019learning,daniely2017sgd,neyshabur2015norm,bartlett2017spectrally,hardt2015train, zhang2016understanding, li2018tighter, jakubovitz2019generalization}, among others. However, for the vast majority of the existing results on the deep NNs, the generalization error bounds grow polynomially in the dimensionality (which is equal to the number of fitting parameters  and is also called the network size) and sometimes even increase exponentially in the depth of the network.   Such a high sensitivity to dimensionality and depth is inconsistent with the empirical performance of the NNs in many practical applications, where over-parameterization and deep architectures are   common and often preferred by practitioners.

In contrast, we  analyze the NNs through the lens of HDSL under A-sparsity and  consider an FCP-regularized NN training formulation  as
 a special case of \eqref{general formulation} in binary classification. Our results indicate that the NN's generalization errors at local solutions can  be both poly-logarithmic   in the number of fitting parameters   and  polynomial  in the network depth. Thus, we think that the results herein can facilitate understanding the powerful performance of the NNs in practice, especially for the over-parameterized and deep models.  \cite{barron2018approximation} have shown the existence  of fitting parameters for an NN with ramp activation functions to achieve the poly-logarithmic sample complexity. Compared with \cite{barron2018approximation}, our analysis may present better flexibility in the choice of activation functions and provide more insights towards the computability of the desired fitting parameters in training a deep NN to ensure the proven error bounds.

More specifically, we show that  the generalization error incurred by an S$^3$ONC solution to the FCP-regularized training formulation of an NN is bounded by
\begin{align}
{\mathcal O}\left(\underbrace{\frac{{s_A}\cdot \mathcal D\cdot \ln p }{n^{2/3}}+\sqrt{\frac{{s_A}\cdot \mathcal D \cdot \ln p}{n}}+\frac{1}{n^{1/3}}}_{\text{$\mathcal O\left(n^{-1/3}+ n^{-1/2}\mathcal D\cdot \ln p\right)$}} +\underbrace{\Gamma}_{\text{Suboptimality gap}} +  \underbrace{\Omega({s_A})}_{\text{Representability gap}} +  \underbrace{\sqrt{\frac{\Gamma+ \Omega({s_A}) }{n^{1/3}}}}_{\text{Interaction term}}\right),\label{general inequality NN bound}
\end{align}
for any fixed ${s_A}:\,1\leq {s_A}\leq p$,
with overwhelming probability. Here, $\mathcal D$ is the number of NN layers,   $\Gamma\geq 0$ is the suboptimality gap incurred by the S$^3$ONC solution of consideration, and $\Omega({p'})$, for any ${p'}:\,1\leq {p'}\leq p$, is the   architecture-dependent representability gap (a.k.a., the model misspecification error or  the expressive power) of an NN with ${p'}$-many nonzero fitting  parameters.
By \eqref{general inequality NN bound} above, the generalization error of an NN consists of four terms: (i) a  generalization error term of the order $\mathcal O\left(n^{-1/3}+n^{-1/2}\mathcal D\ln p\right)$; (ii) the suboptimality gap; (iii) a term that  measures the NN's representability; and (iv) a term that is dependent on suboptimality gap, sample size, and representability, simultaneously. It is worth noting that \eqref{general inequality NN bound} is obtained with little restriction on the NN architecture and the data generation process.   Combining \eqref{general inequality NN bound}  with the existing results on the representability analysis of NNs, we further derive more explicit  generalization error bounds. For example, we show that the error yielded by an NN with smooth activation functions can be bounded by  ${\mathcal O}\left(\frac{\mathcal D\cdot \ln p}{ n^{1/3}}+\sqrt{{\frac{\Gamma}{n^{1/3}}}}+\Gamma\right)$, when we assume that data from different categories are separable by a polynomial function (as well as a couple of other conditions on the NN architecture).
  
{ {\label{changed text reference}\Copy{Error bound in depends}{The error bound in \eqref{general inequality NN bound} depends on $\Gamma$, the suboptimality gap.  To explicitly bound its value is challenging in general because of the nonconvexity of an NN's  training formulation. Nonetheless, we show that  some pseudo-polynomial-time computable solutions generated with the aid of an efficient initialization provably ensure the explicit control of $\Gamma$  in  the same settings considered by \cite{cao2019a}.  In such a case, the generalization error  is further explicated into
 \begin{align}
 \mathcal O\left( \frac{\mathcal D}{n^{1/3}}\cdot \ln p\right),\label{desired local}
  \end{align}
 which becomes independent of $\Gamma$. In achieving this result, our settings seem more general than  \cite{wang2019a}, and our rates on both $\mathcal D$ and $p$ are perhaps more appealing than most of the existing results.   In particular, \cite{wang2019a}  focus on  ReLU-NNs  \cite[that is, the NNs where the activation functions are ReLU, as discussed by][]{glorot2011deep} with one hidden layer, but our approach can handle deep NNs under more general hyper-parameters.  For  deep and wide NNs, \cite{cao2019a} have established generalization error bounds, which,  however, increase exponentially  in   the number of layers in the same settings of our discussion. In contrast, our bound is both poly-logarithmic in dimensionality and polynomial in the number of layers.   The computational complexity  of training an NN with the claimed error bound is in pseudo-polynomial time.}}}
 
  In obtaining our results, we do not artificially impose any condition on sparsity or alike. As we articulate in Section \ref{NN training}, our findings are based on the observation that the   A-sparsity (as in Assumption \ref{Assumption A-sparsity}) is an intrinsic property   implied by the NN's  expressive power.

\subsection{Summary of results}\label{summary of results section sub}
Table \ref{summary of complexity} summarizes the sample complexity results proven in this paper. In contrast to the literature, we claim that our results  could lead to the following contributions:
\begin{itemize}
\item[1.] We provide the first  HDSL theory for problems where the three conditions---the twice-differentiability,  the RSC or alike, and the  sparsity---are simultaneously relaxed.   In the more general settings, we show that HDSL is still possible   even if  the sample size is only poly-logarithmic in the dimensionality. In Table \ref{summary of complexity}, the results are presented in  the rows for ``HDSL under A-sparsity''.   
\item[2.] We have derived  a  pseudo-polynomial-time gradient-based method to compute an S$^3$ONC solution.  Even though the S$^3$ONC is a set of second-order necessary conditions, the proposed algorithm does not need to access the Hessian matrix. Furthermore,   the iteration complexity of the proposed method is provably $\mathcal O(\frac{1}{\gamma_{opt}^2})$  in achieving a $\gamma_{opt}$-approximation to the S$^3$ONC, which is sharper than the more generic algorithms such as the variations of Newton's method.
\item[3.]{\label{point 3 of contribution summary}\Copy{As theoretical applications Copy}{As theoretical applications of our error bounds for HDSL under A-sparsity, we derive generalizability results for nonsmooth HDSL problems and deep NNs. More specifically, for a flexible class of high-dimensional nonsmooth M-estimation problems, we prove perhaps the first poly-logarithmic sample complexity bound without the RSC assumption.  The corresponding result  is summarized in Table \ref{summary of complexity} in the rows for ``Nonsmooth HDSL under A-sparsity''. As for the NNs, our sample requirement is only poly-logarithmic in the  network size  and polynomial in the number of layers, providing theoretical underpinnings for the generalizability of an NN under over-parameterization. These results are summarized  in the rows for ``Neural Network'' of Table \ref{summary of complexity}. }}  

\end{itemize}
 \begin{table}[h!]
{\caption{Summary of sample complexities.  $\varepsilon_{A}$ is the parameter for A-sparsity as in Assumption \ref{Assumption A-sparsity original}.  $p$ and $n$ are the sample size and the dimensionality, respectively. ``ReLU-NN'' stands for an NN with ReLU activation. }\label{summary of complexity}
\bigskip
\begin{center}
\begin{tabular}{cccccccccccc}\hline\hline
\multicolumn{2}{c}{\it HDSL under A-sparsity}
\\
\hline 
\thead{S$^3$ONC 
  initialized with Lasso}&  \thead{$\frac{\ln p}{n^{2/3}}+{\frac{\sqrt{\ln p}}{n^{1/3}}}+\sqrt{\frac{\varepsilon_{A}}{n^{1/3}}}+\varepsilon_{A}$}       
\\\hline
\thead{S$^3$ONC 
 with suboptimality gap $\Gamma$} &  \thead{$\frac{{\ln p}}{n^{2/3}}+\sqrt{\frac{{\ln p}}{n}}+\frac{1}{n^{1/3}}+\sqrt{\frac{\Gamma+\varepsilon_{A}}{n^{1/3}}}+\Gamma+\varepsilon_{A}$}   
\\\hline\hline
\multicolumn{2}{c}{\it  Nonsmooth HDSL under A-sparsity}
\\\hline 
\thead{S$^3$ONC 
initialized with Lasso}              & $\frac{{\ln p}}{n^{3/4}}+\frac{\sqrt{\ln p}}{n^{1/4}}+\sqrt{\frac{\varepsilon_{A}}{n^{1/4}}}+\varepsilon_{A}$      
\\\hline\hline
\multicolumn{2}{c}{\it Neural network (with $\mathcal D$-many layers and $p$-many fitting parameters)}
\\\hline 
\thead{S$^3$ONC to a general NN
 with suboptimality gap $\Gamma$\\ and any $s_A:\,1\leq s_A\leq p$}    & \thead{$  {\frac{{s_A}\cdot \mathcal D\cdot \ln p }{n^{2/3}}+\sqrt{\frac{{s_A}\cdot \mathcal D \cdot \ln p}{n}}+\frac{1}{n^{1/3}}} +   {\Omega({s_A})} + {\Gamma}+   {\sqrt{\frac{\Gamma+ \Omega({s_A}) }{n^{1/3}}}}$} 
\\\hline 
\thead{S$^3$ONC 
to  an NN for a flexible choice of activation functions \\ with suboptimality gap $\Gamma$, when the target function is polynomial}    & \thead{$ \frac{\mathcal D}{ n^{1/3}}\cdot \ln p+\sqrt{{\frac{\Gamma}{n^{1/3}}}} +\Gamma$} 
\\\hline 
\thead{A pseudo-polynomial-time computable solution 
in training   \\ a ReLU-NN in the same settings by \cite{cao2019a}}    & \thead{$\frac{\mathcal D}{n^{1/3}}\cdot \ln p$} 
\\\hline\hline
\end{tabular}
\end{center}
}
\end{table}
\subsection{Organization of the paper}

The rest of the paper is organized as below: 
Section \ref{assumptions here} summarizes the  settings and assumptions.
Section \ref{3} introduces the S$^3$ONC. Section \ref{sec: main results} states our main results concerning HDSL under A-sparsity. A pseudo-polynomial-time  solution scheme that guarantees the S$^3$ONC is discussed in Section \ref{solution section}. Section \ref{sec theoretical app} discusses the theoretical applications  to nonsmooth HDSL and the regularized (deep) NNs.  Some numerical experiments are presented in Section \ref{Numerical section}.  Sections \ref{additional theoretical results} and \ref{sec Numerical Experiments EC} of the electronic companion, respectively, present some additional theoretical results on the NN and supplementary numerical results on both the SVM and the NN.  Section \ref{conclusion: sec} concludes the paper.

Our notations are summarized below.   We use $p$ and $n$ to represent the numbers of dimensions (fitting parameters) and the sample size.  We  let $\Vert\,\cdot\,\Vert_{\mathbf p}$  ($1\leq \mathbf p\leq \infty$) be the  $\mathbf p$-norm, except that $1$- and $2$-norms are denoted by $\vert\,\cdot\,\vert$ and $\Vert\,\cdot\,\Vert$, respectively. When there is no ambiguity, we also denote by $\vert\,\cdot\,\vert$ the cardinality of a set, if the argument is  a finite set. Let $\Vert \cdot\Vert_F$ of a matrix be  its Frobenius norm and let $\Vert\cdot\Vert_0$ of a vector be the number of its nonzero entries. For a random vector $\mathbf v=(v_j)\in\Re^p$, we  denote that $\Vert \mathbf v\Vert_\infty\leq R$ if $\mathbb P[\vert v_j\vert\leq R,\,\forall j=1,...,p]=1$. For a random variable $X$, its subexponential and subgaussian norms  are denoted by  $\Vert X\Vert_{\psi_1}$ and $\Vert X\Vert_{\psi_2}$, respectively.    {\label{definition of 2-1 norm}\Copy{defi of norm Copy}{$\Vert \mathbf A\Vert_{1,2}:=\max_{\mathbf x\in\Re^{m_1},\,\mathbf u\in\Re^{m_2}}\{\mathbf u^\top\mathbf A\mathbf x:\,\Vert \mathbf x\Vert_1=1,\,\Vert\mathbf u\Vert_2=1\}$ for integers $m_1,\,m_2$ and a matrix $\mathbf A\in\Re^{m_2\times m_1}$.}}  
 For a function $f$, denote by $\nabla f$ its gradient, whenever it exists. For a vector $\boldsymbol\beta=(\beta_j)\in\Re^p$  and a set $S\subset\{1,...,p\}$, let $\boldsymbol\beta_S=(\beta_j:\, j\in S)$ be a sub-vector of $\boldsymbol\beta$.  For any vector $\mathbf v=(v_j)$,  the notation $diag(\mathbf v)$ represents the diagonal matrix whose $j$th diagonal entry is $v_j$. We denote by $vec(M_1,M_2,..., M_{m})$  the vector that collects all the entries of the matrices $M_1$, $M_2,$ ..., $M_m$. The vector   $e_j$ is the $j$th standard basis. $\lceil x\rceil$ (or $\lfloor x\rfloor$) for any $x\geq 0$ is the smallest  (or largest) integer that is greater (or smaller, respectively) than or equal to $x$.   Finally, we   denote by $O(\cdot)$'s and  $\mathcal O(\cdot)$'s, respectively, the complexity rates that hide (potentially different) universal constants and quantities at most logarithmically dependent on ``$\cdot$''.
 
 \section{Settings and assumptions}\label{assumptions here}
In this section, we summarize our  assumptions in addition to the aforementioned settings. We  assume    that  
 the gradient $\nabla L(\boldsymbol\beta,z):=(\frac{\partial  L(\boldsymbol\beta,z)}{\partial \beta_j}:\,j=1,...,p)$ of 
$L(\boldsymbol\beta,z)$ w.r.t. $\boldsymbol\beta$  is well-defined for all $\boldsymbol\beta\in\Re^p$ and almost every $z\in\mathcal W$. 
Furthermore, we  also suppose that  $\frac{\partial L(\boldsymbol\beta,z)}{\partial\beta_j}$ is  Lipschitz   continuous for all $\boldsymbol\beta\in\Re^p$; that is, there exists a scalar $U_L>0$ such that
  \begin{align}\left\vert\left[\frac{\partial L(\boldsymbol\beta,z)}{\partial\beta_j} \right]_{\boldsymbol\beta=\widetilde{\boldsymbol\beta}+\delta \cdot e_j}-\left[\frac{\partial L(\boldsymbol\beta,z)}{\partial\beta_j}\right]_{\boldsymbol\beta=\widetilde{\boldsymbol\beta}}\right\vert\leq  U_{L}\cdot\vert\delta \vert,\label{Lipschitz gradient condition value new value}
 \end{align}
  for almost every $z\in \mathcal W$ and for all $\widetilde{\boldsymbol\beta}\in\Re^p$, $\delta\in\Re$, $j=1,...,p$. These regularities are to be relaxed when we  later discuss the nonsmooth HDSL problems and the  ReLU-NNs.  Apart from the above, two additional assumptions are imposed.

\begin{assumption}\label{sub exponential condition}
For all  $\boldsymbol\beta\in\Re^p:\,\Vert\boldsymbol\beta\Vert_\infty\leq R$ and $i=1,...,n$, it holds that $\mathbb E[ L(\boldsymbol\beta,\, Z_i)]$ is finite-valued and  $L(\boldsymbol\beta,\, Z_i)-\mathbb E[ L(\boldsymbol\beta,\, Z_i)]$  follows a subexponential distribution; that is,
$
\Vert L(\boldsymbol\beta,\, Z_i)-\mathbb E[ L(\boldsymbol\beta,\, Z_i)]\Vert_{\psi_1}\leq \sigma,
$
for some $\sigma \geq 1$. 
\end{assumption}

\begin{remark}\label{remark beinstein}
As an implication of Assumption \ref{sub exponential condition},  for all $\boldsymbol\beta\in\Re^p:\,\Vert\boldsymbol\beta\Vert_\infty\leq R$, (combined with the assumption that   $Z_i$, $i=1,...,n$, are i.i.d.) a well-known Bernstein-like inequality holds as below:
\begin{align}
\mathbb P\left(\left\vert \sum_{i=1}^n a_i\left\{L(\boldsymbol\beta,\, Z_i)-\mathbb E[L(\boldsymbol\beta,\, Z_i)] \right\}\right\vert>\sigma \cdot \left(\Vert \mathbf a\Vert\sqrt{t}+\Vert \mathbf a\Vert_\infty t\right)\right)\leq 2\exp\left(-ct\right),\quad
\forall t\geq 0, \,\mathbf a=(a_i)\in\Re^n,\label{Bernstein inequality result}
\end{align}
for some absolute constant $c\in(0,\,0.5]$.  Interested readers are referred to \cite{vershynin2012a} for more detailed discussions on the subexponential distributions.
 
\end{remark}

\begin{assumption}\label{Lipschitz condition}
For some measurable and deterministic function $\mathcal C:\, \mathcal W\rightarrow \Re_+$, the random variable $\mathcal C(Z_i)$ satisfies that
$
\left\Vert \mathcal C(Z_i)-\mathbb E\left[\mathcal C(Z_i)\right]\right\Vert_{\psi_1}\leq \sigma_L, $
 for all $i=1,...,n$, for some $\sigma_L\geq 1$. Furthermore,
$\vert L(\boldsymbol\beta_1,\, z)-L(\boldsymbol\beta_2,\, z)\vert\leq \mathcal C(z)\Vert \boldsymbol\beta_1-\boldsymbol\beta_2\Vert,$
 for all $\boldsymbol\beta_1,\,\boldsymbol\beta_2\in\Re^p\cap\{\boldsymbol\beta:\,\Vert\boldsymbol\beta\Vert_\infty\leq R\}  $  and almost every  $z\in\mathcal W$.
 \end{assumption}
 
 Hereafter, we let $\mathbb E[ \mathcal C(Z_i) ]\leq \mathcal  C_\mu$ for all $i=1,...,n$ for some $\mathcal  C_\mu\geq 1$.

\begin{remark}
Assumptions \ref{sub exponential condition} and \ref{Lipschitz condition} are general enough to cover a wide spectrum of  M-estimation problems. More specifically, Assumption \ref{sub exponential condition}   requires   that the underlying distribution is sub-exponential, and Assumption \ref{Lipschitz condition} essentially imposes the Lipschitz(-like) continuity on $\mathcal L_n(\,\cdot\,,\mathbf Z_1^n)$. Examples of sub-exponential distributions include uniform, Gaussian, exponential,  and $\chi^2$ distributions, as well as any distribution that has a bounded support set.   {
\label{remark 2}\Copy{to copy remark 2}{As for the Lipschitz continuity, it is a condition satisfied by many statistical learning problems, such as linear regression, Huber regression, SVM, and NNs. We are to show that the generalization error bounds only grow logarithmically in the Lipschitz constant.}} The combination of our Assumptions is  non-trivially weaker than the settings in \cite{liu2017a,liu2018a}. 
 It is also worth mentioning that the stipulations of $\sigma\geq 1$, $\mathcal  C_\mu\geq 1$, and $ \sigma_L\geq 1$ can be easily relaxed and are needed only for notational simplicity in   presenting our results.
\end{remark}

\section{Significant subspace second-order necessary conditions}\label{3}
 
Because the FCP is nonconvex, so is Eq. \eqref{general formulation}. Thus, computing the global solution to \eqref{general formulation} is intractable. Nonetheless, our theories concern only local stationary points. We show that these local solutions  are good enough to ensure the promised statistical performance. 

In particular, we consider the stationary points that are characterized by the satisfaction of the significant subspace second-order necessary conditions (S$^3$ONC), which are closely similar to the necessary conditions discussed by \cite{chen2010a} for   linear regression with bridge regularization and by \cite{liu2017a,liu2018a}   under the assumption that the empirical risk function is everywhere twice differentiable. This paper generalizes the  characterizations of the S$^3$ONC    to scenarios where the twice-differentiability may not hold everywhere. 
 
 \begin{definition}\label{SONC definition}
Given $\mathbf Z_1^n\in\mathcal W^n$, a vector $\widehat{\boldsymbol\beta}\in\Re^p$ is said to satisfy the S$^3$ONC  (denoted by {\it S$^3$ONC}$(\mathbf Z_1^n)$) of Problem \eqref{general formulation} if both of the following  sets of conditions are satisfied:

\begin{enumerate} 
\item[a.] {\label{first order kkt def}\Copy{the first-order KKT copy}{The first-order KKT conditions are met at $\widehat{\boldsymbol\beta}:=(\widehat\beta_j)$; that is, there exists $\varkappa_j\in \partial(\vert \widehat\beta_j \vert)$, for all $j=1,...,p$, such that
\begin{align}
\nabla\mathcal L_{n}(\widehat{\boldsymbol\beta},\, \mathbf Z_1^n)+(P'_\lambda(\vert\widehat\beta_j\vert)\cdot\varkappa_j:\,j=1,...p)=\mathbf 0,\label{first-order kkt}
\end{align}
where $\nabla\mathcal L_{n}(\widehat{\boldsymbol\beta},\, \mathbf Z_1^n)$ is the gradient of $\mathcal L_{n}(\,\cdot,\, \mathbf Z_1^n)$ as defined in \eqref{test new results},  $\partial(\vert \widehat\beta_j \vert)$ is the subdifferential  of $\vert \,\cdot\, \vert$ at $\widehat\beta_j$, and $P'_\lambda(\,\cdot\,)$ is the first derivative of $P_\lambda(\,\cdot\,)$. 
}}
\item[b.] The following inequality holds at $\widehat{\boldsymbol\beta}$: for all $j=1,...,p$, if $\vert\widehat\beta_j\vert\in(0,\,a\lambda)$, then
\begin{align}
U_L+P''_\lambda(\vert\widehat\beta_j\vert)\geq 0.\label{second condition 1}
\end{align}
where $P''_\lambda$ is the second derivative of $P_\lambda(\,\cdot\,)$, the quantity   $U_L$ is defined as in \eqref{Lipschitz gradient condition value new value}, and $a$ and $\lambda$ are (hyper-)parameters of the FCP as in \eqref{FCP penalty formulation}. 
\end{enumerate}
\end{definition}
 %
It is   worth noting that the S$^3$ONC is  verifiably implied by the conventional second-order KKT conditions when they are well-defined. We  show in Section \ref{solution section}  that an S$^3$ONC solution (i.e., a solution that satisfies the S$^3$ONC) can be computed by the proposed gradient-based method  at  pseudo-polynomial-time complexity.

\section{Statistical performance bounds}\label{sec: main results}

This section presents the promised sample complexity results for  a generic HDSL problem under A-sparsity. More specifically, Proposition \ref{second theorem} shows the most general result of this paper.  In that proposition, a hyper-parameter $\varrho$ is left to be  determined in different special cases. One of those cases is then presented in Theorem \ref{second theorem theorem}. For convenience, we adopt a short-hand notation as follows:  $\widetilde\zeta:=\ln\left( 3eR\cdot(\sigma_L+\mathcal  C_\mu)\right)$.

\begin{newproposition}\label{second theorem} Suppose that Assumptions \ref{Assumption A-sparsity}, \ref{sub exponential condition}, and  \ref{Lipschitz condition} hold. For any $\varrho:\,0<\varrho< \frac{1}{2}$ and the same $c$ in \eqref{Bernstein inequality result}, let    $a<\frac{1}{U_L}$ and  $\lambda:=\sqrt{\frac{8\sigma }{c\cdot a\cdot n^{2\varrho}}[\ln(n^{\varrho}p)+\widetilde\zeta]}$.      Consider any random vector $\widehat{\boldsymbol\beta}\in\Re^p$ such that $\Vert \widehat{\boldsymbol\beta}\Vert_\infty\leq R$ and  the S$^3$ONC$(\mathbf Z_1^n)$   to \eqref{general formulation}  is satisfied at   $\widehat{\boldsymbol\beta}$ almost surely. The following statements hold:
\begin{itemize}
\item[(i)]
For any fixed $\Gamma\geq 0$ and some    universal constant $C_1>0$, if
\begin{align}
n > C_1\cdot \left[\left(\frac{\Gamma+\varepsilon_{A}}{\sigma }\right)^{\frac{1}{1-2\varrho}}+s\cdot\left(\ln(n^{\varrho}p)+\widetilde \zeta\right)\right],\label{sample initial requirement 2 ori}
\end{align}
 and   $\mathcal L_{n,\lambda}(\widehat{\boldsymbol\beta},\,\mathbf Z_1^n)\leq \mathcal L_{n,\lambda}({\boldsymbol\beta}^{*}_{\varepsilon_{A}},\,\mathbf Z_1^n)+\Gamma$ almost surely, then  
\begin{multline}
\mathbb L(\widehat{\boldsymbol\beta})-L_g^*
\leq C_1\cdot\left(\frac{s\cdot\left(\ln(n^{\varrho}p)+\widetilde \zeta\right)}{n^{2\varrho}}+\sqrt{\frac{s\cdot\left(\ln(n^{\varrho}p)+\widetilde \zeta\right)}{n}}+\frac{1}{n^{\varrho}}+\frac{1}{n^{1-2\varrho}}+\frac{1}{n^{(1-\varrho)/2}}\right)\cdot \sigma \\+C_1\cdot\sqrt{\frac{\sigma (\Gamma+\varepsilon_{A})}{n^{1-2\varrho}}}+\Gamma+\varepsilon_{A},\label{bound global}
\end{multline}
with probability at least 
$
1-2 (p+1) \exp(- n/C_1)-6\exp\left(-2cn^{4\varrho-1}\right) 
$, where $\mathbb L$ is defined in Eq.\,\eqref{population-level model M-estimation} and $L_g^*$ is defined in Assumption \ref{Assumption A-sparsity}.
\item[(ii)]
For almost every $\mathbf Z_1^n\in\mathcal W^n$, assume that the minimization problem  in \eqref{Lasso problem} admits a finite optimal solution denoted by $\widehat{\boldsymbol\beta}^{\ell_1}:=\widehat{\boldsymbol\beta}^{\ell_1}(\mathbf Z_1^n)$. For  some    universal constant $C_2>0$, if
\begin{align}
n > C_2\cdot\left(\frac{\varepsilon_{A}}{\sigma }\right)^{\frac{1}{1-2\varrho}}+ C_2\cdot a^{-1}\cdot [\ln(n^{\varrho}p)+\widetilde\zeta]\cdot s^{\max\{1,\frac{1}{2-4\varrho},\,\frac{1}{2\varrho}\}}\left(\max\left\{1,\,\Vert \boldsymbol\beta^*_{\varepsilon_{A}}\Vert_\infty\right\}\right)^{\max\{\frac{1}{2-4\varrho},\,\frac{1}{2\varrho}\}},\label{sample initial requirement}
\end{align}
 and   $\mathcal L_{n,\lambda}(\widehat{\boldsymbol\beta},\,\mathbf Z_1^n)\leq \mathcal L_{n,\lambda}(\widehat{\boldsymbol\beta}^{\ell_1},\,\mathbf Z_1^n)$ almost surely, then  
\begin{multline}
\mathbb L(\widehat{\boldsymbol\beta})-L_g^* 
\leq C_2\cdot\left[\frac{s\left(\ln(n^{\varrho}p)+\widetilde \zeta\right)}{n^{2\varrho}}+\frac{1}{n^{\varrho}}+\frac{1}{n^{1-2\varrho}}\right]
\cdot \sigma 
\\+C_2\cdot \frac{s\cdot \max\left\{1,\,\Vert \boldsymbol\beta^*_{\varepsilon_{A}}\Vert_\infty\right\}\cdot \sigma ^{3/4}}{\min\left\{a^{1/2}n^\varrho,\,a^{1/4}n^{\frac{1-\varrho}{2}}\right\}} \left[\ln(n^{\varrho}p)+\widetilde \zeta\right]^{1/2}
+C_2\cdot\sqrt{\frac{\sigma \varepsilon_{A}}{n^{1-2\varrho}}}+\varepsilon_{A},\label{bound 12}
\end{multline}
with probability at least 
$
1-2 (p+1) \exp(- n/C_2)-6\exp\left(-2cn^{4\varrho-1}\right)   
$.
\end{itemize}
\end{newproposition}
\begin{proof}{Proof.}
See Section \ref{subsection proof proposition 6}.
\hfill \ensuremath{\Box}\end{proof} 

\begin{remark}Proposition \ref{second theorem} is the most general result in this paper. It does not rely on convexity, RSC, or alike, although to ensure $\mathcal L_{n,\lambda}(\widehat{\boldsymbol\beta},\,\mathbf Z_1^n)\leq \mathcal L_{n,\lambda}(\widehat{\boldsymbol\beta}^{\ell_1},\,\mathbf Z_1^n)$ almost surely in Part (ii) usually requires $\mathcal L_{n,\lambda}(\,\cdot\,,\,\mathbf Z_1^n)$ to be convex. 
\end{remark}
{
\begin{remark}\label{additional remark on R}
\Copy{to copy 3}{The assumption that $\Vert \widehat{\boldsymbol\beta}\Vert_\infty\leq R$  is comparable to, or less restrictive than, some similar conditions in the literature. For example,   \cite{p2017a}  and \cite{loh2015a} require that the estimator is within the set of $\{\boldsymbol\beta:\,\vert{\boldsymbol\beta}\vert\leq R_{\ell_1}\}$. Under the same requirement, we may have $R_{\ell_1}\geq R$. 
 Because the error bounds in \eqref{sample initial requirement 2 ori} and \eqref{bound 12} are logarithmic in $R$ (with $\widetilde\zeta:=\mathcal O\left(\ln R\right)$), one may let the value of $R$ to be a coarse overestimation of $\Vert \widehat{\boldsymbol\beta}\Vert_\infty$.} 
\end{remark}}
\begin{remark}
  Because $\mathbb L(\widehat{\boldsymbol\beta})-\inf_{\boldsymbol\beta}~\mathbb L({\boldsymbol\beta})\leq \mathbb L(\widehat{\boldsymbol\beta})-L_g^*$, the first part of this proposition indicates that, for all the S$^3$ONC solutions, the excess risk can be bounded by a function in the parameterization of the suboptimality gap $\Gamma$. (Technically speaking, $\Gamma$  is an underestimation of the suboptimality gap in this proposition.) This  bound on  the excess risk explicates the consistency  between the statistical performance of a stationary point to an HDSL problem and the optimization quality of that stationary point in minimizing the   objective function of Problem  \eqref{general formulation}.  
The  second part of Proposition \ref{second theorem} concerns an arbitrary S$^3$ONC solution $\widehat{\boldsymbol\beta}$ that has an objective function value smaller than that of $\widehat{\boldsymbol\beta}^{\ell_1}$. The corresponding error bound becomes independent of $\Gamma$. 
\end{remark}
\begin{remark}
To compute  $\widehat{\boldsymbol\beta}$ in Part (ii) of this proposition, we can adopt  a two-step approach: In the first step, we  solve for  $\widehat{\boldsymbol\beta}^{\ell_1}$, which is often polynomial-time computable if $\mathcal L_{n,\lambda}(\,\cdot\,,\mathbf Z_1^n)$ is convex given $\mathbf Z_1^n$. Then, in the second step, we invoke an   S$^3$ONC-guaranteeing algorithm (such as the gradient-based method to be discussed in Section \ref{solution section}). This algorithm should be initialized with $\widehat{\boldsymbol\beta}^{\ell_1}$. 
\end{remark}

\begin{remark}
We may as well let $a^{-1}=2U_L$  to satisfy the stipulation on $a$ in Proposition \ref{second theorem theorem}. Here, $U_L$ can be considered as the largest diagonal   of the Hessian matrix of $\mathcal L(\,\cdot\,,\,z)$, if it exists. In many applications of HDSL, this quantity can satisfy $U_L\leq O(1)\ln p$ with high probability under  data normalization. For example, in the special case of high-dimensional linear models, $U_L\leq 1$ is implied by the common assumption of column normalization \citep{raskutti2011a,n2012a}.
\end{remark}

{
\begin{remark}\label{remark added in response to reviewer}
\Copy{to copy 1}{The proof of Proposition \ref{second theorem} makes use of the coincidence that, at the S$^3$ONC solutions, the FCP behaves similarly as the  $\ell_0$ penalty (as discussed by, e.g., \cite{shen2013constrained}). Thus, it is possible that adopting  the $\ell_0$ penalty instead of the FCP  in our formulation \eqref{general formulation} may lead to similar results on the generalization errors with less technical difficulty.  Nonetheless, the $\ell_0$ penalty introduces discontinuity to the formulation and thus may usually lead to higher computational ramification. We leave for the future research the study of the trade-offs between computational and sample complexities for the formulations with  alternative regularization terms.}
\end{remark}}

\begin{remark}
For any fixed $\varrho:\ 0<\varrho< \frac{1}{2}$, each of the two parts  of  Proposition \ref{second theorem} has already established the poly-logarithmic sample complexity. Based on this proposition, polynomially increasing the sample size can compensate for the exponential growth in the dimensionality. We may further pick a reasonable value for $\varrho$ and obtain more detailed bounds as in Theorem \ref{second theorem theorem} below, which confirms the promised complexity rates as previously mentioned in \eqref{lasso initialized} and \eqref{all solutions} for a general HDSL problem under A-sparsity.
\end{remark}

\begin{theorem}\label{second theorem theorem} Let   $a<\frac{1}{U_L}$  and $\lambda:=\sqrt{\frac{8\sigma }{c\cdot a\cdot n^{2/3}}[\ln(n^{2/3}p)+\widetilde\zeta]}$ for the same $c$ in \eqref{Bernstein inequality result}. Suppose that Assumptions \ref{Assumption A-sparsity original}, \ref{sub exponential condition}, and \ref{Lipschitz condition} hold.    For any random vector $\widehat{\boldsymbol\beta}\in\Re^p$   such that $\Vert \widehat{\boldsymbol\beta}\Vert_\infty\leq R$ and S$^3$ONC$(\mathbf Z_1^n)$   to \eqref{general formulation}  is satisfied at   $\widehat{\boldsymbol\beta}$  almost surely,  the following statements hold: 
\begin{itemize}
\item[(i)]
 For any fixed $\Gamma\geq 0$ and some    universal constant $C_3>0$, if
\begin{align}
n > C_3\cdot \left[\left(\frac{\Gamma+\varepsilon_{A}}{\sigma }\right)^{3}+s\cdot\left(\ln(np)+\widetilde \zeta\right)\right],\label{sample initial requirement 2 3}
\end{align}
 and   $\mathcal L_{n,\lambda}(\widehat{\boldsymbol\beta},\,\mathbf Z_1^n)\leq \mathcal L_{n,\lambda}({\boldsymbol\beta}^{*}_{\varepsilon_{A}},\,\mathbf Z_1^n)+\Gamma$ almost surely, then the  excess risk is bounded by 
\begin{align}
\mathbb L(\widehat{\boldsymbol\beta})-\inf_{\boldsymbol\beta}~\mathbb L(\boldsymbol\beta)
\leq C_3\sigma \cdot\left[\frac{s\cdot \left(\ln(n p)+\widetilde \zeta\right)}{n^{2/3}}+\sqrt{\frac{s\cdot \left(\ln(n p)+\widetilde \zeta\right)}{n}}+\frac{1}{n^{1/3}}\right]   
+C_3\cdot\sqrt{\frac{\sigma (\Gamma+\varepsilon_{A})}{n^{1/3}}}+\Gamma+\varepsilon_{A}\label{bound global 3}
\end{align}

with probability at least 
$
1-2 (p+1) \exp\left(-\frac{n}{C_3}\right)-6\exp\left(-\frac{n^{1/3}}{C_3}\right)$.
\item[(ii)]
For almost every $\mathbf Z_1^n\in\mathcal W^n$, assume that the minimization problem  in \eqref{Lasso problem} admits a finite optimal solution denoted by $\widehat{\boldsymbol\beta}^{\ell_1}:=\widehat{\boldsymbol\beta}^{\ell_1}(\mathbf Z_1^n)$.
For  some    universal constant $C_4>0$, if
\begin{align}
n > C_4\cdot\left(\frac{\varepsilon_{A}}{\sigma }\right)^{3}+ C_4\cdot a^{-1}\cdot [\ln(n p)+\widetilde\zeta]\cdot s^{ \frac{3}{2} }\max\left\{1,\,\Vert\boldsymbol\beta^*_{\varepsilon_{A}}\Vert_\infty^{\frac{3}{2}}\right\},\label{sample initial requirement 2} 
\end{align}
 and   $\mathcal L_{n,\lambda}(\widehat{\boldsymbol\beta},\,\mathbf Z_1^n)\leq \mathcal L_{n,\lambda}(\widehat{\boldsymbol\beta}^{\ell_1},\,\mathbf Z_1^n)$ almost surely, then the excess risk is bounded by 
\begin{align}
\mathbb L(\widehat{\boldsymbol\beta})-\inf_{\boldsymbol\beta}~\mathbb L(\boldsymbol\beta) 
\leq C_4\cdot a^{-1/2}\cdot s\cdot \sigma \cdot\left[\frac{ \left(\ln(n p)+\widetilde \zeta\right)}{n^{\frac{2}{3}}}+\frac{\max\{1,\,\Vert\boldsymbol\beta^*_{\varepsilon_{A}}\Vert_\infty\} \cdot \sqrt{\ln(n p)+\widetilde \zeta}}{n^{\frac{1}{3}}}\right]
+C_4\cdot\sqrt{\frac{\sigma \varepsilon_{A}}{n^{1/3}}}+\varepsilon_{A}\label{bound 123}
\end{align}
with probability at least 
$
1-2 (p+1) \exp\left(-\frac{n}{C_4}\right)-6\exp\left(-\frac{n^{1/3}}{C_4}\right)  $.
 \end{itemize}
\end{theorem}
\begin{proof}{Proof.}
Invoking   Proposition \ref{second theorem} with    $\varrho=\frac{1}{3}$ and noticing that Assumption \ref{Assumption A-sparsity original} implies  Assumption \ref{Assumption A-sparsity} with $L_g^*:=\inf_{\boldsymbol\beta}~\mathbb L(\boldsymbol\beta)$, we obtain both parts of the desired results.
\hfill \ensuremath{\Box}\end{proof} 

Theorem \ref{second theorem theorem} ensures the desired poly-logarithmic sample complexity for HDSL under A-sparsity. Our remarks  concerning Proposition \ref{second theorem} above also apply to Theorem \ref{second theorem theorem}, since the latter is a special case   when   $\varrho=\frac{1}{3}$ and $L_g:=\inf_{\boldsymbol\beta}~\mathbb L({\boldsymbol\beta})$. We would like to point out that, if $\varepsilon_{A}=0$, then  A-sparsity is reduced to the conventional sparsity. In such a case, the excess risk in \eqref{bound 123} is simplified into $\mathbb L(\widehat{\boldsymbol\beta})-\inf_{\boldsymbol\beta}~\mathbb L({\boldsymbol\beta}) \leq {\mathcal O}(\frac{{\ln p}}{n^{2/3}}+\frac{\sqrt{\ln p}}{n^{1/3}})$.

\section{An S$^3$ONC-Guaranteeing Algorithm}\label{solution section} 

This section  presents a pseudo-polynomial-time  S$^3$ONC-guaranteeing algorithm. For convenience, we consider a slightly more abstract  optimization problem than \eqref{general formulation} as below:
\begin{align}
\begin{split}
\min_{\boldsymbol\beta:=(\beta_j)\,\in\,\Re^p}\,\widetilde f_\lambda(\boldsymbol\beta):=\widetilde f(\boldsymbol\beta)+\sum_{j=1}^p P_\lambda(\vert\beta_j\vert).
\end{split}\label{special problem here 1}
\end{align}
where $\widetilde f:\,\Re^p\rightarrow\Re$ is a continuously differentiable function with $\Vert \nabla \widetilde f(\boldsymbol\beta_1)-\nabla \widetilde f(\boldsymbol\beta_2)\Vert\leq \widetilde U_{L,2}\cdot \Vert \boldsymbol\beta_1-\boldsymbol\beta_2\Vert$ for some  $\widetilde U_{L,2}\geq 1$ and all $\boldsymbol\beta_1,\,\boldsymbol\beta_2\in\Re^p$. Consequently, the partial derivative $\frac{\partial \widetilde f (\boldsymbol\beta)}{\partial\beta_j}$, for all $j=1,...,p$, is also globally Lipschitz continuous in the sense that $\left\vert\left[\frac{\partial \widetilde f (\boldsymbol\beta)}{\partial\beta_j}\right]_{\boldsymbol\beta=\widetilde{\boldsymbol\beta}+\delta \cdot e_j}-\left[\frac{\partial \widetilde f (\boldsymbol\beta)}{\partial\beta_j}\right]_{\boldsymbol\beta=\widetilde{\boldsymbol\beta}}\right\vert\leq  \widetilde U_{L,\infty}\cdot\vert \delta \vert$ for  every $\widetilde{\boldsymbol\beta}\in\Re^p$, any $\delta\in\Re$,  and some $1\leq U_{L,\infty}\leq U_{L,2}$.  (Note that $U_L$ in \eqref{Lipschitz gradient condition value new value} becomes $\widetilde U_{L,\infty}$ here.) The pseudo-code of the proposed algorithm is summarized below.
\vspace{4mm}
\hrule
\vspace{1mm}
{\bf\noindent Algorithm 1. An S$^3$ONC-guaranteeing gradient-based algorithm}
\vspace{1mm}
\hrule
\vspace{2mm}
\begin{description}
\item[Step 1.] Fix parameters $\gamma_{opt},\,\mathcal M,\,\lambda,$ and $a$ such that  $a<\mathcal M^{-1}$. Initialize $k=0$ and  $\boldsymbol\beta^0\in\Re^p$.
\item[Step 2.] Compute $\boldsymbol\beta^{k+\frac{1}{2}}$ by solving the following problem
\begin{align}
\boldsymbol\beta^{k+\frac{1}{2}}\in \underset{\boldsymbol\beta}{\arg\,\min}\,\left\langle \nabla \widetilde f(\boldsymbol\beta^k),\,\boldsymbol\beta-\boldsymbol\beta^k \right\rangle+\frac{{\mathcal M}}{2}\Vert \boldsymbol\beta-\boldsymbol\beta^k \Vert^2+\sum_{j=1}^pP_\lambda'(\vert \beta_j^k\vert)\cdot\vert\beta_j\vert.\label{first subproblem}
\end{align}
\item[Step 3.] Compute $\boldsymbol\beta^{k+1}$ by solving  the following problem
\begin{align}
\boldsymbol\beta^{k+1}\in \underset{\boldsymbol\beta}{\arg\,\min}\,\left\langle \nabla \widetilde{f}(\boldsymbol\beta^{k+\frac{1}{2}}),\,\boldsymbol\beta-\boldsymbol\beta^{k+\frac{1}{2}} \right\rangle+\frac{{\mathcal M}}{2}\Vert \boldsymbol\beta-\boldsymbol\beta^{k+\frac{1}{2}} \Vert^2+\sum_{j=1}^pP_\lambda(\vert \beta_j\vert).\label{second subproblem}
\end{align}

\item[Step 4.] Algorithm terminates and outputs $\boldsymbol\beta^k$ if the stopping criteria are met. Otherwise, let $k:=k+1$ and go to Step 2.
\end{description}
\vspace{2mm}
\hrule
\vspace{10mm}

We design the termination criterion to be that the algorithm stops when the below is satisfied for the first time\begin{align}\widetilde{f}_\lambda(\boldsymbol\beta^{k+1})>\widetilde{f}_\lambda(\boldsymbol\beta^{k})-\frac{\gamma_{opt}^2}{{2\mathcal M}},\label{termination criterion Algorithm 1}
\end{align}
where $\mathcal M>0$ and $\gamma_{opt}>0$ are    specified in Step 1 of Algorithm 1. Intuitively, $\mathcal M^{-1}$ can be interpreted as the step size of the algorithm, and $\gamma_{opt}$, as the error tolerance in approximating the S$^3$ONC. At termination, the iteration count is denoted by $k^*$.

{\color{black}\Copy{to our analysis}{To our analysis, Algorithm 1 relies on  solving  two per-iteration subproblems \eqref{first subproblem} and \eqref{second subproblem}, repetitively. Subproblem \eqref{first subproblem} in Step 2 ensures that a non-trivial reduction in the objective function value can be achieved whenever the first-order KKT conditions are not met. This step is essential to the promised $\mathcal O(1/\gamma_{opt}^2)$-rate  of the algorithm. Meanwhile, the presence of Subproblem \eqref{second subproblem}  in Step 3  leads to a solution sequence that approaches a desired S$^3$ONC solution without affecting the convergence rate.}\label{to our analysis ref}}  We may formalize the above analysis to prove the theorem below  on  the iteration complexity of Algorithm 1 in computing an S$^3$ONC solution. 
\begin{theorem}\label{Algorithm complexity theorem}
Suppose that $\widetilde f_{\lambda}^*:=\inf_{\boldsymbol\beta}\widetilde f_{\lambda}(\boldsymbol\beta)>-\infty$, $\mathcal M\geq \widetilde U_{L,2}$, and $a<\frac{1}{\mathcal M}$. For any $\gamma_{opt}:\,0<\gamma_{opt}<a\lambda\cdot\mathcal M$, the following statements  hold true:
\begin{itemize}
\item[(a)] Algorithm 2 terminates at iteration $ k^*\leq \left\lfloor2\mathcal M\cdot \frac{\widetilde f_{\lambda}(\boldsymbol\beta^{0})-\widetilde f_{\lambda}^*}{\gamma_{opt}^2}\right\rfloor+1.$   
\item[(b)] At termination, $\boldsymbol\beta^{k^*}=(\beta^{k^*}_j)$ is a $\gamma_{opt}$-S$^3$ONC solution to \eqref{special problem here 1}; that is, 
there exists $\varkappa_j\in \partial(\vert \beta^{k^*}_j\vert)$, for all $j=1,...,p$, such that
\begin{align}
\left\Vert\nabla{\widetilde f}(\boldsymbol\beta^{k^*})+(P'_\lambda(\vert \beta^{k^*}_j\vert)\cdot\varkappa_j:\,j=1,...p)\right\Vert\leq \gamma_{opt},\label{desired first order conditions}
\end{align}
and, for all $j=1,...,p$, if $\vert \beta^{k^*}_j\vert\in(0,\,a\lambda)$, then
$
\widetilde U_{L,\infty}+P''_\lambda(\vert\beta^{k^*}_j\vert)\geq 0,$
where  $a$ and $\lambda$ are defined in \eqref{FCP penalty formulation}.  
\item[(c)] $\widetilde f_\lambda(\boldsymbol\beta^{k^*})\leq \widetilde f_\lambda(\boldsymbol\beta^{0})$.
\item[(d)] $\beta_j^k\notin(0,a\lambda)$ for all $k=1,...,k^*$, where $\beta_j^k$ is the $j$th entry of $\boldsymbol\beta^k$.
\end{itemize}
\end{theorem}
\begin{proof}{Proof.} See proof in Section \ref{Algorithm complexity proof}
\hfill \ensuremath{\Box}\end{proof} 

\begin{remark}
We would like to make a few remarks on Theorem \ref{Algorithm complexity theorem} in the following.
\begin{itemize}
\item The assumptions of this theorem include  the stipulation of $a<\frac{1}{\mathcal M}$, which is consistent with the requirement on $a$ in the generalizability results in the previous section. More specifically,    we may let $a<\min\{\widetilde U_{L,\infty}^{-1},\,\mathcal M^{-1}\}$ to satisfy the conditions for both  Theorem \ref{Algorithm complexity theorem} and Proposition \ref{second theorem}, simultaneously. This observation can be generalized to almost all of our main sample complexity results. Another important assumption we have made is that $\widetilde f$ is smooth;  that is,  $\nabla \widetilde f$  is (globally) Lipschitz continuous. While many machine learning problems satisfy such  a condition, it  is violated by a nonsmooth HDSL problem and a ReLU-NN.  Nonetheless, as we  show in Section \ref{sec theoretical app}, the nonsmooth learning problems, including the SVM, can be analyzed through a smooth approximation.  As for a ReLU-NN, we demonstrate that Algorithm 1 can still be effective with the aid of a tractable initialization scheme.
 \item From Part (b) of the result, the   $\gamma_{opt}$-S$^3$ONC solution is an $\gamma_{opt}$-approximation to the  S$^3$ONC as in Definition \ref{SONC definition}, if we let $\mathcal L_{n}(\,\cdot\,,\, \mathbf Z_1^n):=\widetilde{f}(\,\cdot\,)$. One may see that \eqref{desired first order conditions} is a $\gamma_{opt}$-approximation to   the first-order KKT conditions   in \eqref{first-order kkt}. Meanwhile, the second set of conditions in \eqref{second condition 1} are met exactly.

\item It is easy to re-organize the results from Parts (a) and (b) of Theorem \ref{Algorithm complexity theorem} to see that the algorithm runs for $\mathcal O({\gamma_{opt}^{-2}})$-many iterations  to generate an $\gamma_{opt}$-S$^3$ONC solution.  This iteration complexity is   polynomial in the problem dimensionality and the numeric value of the problem data input. 
Since the per-iteration problems   admit closed forms, we can then see that Algorithm 2 is among the class of pseudo-polynomial-time algorithms. It is worth noting that many existing alternatives are more generic and can compute   stronger necessary conditions than the S$^3$ONC. Nonetheless, the new algorithm can still be of independent interest. Compared to $\mathcal O({\gamma_{opt}^{-3}})$, the best-known  rate to ensure an   $\gamma_{opt}$-approximation to the second-order necessary conditions in the literature, our proposed gradient-based method  yields a significantly better  computational complexity.  

\item Part (c) indicates that the output of the algorithm is no worse than the initial solution in terms of minimizing the objective function $\widetilde f_\lambda$. This property ensures conditions like $\mathcal L_{n,\lambda}(\widehat{\boldsymbol\beta},\,\mathbf Z_1^n)\leq \mathcal L_{n,\lambda}(\widehat{\boldsymbol\beta}^{\ell_1},\,\mathbf Z_1^n)$ in the sample complexity results in, e.g., Part (ii) of Theorem \ref{second theorem theorem}, if   Algorithm 1   is initialized with $\widehat{\boldsymbol\beta}^{\ell_1}$. 

\item Part (d) is useful for our subsequent analysis. One may verify that the proof of this part  holds even if $\widetilde f(\,\cdot\,)$ is not continuously differentiable.
\end{itemize}

\end{remark}

 We observe that both the per-iteration problems \eqref{first subproblem} and \eqref{second subproblem} admit closed-form solutions. To see this, we note that \eqref{first subproblem} is essentially a soft thresholding problem, whose closed form is well-known. As for \eqref{second subproblem}, we observe that it can be decomposed into $p$-many one-dimensional problems. Enumerating all the   KKT solutions to each of these decomposed problems and noticing that $a<\mathcal M^{-1}$, one may verify that, for all $j=1,...,p$,
\begin{align}\beta_j^{k+1}=
\begin{cases}
\beta_j^{k+\frac{1}{2}}-\frac{1}{\mathcal M}\cdot \left[\frac{\partial \widetilde f(\boldsymbol\beta)}{\partial\beta_j}\right]_{\boldsymbol\beta=\boldsymbol\beta^{k+\frac{1}{2}}}&\text{if $\left\vert\beta_j^{k+\frac{1}{2}}-\frac{1}{\mathcal M}\cdot \left[\frac{\partial \widetilde f(\boldsymbol\beta)}{\partial\beta_j}\right]_{\boldsymbol\beta=\boldsymbol\beta^{k+\frac{1}{2}}}\right\vert\geq a\lambda$};\\
0&\text{otherwise}.
\end{cases}\nonumber
\end{align}

\section{Theoretical Applications}\label{sec theoretical app}
In this section, we discuss two important theoretical applications of Proposition \ref{second theorem} and Theorem \ref{second theorem theorem}. Section \ref{nonsmooth learning} presents our results for a flexible class of  nonsmooth HDSL problems. Section \ref{NN training} then considers the generalizability of an FCP-regularized (deep) NN.

 \subsection{Nonsmooth HDSL under A-sparsity}\label{nonsmooth learning}
The   nonsmooth HDSL  problem of our consideration is formulated as below:
\begin{align}\min_{\boldsymbol\beta}\frac{1}{n}\sum_{i=1}^n\left[L_{ns}(\boldsymbol\beta,Z_i):=f_1(\boldsymbol\beta,Z_i)+\max_{\mathbf u\in\mathbb U} \left\{\mathbf u^\top\mathbf A(Z_i) \boldsymbol\beta-\phi(\mathbf u,\,Z_i)\right\}\right],\label{original nonsmooth}
\end{align}
 where  $\mathbf A(\cdot):\,\mathcal W\rightarrow\Re^{m\times p}$ is deterministic and measurable (and may be nonlinear in ``$\cdot$''),  $\mathbb U\subseteq \Re^{m}$ is a convex and compact set with a diameter $D:=\max\{\Vert \mathbf u_1-\mathbf u_2\Vert:\,\mathbf u_1,\mathbf u_2\in\mathbb U\}$, and $f_1:\,\Re^p\times \mathcal W\rightarrow \Re$  and $\phi:\,\mathbb U\times \mathcal W\rightarrow \Re$ are   deterministic, measurable functions. Let $f_1(\,\cdot\,,\,z)$ be  continuously differentiable with
$\left\vert\left[\frac{\partial  f_1(\boldsymbol\beta,z)}{\partial\beta_j}\right]_{\boldsymbol\beta=\widetilde{\boldsymbol\beta}+\delta \cdot e_j}-\left[\frac{\partial f_1(\boldsymbol\beta,z)}{\partial\beta_j}\right]_{\boldsymbol\beta=\widetilde{\boldsymbol\beta}}\right\vert\leq  U_{f_1}\cdot\vert \delta \vert$ for almost every $z\in\mathcal W$ and for all $\widetilde{\boldsymbol\beta}\in\Re^p$, $\delta\in\Re$, and $j=1,...,p$.  Let  $\phi(\,\cdot\,, z)$ be convex and continuous for almost every $z\in\mathcal W$. As some standard and non-critical regularity conditions, it is assumed that  $\mathbb E\left[n^{-1}\sum_{i=1}^nL_{ns}(\boldsymbol\beta,Z_i)\right]$ is well-defined for all $\boldsymbol\beta\in\Re^p$ with $\inf_{\boldsymbol\beta}\mathbb E\left[n^{-1}\sum_{i=1}^nL_{ns}(\boldsymbol\beta,Z_i)\right]>-\infty$ and there exists some vector  $\boldsymbol\beta^*_{\varepsilon_A'}\in\Re^p:\, \Vert\boldsymbol\beta^*_{\varepsilon_A'}\Vert_\infty\leq R$, such that $\mathbb E\left[n^{-1}\sum_{i=1}^nL_{ns}(\boldsymbol\beta_{\varepsilon_A'},Z_i)\right]-\underset{\boldsymbol\beta}{\inf}\,\,\mathbb E\left[n^{-1}\sum_{i=1}^nL_{ns}(\boldsymbol\beta,Z_i)\right]\leq \varepsilon_A'$ for some $\varepsilon_A'\geq 0$.  In the foregoing settings, A-sparsity    (in the sense of Assumption \ref{Assumption A-sparsity original}) holds with $\varepsilon_A:=\varepsilon_A'$ and we are again interested in estimating the vector of true parameters $\boldsymbol\beta^*\in\arg\inf_{\boldsymbol\beta}\mathbb E\left[n^{-1}\sum_{i=1}^nL_{ns}(\boldsymbol\beta,Z_i)\right]$.
Such a problem  is general enough to cover some important nonsmooth learning problems, such as  the  least quantile linear regression, the   least absolute deviation regression, and the SVM.

Compared to our results in Section \ref{sec: main results}, a nuance here is that Problem \eqref{original nonsmooth} has an  empirical risk function that is not everywhere differentiable due to the presence of a maximum operator. The non-differentiable point may reside anywhere, such as at, or in some near neighborhood of, the vector of true parameters.
 In view of this subtlety, we propose the following FCP-based formulation.
\begin{align}
\min_{\boldsymbol\beta}\left[\widetilde{\mathcal L}_{n,\delta,\lambda}(\boldsymbol\beta,\mathbf Z_1^n):=\frac{1}{n}\sum_{i=1}^nf_1(\boldsymbol\beta,Z_i)+\sum_{i=1}^n\frac{1}{n}\max_{\mathbf u\in \mathbb U}\left\{\mathbf u^\top\mathbf A(Z_i) \boldsymbol\beta-\phi(\mathbf u,\,Z_i)-\frac{\Vert\mathbf u-\mathbf u_0\Vert^2}{2n^{\delta}}\right\}+\sum_{j=1}^pP_\lambda(\vert\beta_j\vert)\right],\label{delta function}
\end{align}
for a user-specific $\mathbf u^0\in \mathbb U$ and  $\delta>0$ (which is chosen to be $\delta=\frac{1}{4}$ later    in our theory).

Note that  the proposed formulation in \eqref{delta function} is not an immediate instantiation of \eqref{general formulation} for the population-level problem $\underset{\boldsymbol\beta}{\inf}\,\,\mathbb E\left[n^{-1}\sum_{i=1}^nL_{ns}(\boldsymbol\beta,Z_i)\right]$. Indeed, apart from the FCP-based   regularization term, an additional quadratic function $-\frac{\Vert\mathbf u-\mathbf u_0\Vert^2}{2n^{\delta}}$ is also included in   \eqref{delta function}. The purpose of this extra term is to add  regularities in order to facilitate our analysis; although $\widetilde{\mathcal L}_{n}(\boldsymbol\beta,\mathbf Z_1^n):=\frac{1}{n}\sum_{i=1}^n L_{ns}(\boldsymbol\beta,Z_i)$ is not everywhere differentiable,
\begin{align}\widetilde{\mathcal L}_{n,\delta}(\boldsymbol\beta,\mathbf Z_1^n):=\frac{1}{n}\sum_{i=1}^nf_1(\boldsymbol\beta,Z_i)+\sum_{i=1}^n\frac{1}{n}\max_{\mathbf u\in \mathbb U}\left\{\mathbf u^\top\mathbf A(Z_i) \boldsymbol\beta-\phi(\mathbf u,\,Z_i)-\frac{\Vert\mathbf u-\mathbf u_0\Vert^2}{2n^{\delta}}\right\},\label{smoothed nonsmooth problem}
\end{align}  is verifiably a continuously differentiable  approximation to $\widetilde{\mathcal L}_{n}(\boldsymbol\beta,\mathbf Z_1^n)$. The error incurred by this approximation can be controlled by properly determining the hyper-parameter $\delta$. Furthermore, invoking  Theorem 1  by \cite{nesterov2005a} (restated as Theorem \ref{lipschitz} for completeness), one may derive the  Lipschitz constant of the gradient of $\widetilde{\mathcal L}_{n,\delta}(\,\cdot\,,\mathbf Z_1^n)$. This observation is formalized in Part (a) of Theorem \ref{nonsmooth case} below.

With this approximation, the nonsmooth HDSL problem can now be analyzed via the framework of HDSL under A-sparsity; we can consider the approximation error as a composite of $\varepsilon_A$ in the definition of A-sparsity.  Via this perspective, we may easily  apply  results from Proposition \ref{second theorem} or  Theorem \ref{second theorem theorem}  to \eqref{smoothed nonsmooth problem} after   some  conversions of the settings.
In doing so, we    impose  the following two assumptions, which are   instantiations of Assumptions \ref{sub exponential condition} and \ref{Lipschitz condition}, respectively:
\begin{assumption}\label{SVM assumption 1}
For all  $\boldsymbol\beta\in\Re^p:\,\Vert\boldsymbol\beta\Vert_\infty\leq R$ and $i=1,...,n$,  it holds that $
\Vert L_{ns}(\boldsymbol\beta,\, Z_i)-\mathbb E[ L_{ns}(\boldsymbol\beta,\, Z_i)]\Vert_{\psi_1}\leq \sigma ,
$
for some $\sigma \geq 1$.
\end{assumption}

\begin{assumption}\label{SVM Lipschitz problem}
For some measurable and deterministic function $\mathcal C:\, \mathcal W\rightarrow \Re_+$, the random variable $\mathcal C(Z_i)$ satisfies that 
\begin{itemize}
\item[(i)]
$
\left\Vert \mathcal C(Z_i)-\mathbb E\left[\mathcal C(Z_i)\right]\right\Vert_{\psi_1}\leq \sigma_L, $
 for all $i=1,...,n$ for some $\sigma_L\geq 1$, and
 \item[(ii)] $\mathbb E[ \mathcal C(Z_i)]\leq \mathcal  C_\mu$ for all $i=1,...,n$ for some $\mathcal  C_\mu\geq 1$. 
 \end{itemize}
 Furthermore,
$\vert L_{ns}(\boldsymbol\beta_1,\, z)-L_{ns}(\boldsymbol\beta_2,\, z)\vert\leq \mathcal C(z)\Vert \boldsymbol\beta_1-\boldsymbol\beta_2\Vert,$
 for all $\boldsymbol\beta_1,\,\boldsymbol\beta_2\in\Re^p \cap\{\boldsymbol\beta:\,\Vert\boldsymbol\beta\Vert_\infty\leq R\}$  and almost every  $z\in\mathcal W$.
\end{assumption}
\begin{remark}
Similar to Assumptions \ref{sub exponential condition} and \ref{Lipschitz condition}, the foregoing two conditions ensure that the underlying distribution is subexponential and that a Lipschitz-like inequality holds for $L_{ns}(\,\cdot\,,z)$. 
\end{remark}

 We are now ready to present our results on nonsmooth HDSL in the following theorem, which leads to what is   claimed  in Eq.\,\eqref{Nonsmooth rate}. Similar to Section  \ref{sec: main results}, we  adopt the short-hand, $\widetilde\zeta:=\ln\left( 3eR\cdot(\sigma_L+\mathcal  C_\mu)\right)$.

\begin{theorem}\label{nonsmooth case} Suppose that    $\Vert \mathbf A(z)\Vert^2_{1,2}\leq U_A$ for some $U_A\geq 0$ and for almost every $z\in\mathcal W$. Let {Assumptions} \ref{Assumption A-sparsity original}, \ref{SVM assumption 1}, and \ref{SVM Lipschitz problem}  hold (where $\varepsilon_A$ and $\mathbb L(\,\cdot\,)$ from Assumption \ref{Assumption A-sparsity original} become $\varepsilon_A'$ and $\mathbb E[L_{ns}(\,\cdot\,,\,\mathcal Z)]$, respectively). The following statements hold:
\begin{itemize}
\item[(a)] For any $\delta>0$, all $j=1,...,p$, every $\widetilde{\boldsymbol\beta}\in\Re^p$, and almost every $\mathbf Z_1^n\in\mathcal W^n$, the partial derivative $\frac{\partial \widetilde{\mathcal L}_{n,\delta}(\widetilde{\boldsymbol\beta},\, \mathbf Z_1^n)}{\partial\beta_j}$  is well-defined and  Lipschitz continuous with $\left\vert\left[\frac{\partial   \widetilde{\mathcal L}_{n,\delta}( {\boldsymbol\beta},\, \mathbf Z_1^n)}{\partial\beta_j}\right]_{\boldsymbol\beta=\widetilde{\boldsymbol\beta}+h \cdot e_j}-\left[\frac{\partial   \widetilde{\mathcal L}_{n,\delta}({\boldsymbol\beta},\, \mathbf Z_1^n)}{\partial\beta_j}\right]_{\boldsymbol\beta=\widetilde{\boldsymbol\beta}}\right\vert\leq  (U_{f_1}+n^{\delta}U_A)\cdot\vert h \vert$ for any $h\in\Re$.
\item[(b)] Let $\delta=\frac{1}{4}$, $a=\frac{1}{2(U_{f_1}+n^{1/4}U_A)}$, and $\lambda:=\sqrt{\frac{8\sigma }{c\cdot a\cdot n^{3/8}}[\ln(n^{\frac{3}{8}}p)+\widetilde\zeta]}$ for the same $c$ in \eqref{Bernstein inequality result}.  For almost every $\mathbf Z_1^n\in\mathcal W^n$, assume that the minimization problem $\min_{\boldsymbol\beta}\widetilde{\mathcal L}_{n,\delta}(\boldsymbol\beta,\mathbf Z_1^n)+\lambda\vert\boldsymbol\beta\vert$ admits a finite optimal solution denoted by $\widehat{\boldsymbol\beta}^{\ell_1,\delta}:=\widehat{\boldsymbol\beta}^{\ell_1,\delta}(\mathbf Z_1^n)$. Consider any random vector  $\widehat{\boldsymbol\beta}\in\Re^p$ such that  $\Vert \widehat{\boldsymbol\beta}\Vert_\infty\leq R$, 
$
\widetilde{\mathcal L}_{n,\delta,\lambda}(\widehat{\boldsymbol\beta},\mathbf Z_1^n) \leq \widetilde{\mathcal L}_{n,\delta,\lambda}(\widehat{\boldsymbol\beta}^{\ell_1,\delta},\mathbf Z_1^n)
$
almost surely, and  $\widehat{\boldsymbol\beta}$ satisfies the S$^3$ONC$(\mathbf Z_1^n)$     to \eqref{delta function}
 w.p.1.  For some universal constant $C_5>0$, if
\begin{align}
n >  C_5\cdot\frac{D^4}{\sigma ^2}+C_5\cdot\frac{(\varepsilon_{A}')^4}{\sigma^4}+ C_5\cdot(U_{f_1}+U_A)^{4/3}\cdot [\ln(np)+\widetilde\zeta]^{4/3}\cdot s^{8/3}\max\left\{1, \, \Vert\boldsymbol\beta^*_{\varepsilon_{A}}\Vert_\infty^{8/3}\right\},\label{sample initial requirement nonsmooth}
\end{align}
where $D:=\max\{\Vert \mathbf u_1-\mathbf u_2\Vert:\,\mathbf u_1,\mathbf u_2\in \mathbb U\}$, then
\begin{multline}
\mathbb L(\widehat{\boldsymbol\beta})-\inf_{\boldsymbol\beta}~\mathbb L({\boldsymbol\beta})
\leq  {\frac{C_5 \cdot  \sigma \cdot s \cdot (\ln(np)+\widetilde\zeta)}{n^{3/4}}}+C_5\cdot\sqrt{\frac{\sigma \varepsilon_{A}' }{n^{1/4}}}+\varepsilon_{A}'
\\+  \frac{C_5 \cdot s\cdot \sigma \cdot \max\left\{1,\,\Vert\boldsymbol\beta^*_{\varepsilon_{A}}\Vert_\infty\right\}\cdot (U_{f_1}+U_A)^{1/2}\sqrt{\ln(np)+\widetilde \zeta}+ \max\left\{\sqrt{\sigma }\cdot D,\,D^2\right\}}{n^{1/4}}\label{nonsmooth excess risk bound detail}
\end{multline}
with probability at least $
1-2 (p+1) \exp(-n/C_5)-6\exp\left(-2cn^{1/2}\right)
$.
\end{itemize}
\end{theorem}
\begin{proof}{Proof.}
See Section \ref{proof of theorem 2}.
\hfill \ensuremath{\Box}\end{proof} 

\begin{remark}\label{SVM remark 1st additional}
{\color{black}\Copy{theorem Copy remark additional}{It is possible to generalize Part (b) of the above theorem to obtain an error bound  in the parameterization of any   $\delta > 0$. Nonetheless, the optimal choice to balance all the error terms would be $\delta=1/4$.}}
\end{remark}

\begin{remark}\label{SVM remark}
{{Theorem \ref{nonsmooth case} is  general enough to cover a flexible class of nonsmooth HDSL problems under A-sparsity. Particularly, in the case of the high-dimensional SVM, Problem \eqref{original nonsmooth} becomes
\begin{align}
\min_{\boldsymbol\beta}\,\rho\Vert \boldsymbol\beta\Vert^2+\frac{1}{n}\sum_{i=1}^n\left[1-y_i\mathbf x_i^\top\boldsymbol\beta \right]_+\quad\Longleftrightarrow\quad\min_{\boldsymbol\beta}\,\rho\Vert \boldsymbol\beta\Vert^2+\frac{1}{n}\sum_{i=1}^n~\max_{u_i:\,0\leq u_i\leq 1}~\left\{u_i\cdot\left(1-y_i\mathbf x_i^\top\boldsymbol\beta \right)\right\},\label{SVM original}
\end{align}
where $(\mathbf x_i, y_i)$, for $i=1,...,n$, are i.i.d.\,random pairs of the feature values  and the categorial labels with support $\{\mathbf x\in\Re^p:\,\vert\mathbf x\vert\leq 1\}\times \{-1,\,+1\}$, and $\rho\geq 0$ is  a user-specific constant.   (The assumption that $\vert\mathbf x_i\vert\leq1$, a.s., can always be ensured by normalization.)
We may enable the SVM to handle high dimensionality via   the   formulation below:
\begin{align}
\min_{\boldsymbol\beta}\,\rho\Vert \boldsymbol\beta\Vert^2+\frac{1}{n}\sum_{i=1}^n \max_{u_i:\,  0\leq   u_i\leq  1}~\left\{u_i\cdot\left(1-y_i\mathbf x_i^\top\boldsymbol\beta \right)-\frac{( u_i- u_0)^2}{2n^{\delta}}\right\}+\sum_{j=1}^pP_\lambda(\vert\beta_j\vert),\label{SVM reformulated}
\end{align}
where the value of $u_0\in[0,\,1]$ can be specified arbitrarily.
As a special case to \eqref{delta function}, Problem \eqref{SVM reformulated} satisfies both Assumptions \ref{SVM assumption 1} and \ref{SVM Lipschitz problem}. For example, when $\rho=0.01$, both of the assumptions are met with $\sigma\leq O(1)$, $R\leq O(1)$, $\sigma_L=0$, and $\mathcal C_\mu\leq O(1)\cdot  \sqrt{p}$. (More detailed derivations are provided in Section \ref{sec: applicability to SVM} of the electronic companion.) Also observe that we may let $f_1$, $U_{f_1}$, $D$, and $\mathbf A(\mathbf Z_i^n)$ from Theorem \ref{nonsmooth case}  to be 
\[f_1(\boldsymbol\beta,\mathbf Z_1^n):=\rho\Vert \boldsymbol\beta\Vert^2,~~~~U_{f_1}:= 2\rho,~~~~D:=\max\{(u_1-u_2)^2:\,u_1,u_2\in[0,\,1]\},~~~~\text{and}~~~~\mathbf A(\mathbf Z_i^n):=y_i\cdot \mathbf x_i^\top,\] respectively, in the SVM. Thus, $U_{f_1}\leq O(1)$, $D=1$ and $U_A\leq \max_{y,\,\mathbf x}\left\{\Vert y\cdot \mathbf x^\top\Vert_{1,2}^2:\,y\in\{-1,\,1\},\,\vert\mathbf x\vert\leq 1\right\}\leq 1$ in this special case. Recall here that the error bound in, e.g., \eqref{nonsmooth excess risk bound detail} is poly-logarithmic in $\mathcal C_\mu$.  Theorem \ref{nonsmooth case} then implies that the poly-logarithmic sample complexity can also be achieved for the FCP-regularized SVM.}}


In contrast to \eqref{SVM reformulated}, an alternative formulation as below has been previously discussed in the literature:
\begin{align}
\min_{\boldsymbol\beta}\,\rho\Vert \boldsymbol\beta\Vert^2+\frac{1}{n}\sum_{i=1}^n\left[1-y_i\mathbf x_i^\top\boldsymbol\beta \right]_++\sum_{j=1}^p\widetilde P_\lambda(\vert\beta_j\vert),\label{SVM}
\end{align}
where $\widetilde P_\lambda(\vert\,\cdot\,\vert):\,\Re\rightarrow\Re$ is some sparsity-inducing regularization function, such the SCAD  and the Lasso.
Compared with \eqref{SVM reformulated}, this alternative does not incorporate the smoothing term of $-\frac{(u_i- u_0)^2}{2n^{\delta}}$.
Such a formulation  has been shown to be successful in multiple realistic classification problems \citep[e.g.,][]{zhang2006a}.  Furthermore,  recovery theories in different high-dimensional settings have been established by \cite{zhang2016b,zhang2016c} and \cite{peng2016a}, etc.  Nonetheless, the existing results commonly stipulate a strictly positive lower bound on the  eigenvalues of some principal submatrices of  $\mathbf X^\top\mathbf X$ or $\mathbb E[\mathbf X^\top\mathbf X]$, where $\mathbf X:=(\mathbf x_i^\top:\,i=1,...,n)$. Some of these conditions are the instantiations of the RE condition in the SVM problem. In contrast, our  bound on the excess risk is established without these eigenvalue conditions. 
\end{remark}

 \subsection{Regularized deep neural networks}\label{NN training}
 
This subsection   presents a generalization error bound for a flexible set of NN architectures. Additional results are  provided in Section \ref{additional theoretical results} of the electronic companion, where we derive more explicit error bounds under additional regularities. 


{\color{black}\label{NN binary classification setup}\Copy{For some Copy}{While NNs can be applied to a wide spectrum of data-driven tasks, our analysis herein is focused on a binary classification problem in the following settings.}}
{{For some   $\mathcal X:=\{\mathbf x\in\Re^d:\,\Vert\mathbf x\Vert=1\}$ and $\mathcal Y\in\{-1,\,1\}$ (where $d>0$ is some integer), let  $(\mathbf x,\,y)\in\mathcal X\times\mathcal Y$ be a random pair  that follows an unknown probability distribution $\mathbb D$ on $\mathcal X\times \mathcal Y$ with support $supp(\mathbb D)$. Here, $\mathbf x$ is the vector of random feature values and $y$ is the corresponding class label. We assume that there exists an unknown, deterministic, and measurable separating function  $g:\,\mathcal X\rightarrow\Re$ such that $\inf_{(\mathbf x,y)\in supp(\mathbb D)}\,\left\{y\cdot g(\mathbf x)\right\}\geq v$ for some $v\in(0,1)$; that is, the two categories of data are separable by    function $g$. Also assume that   $\mathbb E\left[ \vert g(\mathbf x)\vert\right]<\infty$.
 The learning problem of interest here, as a special case  of \eqref{population-level model M-estimation}, is  to train a classifier   using  the knowledge of a sequence of i.i.d. random samples, $(\mathbf x_i,\, y_i)$,\,$i=1,...,n$, of  $(\mathbf x,\,y)$. }}

In applying an NN to  solving this learning problem, we  narrow down the search of the optimal classifier to the determination of the best fitting parameters for the NN. Some relative details     are     below. Denote by  $\Psi:\,\Re\rightarrow\Re$  an activation function, such as the ReLU, $\Psi_{ReLU}(x)=\max\{0,\,x\}$,  the softplus, $\Psi_{softplus}(x)=\ln(1+e^x),$ and  the sigmoid, $\Psi_{sigmoid}(x)=\frac{e^x}{1+e^x}.$ The NN model is then  a network  that consists of  multiple   layers (groups) of     neurons (or units). Each neuron is a computing unit that performs the operations of the chosen activation function on the input signals.  
  Architectures among those layers  are formed in the sense that the signals are passed from the   layer of input neurons to the layer of output units, transversing  a predetermined collection of candidate paths. Each path may comprise multiple neurons and connections. Fitting parameters often exist in the forms of connection weights and biases  to (dis)amplify and offset the signals, respectively.  A layer that is neither the input layer nor the output layer is called a hidden layer. Throughout our discussions on the NNs,  we let $\mathcal D\geq 2$ be the number of layers (excluding the input layer but including the output layer).  A neuron in a hidden layer is called a hidden neuron. 
  We denote this NN by $F_{NN}(\mathbf x,\boldsymbol\beta)$, where $F_{NN}:\,\mathcal X\times \Re^p\rightarrow\Re$  is a deterministic, measurable function  that captures the output  of an NN given input $\mathbf x$ and  fitting parameters $\boldsymbol\beta$.  We also assume that there exists a deterministic function $\Omega:\{1,...,p\}\rightarrow\Re_+$   such that
\begin{align}
\Omega({p'})\geq \inf_{\substack{\boldsymbol\beta:\,\Vert \boldsymbol\beta\Vert_0\leq {p'}
}}~\mathbb E\left[\vert F_{NN}(\mathbf x,\,\boldsymbol\beta)-g(\mathbf x)\vert\right],~~~~\forall {p'}:\,1\leq {p'}\leq p.\label{Define Omega}
\end{align}
Intuitively, $\Omega(p')$ measures the model misspecification error incurred by  the NN in  representing   $g$, when only $p'$-many fitting parameters  are nonzero (active).

 In training the NN, we focus on the following  formulation as a special case to \eqref{general formulation}:
  \begin{align}
 \inf_{\boldsymbol\beta} ~\mathcal T_{n,\lambda}(\boldsymbol\beta):=n^{-1}\sum_{i=1}^n\mathcal F\left(y_i\cdot F_{NN}(\mathbf x_i,\boldsymbol\beta)\right)+\sum_{j=1}^pP_{\lambda}(\vert\beta_j\vert),
 \label{regularized NN model}\end{align}
where we follow \cite{cao2019a,cao2019generalization} in defining $\mathcal F:\,\Re\rightarrow\Re_+$   to be $\mathcal F\left(z\right):=\ln(1+\exp(-z))$.  Note that, if we drop the regularization term $\sum_{j=1}^pP_{\lambda}(\vert\beta_j\vert)$, then \eqref{regularized NN model} is reduced to the conventional training formulation for an NN.  Hereafter, we assume that $\mathbb E\left[\vert F_{NN}(\mathbf x,\boldsymbol\beta)\vert\right]<\infty$  for all $\boldsymbol\beta:\,\Vert \boldsymbol\beta\Vert_\infty\leq R_{\Omega}$ for some $R_{\Omega}>0$.  This quantity should be properly large to ensure the satisfaction of the assumption below.
{  \begin{assumption}\label{target function assumption} \Copy{For all Copy}{
For all $1\leq{s_A}\leq p$, it holds that $ 
\emptyset\neq  [-R_{\Omega},\,R_{\Omega}]^p\cap  \left\{\boldsymbol\beta\in\Re^p:\,\mathbb E\left[ \left\vert g(\mathbf x)-F_{NN}(\mathbf x,\,\boldsymbol\beta)\right\vert\right]\leq \Omega(s_A),\,\Vert \boldsymbol\beta\Vert_0\leq {s_A}\right\}.
$}
  \end{assumption}}
{\label{new assumption intuition first}\Copy{Intuitively,  Assumption Copy}{ Intuitively,  Assumption \ref{target function assumption} means that the NN can represent the separating function $g$ with a model misspecification error of no more than $\Omega(s_A)$ when  (a) no more than $s_A$-many fitting parameters are nonzero and (b) the absolute values of  these fitting parameters are bounded from above by $R_{\Omega}>0$.  }}
 
 We  also impose the following non-critical condition on the architecture of an NN.
{\label{new assumption intuition second}  \begin{assumption}\label{mild condition invariance}\Copy{For any constant Copy}{
 For any constants $C\in\Re$, $R_\Omega>0$,   $p'\geq 1$, and fitting parameters $\boldsymbol\beta_1\in\Re^p:\,\Vert\boldsymbol\beta_1\Vert_{\infty}\leq R_{\Omega},\,\Vert\boldsymbol\beta_1\Vert_{0}\leq p'$, it holds that $ F_{NN}(\mathbf x,\boldsymbol\beta_1)\cdot C= F_{NN}(\mathbf x,\boldsymbol\beta_2)$   for some $\boldsymbol\beta_2\in\Re^p:\,\Vert\boldsymbol\beta_2\Vert_{\infty}\leq C\cdot R_{\Omega},\,\Vert\boldsymbol\beta_2\Vert_{0}\leq p'$, for every $\mathbf x\in\mathcal X$.}
\end{assumption}}
{\color{black}\Copy{It can be}{\noindent It can be verified that Assumption \ref{mild condition invariance} holds for many NN architectures, including many convolutional neural networks and residual networks that have  linear or ReLU activation functions in the output layer.}}

\begin{remark}\label{A-sparsity of an NN}{{\Copy{By the satisfaction of Assumptions}{
By the satisfaction of Assumptions \ref{target function assumption} and \ref{mild condition invariance}, we argue that the generalizability of an NN trained by solving \eqref{regularized NN model}  can be analyzed through the  framework of HDSL under A-sparsity. Based on the existing results on the representability of NNs, e.g., by \cite{devore1989a}, \cite{yarotsky2017a}, \cite{mhaskar2016a}, and \cite{mhaskar1996a},   an NN with a reasonably small network size $s_A$ may well represent $g$ (such that $\Omega(s_A)$ is small) under some plausible conditions. These representability results imply the innate presence of    A-sparsity in an NN model. Observe that $\mathcal F$ is 1-Lipschitz continuous. Thus,   $\mathbb E [\mathcal F(y\cdot  \frac{\ln n}{2v}\cdot F_{NN}(\mathbf x,\boldsymbol\beta_1))]-\mathbb E[\mathcal F(y\cdot  \frac{\ln n}{2v}\cdot  g(\mathbf x))]\leq   \frac{\ln n}{2v}\cdot  \mathbb E\left[ \left\vert F_{NN}(\mathbf x,\boldsymbol\beta_1)-g(\mathbf x)\right\vert\right]$ for any $\boldsymbol\beta:\,\Vert\boldsymbol\beta\Vert_\infty\leq R_{\Omega}$. Invoking Assumption \ref{target function assumption} and the fact that $\inf_u \mathcal F(u)=0$, we obtain that
\begin{align}
&\min_{\substack{\boldsymbol\beta:\,\Vert\boldsymbol\beta\Vert_0\leq {s_A},\\\,\Vert\boldsymbol\beta\Vert_\infty\leq R_\Omega}}\mathbb E\left[\mathcal F\left(y\cdot  \frac{\ln n}{2v}\cdot F_{NN}(\mathbf x,\boldsymbol\beta)\right)\right]-\inf_{u}\,\mathcal F(u)\nonumber
\\
\,\leq &\min_{\substack{\boldsymbol\beta:\,\Vert\boldsymbol\beta\Vert_0\leq {s_A},\\\,\Vert\boldsymbol\beta\Vert_\infty\leq R_\Omega}}\mathbb E\left[ \frac{\ln n}{2v}\left\vert  F_{NN}(\mathbf x,\boldsymbol\beta)   -g(\mathbf x)\right\vert\right] +\mathbb E \left[\mathcal F\left(y\cdot  \frac{\ln n}{2v}\cdot g(\mathbf x)\right)\right]\nonumber
\\\leq &  \frac{\ln n}{2v}\cdot\Omega(s_A)+\mathbb E \left[\mathcal F\left(y\cdot  \frac{\ln n}{2v}\cdot g(\mathbf x)\right)\right]\leq \frac{\ln n}{2v}\cdot\Omega(s_A)+\frac{1}{\sqrt{n}},\nonumber
\end{align}
where the last inequality is due to the assumption that, for all $(\mathbf x, y)\in supp(\mathbb D)$, it holds that $y\cdot g(\mathbf x)\geq v\Longrightarrow \mathbb E \left[\mathcal F\left(y\cdot  \frac{\ln n}{2v}\cdot g(\mathbf x)\right)\right]\leq\ln\left(1+\exp(-0.5\ln n)\right)\leq \frac{1}{\sqrt{n}}$. Further note that, by Assumption \ref{mild condition invariance}, $\frac{\ln n}{2v}\cdot F_{NN}(\mathbf x,\boldsymbol\beta) $ can   be represented by the same NN architecture; that is, $\frac{\ln n}{2v}\cdot F_{NN}(\mathbf x,\boldsymbol\beta)= F_{NN}(\mathbf x,\boldsymbol\beta')$ for some new fitting parameters $\boldsymbol\beta':\,\Vert\boldsymbol\beta'\Vert_\infty\leq \frac{\ln n}{2v} R_\Omega$. Thus, we may have 
\begin{align}\min_{\substack{\boldsymbol\beta:\,\Vert\boldsymbol\beta\Vert_0\leq {s_A},\\\,\Vert\boldsymbol\beta\Vert_\infty\leq \frac{\ln n}{2v}\cdot R_\Omega}}\mathbb E\left[\mathcal F\left(y\cdot F_{NN}(\mathbf x,\boldsymbol\beta)\right)\right]-\inf_{u}\,\mathcal F(u)\leq \frac{\ln n}{2v}\cdot\Omega(s_A)+\frac{1}{\sqrt{n}},\label{test inequality final result representability}
\end{align}
  which matches the statement of Assumption \ref{Assumption A-sparsity}    with $s:={s_A}$, \, $R:=\frac{\ln n}{2v}\cdot R_\Omega$,  \, $\varepsilon_A :=   \frac{\ln n}{2v}\cdot \Omega(s_A)+\frac{1}{\sqrt{n}}$, and $L_{g}^*:=\inf_{u}\,\mathcal F(u)=0$.
As mentioned, explicit forms of $\Omega(\cdot)$ have been provided, e.g., by  \cite{devore1989a}, \cite{yarotsky2017a}, \cite{mhaskar2016a}, and \cite{mhaskar1996a}.}}} With the above discussion, the generalizability  of an NN can then be derived using the same machinery for  HDSL under  A-sparsity, under  one more  flexible assumption on the NN's architecture  as  below.
\end{remark}

{
 \begin{assumption}\label{new bound results 2.5} \Copy{Assumption to copy}{ 
 For almost every   $\mathbf x\in\mathcal X$, it holds that the gradient $\nabla_{\boldsymbol\beta} F_{NN}(\mathbf x,{\boldsymbol\beta})$ and Hessian $\nabla_{\boldsymbol\beta}^2 F_{NN}(\mathbf x,{\boldsymbol\beta})$ of $F_{NN}(\mathbf x,\,\cdot\,)$ are  everywhere well-defined and satisfy that 
\[ \max\left\{\underset{ \mathbf x\in\mathcal X}{\text{\upshape ess sup}}\,\left\Vert \nabla_{\boldsymbol\beta}  F_{NN}(\mathbf x,{\boldsymbol\beta}) \right\Vert,~~\underset{ \mathbf x\in\mathcal X}{\text{\upshape ess sup}}\left\Vert \nabla^2_{\boldsymbol\beta}  F_{NN}(\mathbf x,{\boldsymbol\beta})\right\Vert\right\} \leq \exp\left[\mathcal U_{NN}\cdot \mathcal D\cdot \ln\left(\mathcal U_{NN}\cdot \Vert \boldsymbol\beta\Vert +\mathcal U_{NN}\right)\right]\]
for all $\boldsymbol\beta\in\Re^p$ and some $\mathcal U_{NN}\geq 1$.}
\end{assumption}
}
{\label{explanation 3rd assumption}\Copy{Assumption 222 Copy}{Assumption  \ref{new bound results 2.5} essentially allows the norms of gradient and Hession to grow   exponentially in the number of  layers $\mathcal D$. Such an assumption is  satisfied by a wide spectrum of NN architectures, especially when the    activation functions are smooth. Some NNs with nonsmooth activation functions, such as the ReLU, may still be analyzed. We discuss such a case later in Subsection  \ref{computable local solution}.}} 
 
We are now ready to present our  result on the generalizability of a regularized NN. With some abuse of notations,  the S$^3$ONC($\mathbf Z_1^n$), in this special case,   is  referred to  as the S$^3$ONC$(\mathbf X,\,\mathbf y)$ to problem \eqref{regularized NN model}, where $\mathbf X:=(\mathbf x_i^\top)$ and $\mathbf y:=(y_i)$.  
\begin{theorem}\label{general result theorem regularized NN}
Consider any random vector $\widehat{\boldsymbol\beta}$ such that   $\Vert\widehat{\boldsymbol\beta}\Vert_\infty\leq \frac{\ln n}{2v}\cdot R_\Omega$ and  the S$^3$ONC$(\mathbf X,\mathbf y)$ holds at  $\widehat{\boldsymbol\beta}$ almost surely. 
Suppose that Assumptions \ref{target function assumption}, \ref{mild condition invariance}, and \ref{new bound results 2.5} hold.
For any fixed $\Gamma\geq 0$, assume that $\mathcal T_{n,\lambda}(\widehat{\boldsymbol\beta})-\inf_{{\boldsymbol\beta}}\mathcal T_{n,\lambda}({\boldsymbol\beta})\leq \Gamma$, w.p.1.,
where $\mathcal T_{n,\lambda}$ is as defined in \eqref{regularized NN model}.
There exists a universal constant $C_6>0$, such that, for any ${s_A}:\, 1\leq {s_A}\leq p$, if    $a<\frac{1}{2}\cdot \exp\left\{-2\mathcal U_{NN}\cdot \mathcal D\cdot \ln\left[2 p\cdot v^{-1}\cdot \mathcal U_{NN}\cdot R_{\Omega}\cdot\ln n \right]\right\}$,    $\lambda:=\sqrt{\frac{8\sigma }{c\cdot a\cdot n^{2/3}}\left[\ln(\frac{3e}{2v}\cdot R_{\Omega} p n^{4/3})+\mathcal U_{NN}\cdot \mathcal D\cdot \ln\left(\mathcal U_{NN}  R_{\Omega}pnv^{-1} \right)\right]}$,
 and 
 \begin{align}
n > C_6\cdot \left[\left( {\Gamma+v^{-1}\cdot\Omega({s_A})\cdot \ln n  }\right)^{3}+{s_A}\cdot \mathcal D\cdot \mathcal U_{NN}\cdot \ln\left(\mathcal U_{NN}\cdot (1+n pR_{\Omega}v^{-1})\right) \right],\label{sample initial requirement 3}
\end{align}
then it holds  that
 \begin{multline}\label{part 1 of nn result 1}
\mathbb E\left[\mathbb 1\left(y\cdot F_{NN}(\mathbf x,\widehat{\boldsymbol\beta})<0\right)\right]
\\\leq    C_6\cdot\left(\frac{{s_A}\cdot \mathcal D\cdot \mathcal U_{NN}\cdot \ln\left(\mathcal U_{NN}\cdot (1+n pR_{\Omega}v^{-1})\right)}{n^{2/3}}+\sqrt{\frac{{s_A}\cdot \mathcal D\cdot \mathcal U_{NN}\cdot \ln\left(\mathcal U_{NN}\cdot (1+n pR_{\Omega}v^{-1})\right)}{n}}+\frac{1}{n^{1/3}}\right)
\\ +v^{-1}\cdot\Omega({s_A})\cdot \ln n+\Gamma+C_6\cdot  \sqrt{\frac{\Gamma+v^{-1}\cdot\Omega({s_A})\cdot \ln n}{n^{1/3}}}
\end{multline}
with probability at least $1-C_6 p \exp\left(-\frac{n}{C_6}\right)-C_6\exp\left(-\frac{n^{1/3}}{C_6}\right).$ Here, $\Omega(\,\cdot\,)$ is defined as in \eqref{Define Omega}. 
\end{theorem}
\begin{proof}{Proof.}
See Section \ref{sec: proof of generalized NN in general}. 
\hfill \ensuremath{\Box}
\end{proof}

\begin{remark}\label{remark on the first NN result}
We would like to make a few remarks on the results presented in this theorem.
\begin{enumerate}
\item[(i)]  $\mathbb E\left[\mathbb 1\left(y\cdot F_{NN}(\mathbf x,\widehat{\boldsymbol\beta})<0\right)\right]=\mathbb P[y\cdot F_{NN}(\mathbf x,\widehat{\boldsymbol\beta})<0]$ is also referred to as the expected 0-1 loss and is a commonly adopted measure of generalization performance,  such as by \cite{cao2019a,cao2019generalization},  in a binary classification problem. 
\item[(ii)] This theorem provides the promised poly-logarithmic dependence between the sample size $n$ and the dimensionality $p$; polynomially increasing $n$ can compensate  for the exponential growth in  $p$.   With this result,  the generalizability of an over-parameterized NN is ensured, and the promised result in \eqref{general inequality NN bound} is proven. The error bound can be made more explicit under some additional conditions as discussed in Section \ref{local solutions NN training results ReLu}.
 
\item[(iii)] Although  Assumption \ref{new bound results 2.5} allows the   Lipschitz constant to   grow exponentially in the number of layers $\mathcal D$, the generalization error increases  no more than linearly in $\mathcal D$.
\item[(iv)] Many sparsity-inducing regularization schemes have been discussed in the literature, including Dropout \citep{srivastava2014a}, sparsity-inducing penalization \citep{han2015a,scardapane2017a,louizos2017a,Wen2016a}, DropConnect \citep{wan2013a}, randomDrop \citep{huang2016a}, and pruning \citep{alford2018a}, etc. Many of these studies are focused on the numerical aspects, yet the  theoretical guarantees on the effectiveness of  regularization   are still largely lacking.  Although  \cite{wan2013a} presented   generalization error analyses for DropConnect, the dependence among the dimensionality, the generalization error, and the sample size is not explicated therein. 
{\color{black}{\label{statement change minor}\Copy{It is our}{It is our conjecture that our results could be extended to and combined with the alternative regularization schemes to facilitate the analysis of the regularized NNs.}}}
 
\item[(v)] Theorem \ref{general result theorem regularized NN} informs us that the generalization performance of the NNs is consistent with the optimization quality. If all other quantities are fixed, the generalization error can be bounded by $\mathcal O\left(\sqrt{\Gamma}+\Gamma\right)$, where we recall that $\Gamma\geq 0$ is the suboptimality gap.   
\item[(vi)] {\label{comment on optimization}\Copy{Admittedly Copy}{Admittedly, how to control $\Gamma$ is still an open  question.  The traditional training formulation of an NN is usually nonconvex. Thus, it is  generally prohibitive to compute  a global solution. The  challenge is further increased by the  incorporation of the FCP, which is also nonconvex. Fortunately,  in spite of the current theoretical challenge, it has been observed empirically that some local optimization algorithms could well approximate a global optimum in NN training, e.g., in the experiments reported by \cite{wan2013a} and \cite{alford2018a}. To explain these observations, several   theoretical paradigms have already been provided  by, e.g., \cite{du2018a,liang2018a,haeffele2017a} and \cite{wang2019a}.   Based on those results, it is promising that the structures of an NN (even with  regularization) can often be exploited to facilitate global optimization. An excellent review of this topic is provided by \cite{sun2019a}.  To add to the literature, we   present an interesting special case where a suboptimality-independent generalization error bound for the FCP-regularized NN can be  achieved at a pseudo-polynomial-time computable solution in Subsection \ref{computable local solution} of the electronic companion.}} \label{last bullet point} 


 \end{enumerate}
\end{remark}

\section{Numerical Experiments}\label{Numerical section}
We report in this section   several numerical experiments. In Sections \ref{subsec: numerical A-sparsity} and \ref{MNIST results}, we  consider the high-dimensional Huber regression under A-sparsity and the NNs, respectively. Then, Section \ref{sec Numerical Experiments EC} of the electronic companion presents our test results on  the high-dimensional SVM  (as a special nonsmooth learning problem) and some additional numerical examples  on the NNs.
Unless otherwise stated explicitly, most of our experiments, including those in the electronic companion, were implemented in Matlab 2014b and run with a single thread on a PC with 40 Intel (R) Xeon (R) E5-2640-v4 CPU cores (2.40 GHz, 64 bits), and 128 GB memory. A different implementation environment was involved in the tests on some larger-scale NN models, as presented in Section \ref{MNIST results}.
 
\subsection{Experiments on HDSL under A-sparsity}\label{subsec: numerical A-sparsity}
This section reports our test results   on  high-dimensional Huber regression (HR) under A-sparsity (in the sense of Assumption  \ref{Assumption A-sparsity original}). Our settings for experiments are summarized below:  Denote by $\mathcal N(0,\sigma^2)$ a centered normal distribution  with variance $\sigma^2>0$ and by $\mathcal N_{p}(\mathbf 0,\Sigma)$  a centered $p$-variate normal distribution   with covariance matrix $\Sigma=(\varsigma_{j_1,j_2})$ and $\varsigma_{j_1,j_2}=0.3^{\vert j_1-j_2\vert}$. The training data set $\{(\mathbf x_i,y_i):\,i=1,...,n\}$ was generated as per a linear system $y_i=\mathbf x_i^\top\boldsymbol\beta^{*}+\omega_i$, for  $i=1,...,n$. Here, $(\mathbf x_i,\,y_i)$ denotes a pair of (observed) design and response, and $\boldsymbol\beta^{*}$ denotes the vector of true parameters to be recovered. Some additional details are summarized below:
\begin{itemize}
\item The training sample size was chosen as $n=100$.
\item $\omega_i$, $i=1,...,n$, were   i.i.d.\ white noises such that  $\omega_i\sim \mathcal N(0,\sigma^2)$ for all $i$.
\item $\mathbf x_i\sim \mathcal N_{p}(\mathbf 0,\Sigma)$, $i=1,...,n$, were  i.i.d. random vectors.
\item The vector of true parameters was prescribed as $\boldsymbol\beta^{*}=\boldsymbol\beta_{\varepsilon_{A}}^*+E\cdot \boldsymbol v\cdot \frac{1}{\vert \boldsymbol v\vert}$, where $\boldsymbol\beta_{\varepsilon_{A}}^*:=(3,\,5,\,0,\,0,\,1.5,\underbrace{0,\,...,0}_{\text{$(p-5)$-many 0's}})^\top$ and $E\cdot \boldsymbol v\cdot \frac{1}{\vert \boldsymbol v\vert}$ stands for some dense perturbation. Here, $E>0$ denotes a user-specific scalar and  $\mathbf v=(v_j)$ denotes a random vector with i.i.d. entries of uniform random variables on $[-1,\,1]$. Note that  the magnitude of the perturbation can be calculated as $\left\vert E\cdot \boldsymbol v\cdot \frac{1}{\vert \boldsymbol v\vert}\right\vert=E$
\end{itemize}
 
Given the above, this experiment was focused on the following HR problem:
\[
\min_{\boldsymbol\beta} ~\frac{1}{n}\sum_{i=1}^n \left[L_{HR}(\boldsymbol\beta,\mathbf x_i,y_i):=\frac{1}{2}(\mathbf x_i\boldsymbol\beta-y_i)^2\cdot \mathbb I\left(\vert \mathbf x_i\boldsymbol\beta-y_i\vert \leq \eta\right)+ \left(\eta\vert \mathbf x_i\boldsymbol\beta-y_i\vert -\frac{\eta^2}{2}\right)\cdot \mathbb I\left(\vert \mathbf x_i\boldsymbol\beta-y_i\vert  >\eta\right)\right].
\]
The corresponding FCP-regularized formulation, referred to as the HR-FCP, is then given as
\begin{align}
\min_{\boldsymbol\beta} \,n^{-1}\sum_{i=1}^n L_{HR}(\boldsymbol\beta,\mathbf x_i,y_i)+\sum_{j=1}^pP_\lambda(\vert\beta_j\vert).\label{FCP HR form}
\end{align}
This problem was solved via Algorithm 1, for which  the initial solution was prescribed as $\widehat{\boldsymbol\beta}^{\ell_1}\in
\arg\,\min_{\boldsymbol\beta} \,n^{-1}\sum_{i=1}^n L_{HR}(\boldsymbol\beta,\mathbf x_i,y_i)+\lambda\cdot \sum_{j=1}^p \vert\beta_j\vert$ for the same $\lambda$ as in \eqref{FCP HR form}.

The hyper-parameters of Algorithm 1 were set to be $\mathcal M=10$ and $\gamma_{opt}=10^{-5}$. For  the FCP,   we  fixed $a=0.09$  (such that $a<\mathcal M^{-1}$)  and prescribed that $\lambda:=\mathcal C_{fcp}\cdot \sqrt{\frac{\ln p}{n^{2/3}}}$ for some  $\mathcal C_{fcp}>0$.  In choosing  $\mathcal C_{fcp}$,   three independent validation datasets, with 100 data observations for each,  were generated following the same approach as  the training data above. The dimensions of those validation sets were   $p\in\{500,\,  750,\, 1000\}$.  The value of $\mathcal C_{fcp}$ was chosen to be the best-performing   on the validation data   among the candidate values of $\{0.5,\,0.75,\,1,\,1.25,\,1.5\}$. More specifically, a linear model was trained on the training data  when $\mathcal C_{fcp}$ and $p$ were fixed at every combination of their candidate values listed above. We let $\widehat{\boldsymbol\beta}^{1,\mathcal C_{fcp}}$, $\widehat{\boldsymbol\beta}^{2,\mathcal C_{fcp}}$, and $\widehat{\boldsymbol\beta}^{3,\mathcal C_{fcp}}$ be  the resultant estimators   for a fixed  $\mathcal C_{fcp}$ when $p=500$, $750$, and $1000$, respectively.  The chosen  value of  $\mathcal C_{fcp}$  was the one that  minimized the average performance on all the validation sets, calculated as per the below:
\begin{align}
\frac{1}{300}\left[ \sum_{i=1}^{100}L_{HR}(\widehat{\boldsymbol\beta}^{1,\mathcal C_{fcp}},\mathbf x_i^{val,1},y_i^{val,1})+\sum_{i=1}^{100}L_{HR}(\widehat{\boldsymbol\beta}^{2,\mathcal C_{fcp}},\mathbf x_i^{val,2},y_i^{val,2})+\sum_{i=1}^{100}L_{HR}(\widehat{\boldsymbol\beta}^{3,\mathcal C_{fcp}},\mathbf x_i^{val,3},y_i^{val,3})\right].\label{validation set evaluate HR}
\end{align}
Here, $(x_i^{val,k'},y_i^{val,k'})$, for $k'\in\{1,2,3\}$, is the $i$th data from the $k'$th validation set. As it turned out, $\mathcal C_{fcp}:=1$.

 The HR-FCP  was  compared with two alternative schemes: (i) the HR without any regularization, denoted by HR, and (ii) the HR with the $\ell_1$-norm regularization, denoted by HR-L1. (The HR-L1 has been discussed by \cite{owen2007a}, among others.) 
 The coefficient for the $\ell_1$-norm penalty was chosen to be $\lambda_{\ell_1}:=\mathcal C_{\ell_1}\cdot \sqrt{\frac{\ln p}{n}}$ for some $\mathcal C_{\ell_1}>0$. The dependence of $\lambda_{\ell_1}$ on $p$ and $n$ is consistent with the theoretical results for the $\ell_1$-norm regularization (e.g., by \cite{n2012a}).  We determined $\mathcal C_{\ell_1}:=0.5$ using the same  approach as in choosing $\mathcal C_{fcp}$ above.  
 
 To evaluate  the out-of-sample performance,  $5000$-many  independent test  data observations were simulated for each problem instance, following the same data generation process for the training data  above. If we let  $(\mathbf x_i^{test},y_i^{test})$, $i=1,...,5000$, be the test data of a problem instance,    the   out-of-sample error of an estimator $\widehat{\boldsymbol\beta}$ was calculated by
\begin{align}
\frac{1}{5000}\sum_{i=1}^{5000}L_{HR}(\widehat{\boldsymbol\beta},\mathbf x_i^{test},y_i^{test})-\frac{1}{5000}\sum_{i=1}^{5000}L_{HR}(\boldsymbol\beta^*,\mathbf x_i^{test},y_i^{test}).\label{test evaluate HR}
\end{align}
Each experiment was  randomly replicated 100 times. Figure \ref{Plots for different gap A-sparsity} presents the numerical results.  We discuss this figure in relative detail below.

\begin{figure}[htbp]
\begin{center}
\begin{tabular}{ c c  }\setlength\extrarowheight{-10pt}
 \includegraphics[width=0.35\textwidth]{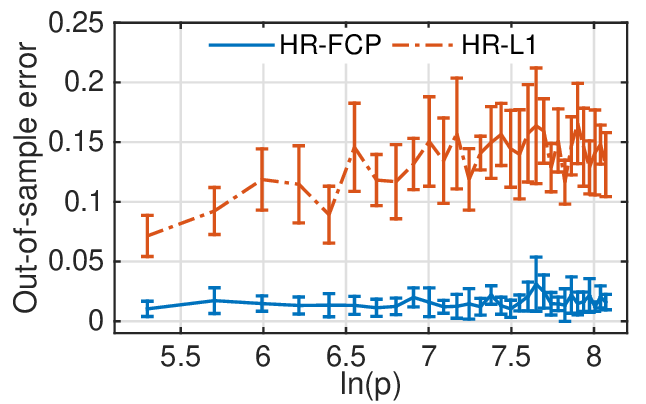} & \includegraphics[width=0.35\textwidth]{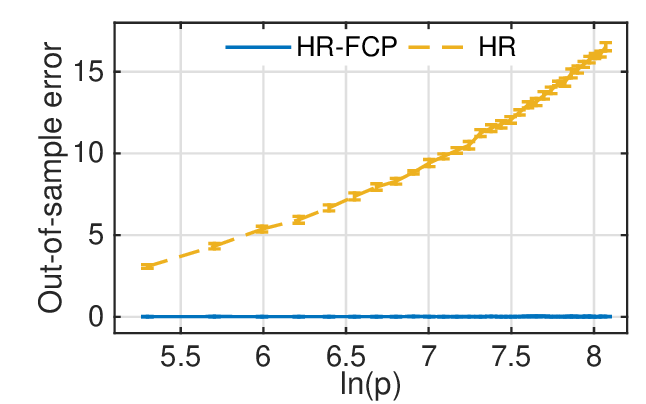}    \\ 
 
 (a)& (b)\\

  \includegraphics[width=0.35\textwidth]{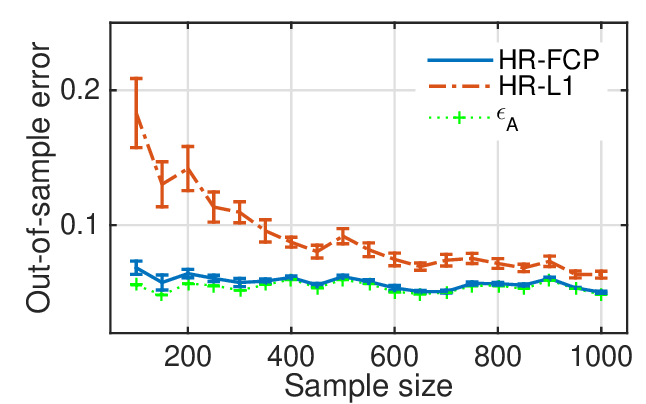} & \includegraphics[width=0.35\textwidth]{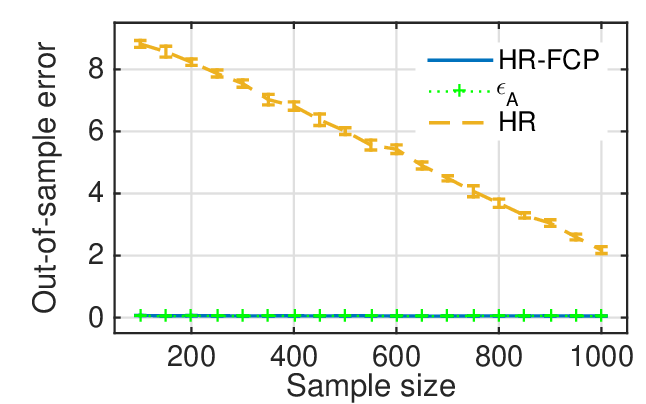}    \\ 
 (c)& (d)\\
  \includegraphics[width=0.35\textwidth]{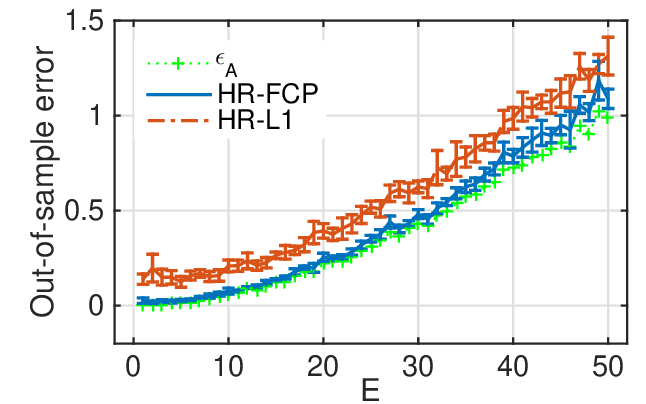} & \includegraphics[width=0.35\textwidth]{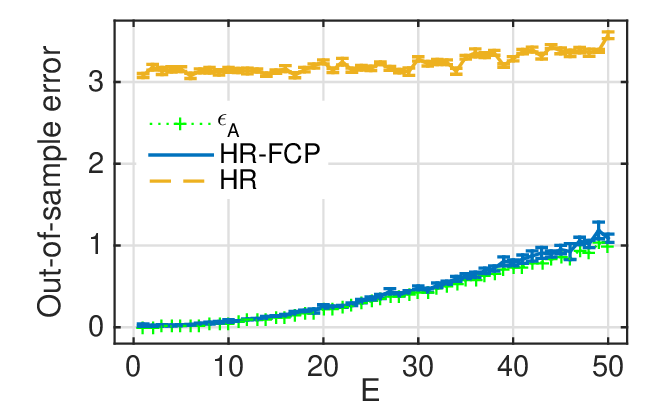}    \\ 
 (e)& (f)\\
   \includegraphics[width=0.35\textwidth]{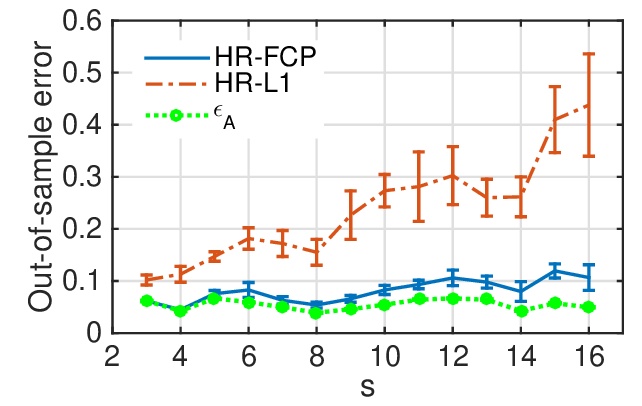} & \includegraphics[width=0.35\textwidth]{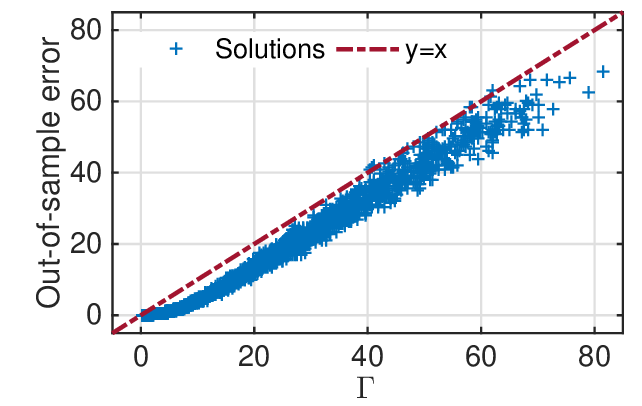}    \\ 
 (g)&(h)\\
\end{tabular}
\end{center}
    \caption{Numerical tests on the dependence of the out-of-sample errors in high-dimensional Huber regression on different quantities, including the logarithm of dimensionality $\ln p$ in Subplots (a)-(b), the sample size $n$ in subplots (c)-(d), the  quantities $E$ and $\varepsilon_A$ in subplots (e)-(f), the sparsity level $s$ in subplot (g), and the underestimation of the suboptimality gap $\Gamma$ in Subplot (h). All the error bars are centered at the average levels out of 100 random replications, and the radii of the error bars are equal to  1.96 times the standard errors. }\label{Plots for different gap A-sparsity}
\end{figure}

\begin{itemize}[leftmargin=*]
\item {\label{explain figure first}\Copy{In all the subplots Copy}{In all the subplots (a) through (g) of Figure \ref{Plots for different gap A-sparsity}, blue solid lines, red dot-dashed lines, and yellow dashed lines represent the out-of-sample errors generated by the HR-FCP, the HR-L1, and the HR. The green dotted lines stand for the estimated values of $\varepsilon_A$, a quantity involved in the definition of A-sparsity. The values of $\varepsilon_A$ were estimated by  \eqref{test evaluate HR} with $\widehat{\boldsymbol\beta}:={\boldsymbol\beta}_{\varepsilon_A}^*$.  The error bars in the plot are all centered at the average levels out of 100 random replications, and the radii of the error bars are 1.96 times the corresponding standard errors.}}
\item {\Copy{Subplots (a) and (b)}{Subplots (a) and (b) show the comparison of the HR-FCP with the HR-L1 and with the HR, respectively, when the logarithm of the dimensionality ($\ln p$) was increased  gradually  with $p\in\{200, 300, ..., 5000\}$ and $E=0$. From both subplots (a) and (b), one can see that the out-of-sample errors generated by HR-FCP were small for all the values of $\ln p$,  especially when the HR-FCP was compared with both the HR and the HR-L1. In particular, (as in Subplot (b)), the performance of the HR   deteriorated  rapidly as $\ln p$ grew, while the performance of the HR-FCP   remained approximately constant. Because our error bounds for HR-FCP are polynomial in $\ln p$, it appears that  an even sharper dependence on $\ln p$ may be pursued in our analysis, at least for certain HDSL special cases.}}

\item {\Copy{Subplots (c) and (d)}{Subplots (c) and (d) present the performance of all the three schemes above when  the sample size $n$ was increased from 100 to 1000 (with $E=10$ and $p=1000$). From both subplots, one can observe that the HR-FCP outperformed both the HR and  the HR-L1. Also shown in these two subplots are the values of $\varepsilon_A$ (denoted by ``$\epsilon_A$'' in the figure).
It can be observed that the out-of-sample errors of the HR-FCP matched with the values of $\varepsilon_A$, especially when the sample size was relatively large.  This  pattern was consistent with our error bounds.}}
\item {\Copy{Subplots (e) and (f) Copy}{As shown in Subplots (e) and (f), all the three schemes above were compared again when   $E$ was    increased gradually
 (and, as a result,  $\varepsilon_A$ would tend to grow). Consistent with our theoretical results, the out-of-sample errors yielded by the HR-FCP approximately matched the values of $\varepsilon_A$ (denoted by ``$\epsilon_A$'' in the plots). Furthermore, regardless of  the values of $\varepsilon_A$, the HR-FCP achieved better generalization errors than the HR and the HR-L1 in almost all of the instances.   We can also observe from both subplots that,  even if the magnitudes of the perturbation $E$ were comparable to $\vert\boldsymbol\beta_{\varepsilon_A}^*\vert$, the corresponding values of $\varepsilon_A$ remained to be small. So did the out-of-sample errors generated by the HR-FCP, especially when compared with the HR's performance.   For example, when $E=10$, the magnitude of  perturbation was larger than $\vert\boldsymbol\beta_{\varepsilon_A}^*\vert=9.5$.  Yet, the corresponding $\varepsilon_A$ was below 0.1, and the out-of-sample error of the HR-FCP was almost equal to $\varepsilon_A$. Both values were significantly lower than the corresponding out-of-sample error of the HR.}}
 \item  {\Copy{The dependence of the HR-FCP Copy}{In Subplot (g), the dependence of the HR-FCP and the HR-L1 on the sparsity level $s$ was  evaluated when $E=10$, $p=1000$, $n=100$, and  $\boldsymbol\beta_{\varepsilon_A}^*:=(3,\,5,\,0,\,0,\,1.5,\underbrace{2,\,...,2}_{\text{$(\tau)$-many 2's}},\,~\underbrace{0,\,...,0}_{\substack{\text{$(p-\tau-5)$}\\\text{-many 0's}}})^\top$ for all $\tau=0,1,...,13$. Thus, the corresponding values of $s$ were $s=3,4,...,16$.   As one may see from Subplot (g), the performance of both the HR-FCP and the HR-L1 deteriorated  when  $s$ increased. Yet, the HR-L1 seemed to be more sensitive to the change in $s$ than the HR-FCP.  
 }}
 \item {\Copy{Finally, Subplot (h) Copy}{Finally, Subplot (h) presents the numerical evaluation of the dependence of the HR-FCP's out-of-sample performance on $\Gamma$. Note that, in the case of Huber regression,   $\Gamma:=\left[n^{-1}\sum_{i=1}^n L_{HR}(\widehat{\boldsymbol\beta},\mathbf x_i,y_i)+\sum_{j=1}^pP_\lambda(\vert\widehat\beta_j\vert)\right]-\left[n^{-1}\sum_{i=1}^n L_{HR}({\boldsymbol\beta}^{*}_{\varepsilon_{A}},\mathbf x_i,y_i)+\sum_{j=1}^pP_\lambda(\vert\beta_{\varepsilon_A,j}^*\vert)\right]$   is an underestimation of the suboptimality gap in  minimizing   \eqref{FCP HR form}. To generate this plot, we solved for the S$^3$ONC solutions with random initialization for 2000-many repetitions. A ``$+$'' in the plot corresponds to one of those S$^3$ONC solutions, and the dot-dashed line stands for the linear function of $Y=X$. If a ``$+$'' is below the line of $Y=X$, then it indicates that the out-of-sample error of that point was smaller than the    corresponding  value of $\Gamma$. As can be seen from this subplot, almost all the ``+''s are  below (but in the proximity of) the aforementioned linear function.   This pattern was consistent with our error bound in \eqref{bound global 3}, which is indeed of $\mathcal O(\Gamma)$ when $\Gamma\geq 1$. }}
\end{itemize}

\subsection{Experiments on   neural networks}\label{MNIST results} We report two sets of experiments on the FCP-regularized NNs.   The first set, as  presented in this subsection, was focused on image classification  using two mainstream testbeds, the MNIST \citep{lecun2013a} and the CIFAR-10 datasets \citep{krizhevsky2009a}.  Leaderboards that report the state-of-the-art results can be found at, e.g., \url{https://paperswithcode.com/}.  The second set of tests, as presented in Section \ref{sec NN test binary classification} of the electronic companion, involved the comparison between the non-regularized NNs and their  FCP-regularized counterparts in a task of binary classification with simulated data.  

In this experiment of image classification, we considered a few  popular or highly-ranked NN architectures (as well as their regularization and data augmentation schemes, if applicable), as below:


\noindent(A) For the MNIST dataset:
\begin{itemize} 
\item {\it CNN}:  A simple convolutional neural network with two convolutional layers. The codes for this model are available at \url{https://github.com/pytorch/examples/tree/master/mnist}.
\item {\it LN-S}:  A convolutional neural network called  LeNet5 \citep{lecun1995a} trained with a sparse learning strategy by \cite{dettmers2019a}.
\item  {\it VGG-g}: A deep convolutional neural network (a.k.a., VGG8B) that is trained with  global loss and cutout \citep{devries2017a} regularization. This model is presented by \cite{nokland2019training}.
\end{itemize}
(B)  For the CIFAR-10 dataset:
\begin{itemize}
\item {\it VGG19}: A deep convolutional neural network with 19 layers. The architecture was first discussed by \citep{simonyan2014a}, and the codes for this network were made available by  \cite{li-a}. 
\item {\it shk-RN}: A residual network \citep{he2016a} with a regularization scheme that combines  shake-shake \citep{gastaldi2017a}, cutout \citep{devries2017a}, and mixup \citep{zhang2017a}.  The code for this network were made available by  \cite{li-a}.
\item {\it FMix} \citep{harris-a}: An NN architecture that adopts  a modified mixed sample data augmentation (MSDA).  
\end{itemize}
 
We replaced the training algorithms of the above  NN implementations into Algorithm 1 with $\gamma_{opt}=10^{-6}$, using the outputs of the original implementations as the initial solutions.  
Some heuristic modifications were incorporated into Algorithm 1 in the above replacement: First, the gradient in Algorithm 1 was changed into an unbiased estimator of the gradient constructed on a mini-batch of the whole dataset. The mini-batch sizes remained the same as the original implementations. Second, the values of $\mathcal M$ could be varying over the iterations and were specified to be the multiplicative inverse for the learning rates (a.k.a., step sizes) of the original implementations. Third, $a$,  the parameter in FCP, was always set to be 0.99 times the current value of $\mathcal M^{-1}$ at each iteration (a.k.a., epoch) during the NN training.  {\color{black}\label{ref: lambda determine}\Copy{lambda determine}{Last, the value of   $\lambda$, the other parameter of FCP, was assigned to be $\lambda:=\mathcal C_\lambda\cdot \mathcal U^{-1}$ heuristically, where $\mathcal C_\lambda\geq 0$ was determined as below for each NN: We first randomly selected 10\% of the training data points to construct a balanced validation set. Then, we found the 1st, 1.25th, 2.5th, 5th, 10th, and 15th percentile absolute values of the nonzero fitting parameters in the initial solution. After rounding these percentile values to their first significant digits, the resulting numbers were considered as the candidates for $\mathcal C_\lambda$. From these candidates, we then selected the one  that led to the  best classification result for the validation set, when the NN model was trained on the rest of the training set. 
As it turned out, $\mathcal C_\lambda$ was $1 \times 10^{-2}$, $5 \times 10^{-6}$, and   $2 \times 10^{-4}$, respectively, for CNN-FCP, LN-S-FCP,  and  VGG-g-FCP in the experiments on the MNIST dataset, and $1 \times 10^{-3}$, $3 \times 10^{-2}$, and $1 \times 10^{-3}$, respectively, for  VGG-19-FCP, shk-RN-FCP, and FMix-FCP in the experiments on the CIFAR-10 dataset.}}

The tests in this subsection were implemented using Pytorch \citep{paszke2017a}, and most of the tests were conducted on a single thread on a PC with 40 Intel (R) Xeon (R) E5-2640-v4 CPU cores (2.40 GHz, 64 bits), 128 GB memory, and one Quadro M4000 GPU (8GB memory), except that shk-RN and shk-RN-FCP were implemented using one GPU-enabled  thread on Floydhub, a cloud computing platform with an Intel Xeon CPU (4 Cores), 61GB RAM, and an NVIDIA Tesla K80 GPU (12 GB Memory)  and FMix and FMix-FCP were tested on the same cloud computing platform with different configurations (Intel Xeon CPU with 8 Cores, 61GB RAM, and an NVIDIA Tesla  V100 GPU  with 16 GB Memory).

The out-of-sample classification errors are reported in Tables \ref{table summary 12 NN} and  \ref{table summary 2 NN} for results on MNIST and CIFAR-10, respectively. One may tell from the tables that  the performance of all the NN architectures involved in the test were sharpened by incorporating the proposed FCP regularization. In particular, the best  out-of-sample classification errors achieved by the FCP-regularized schemes for MNIST and CIFAR-10 were 0.23\% and 1.31\%, respectively, both of which were competitive against some  high-performance NNs on the leaderboards (available at \url{https://paperswithcode.com/}), especially if we notice that no external data were used. 

The number of nonzero fitting parameters of the NNs after training with and without the FCP are also reported in Tables \ref{table summary 12 NN} and \ref{table summary 2 NN}. One may observe that the FCP  significantly reduced the number of active fitting parameters. For the case of LN-S, the FCP was able to further reduce the dimensionality on top of the  sparsity-inducing mechanisms in the original model.
\begin{table}[h!]
{\caption{Classification errors of NN variants with and without the FCP on MNIST   dataset.   ``$\langle$Model Name$\rangle$-FCP'' refers to the an FCP-regularized NN.   ``Param \#'' stands for  the number of nonzero fitting parameters after training. ```R.Gap'' standards for the relative gap; that is, the ratio between the difference  and the value obtained before introducing the FCP. \\}\label{table summary 12 NN}
\begin{center}
\begin{threeparttable}
\begin{tabular}{cccccccccccccc}\hline\hline
 Model & CNN & CNN-FCP   &R. Gap  
\\
Test Error&0.80\%&0.70\%& 12.50\%  
\\
Param \#&1,199,882&265,517& 77.87\% 
\\\hline 
 Model &    {LN-S}& LN-S-FCP  & R. Gap
\\
Test Error&   0.66\%&0.64\%& 3.03\%
\\
Param \#& 22,000\tnote{*} &14,417& 34.47\%
\\\hline 
 Model &    {VGG-g} & {VGG-g}-FCP  &R. Gap
\\
Test Error& 0.25\%&{0.23\%} &8.00\%
\\
Param \# &16,853,584&15,115,902  &10.31\%
\\\hline\hline
\end{tabular}
\begin{tablenotes}\footnotesize
\item[*] {The original LN-S model has 431,080 fitting parameters. The  built-in sparsity-inducing mechanisms of the LN-S led to a model with 22,000 nonzero fitting parameters.} 
\end{tablenotes}
\end{threeparttable}
\end{center}
}
\end{table}

\begin{table}[h!]
{\caption{Classification errors of NN variants with and without the FCP on  CIFAR-10 dataset.   ``$\langle$Model Name$\rangle$-FCP'' refers to the an FCP-regularized NN. ``Param \#'' stands for  the number of nonzero fitting parameters after training. ``R.Gap'' standards for the relative gap; that is, the ratio between the difference  and the value obtained before introducing the FCP. \\}\label{table summary 2 NN}
\begin{center}
\begin{tabular}{cccccccccccccc}\hline\hline
 Model & VGG19 & VGG19-FCP &R.Gap &   
\\
Test Error& 6.86\% &6.84\%& 12.50\%  
\\
Param \#&20,051,546& 10,789,567& 46.19\% 
\\\hline 
 Model &    shk-RN&shk-RN-FCP& R.Gap 
\\
Test Error&    2.29\%& 2.16\% &5.67\%
\\
Param \#& 11,932,743&7,303,200& 38.79\%
\\\hline 
 Model &   FMix & FMix-FCP&R.Gap
\\
Test Error& 1.36\%&{1.31\%}&3.68\%
\\
Param \# &  26,422,068 &21,485,594&18.68\%
\\\hline\hline
\end{tabular}
\end{center}
}
\end{table}

%

\section{Conclusion}\label{conclusion: sec}
In this paper, we provide a   theoretical  framework for HDSL under A-sparsity; that is, the high-dimensional learning problems where the vector of the true parameters may be dense but can be  approximated by a sparse vector. We show that, for a   problem of this type, an S$^3$ONC solution for an FCP-based learning formulation yields a poly-logarithmic sample complexity:  the required sample size is only poly-logarithmic in the number of dimensions, even if the common assumption of the RSC is absent.  To compute a solution with the proven sample complexity, we propose a novel, pseudo-polynomial-time  gradient-based algorithm.  

Our results on HDSL under A-sparsity can be applied to the  analysis of two important learning problems that are currently less understood: (i) the nonsmooth HDSL problems, where the  empirical risk functions are not necessarily differentiable; and (ii)  an NN with a flexible choice of the network architectures. We show that for both problems, the incorporation of the FCP regularization can ensure the generalization performance, as measured by the excess risk, to be   insensitive to the increase of the dimensionality.  Particularly, our results indicate that, with regularization, an over-parameterized deep NN can be provably generalizable.

Our numerical results are consistent with our theoretical predictions and point to the interesting potential of combining the proposed FCP with some other recent techniques in further enhancing an NN's   performance. 
For future research, we will extend the results to other regularization schemes. {\color{black}\label{additional assumption on weak sparsity 2}\Copy{weak sparsity discussion}{We will also study how our results can be adapted to the analysis of HDSL under the assumption of weak sparsity \citep{n2012a}.} }



\clearpage
 
\ECSwitch
\ECHead{Appendices}
\section{Additional Results on the Neural Networks}\label{additional theoretical results}
This section of the electronic companion is focused on the generalizability of the neural networks (NN) in binary classification. The problem settings of this classification problem follow   Section \ref{NN training}. Section \ref{local solutions NN training results ReLu} presents a corollary of Theorem \ref{general result theorem regularized NN}, where quantities like $\Omega(s_A)$ are made more explicit. Then Section \ref{computable local solution} presents a suboptimality-independent generalization error bound for a ReLU-NN.
\subsection{Generalizability of   NNs under additional regularities.}\label{local solutions NN training results ReLu}
This subsection presents a corollary of Theorem \ref{general result theorem regularized NN}   under some additional assumptions on the separating function $g$, activation functions, and  the network architecture.   Below we start by introducing those assumptions.

First,  we impose additional  regularities on the separating function $g$ following \cite{mhaskar1996a}. {\label{additional assumption on target function}\Copy{Following Copy mhaskar}{We let  $\mathbf D^{\mathbf k}$ represent the partial derivative with order $\mathbf k=(k_1,\,...,k_d)^\top\geq \mathbf 0$ and $\vert \mathbf k\vert= k_1+...+k_d$; that is, $\mathbf D^{\mathbf k} \widetilde g:=\frac{\partial^{\vert\mathbf k\vert} \widetilde g}{\partial x_1^{k_1},\cdots,\,\partial x_d^{k_d}}$, for a function $\widetilde g$.  Define that
$
\mathbb F_{d,r}:=\left\{ \widetilde g\in\mathbb W^{r,\infty}([-1,1]^d):\,\Vert  \widetilde g\Vert_{\mathbb W^{r,\infty}([-1,1]^d)}\leq 1\right\}.
$
Here  $\mathbb W^{r,\infty}([-1,1]^d)$ is the Sobolev space of functions on $[-1,1]^d$  with  continuous  derivatives with order $\mathbf r$ for all $\mathbf r\in\mathbb Z^d\cap[0,r]^d$, where $\mathbb Z$ is the set of integers. Meanwhile,
$\Vert \widetilde g\Vert_{\mathbb W^{r,\infty}([-1,1]^d)}:=\sum_{\mathbf k\in\mathbb Z^d:\,\mathbf k\in[0,\,r]^d}\,\underset{\mathbf x\in[-1,1]^d}{\text{ess\,sup}}\vert \mathbf D^{\mathbf k} {\widetilde g}(\mathbf x)\vert.$ By this definition, $\mathbb F_{d,r}$ is a fairly flexbile class of functions. The corollary to be presented subsequently is focused on the cases that the separating function $g$ is an element from $\mathbb F_{d,r}$.   An important special case   is  where $g$ is a polynomial.}}

Second, we make the following assumption on the activation functions also following \cite{mhaskar1996a}: 
 {\begin{assumption}\label{activation function condition here}
 \Copy{Let the activation Copy}{Let the activation function $\Psi$  be infinitely many times continuously differentiable in some open interval in $\Re$. Furthermore, $\frac{\partial^k\Psi(z)}{\partial z^{k}}\neq 0$ for some $z$ in that interval, for any integer $k\geq 0$.}
  \end{assumption}
 \Copy{The same assumption Copy}{According to \cite{mhaskar1996a}, commonly adopted activation functions, such as sigmoid, hyperbolic tangent, Gaussian, and multiquadratics, all obey Assumption \ref{activation function condition here}.}
}
 
 \begin{figure}
\begin{center}
\includegraphics[width=0.5\textwidth]{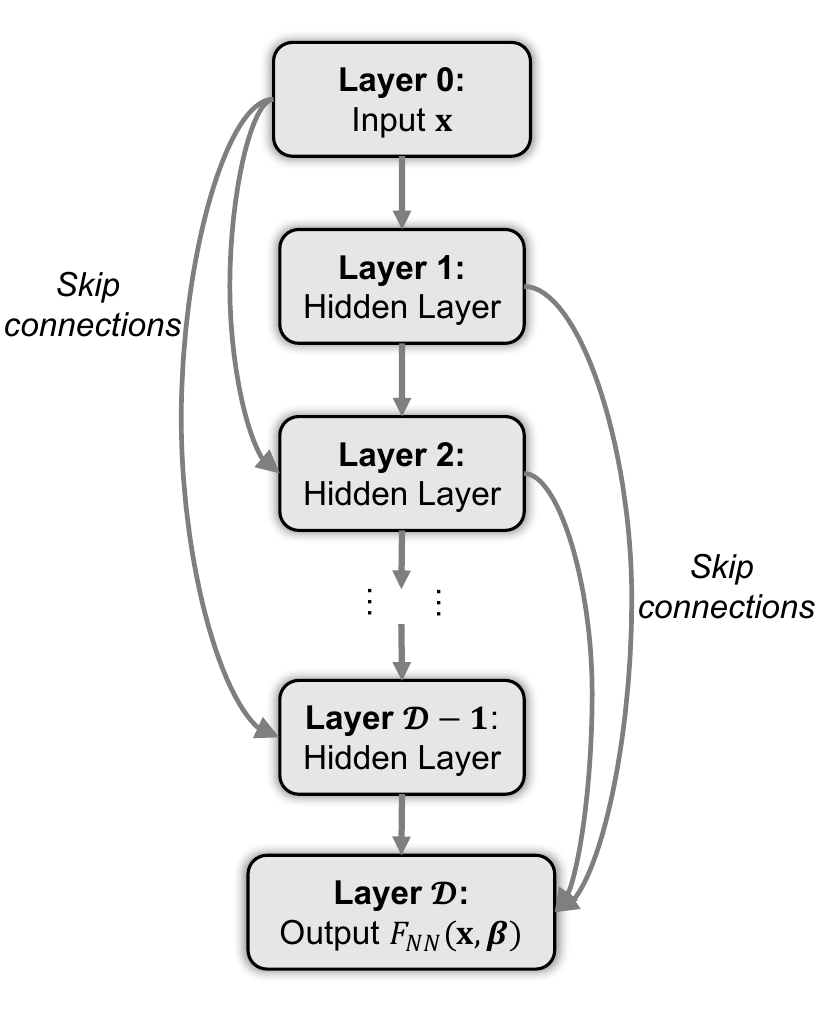}
    \caption{Illustration of the deep NN network. ``Skip connections'' are in presence from the input layer to hidden layers as well as from  hidden layers to the output layer. No nonlinear transform is in presence in the output layer.}\label{skip connection network}
        \end{center}
    \end{figure}
    
Third, for convenience of discussion, {\label{third assumption NN second result}\Copy{we focus on Copy}{we focus on an NN architecture as in Figure \ref{skip connection network}. In this NN, there are  ``skip connections''   from the input layer to the $l$th hidden layer, for all $l=2,...,\mathcal D-1$. Meanwhile, there are also  ``skip connections'' from the $l$ hidden layer, for all $l=1,...,\mathcal D-2$, to the output layer.  We let $\mathcal D$ and $K$ be the network  depth and the number of   neurons in each hidden layer, respectively.}}  Without loss of generality, we assume that all hidden layers have the same number of neurons, and all hidden neurons adopt the same activation function $\Psi$. We also assume that the output layer involves no nonlinear transformation. The output of this NN, given input $\mathbf x$ and fitting parameters $\boldsymbol\beta=vec\left((\mathbf W_{l-1,l}),(\mathbf b_{l-1,l}),\,(\boldsymbol w_{l,\mathcal D}),\,(b_{l,\mathcal D}),\,(\mathbf W_{0,l}),\,(\mathbf b_{0,l})\right)\in\Re^p$, can be captured by the nonlinear system  below, where $f_{NN,l}:\Re^d\times\Re^p\rightarrow\Re^K$ is the output from the $l$th layer.
\begin{align}
F_{NN}(\mathbf x,\boldsymbol\beta)= &\sum_{l=1}^{\mathcal D-1}\left(\boldsymbol w_{l,\mathcal D}^\top f_{NN,l}(\mathbf x,\boldsymbol\beta)+  b_{l,\mathcal D}\right);\label{zero NN equality}
\\
f_{NN,l}(\mathbf x,\boldsymbol\beta)= &\Psi\left( \mathbf W_{l-1,l} f_{NN,l-1}(\mathbf x,\boldsymbol\beta) +\mathbf b_{l-1,l}\right)+  \Psi\left( \mathbf W_{0,l} \mathbf x +\mathbf b_{0,l}\right)
,~~~\forall l=2,...,\mathcal D-1;\label{first NN equality}
\\
f_{NN,1}(\mathbf x,\boldsymbol\beta)=  & \Psi\left( \mathbf W_{0,1} \mathbf x +\mathbf b_{0,1}\right).\label{second NN equality}
\end{align}
  
With the foregoing settings, below is our result on the NN's generalization error.

\begin{corollary}\label{ReLU arch result}     
Let $g\in \mathbb F_{d,r}$.
Consider a   deep neural network $F_{NN}$ defined as in \eqref{zero NN equality}-\eqref{second NN equality}. Suppose that Assumptions \ref{target function assumption}, \ref{new bound results 2.5}, and \ref{activation function condition here} hold.  
Let $\widehat{\boldsymbol\beta}\in\Re^p$ be any random vector such that $\Vert \widehat{\boldsymbol\beta}\Vert_\infty\leq \frac{1}{2}v^{-1}\cdot R_{\Omega}\cdot \ln n$ and  the S$^3$ONC$(\mathbf X,\mathbf y)$ holds at $\widehat{\boldsymbol\beta}$ almost surely. For a fixed $\Gamma\geq 0$, assume that $\mathcal T_{n,\lambda}(\widehat{\boldsymbol\beta})-\inf_{{\boldsymbol\beta}}\mathcal T_{n,\lambda}({\boldsymbol\beta})\leq \Gamma$, w.p.1. Let  $C_7>0$ be a universal constant and  $\mathcal C_{NN}>0$ be some constant that depends only on $d$ and $r$. If $a<\frac{1}{2}\cdot \exp\left\{- \mathcal U_{NN}\cdot {\mathcal D}\cdot \ln\left[p\cdot v^{-1}\cdot  \mathcal U_{NN} \cdot R_{\Omega}\cdot\ln n \right]\right\}$,    $\lambda:=\sqrt{\frac{8\sigma }{c\cdot a\cdot n^{2/3}}\left[\ln(\frac{3e}{2v}\cdot R_{\Omega} p n^{4/3})+\mathcal U_{NN}\cdot {\mathcal D}\cdot \ln\left(\mathcal U_{NN}\cdot (1+n pR_{\Omega}v^{-1})\right)\right]}$, and 
\begin{multline}
n>C_7\cdot \left[\left(\mathcal C_{NN}\cdot v^{-1}\ln n\right)^{3d/r}+ (\Gamma+1)^3\right]
\\+ C_7\cdot\min\left\{\left[ (d+1)\cdot {\mathcal D}\cdot \mathcal U_{NN}\cdot \ln \left(\mathcal U_{NN}\cdot (1+n pR_{\Omega}v^{-1})\right)\right]^{3/2},\right. 
\\\left. \, (d+1)\cdot  K\cdot {\mathcal D}^2\cdot \mathcal U_{NN}\cdot \ln \left(\mathcal U_{NN}\cdot\left(npR_\Omega v^{-1}+1\right)\right) \vphantom{a^{3^3}}\right\},\label{sample initial requirement 3 bound here}
 \end{multline}
 then it holds  that
\begin{multline}
\mathbb E\left[\mathbb 1\left(y\cdot F_{NN}(\mathbf x,\widehat{\boldsymbol\beta})<0\right)\right]
\leq   C_7\cdot  \sqrt{\frac{\Gamma}{n^{1/3}}}+\frac{\mathcal C_{NN}^{1/2}\cdot v^{-1/2}\cdot \sqrt{\ln n}}{\min\left\{n^{\frac{1}{6}+\frac{r}{6d}},\,(K\cdot {\mathcal D})^{\frac{r}{2d}}\cdot n^{\frac{1}{6}}\right\}} +\Gamma
 \\+\mathcal C_{NN}\cdot v^{-1}\cdot \max\left\{n^{-\frac{r}{3d}},\,(K\cdot {\mathcal D})^{-\frac{r}{d}}\right\}\cdot \ln n
+C_7\cdot\frac{(d+1)\cdot {\mathcal D}\cdot \mathcal U_{NN}\cdot \ln\left(\mathcal U_{NN}\cdot (1+n pR_{\Omega}v^{-1})\right)}{n^{1/3}}, \label{ReLU second bound}
\end{multline}
with probability at least $1-C_7 p \exp\left(-\frac{n}{C_7}\right)-C_7\exp\left(-\frac{n^{1/3}}{C_7}\right).$
\end{corollary}

\begin{proof}{Proof.}
See proof in Section \ref{Proof ReLU arch result}.
\hfill \ensuremath{\Box}\end{proof} 
 
\begin{remark} Below are a few remarks on Corollary \ref{ReLU arch result}.
\begin{itemize}
\item
We attain the the poly-logarithmic sample complexity again in this corollary. Similar to Theorem \ref{general result theorem regularized NN}, the generalization error  bound in  \eqref{ReLU second bound} is strictly monotone in the suboptimality gap $\Gamma$.
\item If $g$ is a polynomial function, which is infinitely many times differentiable, and if the network is  over-parameterized with $n\leq (K\mathcal D)^3$, then we may as well let $d=r$ and obtain from \eqref{ReLU second bound} that
\begin{multline}
\mathbb E\left[\mathbb 1\left(y\cdot F_{NN}(\mathbf x,\widehat{\boldsymbol\beta})<0\right)\right]
\leq   O(1)\cdot\frac{(d+1)\cdot {\mathcal D}\cdot \mathcal U_{NN}\cdot \ln\left(\mathcal U_{NN}\cdot (1+n pR_{\Omega}v^{-1})\right)}{n^{1/3}}
\\+O(1)\cdot  \sqrt{\frac{\Gamma}{n^{1/3}}}+\frac{\left(\mathcal C_{NN}^{1/2}+\mathcal C_{NN}\right)\cdot v^{-1}\cdot \ln n}{n^{1/3}} +\Gamma, \nonumber
\end{multline}
with overwhelming probability.
\item {\color{black}{\label{comments on 2nd result 3rd assumption}\Copy{By a closer Copy}{By a closer examination, Corollary \ref{ReLU arch result} is obtained by explicating the misspecification error $\Omega(\cdot)$ in Theorem \ref{general result theorem regularized NN}.  In doing so, we reduce the NN defined as in \eqref{zero NN equality}-\eqref{second NN equality} to a one-hidden-layer subnetwork with $(K\cdot \mathcal D)$-many hidden neurons by assigning 0 to all the connection weights between any pair of hidden layers. We can then use the  existing upper bounds on the misspecification error of a one-hidden-layer NN, such as the results by \cite{mhaskar1996a}, to provide a (conservative) estimate of $\Omega(\cdot)$.  We conjecture that the same argument can be  extendable to  many other NN  architectures, given that  they can represent a one-hidden-layer subnetwork   with $(K\cdot \mathcal D)$-many hidden neurons. Here, we say that one NN (denoted by $F_{NN,1}$) can be represented by another NN (denoted by $F_{NN,2}$), if it holds that,  for any $\boldsymbol\beta_1$ and almost every $\mathbf x\in\mathcal X$,  $F_{NN,1}(\mathbf x,\boldsymbol\beta_1)=F_{NN,2}(\mathbf x,\boldsymbol\beta_2)$ for some $\boldsymbol\beta_2$. Because many NN architectures entail strong representability, we think that Corollary \ref{ReLU arch result} can be used to understand a broader spectrum of NN-based models.}}}
\end{itemize}
\end{remark}

\subsection{A suboptimality-independent generalization bound at tractable local solutions.}\label{computable local solution}

This subsection presents a result on the generalizability of a ReLU-NN at a pseudo-polynomial-time computable solution. Different from the above, the error bound herein is independent of the suboptimality gap $\Gamma$. This is possible under the  following assumption  on the data generation process. 
{ \begin{assumption}\label{Gaussian assumption NN}
\Copy{There exists a constant Copy Assumption 11}{There exists a constant $v\in(0,1)$ and
$$g(\cdot)\in\left\{G(\,\cdot\,):\,G(\mathbf x)=\int_{\Re^d} C_g(\mathbf u)\cdot\max\left\{0,\,\mathbf u^\top\mathbf x\right\}\cdot P(\mathbf u) d\mathbf u:\,\sup_{\mathbf u}\vert C_{g}(\mathbf u)\vert\leq 1\right\},$$
 where $P(\mathbf u)$ is the density of a standard Gaussian vector, such that $y\cdot g(\mathbf x)\geq v$ for all $(\mathbf x,\,y)\in supp(\mathbb D)$.}
\end{assumption}}
{\label{discussion on assumption data generation 3rd result}\Copy{Assumption follows in their analysis}{Assumption \ref{Gaussian assumption NN}  follows Assumption 4.10 by \cite{cao2019a} and Assumption A.1 by \cite{cao2019generalization} in their analysis on the generalization performance of the ReLU-NNs trained with a  stochastic gradient descent (SGD) algorithm. The same assumption is also equivalent to the condition discussed by \cite{rahimi2009weighted}, for some choices of parameters, in analyzing a one-hidden-layer NN. According to  \cite{cao2019a}, Assumption \ref{Gaussian assumption NN} holds for all the functions representable by an infinite-width one-hidden-layer ReLU-NN with a rapidly decaying second-layer weights (faster than $P(\mathbf u)$). Because of the strong representability of an infinite-width ReLU-NN, we think that the set of functions defined in  Assumption \ref{Gaussian assumption NN} is reasonably flexible.}} 

Though our results can be adapted to   facilitate the analysis of a more flexible class of NN architectures,   we  focus on a ReLU-NN architecture $F_{NN}:\,\mathcal X\times\Re^p\rightarrow\Re$ that is in accordance with the following system, given fitting parameters $\boldsymbol\beta=vec\left((\mathbf W_{l-1,l}:\,2\leq l\leq \mathcal D-1),(\mathbf b_{l-1,l}:\,2\leq l\leq \mathcal D-1),\,\boldsymbol w_{\mathcal D-1,\mathcal D},\,\boldsymbol w_{1,\mathcal D},\,,b_{\mathcal D-1,\mathcal D},\mathbf W_{0,1},\mathbf b_{0,1}\right)\in\Re^p$:
\begin{align}
F_{NN}(\mathbf x,\boldsymbol\beta)= &\boldsymbol w_{{\mathcal D}-1,{\mathcal D}}^\top f_{NN,{\mathcal D}-1}(\mathbf x,\boldsymbol\beta)+\boldsymbol w_{1,{\mathcal D}}^\top\Psi\left( \mathbf W_{0,1} \mathbf x +\mathbf b_{0,1}\right)+b_{{\mathcal D}-1,{\mathcal D}};\label{zero NN equality 2}
\\
f_{NN,l}(\mathbf x,\boldsymbol\beta)= &\Psi\left( \mathbf W_{l-1,l} f_{NN,l-1}(\mathbf x,\boldsymbol\beta) +\mathbf b_{l-1,l}\right),~~~\forall l=2,...,{\mathcal D}-1;\label{first NN equality 2}
\\
f_{NN,1}(\mathbf x,\boldsymbol\beta)= &\Psi\left( \mathbf W_{0,1}  \mathbf x +\mathbf b_{0,1}\right).\label{first NN equality 3}
\end{align}
where we let $\Psi(z):=\max\{0,\,z\}$ be the ReLU activation function. The system in \eqref{zero NN equality 2}-\eqref{first NN equality 3} captures a fully-connected ${\mathcal D}$-layer NN (with $\mathcal D-1$ hidden layers),   {where the first hidden layer is connected with the output layer directly through ``skip connections''.\label{additional note on skip connection ReLU-NN}} We  assume that there  are $K$-many neurons in the every hidden layer.

In order to effectively train the above ReLU-NN, we propose the following initialization scheme (Algorithm 2) modified from the Weighted Sums of Random Kitchen Sinks (WSRKS) fitting procedure by \cite{rahimi2009weighted} for training shallow networks.

\vspace{2mm}
\hrule
\vspace{1mm}
{\bf\noindent Algorithm 2. A tractable initialization scheme}
\vspace{1mm}
\hrule
\vspace{2mm}
\begin{description}
\item[Step 0.] 
Specify an integer $K^*:\,1\leq K^*\leq K$. Consider a  subnetwork   in  Figure \ref{initialization scheme}  (where the subnetwork is highlighted in red) of   the complete ReLU-NN \eqref{zero NN equality 2}-\eqref{first NN equality 3}. Denote  this subnetwork by $F^{sub}_{NN}:\,\mathcal X\times \Re^p\rightarrow\Re$, which writes as $F^{sub}_{NN}\left(\mathbf x,(\widetilde{\mathbf W}_{0,1},\,\widetilde{\boldsymbol w}_{1,{\mathcal D}})\right):=\widetilde{\boldsymbol w}_{1,{\mathcal D}}^\top\Psi\left( \widetilde{\mathbf W}_{0,1} \mathbf x \right)$. Here, we let $\widetilde{\mathbf W}_{0,1}=(\omega_{0,1,k,\iota}:\,k=1,...,K^*,\,\iota=1,...,d)\in\Re^{K^{*}\times d}$ and $\widetilde{\boldsymbol w}_{1,{\mathcal D}}=(\omega_{1,{\mathcal D},k}:\,k=1,...,K^*)\in\Re^{K^{*}}$. 
\item[Step 1.] Generate each entry of $\mathbf W^{initial}_{0,1}=\left((\mathbf w_{0,l,k}^{initial})^\top:\,k=1,...,K^*\right)$, independently, from a standard normal distribution $\mathcal N(0,1)$.

\item[Step 2.]  Compute $\boldsymbol w_{1,{\mathcal D}}^{initial}=\left(w^{initial}_{1,{\mathcal D},k}:\,k=1,...,K^*\right)$ by solving the following (convex) optimization problem, where all the entries of $\mathbf W^{initial}_{0,1}$ are fixed to be the values from Step 1:
\begin{align}
{\boldsymbol w}_{1,{\mathcal D}}^{initial}\in&\,\underset{\Vert\widetilde{\boldsymbol w}_{1,{\mathcal D}}\Vert_\infty\leq n}{\arg\,\min}\,\frac{1}{n}\sum_{i=1}^n\mathcal F\left(y_i\cdot F^{sub}_{NN}\left(\mathbf x_i,({\mathbf W}^{initial}_{0,1},\,\widetilde{\boldsymbol w}_{1,{\mathcal D}})\right)\right),\label{initial NN subproblem}
\end{align} 

\item[Step 3.]  Let $\widehat{\boldsymbol\beta}^{initial}\in\Re^p$ be a vector of  fitting parameters. Set the components of $\widehat{\boldsymbol\beta}^{initial}$ that correspond to the subnetwork   to be  $vec(\mathbf W^{initial}_{0,1},\,\boldsymbol w_{1,{\mathcal D}}^{initial})$. Let all other components of $\widehat{\boldsymbol\beta}^{initial}$   be  zero. 
\item[Step 4.] Output $\widehat{\boldsymbol\beta}^{initial}$. 
\end{description}
\vspace{3mm}
\hrule
\vspace{3mm}
 
Algorithm 2  essentially trains the subnetwork constructed in Step 0 of Algorithm 2 with the WSRKS fitting procedure. Meanwhile, all the fitting parameters outside the subnetwork are set to be zero. 
Subsequent to this initialization scheme, we may then invoke Algorithm 1 to generate the desired solution to the FCP-regularized training formulation in \eqref{regularized NN model}.

A subtlety arises when applying Algorithm 1 to  the ReLU-NN. The ReLU activation function $\Psi(z):=\max\{0,\,z\}$ is nonsmooth. Resultantly, the empirical risk function is not everywhere differentiable in general. A common approach in the literature (e.g., \cite{berner2019a}) to avoid this irregularity is to consider a modified  first  derivative of $\Psi$ defined as  $\frac{\partial \Psi(z)}{\partial z}:=\mathbb 1(z>0)$. By this definition, a chain rule is preserved as per \cite{berner2019a}.  Correspondingly, the (modified) gradient can be calculated with the detailed formula  provided in Section \ref{detailed gradient NN}.    We  adopt this modification in Algorithm 1. Despite the use of these modifications, we show that the combination of Algorithms 1 and 2 can lead to a generalizable ReLU-NN within pseudo-polynomial time, and the resulting sample complexity is poly-logarithmic in $p$. Furthermore, the generalization error is independent of $\Gamma$, the suboptimality gap.

Theorem \ref{local solution results} below shows  the promised suboptimality-independent generalization error bound. Note that this theorem  adopts the following settings  and   hyper-parameters:
\begin{align}
\mathcal M>0;\quad{K^*}= \left\lceil 10 n^{1/3}\cdot (\ln n)^{5/3}\right\rceil;\quad a<\frac{1}{\mathcal M};\quad\text{and}\quad\lambda:=\sqrt{\frac{8}{ c\cdot a\cdot n^{2/3}}[\ln(9 en^{4/3}p^{3/2}{\mathcal D})+{\mathcal D}\ln(KR)]},\label{choice of hyper-parameters NN}
\end{align}
where $K^*$ is defined  in Algorithm 2 and $(a,\,\lambda)$ are tuning parameters of the FCP.   For invoking Algorithm 1 in training the ReLU-NN, we let $\widetilde f(\,\cdot\,):=n^{-1}\sum_{i=1}^n\mathcal F\left(y_i\cdot F_{NN}(\mathbf x_i,\,\cdot\,)\right)$  and $\nabla\widetilde f(\,\cdot\,):=n^{-1}\sum_{i=1}^n\widetilde\nabla_{\boldsymbol\beta}\mathcal F\left(y_i\cdot F_{NN}(\mathbf x_i,\,\cdot\,)\right)$ with $\widetilde\nabla_{\boldsymbol\beta}\mathcal F\left(y_i\cdot F_{NN}(\mathbf x_i,\,\cdot\,)\right)$  defined in Section \ref{detailed gradient NN}.     Finally, it is worth noting that the output of Algorithm 1 can be understood as a deterministic (and implicit) function of its initial solution $\boldsymbol\beta^0$ and training data $(\mathbf X,\,\mathbf y)$. When $\boldsymbol\beta^0$, $\mathbf X$, and $\mathbf y$ are random, the algorithm's output is also a random vector.

\begin{theorem}\label{local solution results}Consider the ReLU-NN in \eqref{zero NN equality 2}-\eqref{first NN equality 3} with $K\geq \max\{2,\,d,\,10n^{1/3}\cdot(\ln n)^{5/3}+1\}$.   Suppose that Assumption \ref{Gaussian assumption NN} holds and that $\widehat{\boldsymbol\beta}\in\Re^p$ with $\Vert \widehat{\boldsymbol\beta}\Vert_\infty\leq R$ for some $R\geq n$ is the output of Algorithm 1 when it terminates as per the stopping criterion in \eqref{termination criterion Algorithm 1}. Given hyper-parameters as in \eqref{choice of hyper-parameters NN},  the following statements hold.
 \begin{itemize}
 \item[(a)] For any initial solution $\boldsymbol\beta^0\in\Re^n$ and training data $(\mathbf X,\mathbf y)$, Algorithm 1 terminates at the $k^*(\boldsymbol\beta^{0},\mathbf X,\mathbf y)$-th iteration, for some  integer  $k^*(\boldsymbol\beta^{0},\mathbf X,\mathbf y)< \left(\left\lceil2\mathcal M\cdot \frac{\mathcal T_{n,\lambda}(\widehat{\boldsymbol\beta}^{initial})}{\gamma_{opt}^2}\right\rceil+1\right)$. 
 \item[(b)] 
 Further assume that the initial solution of Algorithm 1   is the output of   Algorithm 2; that is, ${\boldsymbol\beta}^0:=\widehat{\boldsymbol\beta}^{initial}$. At the termination of Algorithm 1, there exists a universal constant $C_8>0$ such that, if
\begin{align}
  n>C_8\cdot  d^{3}\cdot v^{-3}\cdot {\mathcal D}^{3/2}\cdot\left(\ln n\right)^{4}\cdot [\ln(pR)]^{3},\label{sample size initial requirement 1}
  \end{align}
 then, with probability at least $1-C_8 \cdot p  \exp\left(-\frac{n^{1/3}}{C_8}\right)-C_8\cdot n^{1/3} d\exp(-n^2/2)-C_8\cdot(d\cdot n)^{-d/3}$,  the generalization error of the trained ReLU-NN is bounded by
\begin{align}
\mathbb E\left[\mathbb 1\left(y\cdot F_{NN}(\mathbf x,\widehat{\boldsymbol\beta})< 0\right)\right]\leq C_8\cdot   \frac{d \cdot {\mathcal D}}{n^{1/3}v^2}\cdot \left[\left(\ln n\right)^{4/3}\cdot\ln(pR)\right]-\frac{\gamma^2_{opt}}{2\mathcal M}\cdot k^*(\widehat{\boldsymbol\beta}^{initial},\mathbf X,\mathbf y).\label{part b result NN tractable}
\end{align}
 \end{itemize}
 \end{theorem}
\begin{proof}{Proof.} See Section \ref{proof of local solution results}.
\end{proof}

\begin{remark}
In this theorem, the generalization error bound, as measured in terms of the expected 0-1 loss, is no longer dependent on the suboptimality gap $\Gamma$, yet the promised poly-logarithmic sample complexity is maintained; the sample size should grow only poly-logarithmically to compensate for the growth in $p$. In addition, the dependence on the number of layers $\mathcal D$ is polynomial. In contrast to the literature, we argue that  our result here may provide a significantly better rate in terms of both $p$ and ${\mathcal D}$, especially when considering that the training algorithm to ensure the desired sample complexity is provably   in pseudo-polynomial time as per the remark below.
\end{remark}

\begin{remark}
The combination of Algorithms 1 and 2 in Theorem \ref{local solution results} yields a pseudo-polynomial-time complexity. 
\begin{itemize}
\item In the initialization step, Algorithm 2 is a polynomial-time algorithm.  The main computational effort is on solving \eqref{initial NN subproblem}, which is convex and thus in polynomial time. (Note that  an approximate solution to \eqref{initial NN subproblem} with a suboptimality gap of $\mathcal O(\frac{d\cdot \mathcal D\cdot \ln p}{n^{1/3}})$ would actually suffice for deriving the same sample complexity as in Theorem \ref{local solution results}.) 

\item Subsequent to  Algorithm 2,  Algorithm 1   computes a solution that entails the desired sample complexity. The  iteration complexity of Algorithm 1, as proven in Part (a) of Theorem \ref{local solution results}, is polynomial in both the dimensionality and the numeric values of the problem data. Thus, Algorithm 1 yields a pseudo-polynomial-time complexity.  
\end{itemize}

With the above, we know that the total   computational effort of the combined algorithm is in pseudo-polynomial time.
\end{remark}

\begin{remark}
The proof of Theorem \ref{local solution results} does not depend on how the  gradient is defined or modified. Nonetheless, there is some benefit of using the ``modified gradient''  as in Section \ref{detailed gradient NN}, as discussed in Remark \ref{remark on whether algorithm 1 works} below.
\end{remark}

\begin{remark}\label{remark on whether algorithm 1 works}
By a closer examination of the proof, one may notice that Algorithm 2 (invoked for initialization) alone is already capable of identifying a solution with provable generalizability. Nonetheless, as per \eqref{part b result NN tractable}, Algorithm 1    sharpens the generalization error; the more iterations that Algorithm 1 would run for, the shaper is the performance of the trained NN. A natural question  would be whether the initial solution identified by Algorithm 2 would render  the stoping criterion in \eqref{termination criterion Algorithm 1} to be satisfied at the first iteration of Algorithm 1. If so, $k^*(\widehat{\boldsymbol\beta}^{initial},\mathbf X,\mathbf y)=0$ and Algorithm 1 would not be effective.
 We think it to be a possible scenario for some problem instances. However, because $\frac{1}{n}\sum_{i=1}^n\mathcal F(y_i\cdot F_{NN}(\mathbf x_i,\cdot))$ is a piecewise smooth function and Algorithm 2  trains only a small subset of the fitting parameters, it is more likely that the initial solution generated by Algorithm 2  is a non-KKT point within a continuously differentiable neighborhood. In such a case, the ``modified gradient''  as in Section \ref{detailed gradient NN} becomes the exact formulation of the gradient. One may then show that $k^*(\widehat{\boldsymbol\beta}^{initial},\mathbf X,\mathbf y)>0$ must hold, if  $\mathcal M$ is properly large and greater than the Lipschitz constant of the gradient of $\frac{1}{n}\sum_{i=1}^n\mathcal F(y_i\cdot F_{NN}(\mathbf x_i,\cdot))$ for every $\boldsymbol\beta$ in that neighborhood.
\end{remark}

{\begin{remark}\label{remark ref ReLU NN assumption}
{\color{black}\Copy{The results of Theorem Copy}{The results of Theorem \ref{local solution results} is obtained via a similar argument as in proving Theorem \ref{general result theorem regularized NN}, except that the misspecification error $\Omega(\cdot)$ and the suboptimality gap $\Gamma$ in  Theorem \ref{general result theorem regularized NN}  are now explicated in Theorem \ref{local solution results} under the specific assumptions made on the neural network and the data generating process. To make explicit both $\Omega(\cdot)$ and  $\Gamma$, our proofs are largely focused on analyzing the subnetwork constructed in Step 0 of Algorithm 2 and illustrated in  Figure \ref{initialization scheme}.  The misspecification error of this subnetwork serves as a conservative estimate of $\Omega(\cdot)$, and the suboptimality gap obtained after training this subnetwork becomes an overestimate of the initial suboptimality gap to bound $\Gamma$. We conjecture that the above argument can be extended to any NN architecture that contains, or can represent, the above subnetwork. Such NN architectures include the conventional ReLU networks and the  residual networks with ReLU activation, among others.}}
\end{remark}}

 \section{Additional Numerical Experiments}\label{sec Numerical Experiments EC}
 
This part of the electronic companion presents some additional numerical experiments. Sections \ref{SVM test results} and \ref{sec NN test binary classification}  below are focused on a high-dimensional SVM and a ReLU-NN, respectively.

 \subsection{Experiments on high-dimensional SVM}\label{SVM test results}
This section presents our experiments on  high-dimensional  SVM, whose training formulation entails a nonsmooth statistical loss function. For each experimental instance,  a training set and a test set   were randomly generated in two different cases below: (a) The first case involved data with less correlated design. With the same notations as in \eqref{SVM original}, let $\mathbf x_1,\,\mathbf x_2,\,..., \,\mathbf x_n$ be i.i.d.\ samples of $\mathcal N_{p}(\mathbf 0,\Sigma)$   with  $\Sigma=(\varsigma_{j_1,j_2})$ and $\varsigma_{j_1,j_2}=0.3^{\vert j_1-j_2\vert}$. Let the class labels  of the samples $y_i,\,i=1,...,n$, be determined by  
 $y_i=+1$ if  $\mathbf x^\top_i\boldsymbol\beta^{*}+\omega_i\geq 0$, and $y_i=-1$, otherwise.
Here, $\omega_1,\,\omega_2,...,\,\omega_n$ are  i.i.d.\  standard normal random variables {and $\boldsymbol\beta^{*}=(3,\,5,\,0,\,0,\,1.5,\underbrace{0,\,...,0}_{\text{$(p-5)$-many 0's}})^\top$}. We let $n=100$ for both the training and test sets.  (b) In the second case,    data with more correlated design were generated. In doing so, the same approach as in the first case above was followed, except that $\Sigma=(\varsigma_{j_1,j_2})$ was simulated differently. We first calculated $\varsigma_{j_1,j_2}=0.3^{\vert j_1-j_2\vert}$ and then shrank all the singular values of $\Sigma$ below the 80th percentile    to be 0.01 times their original values. 

  Linear classifiers were trained on the training data via three different schemes to be explained subsequently. Their performance was measured by the out-of-sample classification error  on the test data, calculated as $\frac{\text{Number of wrongly classified observations}}{\text{Total number of observations}}\times 100\%$.

Our numerical comparisons involved the following schemes:
(i). {\it SVM:} The canonical SVM  in \eqref{SVM original} with $\rho =0$.  (ii) {\it SVM-$\ell_2$:} The SVM   with $\ell_2$ regularization, that is, the estimator generated by solving \eqref{SVM original} with $\rho >0$. (iii) {\it SVM-$\ell_1$:} The SVM variant with $\ell_1$ regularization, that is, the estimator generated by solving \eqref{SVM} with $\rho =0$, and $\widetilde P_\lambda(\vert \cdot \vert)=\lambda \vert\,\cdot\,\vert$. 
(iv) {\it SVM-FCP:} The SVM variant with the proposed FCP-based regularization, that is, the estimator generated by solving for an S$^3$ONC solution via Algorithm 1 to Problem \eqref{SVM reformulated} with $\rho =0$. {Note that Algorithm 1 in (iv) was initialized with solutions generated by the SVM-$\ell_1$}. Hyper-parameters of Algorithm 1 was specified as $\gamma_{opt}=10^{-5}$ and $\mathcal M=3.5\geq n^{1/4}$.  The SVM, the SVM-$\ell_2$, and the SVM-$\ell_1$  were all solved by calling Mosek  \citep{aps2015a}  through CVX \citep{grant2013a,grant2008a}.

{\Copy{In determining copy}{In determining the hyper-parameters, namely, $\rho$ in the SVM-$\ell_2$, $\lambda$ in the SVM-$\ell_1$ as well as  $\lambda$ in the SVM-FCP (where we fixed the value of $a$, the other tuning parameter of the FCP, to be 0.3),   three training sets with $p\in\{100,\, 500,\,1000\}$ and $n=100$  were  generated  as per the above data generation process in the first case (with less correlated design). On these data sets, the SVM-$\ell_2$, the SVM-$\ell_1$, and the SVM-FCP models were then trained for fixed    hyper-parameters,  $\lambda$ or $\rho$, chosen from     $\{0.05,\,0.1,\,0.15,\,0.20,...,0.4\}$.  The trained SVM variants were then evaluated in terms of their  classification errors  on  three validation sets, one for each value of $p\in\{100,\, 500,\, 1000\}$. These validation sets were generated with  the same sample sizes and probability distributions as the three training datasets above. From the pool of candidate values for $\lambda$ and $\rho$,  the best  ones were chosen in terms of minimizing the average classification errors on the validation sets over all the three cases of $p=100,\,500,\,1000$.   It turned out that $\lambda=0.25$ for both  the SVM-FCP and the SVM-$\ell_1$, and $\rho=0.1$ for the SVM-$\ell_2$.}\label{added description on hyper parameters}}

In testing the impact of dimensionality on the out-of-sample performance of all the four SVM variants,  $p$ was increased gradually with values chosen from $\{100,\,200,\,...,1000\}$.  For each choice of dimensionality, 100 random replications were conducted.  The performance  of each SVM variant is reported in Tables \ref{table summary 123} and  \ref{table summary 123 b}, where we compare  the averages and standard errors of the out-of-sample classification errors for the cases with lower and higher correlations in the design, respectively. From both tables, one can see that the classification errors generated by the proposed SVM-FCP were noticeably better than all other alternative approaches involved in this test. A representation of the  comparisons are provided in the two subplots of Figure \ref{C-b-C comparison}, where the center and radius of each of the  error bars are the average classification error and 1.96 times the corresponding standard error, respectively, from the 100 replications. This figure shows that the SVM-FCP   persistently outperformed the other three SVM variants involved in the test.

\begin{figure}[h!]
\begin{center}
\Copy{figures SVM}{
 \begin{tabular}{cc}
 \includegraphics[width=0.48\textwidth]{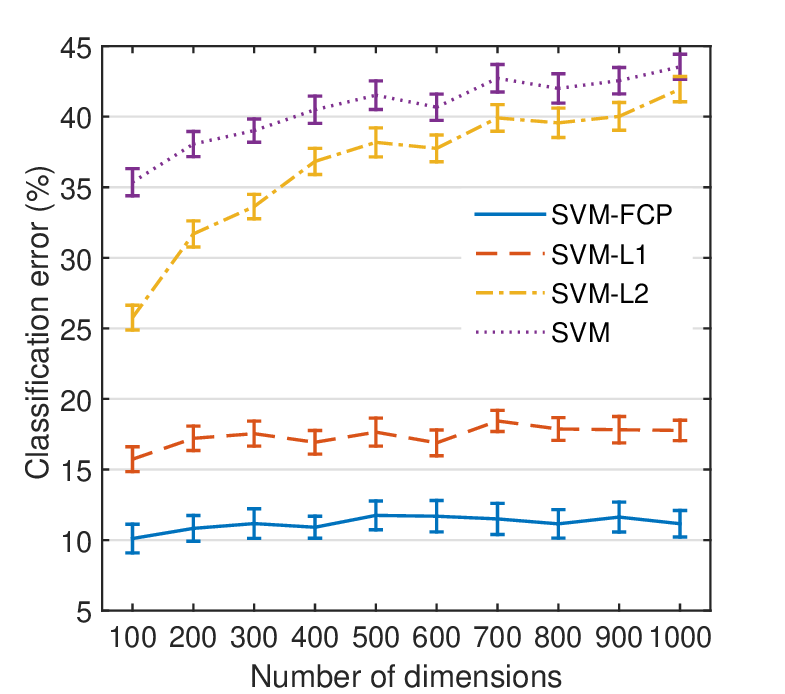}
 &
  \includegraphics[width=0.48\textwidth]{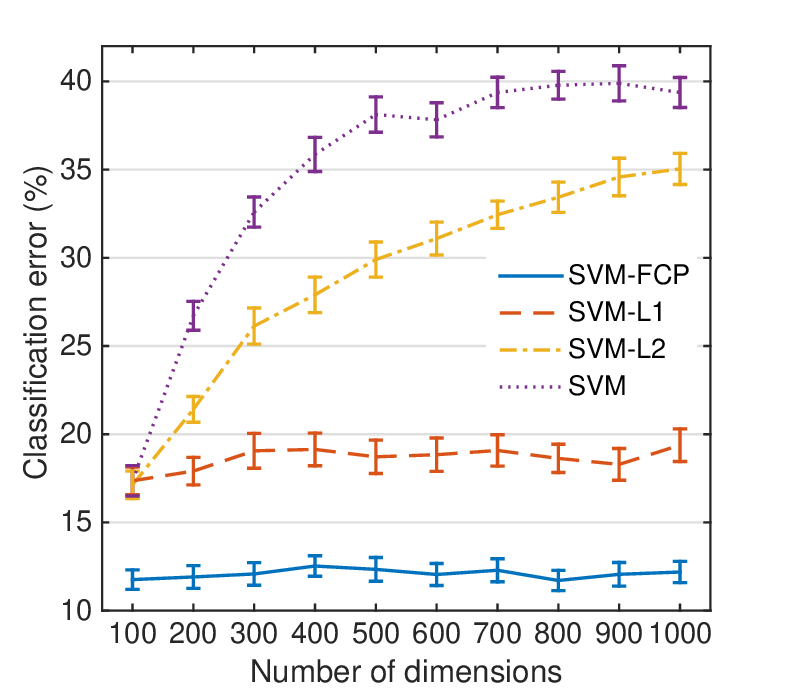}
  \\
   (a)  &(b)  
  \\
  \\
 \end{tabular}}
 \end{center}
\caption{Comparison of classification errors (\%) incurred by the SVM variants. Subplot (a) shows  the case with less correlated design and Subplot (b) presents the case  with more correlated design. ``SVM-FCP'', ``SVM-L1'', ``SVM-L2'', ``SVM'' refers to the SVM variants with the FCP regularization, the $\ell_1$   regularization, the $\ell_2$  regularization, and no regularization, respectively. Centers of error bars are the mean classification errors  out of 100 random replications, and the radius of the error bars are 1.96 times the standard errors. }\label{C-b-C comparison}
\end{figure}
 
\begin{table}[h!]
{\caption{Classification errors of SVM with different regularization schemes when the design has lower correlation. ``Mean'' stands for the average out-of-sample classification error (\%) out of 100 random replications, and ``SE'' is the corresponding standard error (\%).}\label{table summary 123}
\vspace{1mm}
\begin{center}
\begin{tabular}{lcccccccccccc}\hline\hline
 & \multicolumn{2}{c}{SVM-FCP}&& \multicolumn{2}{c}{SVM-$\ell_1$}&& \multicolumn{2}{c}{SVM-$\ell_2$}&& \multicolumn{2}{c}{SVM}
\\\hline
$p$& Mean  & SE && Mean  & SE && Mean  & SE && Mean  & SE
\\\hline 
$100$  & 10.11 & 0.26 && 15.73 & 0.28 && 25.76 & 0.28 && 35.36 & 0.25
\\
$200$  & 10.83 & 0.24 & &17.21 & 0.22 && 31.69 & 0.24 && 38.06 & 0.28
\\
$300$  & 11.17 & 0.28 & &17.54 & 0.28 && 33.63 & 0.22 && 39.01 & 0.21
\\
$400$ & 10.91 & 0.20 && 16.93 & 0.22 && 36.83 & 0.24 && 40.49 & 0.25
\\
$500$ & 11.75 & 0.26 && 17.65 & 0.26 && 38.18 & 0.26 && 41.52 & 0.26 
\\
$600$ & 11.69 & 0.29 && 16.89 & 0.23 && 37.75 & 0.24 && 40.67 & 0.24
\\
$700$ & 11.50 & 0.28 && 18.44 & 0.19 && 39.91 & 0.24 && 42.72 & 0.25 
\\
$800$ & 11.15 & 0.26 && 17.87 & 0.21 && 39.56 & 0.27 && 42.00 & 0.27
\\
$900$ & 11.63 & 0.27 & &17.82 & 0.24 && 40.02 & 0.25 && 42.55 & 0.24 
\\
$1000$ & 11.16 & 0.24 && 17.77 & 0.19 && 41.95 & 0.23 && 43.53 & 0.23 
\\\hline\hline
\end{tabular}
\end{center}
}
\end{table}

\begin{table}[h!]
{\caption{Classification errors of SVM with different regularization schemes when the design has higher correlation. ``Mean'' stands for the average out-of-sample classification error (\%) out of 100 random replications, and ``SE'' is the corresponding standard error (\%).}\label{table summary 123 b}
\vspace{1mm}
\begin{center}
{\color{black}\Copy{additional table}{
\begin{tabular}{lcccccccccccc}\hline\hline
 & \multicolumn{2}{c}{SVM-FCP}&& \multicolumn{2}{c}{SVM-$\ell_1$}&& \multicolumn{2}{c}{SVM-$\ell_2$}&& \multicolumn{2}{c}{SVM}
\\\hline
$p$& Mean  & SE && Mean  & SE && Mean  & SE && Mean  & SE
\\\hline 
100  & 11.76 & 0.55 &  & 17.36 & 0.78 &  & 17.13 & 0.78 &  & 17.36 & 0.86 \\
200  & 11.91 & 0.65 &  & 17.91 & 0.78 &  & 21.41 & 0.73 &  & 26.71 & 0.82 \\
300  & 12.08 & 0.64 &  & 19.06 & 0.99 &  & 26.13 & 1.03 &  & 32.59 & 0.85 \\
400  & 12.53 & 0.58 &  & 19.14 & 0.93 &  & 27.90 & 1.01 &  & 35.86 & 0.97 \\
500  & 12.34 & 0.67 &  & 18.72 & 0.94 &  & 29.90 & 1.00 &  & 38.12 & 1.00 \\
600  & 12.05 & 0.62 &  & 18.84 & 0.94 &  & 31.09 & 0.93 &  & 37.82 & 0.97 \\
700  & 12.29 & 0.65 &  & 19.08 & 0.89 &  & 32.44 & 0.77 &  & 39.37 & 0.86 \\
800  & 11.71 & 0.58 &  & 18.63 & 0.80 &  & 33.43 & 0.86 &  & 39.78 & 0.79 \\
900  & 12.06 & 0.67 &  & 18.29 & 0.90 &  & 34.58 & 1.06 &  & 39.89 & 1.00 \\
1000 & 12.19 & 0.60 &  & 19.38 & 0.92 &  & 35.04 & 0.88 &  & 39.37 & 0.85
\\\hline\hline
\end{tabular}}}
\end{center}
}
\end{table}

\subsection{Numerical Experiments on ReLU-NN in Binary classification}\label{sec NN test binary classification}
This subsection presents  our numerical tests on the efficacy of the FCP-based regularization on a ReLU-NN. 
A training set, a validation set, and a test set  were generated as below:  (A) {\it Training set:}  2000 data were first generated in line with Assumption \ref{Gaussian assumption NN}, where $d=10$ and  $C_g({\bf u}):=\sin({{\sum_{\iota=1}^d {u_\iota}}})/d$ with $\mathbf u=(u_\iota)$. For the given $C_g$,  (the integration involved in defining) the  separating function $g(\mathbf x)$ was evaluated via numerical integration. For each sample data with feature values $\mathbf x_i$, the corresponding (actual) label $y_i$ was set to be +1 if $g(\mathbf x)\geq 0$, and $-1$, otherwise.  Some mislabels were  introduced. Specifically, out of these 2000-many data, a subset of data points was selected as per a Bernoulli distribution; each data point was selected with  probability 0.05.   All the data points in this subset were assigned the wrong labels (opposite to their actual labels calculated previously).
 (B) {\it Validation set:} Following the same approach as the above, we generated another set of 2000 validation data. (C) {\it Test set:}  A set of 5000  independent test data were   generated following Assumption \ref{Gaussian assumption NN}, with the same $d$, $C_g$, and $g$ as the above. However, no  test data was mislabeled. 

We followed  \eqref{zero NN equality 2}-\eqref{first NN equality 3} in constructing the architecture of a $\mathcal D$-layer ReLU-NN model, 
where  the width $K$ (i.e., the number of hidden neurons per hidden layer) was identical across all the hidden layers. 
We employed Algorithm 1, initialized by Algorithm 2, in training the FCP-regularized ReLU-NN formulated in Eq.~\eqref{regularized NN model}.   In choosing the hyper-parameters, we set $a=0.5$ and $\lambda=\mathcal C_{fcp}\cdot \mathcal D\cdot \sqrt{\ln K}$.  Here, $\mathcal C_{fcp}=0.001$ was determined through a process to be detailed subsequently. For Algorithm 1, we let $\gamma_{opt}=10^{-6}$ and $\mathcal M=1$ (such that $a<\frac{1}{\mathcal M}$). For Algorithm 2, ${K^*}= \left\lceil 10 n^{1/3}\cdot (\ln n)^{5/3}\right\rceil$ as per Theorem \ref{local solution results}.  

To determine $\mathcal C_{fcp}$,  three ReLU-NN architectures with 10, 50, and 100 hidden layers and $K=150$ were trained with the combination of Algorithms 1 and 2, when $\mathcal C_{fcp}$ was fixed at each of the candidate values from the set $\{0.0001,\,0.0005, \,0.001,\,0.005, \,0.01,\,0.05,\,0.1\}$.  The performance of these trained ReLU-NNs was evaluated on the validation set in terms of the classification errors. Then, for each candidate value of $\mathcal C_{fcp}$, an classification error  over all  the three NN architectures above was calculated. The value of $\mathcal C_{fcp}$ was chosen to be the one that led to the best average performance. It turned out that  $\mathcal C_{fcp}=0.001$.

Involved as a benchmark in the experiment was  the ReLU-NN model generated by solving the conventional training formulation given as \begin{align}\inf_{\boldsymbol\beta} ~\mathcal T_{n,\lambda}(\boldsymbol\beta):=n^{-1}\sum_{i=1}^n\mathcal F\left(y_i\cdot F_{NN}(\mathbf x_i,\boldsymbol\beta)\right).\label{unpenalized NN train}\end{align}
 In computing a solution to this problem, we employed an  SGD algorithm based on \cite{cao2019generalization}, who have shown the generalizability of the ReLU-NNs trained  by an SGD in spite of the nonconvexity of the formulation. The SGD in our experiment was integrated with a  three-step multi-start strategy: In Step 1, we repeated, for five  times, the training of the same ReLU-NN using the conventional SGD with the   He initialization \citep{he2015a}. Because both the He initialization and the SGD are stochastic, five potentially different local solutions could be generated by Step 1. In Step 2, we  trained the ReLU-NN using the conventional SGD again, but  the initial point was specified as the output of Algorithm 2. Finally, in Step 3, we compared all the  solutions from Steps 1 and 2 and chose the solution with the smallest  objective value (in terms of  \eqref{unpenalized NN train}) as the output of this multi-start strategy. While there could be different strategies in the literature to boost the performance of  the SGD, such as a wise determination of the batch size, the momentum,  and the  learning rate (i.e., the step size), we did not employ those strategies; our purpose was  to compare the non-regularized ReLU-NN formulation in \eqref{unpenalized NN train} with the proposed FCP-regularized ReLU-NN. Thus, given that the SGD   well optimized the problem in \eqref{unpenalized NN train} globally, the performance of the resulting solutions were considered to well represent the efficacy of the non-regularized ReLU-NN.  Indeed, {\label{suboptimality NN non-regu}\Copy{in evaluating  the Copy}{in evaluating  the optimization quality of the SGD, we found that the average, maximal, and minimal objective function values out of all the numerical instances were 0.0013, 0.0052, and 0.0000, respectively. (In contrast, the average  initial objective value of all the SGD runs in this experiment was 19.3101.) In view of the fact that $\inf_{u}\mathcal F(u)\geq 0$, we claim that the global optimal solutions to \eqref{unpenalized NN train} were well approximated, if not always achieved, by the above SGD scheme.}}

\begin{figure}
\begin{center}
\includegraphics[width=0.5\textwidth]{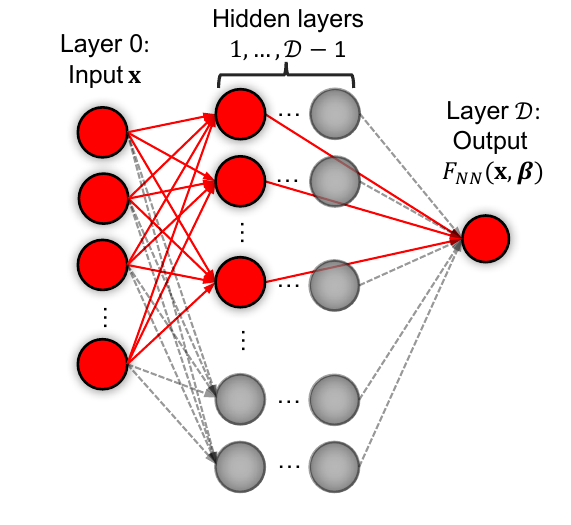}
    \caption{Illustration of the initialization scheme for training the ReLU-NN. The red nodes and links constitute the subnetwork to be initialized via Algorithm 2.}\label{initialization scheme}
        \end{center}
    \end{figure}

\begin{figure}[htbp]
\begin{center}
\begin{tabular}{ c c  }\setlength\extrarowheight{-10pt}
  \includegraphics[width=0.35\textwidth]{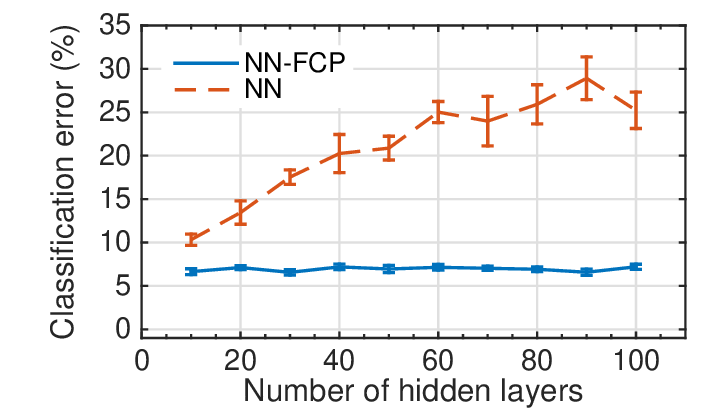} & \includegraphics[width=0.35\textwidth]{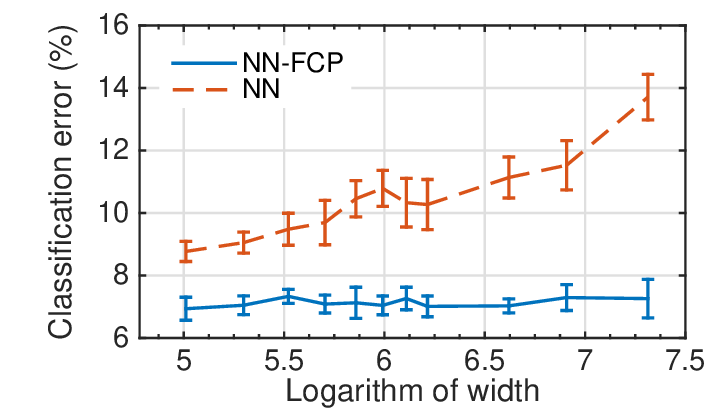}    \\ 
 
 (a)& (b)
 \\
 \includegraphics[width=0.35\textwidth]{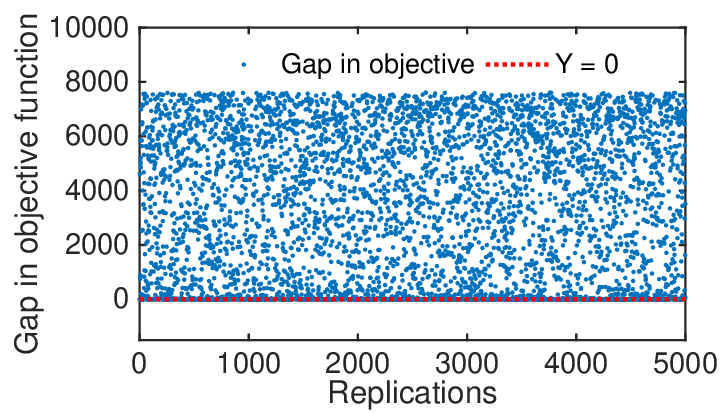} & \includegraphics[width=0.35\textwidth]{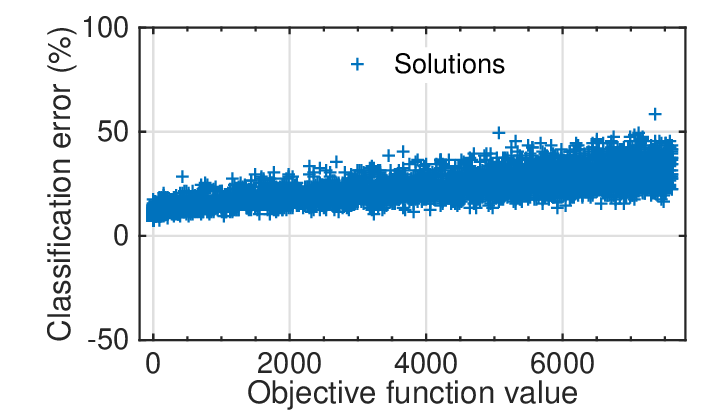}    \\ 
 
 (c)& (d)
 \end{tabular}
\end{center}
    \caption{Numerical evaluation of the FCP-regularized NN (referred to as the NN-FCP): (a) Comparison between the NN-FCP and the non-regularized NN (referred to as the NN) when the number of hidden layers grew; (b) Comparison between  the NN-FCP and the NN when the logarithm of the width of each hidden layer (that is, $\ln K$) increased; (c) Comparison between the objective function values of 5000 randomly generated solutions and that of the solution generated by combining Algorithms 1 and 2; (d) Evaluation of the dependence between in-sample training error and the out-of-sample performance.  All the error bars are centered at the average levels out of 10 random replications and radii are equal to  1.96 times the standard errors. }\label{special NN plot}
\end{figure}

 Our numerical results are presented in Figure \ref{special NN plot}. Some discussions on this figure are as below.
 \begin{itemize}
 \item[(i)] Subplot (a) of Figure \ref{special NN plot} reports the out-of-sample classification errors of the FCP-regularized ReLU-NN and the non-regularized ReLU-NN (referred to as the NN-FCP and the NN, respectively, in the figure) when  the width was fixed at $K=150$, and the number of hidden layers  was chosen from a pool of candidate values $\{10, 20, ..., 150\}$.  For each combination of width and depth, we replicated the experiment for ten  times. The center and the radius of each error bar in the plot are the average classification error and 1.96 times the corresponding standard error out of the ten replications. One can see that the performance of the FCP-regularized ReLU-NN was significantly better than the non-regularized ReLU-NN. Meanwhile, the performance of the former was insensitive to the growth in the depth of the network. This pattern was consistent with  Theorem \ref{local solution results}, but it also may have identified room for   further improvement in terms of the dependence on $\mathcal D$, at least for some regions of the hyper-parameters.
 \item[(ii)] Subplot (b) of Figure \ref{special NN plot} shows the out-of-sample classification errors of the FCP-regularized ReLU-NN  and the non-regularized ReLU-NN  when  the number of hidden layers was fixed to be two  and the width of the hidden layers  was set to be $K\in \{150,\,200,\,250,\,300,\,350,\,400,\,450,\,500,\,750,\,1000,\,1500\}$. Note that the number of fitting parameters $p$ is polynomial in $K$. In order to show the dependence of the generalization performance on $\ln p$, the X-axis of Subplot (b) is on $\ln K$. We can see from this subplot that the performance of the FCP-regularized ReLU-NN remained almost constant as $\ln K$ increased. In contrast, the non-regularized ReLU-NN deteriorated significantly when $\ln K$ became larger.
 \item[(iii)] {\label{discussion on figure suboptimality}\Copy{To show how well the FCP Copy}{To show how well the FCP-regularized ReLU-NN training formulation was optimized in our experiments through the combination of Algorithms 1 and 2, we present in Subplot (c) of Figure \ref{special NN plot} a test on the ReLU-NN with 100 hidden layers and 150 neurons per hidden layer --- the largest network among all the ReLU-NNs involved in (i) and (ii) above.
  For this model, we  generated 5000 random solutions to \eqref{regularized NN model} and compared their objective function values (in terms of \eqref{regularized NN model}) with that of the  solution $\widehat{\boldsymbol\beta}\in\Re^p$ computed by combining Algorithms 1 and 2 as above. The $m$th  (for all $m\in\{1,...,5000\}$) random solution was generated as per the following two-step process: {\bf Step 1.} We generated a random vector $\mathbf v^1_m:=\widehat{\boldsymbol\beta}+\nu_m$, where   $\nu_m\in\Re^p$ was a random sample of a centered Gaussian random vector with i.i.d. entries. The covariance matrix of each $\nu_m$ was prescribed to be $mod(m,\,25)\cdot R_m\cdot I$, where $mod(m,\,25)$ is the remainder of the Euclidean division of $m$ by $25$, $R_m$ denotes a  uniformly distributed random number on (0,\,1), and $I$ stands for the identity matrix. {\bf Step 2}. For all $m=1,...,5000$, we invoked Algorithm 1 to generate a new solution $\mathbf v^2_m\in\Re^p$  using $\mathbf v^1_m$ as the initial point. Here, Algorithm 1 was terminated  whenever either the stopping criterion in \eqref{termination criterion Algorithm 1} was met  ($\gamma_{opt}=10^{-6}$ and $\mathcal M=1$) or a maximal iteration number of 15 was reached. Of all these random solutions, if any could entail a smaller objective   value (w.r.t. the objective function in \eqref{regularized NN model}) than  $\widehat{\boldsymbol\beta}$, then it would mean that $\widehat{\boldsymbol\beta}$ was not the global minimizer.   A blue point in Subplot (c) of Figure \ref{special NN plot} represents one of those random solutions. The corresponding Y-axis of that point indicates  the difference between the objective values of $\mathbf v^2_m$ and $\widehat{\boldsymbol\beta}$. One may observe from the plot that, for all $m=1,...,5000$, the gaps in the objective  were always above zero. This indicates that   $\widehat{\boldsymbol\beta}$   well approximated, if not  coincided with,  a globally minimal solution to \eqref{regularized NN model}.}}
 
 \item[(iv)] In Subplot (d) of Figure  \ref{special NN plot}, we reorganized data from (iii) above to show the correspondence between the in-sample training errors  and the out-of-sample  errors. More specifically, we sorted the random solutions $\mathbf v^2_m$ in the ascending order of their objective values (w.r.t. the objective function in \eqref{regularized NN model}) and showed in this subplot the corresponding out-of-sample  classification errors of those solutions. In the subplot, each blue ``$+$'' represents one of the random solutions $\mathbf v_m^2$. The X- and Y-axis values at the center of each ``$+$'' are the corresponding objective function value and the out-of-sample error, respectively. One may observe that  these ``$+$''s  tend to cluster around an affine function. 
 \end{itemize}
 
Finally, it is worth noting that Algorithm 1 (which was initialized by Algorithm 2) always ran for more than one iteration in all the test instances. If we combine this observation with   Remark \ref{remark on whether algorithm 1 works} about Theorem \ref{local solution results}, we then know that $k^*(\widehat{\boldsymbol\beta}^{initial},\mathbf X,\mathbf y)>0$ (where $k^*(\widehat{\boldsymbol\beta}^{initial},\mathbf X,\mathbf y)$ is defined as in Theorem \ref{local solution results}) and, hence,  Algorithm 1 was indeed effective in our test.

\section{The ``modified gradient'' of the ReLU-NN}\label{detailed gradient NN}
In using Algorithm 1 to train the ReLU-NN of consideration, we follow the commonly adopted definition (e.g., by \cite{berner2019a}) of the (modified) gradient (denoted by $\widetilde\nabla_{\boldsymbol\beta} \mathcal F(y F_{NN}(\mathbf x,\boldsymbol\beta))$) of the training formulation.  In this definition, we denote that $\mathcal H(\mathbf v):=diag\left(\mathbb 1(v_1>0),\,\mathbb 1(v_2>0),...\right)$, for any vector $\mathbf v=(v_1,v_2,...)^\top$. More specifically, we let $\widetilde\nabla_{\boldsymbol\beta} \mathcal F(y F_{NN}(\mathbf x,\boldsymbol\beta)):=\left.\frac{d\mathcal F(t)}{dt}\right\vert_{t=y\cdot F_{NN}(\mathbf x,\boldsymbol\beta))}\cdot y\cdot  \frac{\widetilde dF_{NN}(\mathbf x,\boldsymbol\beta)}{\widetilde d\boldsymbol\beta}$, where  the formula for the components of $\frac{\widetilde dF_{NN}(\mathbf x,\boldsymbol\beta)}{\widetilde d\boldsymbol\beta}=\left(\frac{\widetilde\partial F_{NN}(\mathbf x,\boldsymbol\beta)}{\widetilde\partial \beta_j}:\,j=1,...,p\right)$ are given  below:
$$
\frac{\widetilde\partial F_{NN}(\mathbf x,\boldsymbol\beta)}{\widetilde\partial \boldsymbol w_{\mathcal D-1,\mathcal D}}:= f_{NN,\mathcal D-1} (\mathbf x,\boldsymbol\beta),~~~\frac{\widetilde\partial F_{NN}(\mathbf x,\boldsymbol\beta)}{\widetilde\partial  b_{\mathcal D-1,\mathcal D}}:= \mathbf 1,~~~\text{and}~~~\frac{\widetilde\partial F_{NN}(\mathbf x,\boldsymbol\beta)}{\widetilde\partial \boldsymbol w_{1,\mathcal D}}:=\Psi(\mathbf W_{0,1}\mathbf x+\mathbf b_{0,1}).
$$
Meanwhile, 
\begin{multline}
\frac{\widetilde\partial F_{NN}(\mathbf x,\boldsymbol\beta)}{\widetilde\partial \mathbf W_{l-1,l}}:=   \mathcal H\left(\mathbf W_{l-1,l} f_{NN,l-1}(\mathbf x,\boldsymbol\beta)+\mathbf b_{l-1,l}\right)
\cdot \mathbf W_{l,l+1}^\top\cdot  \mathcal H\left(\mathbf W_{l,l+1} f_{NN,l}(\mathbf x,\boldsymbol\beta)+\mathbf b_{l,l+1}\right)\cdot  ...
\\\cdot  \mathbf W_{\mathcal D-2,\mathcal D-1}^\top\cdot\mathcal H \left(\mathbf W_{\mathcal D-2,\mathcal D-1} f_{NN,\mathcal D-2}(\mathbf x,\boldsymbol\beta)+\mathbf b_{\mathcal D-2,\mathcal D-1}\right)\cdot  \boldsymbol w_{\mathcal D-1,\mathcal D}\cdot\left[ f_{NN,l-1}(\mathbf x,\boldsymbol\beta)\right]^\top,~~~\text{for all $l:\,2\leq l\leq \mathcal D-1$};\end{multline}
\begin{multline}
\frac{\widetilde\partial F_{NN}(\mathbf x,\boldsymbol\beta)}{\widetilde\partial \mathbf b_{l-1,l}}:=   \mathcal H\left(\mathbf W_{l-1,l} f_{NN,l-1}(\mathbf x,\boldsymbol\beta)+\mathbf b_{l-1,l}\right)\cdot \mathbf W^\top_{l,l+1}\cdot  \mathcal H\left(\mathbf W_{l,l+1} f_{NN,l}(\mathbf x,\boldsymbol\beta)+\mathbf b_{l,l+1}\right)\cdot
...\\ \cdot \mathbf W^\top_{\mathcal D-2,\mathcal D-1}\cdot  \mathcal H \left(\mathbf W_{\mathcal D-2,\mathcal D-1} f_{NN,\mathcal D-2}(\mathbf x,\boldsymbol\beta)+\mathbf b_{\mathcal D-2,\mathcal D-1}\right)\cdot  \boldsymbol w_{\mathcal D-1,\mathcal D},~~~\text{for all $l:\,2\leq l\leq \mathcal D-1$};\end{multline}
\begin{multline}
\frac{\widetilde\partial F_{NN}(\mathbf x,\boldsymbol\beta)}{\widetilde\partial \mathbf W_{0,1}}:=  \mathcal H\left(\mathbf W^\top_{0,1}\mathbf x+\mathbf b_{0,1}\right)
\cdot \mathbf W^\top_{1,2}\cdot  \mathcal H\left(\mathbf W^\top_{1,2} f_{NN,1}(\mathbf x,\boldsymbol\beta)+\mathbf b_{1,2}\right)\cdot ...
\\\cdot \mathbf W^\top_{\mathcal D-2,\mathcal D-1}\cdot  \mathcal H \left(\mathbf W^\top_{\mathcal D-2,\mathcal D-1} f_{NN,\mathcal D-2}(\mathbf x,\boldsymbol\beta)+\mathbf b_{\mathcal D-2,\mathcal D-1}\right)\cdot  \boldsymbol w_{\mathcal D-1,\mathcal D}\mathbf x^\top
+   \mathcal H \left(\mathbf W_{0,1} \mathbf x+\mathbf b_{0,1}\right)\cdot \boldsymbol w_{1,\mathcal D}\cdot \mathbf x^\top;
\end{multline}
\begin{multline}
\frac{\widetilde\partial F_{NN}(\mathbf x,\boldsymbol\beta)}{\widetilde\partial \mathbf b_{0,1}}:= \mathcal H\left(\mathbf W_{0,1}\mathbf x+\mathbf b_{0,1}\right)\cdot \mathbf W^\top_{1,2}\cdot  \mathcal H\left(\mathbf W_{1,2} f_{NN,1}(\mathbf x,\boldsymbol\beta)+\mathbf b_{1,2}\right)\cdot ...
\\\cdot \mathbf W^\top_{\mathcal D-2,\mathcal D-1}\cdot  \mathcal H \left(\mathbf W_{\mathcal D-2,\mathcal D-1} f_{NN,\mathcal D-2}(\mathbf x,\boldsymbol\beta)+\mathbf b_{\mathcal D-2,\mathcal D-1}\right)\cdot  \boldsymbol w_{\mathcal D-1,\mathcal D}
+   \mathcal H \left(\mathbf W_{0,1} \mathbf x+\mathbf b_{0,1}\right)\cdot \boldsymbol w_{1,\mathcal D}.
\end{multline}
The above calculation can be conducted via back-propagation. The function $\mathcal F(y F_{NN}(\mathbf x,\cdot))$ is  piecewise continuously differentiable. At points where the gradient is well-defined, the above calculation equals to the gradient exactly.  

  \section{The Applicability of Theorem \ref{nonsmooth case} to the high-dimensional SVM}\label{sec: applicability to SVM}
  This section discusses how Theorem \ref{nonsmooth case} can be used to analyze the generalization performance of SVM.  In particular, we determine here the proper values of $R$, $\sigma$, $\sigma_L$, and $\mathcal C_\mu$ in the instantiation of Assumptions  \ref{SVM assumption 1} and \ref{SVM Lipschitz problem}. We start by introducing a few short-hand notations. Let $\mathbf X=(\mathbf x^\top_i:\,i=1,...,n)$, $\mathbf y=(y_i)$,  
  \begin{align}
  \widetilde{\mathcal L}^{SVM}_{n}(\boldsymbol\beta,(\mathbf X,\mathbf y)):=&\,\rho\Vert \boldsymbol\beta\Vert^2+\frac{1}{n}\sum_{i=1}^n\max_{u_i:\,0\leq u_i\leq 1}\,\left\{u_i\cdot\left(1-y_i\mathbf x_i^\top\boldsymbol\beta \right)\right\},\nonumber
  \\\widetilde{\mathcal L}^{SVM}_{n,\delta}(\boldsymbol\beta,(\mathbf X,\mathbf y)):=&\,\rho\Vert \boldsymbol\beta\Vert^2+\frac{1}{n}\sum_{i=1}^n\max_{u_i:\,0\leq u_i\leq 1}\,\left\{u_i\cdot\left(1-y_i\mathbf x_i^\top\boldsymbol\beta \right)-\frac{(u_i-u_0)^2}{2n^{\delta}}\right\},~~~~~~~~~\text{and}\nonumber
  \\\widetilde{\mathcal L}^{SVM}_{n,\delta,\lambda}(\boldsymbol\beta,(\mathbf X,\mathbf y)):=&\,\rho\Vert \boldsymbol\beta\Vert^2+\frac{1}{n}\sum_{i=1}^n\max_{u_i:\,0\leq u_i\leq 1}\left\{u_i\cdot\left(1-y_i\mathbf x_i^\top\boldsymbol\beta \right)-\frac{(u_i-u_0)^2}{2n^{\delta}}\right\}+\sum_{j=1}^p P_\lambda(\vert \beta_j\vert).\nonumber
  \end{align}
  
  We  first  determine $R$ for the case of SVM. Observe that   $\inf_{\boldsymbol\beta}\mathbb E[L_{ns}(\boldsymbol\beta,Z_i)]\leq \mathbb E[L_{ns}(\mathbf 0,Z_i)]= 1$ and $\boldsymbol\beta^*\in\arg\,\inf_{\boldsymbol\beta}\mathbb E[L_{ns}(\boldsymbol\beta,Z_i)]$.  Recall that we have let $\rho=0.01$. Therefore, $\rho\Vert\boldsymbol\beta^*\Vert^2\leq 1\Longrightarrow \Vert\boldsymbol\beta^*\Vert\leq 10$. If $\boldsymbol\beta^*$ is dense and entails A-sparsity (in Assumption \ref{Assumption A-sparsity original}), there must exist a sparse $\boldsymbol\beta^*_{\varepsilon_A}\in [-10,\,10]^p$ that approximates  $\boldsymbol\beta^*$ in the sense of Assumption \ref{Assumption A-sparsity original}  by the continuity of $\mathbb E[L_{ns}(\,\cdot\,,Z_i)]$. 
Meanwhile, one may also observe that any solution $\widehat{\boldsymbol\beta}$ as defined in Part (b) of Theorem \ref{nonsmooth case} with (where $\widetilde{\mathcal L}_{n,\delta}(\boldsymbol\beta,\mathbf Z_1^n)$ and $\widetilde{\mathcal L}_{n,\delta,\lambda}(\boldsymbol\beta,\mathbf Z_1^n)$ in that theorem become $\widetilde{\mathcal L}^{SVM}_{n,\delta}(\boldsymbol\beta,(\mathbf X,\mathbf y))$ and $\widetilde{\mathcal L}^{SVM}_{n,\delta,\lambda}(\boldsymbol\beta,(\mathbf X,\mathbf y))$, respectively, in the case of SVM) must satisfy that  $\rho\Vert \widehat{\boldsymbol\beta} \Vert^2-1\leq \widetilde{\mathcal L}^{SVM}_{n,\delta,\lambda}(\widehat{\boldsymbol\beta},(\mathbf X,\mathbf y))\leq \widetilde{\mathcal L}^{SVM}_{n,\delta,\lambda}(\widehat{\boldsymbol\beta}^{\ell_1,\delta},(\mathbf X,\mathbf y))\leq \widetilde{\mathcal L}^{SVM}_{n,\delta}(\widehat{\boldsymbol\beta}^{\ell_1,\delta},(\mathbf X,\mathbf y))+\lambda\vert \widehat{\boldsymbol\beta}^{\ell_1,\delta}\vert$, w.p.1., where the last inequality is due to the observation that $P_\lambda(\vert t\vert)\leq \lambda\vert t\vert$ (which is an immediate result of the FCP's definition). Because $\widehat{\boldsymbol\beta}^{\ell_1,\delta}$ is the minimizer to the $\ell_1$-regularized problem, we thus may continue the above as $\rho\Vert \widehat{\boldsymbol\beta} \Vert^2-1 \leq\widetilde{\mathcal L}^{SVM}_{n,\delta}(\widehat{\boldsymbol\beta}^{\ell_1,\delta},\,(\mathbf X,\mathbf y))+\lambda\vert \widehat{\boldsymbol\beta}^{\ell_1,\delta}\vert\leq \widetilde{\mathcal L}^{SVM}_{n}(\widehat{\boldsymbol\beta}^{\ell_1,\delta},\,(\mathbf X,\mathbf y))+\lambda\vert \widehat{\boldsymbol\beta}^{\ell_1,\delta}\vert\leq \left[\widetilde{\mathcal L}_{n}(\boldsymbol\beta,\,(\mathbf X,\mathbf y))+\lambda\vert \boldsymbol\beta \vert\right]_{\boldsymbol\beta=\mathbf 0}\leq  1$, with probability one. Therefore, $\Vert \widehat{\boldsymbol\beta}\Vert\leq \sqrt{200}= 10\sqrt{2}$, a.s.  Thus, $R=10\sqrt{2}$. 

Second, we verify Assumption  \ref{SVM assumption 1} and determine $\sigma$.   Because, with probability one, it holds simultaneously that
\begin{align}\widetilde{\mathcal L}^{SVM}_{n,\delta}(\boldsymbol\beta,(\mathbf X,\mathbf y))\leq& \rho\Vert \boldsymbol\beta\Vert^2+\frac{1}{n}\sum_{i=1}^n\max_{u_i:\,0\leq u_i\leq 1}~\left\{u_i\cdot\left(1-y_i\mathbf x_i^\top\boldsymbol\beta \right)\right\}
\\\leq & \rho\Vert \boldsymbol\beta\Vert^2+ \frac{1}{n}\sum_{i=1}^n\left\{\left(1+\vert\mathbf x_i\vert\Vert\boldsymbol\beta\Vert_\infty \right)\right\}\leq \rho\Vert \boldsymbol\beta\Vert^2+ \frac{1}{n}\sum_{i=1}^n\left\{\left(1+\vert\mathbf x_i\vert\cdot R \right)\right\}\leq 10\sqrt{2}+3,~~a.s.,
\end{align} and 
\begin{align}
\widetilde{\mathcal L}^{SVM}_{n,\delta}(\boldsymbol\beta,(\mathbf X,\mathbf y))\geq&  \rho\Vert \boldsymbol\beta\Vert^2-n^{-1}\sum_{i=1}^n\max_{0\leq u_i\leq 1}\left\{(1+\vert \mathbf x_i\vert\Vert \boldsymbol\beta\Vert_\infty) +\frac{(u_i- u_0)^2}{2n^{\delta}}\right\}
\\\geq& \rho\Vert \boldsymbol\beta\Vert^2-n^{-1}\sum_{i=1}^n\max_{0\leq u_i\leq 1}\left\{(1+\vert \mathbf x_i\vert\cdot R) +\frac{(u_i- u_0)^2}{2n^{\delta}}\right\}\geq -10\sqrt{2}-2,~~a.s.,
\end{align}
 for all $\boldsymbol\beta\in[-R,\,R]^p$.  Thus, the random variable $\widetilde{\mathcal L}^{SVM}_{n,\delta}(\boldsymbol\beta,(\mathbf X,\mathbf y))$ has a bounded support. As an immediate result, $\widetilde{\mathcal L}^{SVM}_{n,\delta}(\boldsymbol\beta,(\mathbf X,\mathbf y))$ is subexponential with $\sigma\leq O(1)$. 
 
Third, we verify Assumption \ref{SVM Lipschitz problem} and  determine $\sigma_L$ and $\mathcal C_\mu$. To that end, we observe that $\widetilde{\mathcal L}^{SVM}_{n,\delta}(\boldsymbol\beta,(\mathbf X,\mathbf y))$ is verifiably Lipschitz continous in $\boldsymbol\beta$. To see this, note that the gradient of the above function w.r.t. $\boldsymbol\beta$ is given as $\nabla_{\boldsymbol\beta}\widetilde{\mathcal L}^{SVM}_{n,\delta}(\boldsymbol\beta,(\mathbf X,\mathbf y))=2\rho\boldsymbol\beta-\frac{1}{n}\sum_{i=1}^n u_i^*y_i\mathbf x_i$, where $u_i^*$, for $i=1,...,n$, is the maximizer to the (inner) maximization problem: $\max_{u_i:\,0\leq u_i\leq 1}~ \left\{u_i\cdot\left(1-y_i\mathbf x_i^\top\boldsymbol\beta \right)-\frac{(u_i- u_0)^2}{2n^{\delta}}\right\}$. The norm of the gradient is bounded from above by $2\rho  R\sqrt{p}+1$, almost surely, for all $\boldsymbol\beta\in[-R,\,R]^p$.  Thus, Assumption \ref{SVM Lipschitz problem} holds with   $\sigma_L=0$ and $\mathcal C_\mu=2\rho  R\sqrt{p}+1=0.2\sqrt{2p}+1\leq O(1)\cdot \sqrt{p}$.

In sum, the FCP-based formulation \eqref{SVM reformulated} for the high-dimensional SVM satisfies both Assumptions \ref{SVM assumption 1} and \ref{SVM Lipschitz problem} with $R\leq O(1)$, $\sigma\leq O(1)$, $\sigma_L=0$, and $\mathcal C_\mu\leq O(1)\sqrt{p}$. Because the generalization error bound in \eqref{nonsmooth excess risk bound detail} is logarithmic in $\mathcal C_\mu$, Theorem \ref{nonsmooth case} can then be applied to show the poly-logarithmic sample complexity for the FCP-regularized SVM. Finally, we would like to remark that some more careful analysis may relax the stipulation on data normalization (such that $\vert\mathbf x_i\vert\leq 1$ a.s.) and improve   the aforementioned quantities.
  
\section{Technical proofs}
 
\subsection{Proof of sample complexities of HDSL under A-sparsity}\label{subsection proof proposition 6}

The proofs for Propositions \ref{second theorem} through \ref{comprehensive proposition} are provided below.  The demonstration of Proposition \ref{second theorem} is an immediate result of Proposition \ref{comprehensive proposition}, which further  relies on Propositions \ref{Proposition 1} through \ref{bound dimensions}.
\begin{proof}{Proof of Proposition \ref{second theorem}.}
Invoking   Proposition \ref{Proposition 1} under the assumption that $a<{U_L}^{-1}$, we have that $\mathbb P\left[\left\{\text{$\vert\widehat\beta_j\vert\notin(0,\,a\lambda)$ for all $j$}\right\}\right]=1$. This,  combined with Proposition \ref{comprehensive proposition}, yields the desired result.  \hfill \ensuremath{\Box}
\end{proof}
\begin{newproposition}\label{Proposition 1}

Suppose that  $a<{U_L}^{-1}$. For any random vector $\widehat{\boldsymbol\beta}\in\Re^p$ such that  $\widehat{\boldsymbol\beta}\in\Re^p:\,\Vert \widehat{\boldsymbol\beta}\Vert_\infty\leq R$ and the {\it S$^3$ONC}$(\mathbf Z_1^n)$  is satisfied at $\widehat{\boldsymbol\beta}$ almost surely. Then,
$$\mathbb P\left[\left\{\text{$\vert\widehat\beta_j\vert\notin(0,\,a\lambda)$ for all $j$}\right\}\right]=1.$$
\end{newproposition}
\begin{proof}{Proof.}  
 Since $\widehat{\boldsymbol\beta}$  satisfies the {\it S$^3$ONC}$(\mathbf Z_1^n)$  almost surely, Eq.\,\eqref{second condition 1} implies that, for any $j\in\{1,...,p\}:\, \vert\widehat\beta_j\vert\in(0,\,a\lambda)$, it holds that
$
0\leq\,U_L+P''_\lambda(\vert\widehat\beta_j\vert)
=  U_L-\frac{1}{a},$
which, combined with the fact that   $\frac{\partial^2 P_\lambda(t)}{\partial t^2}=-a^{-1}$ for $t\in(0,\,a\lambda)$, contradicts with the assumption that $U_L<\frac{1}{a}$.  The above contradiction implies that $\mathbb P[\{\widehat{\boldsymbol\beta}\text{ satisfies the S$^3$ONC$(\mathbf Z_1^n)$}\}\cap\{\vert\widehat\beta_j\vert\in(0,\,a\lambda)\}]=0\Longrightarrow 0\geq 1-\mathbb P[\{\widehat{\boldsymbol\beta}\text{ does not satisfy the S$^3$ONC$(\mathbf Z_1^n)$}\}]-\mathbb P[\{\vert\widehat\beta_j\vert\notin(0,\,a\lambda)\}]$. Since  $\mathbb P[\{\widehat{\boldsymbol\beta}\text{ satisfies the S$^3$ONC$(\mathbf Z_1^n)$}\}]=1$, it holds that $\mathbb P[\{\vert\widehat\beta_j\vert\notin(0,\,a\lambda)\}]=1$ for all $j=1,...,p$, which immediately leads to the desired result.
  \hfill \ensuremath{\Box}\end{proof}


\begin{newproposition}\label{test proposition important} Suppose that Assumptions \ref{sub exponential condition} and \ref{Lipschitz condition} hold.
 Let $\epsilon\in(0,\, 1]$, ${p'}:\,{p'}> s$,  $\zeta_1(\epsilon):=\ln\left(\frac{3 \cdot {(\sigma_L+\mathcal  C_\mu)}\cdot p\cdot eR}{\epsilon}\right)$, and  $\mathcal B_{{p'},R}:=\left\{\boldsymbol\beta\in\Re^{p}:\, \Vert\boldsymbol\beta\Vert_\infty\leq R,\, \left\Vert\boldsymbol\beta \right\Vert_0\leq {p'} \right\}.$  Then, for the same $c\in(0,\,0.5]$ as in \eqref{Bernstein inequality result} and for some  universal constant $\widetilde c>0$,
$$
\sup_{\boldsymbol\beta\in \mathcal B_{{p'},R}}\left\vert\frac{1}{n}\sum_{i=1}^n L(\boldsymbol\beta,Z_i)-\mathbb   L(\boldsymbol\beta)\right\vert
\leq  {\frac{\sigma }{\sqrt{n}}}\sqrt{\frac{2 {p'}}{c}\zeta_1(\epsilon)} +\frac{\sigma }{n}\cdot\frac{2 {p'}}{c}\zeta_1(\epsilon)+\epsilon$$ 
with probability at least $1-2\exp\left(- {p'}\zeta_1(\epsilon)\right)-2\exp(-\widetilde cn)$.
\end{newproposition}
\begin{proof}{Proof.}  
We  follow the ``$\epsilon$-net''  argument as discussed by \citet{vershynin2012a} and   \citet{shapiro2014a} to construct a net of discretization grids $\mathcal S(\epsilon):=\{\widetilde{\boldsymbol\beta}^k\}\subseteq \mathcal B_{{p'},R} $ such that for any $\boldsymbol\beta\in \mathcal B_{{p'},R} $, there is  $\boldsymbol\beta^k\in \mathcal S(\epsilon)$ that satisfies $\Vert\boldsymbol\beta^k-\boldsymbol\beta\Vert\leq \frac{\epsilon}{2\sigma_L+2\mathcal  C_\mu}$  for any fixed $\epsilon\in(0,\,1]$. 

To that end, we first consider a fixed index set $\mathcal I\subseteq\{1,...,p\}:\,\vert \mathcal I\vert ={p'}$ and an arbitrary  $\boldsymbol\beta\in\mathcal B_{{p'},R}\cap\{\boldsymbol\beta\in\Re^p:\,\beta_j=0,\,\forall j\notin   \mathcal I\}$.  To ensure  that there always exists $\widetilde{\boldsymbol\beta}^k \in \mathcal S(\epsilon)$ such that
\begin{align}\left\Vert\boldsymbol\beta_{\mathcal I}-\widetilde{\boldsymbol\beta}^k_{\mathcal I}\right\Vert\leq \frac{\epsilon}{ (2\sigma_L+2\mathcal  C_\mu)},~~~~\text{with~~ $\boldsymbol\beta_{\mathcal I}=(\beta_j:\,j\in\mathcal I)$ and $\widetilde{\boldsymbol\beta}^k_{\mathcal I}=(\widetilde\beta^k_j:\,j\in\mathcal I)$},\label{sample grids new}
\end{align}
it is sufficient to have a covering number of no more than
$\left(\left\lceil\frac{2 {(\sigma_L+\mathcal  C_\mu)p'R}}{\epsilon}\right\rceil\right)^{{p'}}$.
Now we consider how to cover all ${p'}$-dimensional subspaces by enumerating all possible $\mathcal I\subseteq \{1,...,p\}:\,\vert \mathcal I\vert={p'}$.  For each  $\mathcal I$,  an $\epsilon$-net with $\left(\left\lceil\frac{2 {(\sigma_L+\mathcal  C_\mu)Rp'}}{\epsilon}\right\rceil\right)^{{p'}}$-many grids can be constructed   to ensure \eqref{sample grids new} and there could be $p\choose{{p'}}$-many possible choices of $\mathcal I$'s. Therefore, to guarantee the existence of  $\boldsymbol\beta^k\in \mathcal S(\epsilon)$ that satisfies $\Vert\boldsymbol\beta^k-\boldsymbol\beta\Vert\leq \frac{\epsilon}{{2\sigma_L+2\mathcal  C_\mu}}$  for any fixed $\epsilon\in(0,\,1]$ and $\boldsymbol\beta\in\mathcal B_{{p'},R}$, it is sufficient to let $\vert \mathcal S(\epsilon)\vert:
={{p}\choose{ {p'}}}\left(\left\lceil\frac{ {{p'}\cdot(2\sigma_L+2\mathcal  C_\mu)R}}{\epsilon}\right\rceil\right)^{{p'}}$.
We notice that $\frac{{ {p'}} {(\sigma_L+\mathcal  C_\mu)R}}{\epsilon}\geq1$ and thus 
$\left\lceil\frac{{ {p'}}\cdot {(2\sigma_L+2\mathcal  C_\mu)R}}{\epsilon}\right\rceil\leq \frac{{ {p'}}\cdot {(2\sigma_L+2\mathcal  C_\mu)R}}{\epsilon}+1
  \leq \frac{3{ {p'}}\cdot{(\sigma_L+\mathcal  C_\mu)} R}{\epsilon}.$
 Therefore, $\vert \mathcal S(\epsilon)\vert\leq\left(\frac{3 \cdot {(\sigma_L+\mathcal  C_\mu)}  peR}{\epsilon}\right)^{ {p'}}$ due to ${{p}\choose{{p'}}}\leq \left(\frac{ pe}{{p'}}\right)^{{p'}}$ and, further invoking union bound and De Morgan's Law, it holds that
\begin{align}
&\mathbb P\left[\max_{\boldsymbol{\beta}^k\in \mathcal S(\epsilon)}\left\vert\frac{1}{n}\sum_{i=1}^n L(\boldsymbol\beta^k,Z_i)-\mathbb E\left[\frac{1}{n}\sum_{i=1}^n L(\boldsymbol\beta^k,Z_i)\right]\right\vert
\leq \sigma \sqrt{\frac{ t}{n}}
+\frac{\sigma t}{n} \right]  \nonumber
\\=&\mathbb P\left[\bigcap_{\boldsymbol{\beta}^k\in \mathcal S(\epsilon)}\left\{\left\vert\frac{1}{n}\sum_{i=1}^n L(\boldsymbol\beta^k,Z_i)-\mathbb E\left[\frac{1}{n}\sum_{i=1}^n L(\boldsymbol\beta^k,Z_i)\right]\right\vert
\leq \sigma \sqrt{\frac{ t}{n}}
+\frac{\sigma t}{n} \right\}\right]  \nonumber
\\\geq&1-\sum_{\boldsymbol{\beta}^k\in \mathcal S(\epsilon)}\mathbb P\left[\left\vert\frac{1}{n}\sum_{i=1}^n L(\boldsymbol\beta^k,Z_i)-\mathbb E\left[\frac{1}{n}\sum_{i=1}^n L(\boldsymbol\beta^k,Z_i)\right]\right\vert
> \sigma \sqrt{\frac{ t}{n}}
+\frac{\sigma t}{n}\right].  \nonumber
\end{align}
Further invoking the Bernstein-type inequality for a subexponential distribution as mentioned in Remark \ref{remark beinstein}, for  $c$ is as in \eqref{Bernstein inequality result}, it holds that
\begin{multline}
\mathbb P\left[\max_{\boldsymbol{\beta}^k\in \mathcal S(\epsilon)}\left\vert\frac{1}{n}\sum_{i=1}^n L(\boldsymbol\beta^k,Z_i)-\mathbb E\left[\frac{1}{n}\sum_{i=1}^n L(\boldsymbol\beta^k,Z_i)\right]\right\vert
\leq \sigma \sqrt{\frac{ t}{n}}
+\frac{\sigma t}{n} \right] \geq 1-\vert \mathcal S(\epsilon)\vert\cdot2\exp(-ct) 
\\ \geq 1-2\left(\frac{3 \cdot {(\sigma_L+\mathcal  C_\mu)}\cdot  p eR}{\epsilon}\right)^{ {p'}} \cdot\exp(-ct).\label{to use here}
\end{multline}
Furthermore, in view of Lemma  \ref{new lemma result},  it holds that
\begin{multline}\left\vert \frac{1}{n}\sum_{i=1}^n L(\boldsymbol\beta,Z_i) -\frac{1}{n}\sum_{i=1}^n L(\boldsymbol\beta^k,Z_i) \right\vert  +\left\vert \mathbb E\left[\frac{1}{n}\sum_{i=1}^n L(\boldsymbol\beta,Z_i)\right] -\mathbb E\left[\frac{1}{n}\sum_{i=1}^n L(\boldsymbol\beta^k,Z_i) \right]\right\vert
\leq 2(\sigma_L+\mathcal C_{\mu})\Vert \boldsymbol\beta-\boldsymbol\beta^k\Vert,
\end{multline}
 with probability at least $1-2\exp(-\widetilde c\cdot n)$ for some universal constant $\widetilde c>0$. Therefore, for any $\boldsymbol\beta\in \mathcal B_{{p'},R}$ and $\boldsymbol\beta^k\in \mathcal S(\epsilon)$,  it holds with the same probability that
\begin{multline}
 \left\vert\frac{1}{n}\sum_{i=1}^n L(\boldsymbol\beta,Z_i)-\mathbb E\left[\frac{1}{n}\sum_{i=1}^n L(\boldsymbol\beta,Z_i)\right]\right\vert   
\leq \left\vert\frac{1}{n}\sum_{i=1}^n L(\boldsymbol\beta^k,Z_i)-\mathbb E\left[\frac{1}{n}\sum_{i=1}^n L(\boldsymbol\beta^k,Z_i)\right]\right\vert  
\\+ \left\vert \frac{1}{n}\sum_{i=1}^n L(\boldsymbol\beta,Z_i) -\frac{1}{n}\sum_{i=1}^n L(\boldsymbol\beta^k,Z_i) \right\vert  +\left\vert \mathbb E\left[\frac{1}{n}\sum_{i=1}^n L(\boldsymbol\beta,Z_i)\right] -\mathbb E\left[\frac{1}{n}\sum_{i=1}^n L(\boldsymbol\beta^k,Z_i) \right]\right\vert
\\\leq 2(\sigma_L+\mathcal C_\mu)\Vert\boldsymbol\beta-\boldsymbol\beta^k\Vert+ \left\vert\frac{1}{n}\sum_{i=1}^n L(\boldsymbol\beta^k,Z_i)-\mathbb E\left[\frac{1}{n}\sum_{i=1}^n L(\boldsymbol\beta^k,Z_i)\right]\right\vert.\label{to minimize right}
\end{multline}
Combining the above  with \eqref{to use here}, we obtain  that
 \begin{align}
 \left\vert\frac{1}{n}\sum_{i=1}^n L(\boldsymbol\beta,Z_i)-\mathbb E\left[\frac{1}{n}\sum_{i=1}^n L(\boldsymbol\beta,Z_i)\right]\right\vert   
\leq2(\sigma_L+\mathcal C_\mu)\Vert\boldsymbol\beta-\boldsymbol\beta^k\Vert+ \sigma \sqrt{\frac{ t}{n}}
+\frac{\sigma t}{n} \nonumber
\end{align}
with probability at least $1-2\left(\frac{3 \cdot {(\sigma_L+\mathcal  C_\mu)}\cdot  p eR}{\epsilon}\right)^{ {p'}} \cdot\exp(-ct)-2\exp(-\widetilde c\cdot n)$.  Always picking the closest $\boldsymbol\beta^k$ to $\boldsymbol\beta$, we have, in view of \eqref{sample grids new}, for any $\epsilon:\, 0<\epsilon\leq 1$:
$
\mathbb P\left[\max_{\boldsymbol\beta\in \mathcal B_{{p'},R}}\left\vert\frac{1}{n}\sum_{i=1}^n L(\boldsymbol\beta,Z_i)-\mathbb E\left[\frac{1}{n}\sum_{i=1}^n L(\boldsymbol\beta,Z_i)\right]\right\vert
\leq \sigma   \sqrt{\frac{t}{n}}
+\frac{\sigma t}{n}+ \epsilon  \right]
\geq 1-2\cdot\exp(-ct)\cdot \left(\frac{3 \cdot {(\sigma_L+\mathcal  C_\mu)}\cdot  p\cdot eR}{\epsilon}\right)^{ {p'}}
-2\exp(-\widetilde cn).$
Further letting $t:=\frac{2 {p'}}{c}\zeta_1(\epsilon)$, where we recall that $\zeta_1(\epsilon):=\ln\left(\frac{3 \cdot  (\sigma_L+\mathcal  C_\mu)\cdot p\cdot eR}{\epsilon}\right)$, we then obtain the desired result.
    \hfill \ensuremath{\Box}\end{proof}


\begin{newproposition}\label{bound dimensions}
  Let $\Gamma\geq 0$, $\epsilon\in(0,\, 1]$,  and $
\zeta_1(\epsilon):=\ln\left(\frac{3 \cdot {({\sigma }_C+\mathcal  C_\mu)}\cdot p\cdot eR}{\epsilon}\right)$. Suppose that Assumptions \ref{Assumption A-sparsity}, \ref{sub exponential condition} and \ref{Lipschitz condition} hold. Consider any random vector $\widehat{\boldsymbol\beta}=(\widehat\beta_j:\,j=1,...,p)\in\Re^p$ such that $\Vert\widehat{\boldsymbol\beta}\Vert_\infty\leq R$ and $\left\vert\widehat\beta_j\right\vert\notin(0,\,a\lambda)$, for all $j$, almost surely, and
 \begin{align}
& \mathcal L_{n,\lambda}(\widehat{\boldsymbol\beta},\mathbf Z_1^n)\leq \mathcal L_{n,\lambda}({\boldsymbol\beta}_{\varepsilon_{A}}^*,\mathbf Z_1^n)+\Gamma,~~w.p.1.\label{epsilon delta}
\end{align}
For any fixed positive integer ${p'_u}:\,{p'_u}>s$, if 
\begin{align}
({p'}-s)\cdot P_\lambda(a\lambda) > \frac{4\sigma }{cn}\zeta_1(\epsilon)\cdot  {p'}
+{\frac{2\sigma  }{\sqrt{n}}}\sqrt{\frac{2{p'}}{c}\zeta_1(\epsilon)}+\Gamma+2\epsilon+\varepsilon_{A},\label{condition on value lambda 1}
\end{align}
for all ${p'}:\,  {p'_u}\leq {p'}\leq p$,  then 
$
\mathbb P[\Vert\widehat{\boldsymbol\beta} \Vert_0\leq {p'_u}-1]\geq \,1-2 p \exp(-\widetilde cn)-4\exp\left(- {p'_u}\zeta_1(\epsilon)\right)\nonumber
$ for the same $c$ in \eqref{Bernstein inequality result} and some $\widetilde c>0$.
\end{newproposition}
\begin{proof}{Proof.}   
Let $\mathcal B_{R}:=\{\boldsymbol\beta\in\Re^p:\,\Vert\boldsymbol\beta\Vert_{\infty}\leq R\}$.
Consider an arbitrary ${p'}:\,  {p'_u}\leq{p'}\leq p$. Since   ${p'}>s$ by the assumption that ${p'_u}>s$, {we may consider the following sets:} 
\begin{align}\mathcal E_{\Gamma}^1:=&\,\left\{(\widetilde{\boldsymbol\beta},\,\widetilde{\mathbf Z}_1^n)\in\mathcal B_{R}\times \mathcal W^n:\,\mathcal L_{n,\lambda}(\widetilde{\boldsymbol\beta},\,\widetilde{\mathbf Z}_1^n)\leq \mathcal L_{n,\lambda}({\boldsymbol\beta}^*_{\varepsilon_{A}},\,\widetilde{\mathbf Z}_1^n)+\Gamma\right\},\nonumber
\\
\mathcal E_{{p'}}^2:=&\,\left\{ \widetilde{\boldsymbol\beta}\in\mathcal B_{R}:\,\Vert\widetilde{\boldsymbol\beta} \Vert_0={p'}\right\}, \nonumber
\\
\mathcal E_{{p'}}^3:=&\,\left\{\widetilde{\mathbf Z}_1^n\in\mathcal W^n:\,\sup_{\boldsymbol\beta\in\mathcal B_{R}:\,\Vert \boldsymbol\beta\Vert_0\leq {p'}}\left\vert\mathcal L_n(\boldsymbol\beta,\widetilde{\mathbf Z}_1^n)-\mathbb  L (\boldsymbol\beta)\right\vert\leq  {\frac{\sigma }{\sqrt{n}}}\sqrt{\frac{2 {p'}}{c}\zeta_1(\epsilon)}+\frac{\sigma \zeta_1(\epsilon)}{n} \frac{2 {p'}}{c}+\epsilon\right\},\nonumber
\\\mathcal E^4:=&\,\{\widetilde{\boldsymbol\beta}\in\mathcal B_{R}:\, \text{$\vert\widetilde\beta_j\vert\notin(0,\,a\lambda)$ for all $j$}\}.\nonumber
\end{align}
Note that     $\widetilde{\boldsymbol\beta}\in \mathcal E_{{p'}}^2\cap \mathcal E^4$, which means that  $\widetilde{\boldsymbol\beta}$ has     ${p'}$-many   nonzero dimensions and the absolute value for each nonzero dimension must not be within the interval $(0,\,a\lambda)$.  Then, for all $(\widetilde{\boldsymbol\beta},\widetilde{\mathbf Z}_1^n)\in\{(\widetilde{\boldsymbol\beta},\widetilde{\mathbf Z}_1^n)\in \mathcal E_{\Gamma}^1\}\cap\{\widetilde{\boldsymbol\beta}\in \mathcal E^4\cap \mathcal E_{{p'}}^2\}$, where $\widetilde{\mathbf Z}_1^n=(\widetilde  Z_1,...,\widetilde Z_n)$, it holds that
\begin{align}
\frac{1}{n}\sum_{i=1}^n L(\widetilde{\boldsymbol\beta},\widetilde Z_i) +{p'}P_\lambda(a\lambda)\leq \frac{1}{n}  \sum_{i=1}^n L({\boldsymbol\beta^*_{\varepsilon_{A}}},\widetilde Z_i) +s P_\lambda(a\lambda)+\Gamma.\label{first to combine here}
\end{align}

Since $\boldsymbol\beta^*_{\varepsilon_{A}}\in\mathcal B_R:\,\Vert\boldsymbol\beta^*_{\varepsilon_{A}}\Vert_0=s<{p'}$, we may obtain that, for all $\widetilde{\boldsymbol\beta} \in   \mathcal E_{{p'}}^2$, 
\begin{align}
&\,\frac{1}{n}\sum_{i=1}^n L(\boldsymbol\beta^*_{\varepsilon_{A}},\widetilde Z_i) -\frac{1}{n}\sum_{i=1}^n L(\widetilde{\boldsymbol\beta},\widetilde Z_i) \nonumber
\\=&\, \left[\frac{1}{n}\sum_{i=1}^n L(\boldsymbol\beta^*_{\varepsilon_{A}},\widetilde Z_i) -\mathbb  L({\boldsymbol\beta^*_{\varepsilon_{A}}})  \right]+ \left[\mathbb  L (\widetilde{\boldsymbol\beta})  
-\frac{1}{n}\sum_{i=1}^n L(\widetilde{\boldsymbol\beta},\widetilde Z_i) \right]
+  \mathbb   L ({\boldsymbol\beta^*_{\varepsilon_{A}}} ) -\mathbb  L(\widetilde{\boldsymbol\beta})  \nonumber
\\\leq& \, 2\sup_{\boldsymbol\beta\in   \mathcal E_{{p'}}^2} \left\vert\frac{1}{n}\sum_{i=1}^n L(\boldsymbol\beta,\widetilde Z_i)-\mathbb   L(\boldsymbol\beta )  \right\vert +\mathbb   L({\boldsymbol\beta^*_{\varepsilon_{A}}}) -\mathbb  L(\widetilde{\boldsymbol\beta}) \nonumber
\\\leq &\, 2\sup_{\boldsymbol\beta\in   \mathcal E_{{p'}}^2} \,\,\left\vert\frac{1}{n}\sum_{i=1}^n L(\boldsymbol\beta,\widetilde Z_i)-\mathbb  L({\boldsymbol\beta})  \right\vert+\varepsilon_{A}, \label{resolution inequality}
\end{align}
where the last inequality is due to $L_g^*\leq \mathbb   L({\boldsymbol\beta^*_{\varepsilon_{A}}})$ and $\mathbb   L({\boldsymbol\beta^*_{\varepsilon_{A}}}) -L_g^*\leq \varepsilon_{A} \Longrightarrow \mathbb   L({\boldsymbol\beta^*_{\varepsilon_{A}}}) -\mathbb  L({\boldsymbol\beta})\leq \varepsilon_{A}$ for all $\boldsymbol\beta\in\mathcal B_R$.


 For any   ${p'}:\, {p'_u}\leq {p'}\leq p$, if we suppose that $ \emptyset\neq\{(\widetilde{\boldsymbol\beta},\widetilde{\mathbf Z}_1^n)\in\mathcal E_{\Gamma}^1\}\cap\{\widetilde{\boldsymbol\beta} \in\mathcal E_{{p'}}^2\cap \mathcal E^4\}\cap\{\widetilde{\mathbf Z}_{1}^n\in \mathcal E_{{p'}}^3\}$, then   \eqref{first to combine here},  \eqref{resolution inequality} and the definition of $\mathcal E_{{p'}}^3$ together would imply that $
({p'}-s)\cdot P_\lambda(a\lambda)\leq {\frac{2\sigma  }{\sqrt{n}}}\sqrt{\frac{2 {p'}}{c}\zeta_1(\epsilon)}
 +\frac{4\sigma }{n}\frac{ {p'}}{c}\zeta_1(\epsilon)+2\epsilon+\Gamma+\varepsilon_{A}$, which contradicts with the assumed inequality \eqref{condition on value lambda 1}.  Therefore, under the assumption that \eqref{condition on value lambda 1} holds and   $\widehat{\boldsymbol\beta}$  satisfies the {\it S$^3$ONC}$(\mathbf Z_1^n)$ with probability one,
 \begin{align}
 0=&\,\mathbb P\left[\{(\widehat{\boldsymbol\beta},{\mathbf Z}_1^n)\in\mathcal E_{\Gamma}^1\}\cap\{\widehat{\boldsymbol\beta} \in\mathcal E_{{p'}}^2\cap \mathcal E^4\}\cap\{\mathbf Z_1^n\in\mathcal E_{{p'}}^3\}\right]\nonumber
 \\\geq &\,1-\mathbb P\left[ \widehat{\boldsymbol\beta} \notin \mathcal E_{{p'}}^2  \right]-\mathbb P\left[\mathbf Z_1^n \notin \mathcal E_{{p'}}^3\right]-\left\{1-\mathbb P\left[(\widehat{\boldsymbol\beta},{\mathbf Z}_1^n)\in\mathcal E_{\Gamma}^1,\, \widehat{\boldsymbol\beta} \in \mathcal E^4 \right]\right\},
 \label{last inequality to obtain}
 \end{align} 
 for all ${p'}:\, {p'_u}\leq {p'}\leq p$. Now, invoke Proposition \ref{Proposition 1}, $\mathbb P\left[(\widehat{\boldsymbol\beta},{\mathbf Z}_1^n)\in\mathcal E_{\Gamma}^1,\, \widehat{\boldsymbol\beta} \in \mathcal E^4\right]=1$, since $\widehat{\boldsymbol\beta}$ satisfies both the S$^3$ONC$(\mathbf Z_1^n)$ and \eqref{epsilon delta} with probability one. Therefore, \eqref{last inequality to obtain} implies that,  for all ${p'}:\, {p'_u}\leq {p'}\leq p$,
$\mathbb P\left[\mathbf Z_1^n \notin \mathcal E_{{p'}}^3 \right]\geq \mathbb P\left[ \widehat{\boldsymbol\beta} \in \mathcal E_{{p'}}^2 \right].$
Consequently,   
$\mathbb P[\Vert \widehat{\boldsymbol\beta}\Vert_0= {p'}]\leq 1-\mathbb P\left[ \mathbf Z_1^n \in \mathcal E_{{p'}}^3\right]$
for all ${p'}:\, {p'_u}\leq {p'}\leq p$. This, combined with Proposition \ref{test proposition important}, yields that
\begin{align}
&\mathbb P[\Vert\widehat{\boldsymbol\beta} \Vert_0\leq {p'_u}-1]=\mathbb P[\Vert\widehat{\boldsymbol\beta} \Vert_0\notin\{{p'_u},\, {p'_u}+1,..., p\}] \nonumber
\\= &1-\mathbb P\left[\bigcup_{{p'}={p'_u}}^{p}\{\Vert\widehat{\boldsymbol\beta} \Vert_0={p'}\}\right]\geq 1- \sum_{{p'}={p'_u}}^{p} \mathbb P[\Vert\widehat{\boldsymbol\beta} \Vert_0={p'}]
\geq  1-\sum_{{p'}={p'_u}}^{p}\left(1-\mathbb P\left[\mathbf Z_1^n \in \mathcal E_{{p'}}^3\right]\right) \nonumber
\\\geq & 1-\sum_{{p'}={p'_u}}^{p} 2\exp\left(- {p'}\cdot\zeta_1(\epsilon)\right)  -2({p}-{p'_u}+1)\exp(-\widetilde cn).\nonumber
\end{align}
where $\widetilde c>0$ is the same constant as in Proposition \ref{bound dimensions}. Observing that $\zeta_1(\epsilon)=\ln\left(\frac{3 \cdot  (\sigma_L+\mathcal  C_\mu)\cdot p\cdot eR}{\epsilon}\right)>0$ (since $p>s>1$, $R,\,\sigma_L,\, \mathcal C_{\mu}\geq 1$, and $\epsilon\leq 1$) and $\sum_{{p'}={p'_u}}^{p} 2\exp\left(- {p'}\cdot\zeta_1(\epsilon)\right)$ is the sum of a geometric sequence, we have
\begin{align}
\mathbb P[\Vert\widehat{\boldsymbol\beta} \Vert_0\leq {p'_u}-1]\geq & 1-\frac{2\exp\left(- {p'_u}\zeta_1(\epsilon)\right)}{1- \exp\left(-\zeta_1(\epsilon)\right)}-2 {p} \exp(-\widetilde cn). \label{starting point}
\end{align}
The above can be  simplified into
$
\mathbb P[\Vert\widehat{\boldsymbol\beta} \Vert_0\leq {p'_u}-1]  \geq \,1-4\exp\left(- {p'_u}\zeta(\epsilon)\right) -2 p \exp(-\widetilde cn)$.    
  \hfill \ensuremath{\Box}\end{proof}

\bigskip

Proposition \ref{comprehensive proposition} below uses the short-hand notation that $\widetilde\zeta:=\ln\left( 3eR\cdot(\sigma_L+\mathcal  C_\mu)\right)$.
\begin{newproposition}\label{comprehensive proposition} 
Let $a<1$. Suppose that Assumptions \ref{Assumption A-sparsity}, \ref{sub exponential condition}, and \ref{Lipschitz condition} hold. For any $\varrho:\,0<\varrho< \frac{1}{2}$, let     $\lambda:=\sqrt{\frac{8\sigma }{c\cdot a\cdot n^{2\varrho}}[\ln(n^{\varrho}p)+\widetilde\zeta]}$ with the same $c$ in \eqref{Bernstein inequality result}.      Consider any random vector $\widehat{\boldsymbol\beta}=(\widehat\beta_j:\,j=1,...,p)\in\Re^p$ such that $\Vert \widehat{\boldsymbol\beta}\Vert_\infty\leq R$ and $\vert\widehat\beta_j\vert\notin (0,a\lambda)$ for all $j$ almost surely:
\begin{itemize}
\item[(i)]
For any fixed $\Gamma\geq 0$ and some    universal constants $\widetilde c,\,C_1>0$, if
\begin{align}
n > C_1\cdot \left[\left(\frac{\Gamma+\varepsilon_{A}}{\sigma }\right)^{\frac{1}{1-2\varrho}}+s\cdot\left(\ln(n^{\varrho}p)+\widetilde \zeta\right)\right],\label{sample initial requirement 2 ori test new}
\end{align}
 and   $\mathcal L_{n,\lambda}(\widehat{\boldsymbol\beta},\,\mathbf Z_1^n)\leq \mathcal L_{n,\lambda}({\boldsymbol\beta}^{*}_{\varepsilon_{A}},\,\mathbf Z_1^n)+\Gamma$ almost surely, then 
$
\mathbb P[\mathcal E_a\cap \mathcal E_b]\geq 
1-2 (p+1) \exp(-\widetilde cn)-6\exp\left(-2cn^{4\varrho-1}\right)
$, where the events $\mathcal E_a$ and $\mathcal E_b$ are defined as 
\begin{align}\mathcal E_a=\left\{\Vert \widehat{\boldsymbol\beta}\Vert_0\leq \left\lceil\frac{2cn^{4\varrho-1}}{\ln(n^{\varrho}p)+\widetilde \zeta}+\frac{2cn^{2\varrho}}{\sigma\cdot \left(\ln(n^{\varrho}p)+\widetilde \zeta\right)}\cdot\left(\Gamma+\varepsilon_{A}+\frac{2}{n^{\varrho}}\right)+8s\right\rceil\right\};~~~~ \text{and}\label{sparsity result pp5}\end{align}
\begin{multline}
\mathcal E_b:=\left\{\mathbb L(\widehat{\boldsymbol\beta})-L_g^*
\leq C_1\cdot\left(\frac{s\cdot\left(\ln(n^{\varrho}p)+\widetilde \zeta\right)}{n^{2\varrho}}+\sqrt{\frac{s\cdot\left(\ln(n^{\varrho}p)+\widetilde \zeta\right)}{n}}+\frac{1}{n^{\varrho}}+\frac{1}{n^{1-2\varrho}}+\frac{1}{n^{(1-\varrho)/2}}\right)\cdot \sigma \right.\\\left.+C_1\cdot\sqrt{\frac{\sigma (\Gamma+\varepsilon_{A})}{n^{1-2\varrho}}}+\Gamma+\varepsilon_{A}\right\}.\label{bound global}
\end{multline}
\item[(ii)]
For  some    universal constants $\widetilde c,\,C_2>0$, if
\begin{align}
n > C_2\cdot\left(\frac{\varepsilon_{A}}{\sigma }\right)^{\frac{1}{1-2\varrho}}+ C_2\cdot a^{-1}\cdot [\ln(n^{\varrho}p)+\widetilde\zeta]\cdot s^{\max\{1,\frac{1}{2-4\varrho},\,\frac{1}{2\varrho}\}}\cdot \left(\max\left\{1,\, \Vert\boldsymbol\beta^*_{\varepsilon_{A}} \Vert_\infty\right\} \right)^{\max\{\frac{1}{2-4\varrho},\,\frac{1}{2\varrho}\}},\label{sample initial requirement proposition in appendix}
\end{align}
 and   $\mathcal L_{n,\lambda}(\widehat{\boldsymbol\beta},\,\mathbf Z_1^n)\leq \mathcal L_{n,\lambda}(\widehat{\boldsymbol\beta}^{\ell_1},\,\mathbf Z_1^n)$ almost surely, then  
\begin{multline}
\mathbb L(\widehat{\boldsymbol\beta})-L_g^* 
\leq C_2\cdot\left[\frac{s\left(\ln(n^{\varrho}p)+\widetilde \zeta\right)}{n^{2\varrho}}+\frac{1}{n^{\varrho}}+\frac{1}{n^{1-2\varrho}}\right]
\cdot \sigma 
\\+C_2\cdot \frac{s\cdot \max\left\{1,\,\Vert\boldsymbol\beta^*_{\varepsilon_{A}} \Vert_\infty\right\}\cdot \sigma ^{3/4}}{\min\left\{a^{1/2}n^\varrho,\,a^{1/4}n^{\frac{1-\varrho}{2}}\right\}} \left[\ln(n^{\varrho}p)+\widetilde \zeta\right]^{1/2}
+C_2\cdot\sqrt{\frac{\sigma \varepsilon_{A}}{n^{1-2\varrho}}}+\varepsilon_{A},
\end{multline}
with probability at least 
$
1-2 (p+1) \exp(-\widetilde cn)-6\exp\left(-2cn^{4\varrho-1}\right)  
$.
\end{itemize}
\end{newproposition}
\begin{proof}{Proof.}We  denote by $c_0,\,c_1,\,c_2,...$  potentially different universal constants throughout this proof.

To show Part (i), let $\epsilon:=\frac{1}{n^{\varrho}}\in(0,\, 1]$,  and $
\zeta_1(\epsilon):=\ln\left(\frac{3 \cdot {({\sigma }_C+\mathcal  C_\mu)}\cdot p\cdot eR}{\epsilon}\right)=\ln(n^{\varrho} p)+\widetilde\zeta>0$.   Then    $\lambda=\sqrt{\frac{8\sigma \zeta_1(\epsilon)}{c\cdot a\cdot n^{2 \varrho}}}=\sqrt{\frac{8\sigma }{c\cdot a\cdot n^{2\varrho}}[\ln(n^{\varrho} p)+\widetilde\zeta]}$. 
We first invoke Proposition \ref{bound dimensions} to bound the sparsity level of $\widehat{\boldsymbol\beta}$. To that end, we need to derive an explicit form for ${p'_u}$ as defined in that proposition. Let $T_1:=2P_\lambda(a\lambda)-\frac{8\sigma }{cn}\zeta_1(\epsilon)$. We may explicate $p'_u$ by solving the following inequality (where $P_X$ is the unknown), which is equivalent to  \eqref{condition on value lambda 1} of Proposition \ref{bound dimensions} with   $p':=P_X$,
 \begin{align}
\frac{T_1}{2}\cdot P_X-\frac{2\sigma  }{\sqrt{n}}\sqrt{\frac{2P_X\zeta_1(\epsilon)}{c}}>\Gamma+2\epsilon+sP_\lambda(a\lambda)+\varepsilon_{A},\label{initial inequality}
\end{align}
for   the same $c\in(0,\, 0.5]$ in \eqref{Bernstein inequality result}.
The solution to the above inequality yields that
$
\sqrt{P_X}> \frac{2\sigma  }{T_1\sqrt{n}}\sqrt{\frac{2\zeta_1(\epsilon)}{c}}+
\frac{\sqrt{\frac{2(2\sigma  )^2\cdot\zeta_1(\epsilon)}{cn}+2T_1[\Gamma+\varepsilon_{A}+2\epsilon+sP_\lambda(a\lambda)]}}{T_1}.$
Since we aim only to find a feasible $P_X$, we may as well require that
$
{P_X}> \frac{32\sigma ^2\zeta_1(\epsilon)}{cT_1^2\cdot n}+8T_1^{-1}[\Gamma+\varepsilon_{A}+2\epsilon+sP_\lambda(a\lambda)].$
For  $\lambda=\sqrt{\frac{8\sigma \zeta_1(\epsilon)}{c\cdot a\cdot n^{2\varrho}}}$,  we have $P_\lambda(a\lambda)=\frac{a\lambda^2}{2}=\frac{4\sigma \zeta_1(\epsilon)}{c\cdot n^{2\varrho}}$. Further noticing  that $2P_\lambda(a\lambda)=\frac{8\sigma \zeta_1(\epsilon)}{c\cdot n^{2\varrho}}> \frac{4\sigma  \zeta_1(\epsilon)}{c\cdot n^{2\varrho}}+\frac{8\sigma }{nc}\zeta_1(\epsilon)$ as per our assumption (i.e., \eqref{sample initial requirement 2 ori test new} implies that $n^{1-2\varrho}> 2$) we therefore know that $T_1=2P_\lambda(a\lambda)- \frac{8\sigma }{nc}\zeta_1(\epsilon)>\frac{4\sigma \zeta_1(\epsilon)}{c\cdot n^{2\varrho}}$. As a result,  to satisfy  \eqref{condition on value lambda 1} of Proposition \ref{bound dimensions}, it suffices to let ${p'_u}$ be any integer that satisfies 
$
 {p'_u}
\geq \frac{2cn^{4\varrho-1}}{\zeta_1(\epsilon)}+\frac{2cn^{2\varrho}}{\sigma \zeta_1(\epsilon)}\cdot\left[\Gamma+\varepsilon_{A}+2\epsilon+sP_\lambda(a\lambda)\right],$
which is satisfied by specifying
\begin{align}
{p'_u}
:=& \left\lceil\frac{2cn^{4\varrho-1}}{\zeta_1(\epsilon)}+\frac{2cn^{2\varrho}}{\sigma \zeta_1(\epsilon)}\cdot\left(\Gamma+\varepsilon_{A}+2\epsilon\right)+8s\right\rceil  = \left\lceil\frac{2cn^{4\varrho-1}}{\zeta_1(\epsilon)}+\frac{2cn^{2\varrho}}{\sigma \zeta_1(\epsilon)}\cdot\left(\Gamma+\varepsilon_{A}+\frac{2}{n^{\varrho}}\right)+8s\right\rceil,\label{assignment of pu new}
\end{align}
hereafter in this proof. (Here the last equality is due to  our choice of parameter, $\epsilon={\frac{1}{n^{\varrho}}}$.)
In the meantime, the right-hand-side of \eqref{assignment of pu new} is strictly larger than $s$.
 Since \eqref{assignment of pu new} is a sufficient condition to \eqref{initial inequality}, we know that, if \eqref{assignment of pu new} holds, then \eqref{condition on value lambda 1} in Proposition \ref{bound dimensions} holds for all ${p'}:\, {p'_u}\leq {p'}\leq p$.
Invoking Proposition \ref{bound dimensions}, we have   with probability at least
$
 \, 1-4\exp\left(-\left\lceil\frac{2cn^{4\varrho-1}}{\zeta_1(\epsilon)}+\frac{2cn^{2\varrho}}{\sigma \zeta_1(\epsilon)}\cdot\left(\Gamma+\varepsilon_{A}+\frac{2}{n^{\varrho}}\right)+8s\right\rceil\cdot \zeta_1(\epsilon)\right) -2 p \exp(-\widetilde cn),$ it holds that $\Vert\widehat{\boldsymbol\beta} \Vert_0\leq {p'_u}-1=\left\lceil\frac{2cn^{4\varrho-1}}{\zeta_1(\epsilon)}+\frac{2cn^{2\varrho}}{\sigma \zeta_1(\epsilon)}\cdot\left(\Gamma+\varepsilon_{A}+\frac{2}{n^{\varrho}}\right)+8s\right\rceil-1$. 

In view of the assumption that $\mathcal L_{n,\lambda}(\widehat{\boldsymbol\beta},\,{\mathbf Z}_1^n)\leq \mathcal L_{n,\lambda}({\boldsymbol\beta}_{\varepsilon_{A}}^*,\,{\mathbf Z}_1^n)+\Gamma$, w.p.1., together with Assumption \ref{Assumption A-sparsity} and the fact that $P_\lambda(\vert\cdot\vert)\geq 0$, we know that
$\frac{1}{n}\sum_{i=1}^n L(\widehat{\boldsymbol\beta},Z_i) \leq  \frac{1}{n}  \sum_{i=1}^n L({\boldsymbol\beta^*_{\varepsilon_{A}}},Z_i) +s P_\lambda(a\lambda)+\Gamma, \,a.s.\Longrightarrow\left\{\frac{1}{n}\sum_{i=1}^n L(\widehat{\boldsymbol\beta},Z_i)-\mathbb E[\frac{1}{n}\sum_{i=1}^n L(\widehat{\boldsymbol\beta},Z_i)]\right\}+\mathbb E[\frac{1}{n}\sum_{i=1}^n L(\widehat{\boldsymbol\beta},Z_i)] \leq  \left\{\frac{1}{n}  \sum_{i=1}^n L({\boldsymbol\beta^*_{\varepsilon_{A}}},Z_i) -\mathbb E[\frac{1}{n}  \sum_{i=1}^n L({\boldsymbol\beta^*_{\varepsilon_{A}}},Z_i) ]\right\}+\mathbb E[\frac{1}{n}  \sum_{i=1}^n L({\boldsymbol\beta^*_{\varepsilon_{A}}},Z_i) ]+s P_\lambda(a\lambda)+\Gamma, \,a.s.$   
Given the event  $\mathcal E^1\cap \mathcal E^2$, where 
\begin{align}
\mathcal E^1:= 
\left\{ \sup_{\boldsymbol\beta\in \mathcal B_{{p'_u},R}}\left\vert\frac{1}{n}\sum_{i=1}^n L(\boldsymbol\beta,Z_i)-\mathbb E\left[\frac{1}{n}\sum_{i=1}^n L(\boldsymbol\beta,Z_i)\right]\right\vert \right.
\left. \leq {\frac{\sigma }{\sqrt{n}}}\sqrt{\frac{2 {p'_u}}{c}\zeta_1(\epsilon)}  
 +\frac{2\sigma }{n}\frac{{p'_u}}{c}\zeta_1(\epsilon)+\epsilon \right\}, \end{align}
with $ \mathcal B_{{p'_u},R}:=\{\boldsymbol\beta\in\Re^p:\,\Vert \boldsymbol\beta\Vert_\infty\leq R,\,\Vert \boldsymbol\beta\Vert_0\leq {p'_u}\}$ and $\mathcal E^2:=\{\Vert\widehat{\boldsymbol\beta} \Vert_0\leq {p'_u}\}$ with  ${p'_u}>s$, 
 we may  obtain from the above that
$
\mathbb L(\widehat{\boldsymbol\beta})-\mathbb L({\boldsymbol\beta^*_{\varepsilon_{A}}}) 
\leq s\cdot P_\lambda(a\lambda)+ {\frac{2\sigma  }{\sqrt{n}}}\sqrt{\frac{2 {p'_u}}{c}\zeta_1(\epsilon)} \nonumber
+\frac{4\sigma }{n}\frac{{p'_u}}{c}\zeta_1(\epsilon)+2\epsilon+\Gamma$, a.s.

In the analysis above, we have derived the probability for $\{\Vert\widehat{\boldsymbol\beta} \Vert_0\leq {p'_u}-1\}$. Combining this with  Proposition \ref{test proposition important}, we have that  the event $\mathcal E^1\cap\mathcal E^2$ holds with probability at least  $1-6\exp\left(-\left\lceil\frac{2cn^{4\varrho-1}}{\zeta_1(\epsilon)}+\frac{2cn^{2\varrho}}{\sigma \zeta_1(\epsilon)}\cdot\left(\Gamma+\varepsilon_{A}+\frac{2}{n^{\varrho}}\right)+8s\right\rceil\cdot \zeta_1(\epsilon)\right) -2 (p+1) \exp(-\widetilde cn)\geq 1-6\exp(-2cn^{4\varrho-1})-2 (p+1) \exp(-\widetilde cn)$. Further noticing that $\mathbb L({\boldsymbol\beta^*_{\varepsilon_{A}}}) \leq L_g^* +\varepsilon_{A}$ as per  Assumption \ref{Assumption A-sparsity}, we have  both $\Vert\widehat{\boldsymbol\beta} \Vert_0\leq {p'_u}$ and 
\begin{align}
\mathbb L(\widehat{\boldsymbol\beta})-L_g^*
\leq s\cdot P_\lambda(a\lambda)+ {\frac{2\sigma  }{\sqrt{n}}}\sqrt{\frac{2 {p'_u}}{c}\zeta_1(\epsilon)} 
+\frac{4\sigma }{n}\frac{ {p'_u}}{c}\zeta_1(\epsilon)+2\epsilon+\varepsilon_{A}+\Gamma, \label{first results obtained}
\end{align}
 where  ${p'_u} =\left\lceil\frac{2cn^{4\varrho-1}}{\zeta_1(\epsilon)}+\frac{2cn^{2\varrho}}{\sigma \zeta_1(\epsilon)}\cdot\left(\Gamma+\varepsilon_{A}+\frac{2}{n^{\varrho}}\right)+8s\right\rceil$, hold simultaneously with probability at least $1-6\exp(-2cn^{4\varrho-1})-2 (p+1) \exp(-\widetilde cn)$. Thus, we have already proven \eqref{sparsity result pp5} in Part (i) of the Proposition. 
 
 To obtain \eqref{bound global} in Part (i) of the Proposition,  we   simplify \eqref{first results obtained} while preserving the rates in  $n$ and $p$. Firstly, we have \begin{align}\sqrt{\frac{2 {p'_u}}{cn}\zeta_1(\epsilon)}  
 \leq&\,\sqrt{\frac{2}{cn}\zeta_1(\epsilon) \cdot\frac{2cn^{4\varrho-1}}{\zeta_1(\epsilon)}+\frac{2cn^{2\varrho}}{\sigma \zeta_1(\epsilon)}\left(\varepsilon_{A}+\Gamma+\frac{2}{n^{\varrho}}\right)\cdot \frac{2\zeta_1(\epsilon)}{cn}}  +\sqrt{\frac{2}{cn}\zeta_1(\epsilon)\left[ 8s+1\right]} \nonumber
\\
\leq&\,\sqrt{\frac{4}{n^{2-4\varrho}}+\frac{4(\Gamma+\varepsilon_{A}+\frac{2}{n^{\varrho}})}{\sigma n^{1-2\varrho}}}+  \sqrt{\frac{2}{nc}\zeta_1(\epsilon)\left[ 8s+1\right]},\label{first bound here}
\end{align}
which is obtained by observing the fact that  $\sqrt{x+y}\leq \sqrt{x}+\sqrt{y}$ for any $x,\, y\geq 0$ and the relations that $0<a\leq 1$, $0<c\leq 0.5$, $\sigma \geq 1$, and $\zeta_1(\epsilon)\geq \ln 2$ (as a result of the assumed inequality \eqref{sample initial requirement 2 ori test new}).

Similar to the above, we also have
\begin{align}
  {\frac{2 {p'_u}}{cn}\zeta_1(\epsilon)}   
 \leq  {\frac{4}{n^{2-4\varrho}}}+  {\frac{2}{nc}\zeta_1(\epsilon)\left[ 8s+1\right]}
+ {\frac{4(\Gamma+\varepsilon_{A}+\frac{2}{n^{\varrho}})}{\sigma n^{1-2\varrho}}}. \label{to use}
\end{align}
Invoking \eqref{sample initial requirement 2 ori test new} and $\zeta_1(\epsilon)=\ln(n^{\varrho} p)+\widetilde\zeta$, we have
$
\frac{4}{n^{2-4\varrho}}+\frac{4(\Gamma+\varepsilon_{A}+\frac{2}{n^{\varrho}})}{\sigma n^{1 -2\varrho}} \leq c_0$ and $ {\frac{2}{nc}\zeta_1(\epsilon)\left[ 8s+1\right]}\leq c_1$.
Therefore,   it holds that
$
{\frac{2 {p'_u}}{cn}\zeta_1(\epsilon)} \leq c_2\cdot\sqrt{\frac{4}{n^{2-4\varrho}}+\frac{4(\Gamma+\varepsilon_{A}+\frac{2}{n^{\varrho}})}{\sigma n^{1 -2\varrho}}}
+c_2\cdot\sqrt{{\frac{2}{nc}\zeta_1(\epsilon)\cdot\left( 8s+1\right)}}$.
Further invoking \eqref{first bound here} and \eqref{to use},  the inequality in \eqref{first results obtained}   can be simplified into
$
\mathbb L(\widehat{\boldsymbol\beta})-L_g^* \nonumber
\leq  \frac{4s\sigma \zeta_1(\epsilon)}{c\cdot n^{2\varrho}}+ c_3\cdot \sigma \cdot\sqrt{\frac{1}{n^{2-4\varrho}}+\frac{\Gamma+\varepsilon_{A}+\frac{2}{n^{\varrho}}}{\sigma n^{1-2\varrho}}} +c_3\cdot \sigma \sqrt{\frac{s+1}{nc}\zeta_1(\epsilon)}
+\frac{2}{n^{\varrho}}+\Gamma+\varepsilon_{A}$.
Further invoking a few known inequalities such as $\zeta_1(\epsilon)\geq \ln 2$, $0<\varrho<1/2$, $\sigma \geq 1$, and $0<c\leq 0.5$,  we may obtain a further simplification that
\begin{align}\mathbb L(\widehat{\boldsymbol\beta})-L_g^*
\leq c_4\cdot\left(\frac{s\zeta_1(\epsilon)}{n^{2\varrho}}+\frac{1}{n^{\varrho}}+\sqrt{\frac{s\cdot \zeta_1(\epsilon)}{n}}+\frac{1}{n^{1-2\varrho}}+\frac{1}{n^{\frac{1-\varrho}{2}}}\right)\cdot \sigma +c_4\cdot\sqrt{\frac{\sigma (\Gamma+\varepsilon_{A})}{n^{1-2\varrho}}}+\Gamma+\varepsilon_{A},\label{part 1 result here in proof}
\end{align}
 which immediately leads to \eqref{first results obtained} as claimed in Part (i) since $\zeta_1(\epsilon):=\ln\left( 3n^{\varrho} (\sigma_L+\mathcal  C_\mu) \cdot p\cdot eR\right)=\ln(n^{\varrho} p)+\widetilde\zeta$,  $a^{-1}> 1$, $s>1$, $R\geq 1$, and the satisfaction of \eqref{sample initial requirement 2 ori test new}. This immediately leads to the claimed inequality in \eqref{bound global} of Part (i). 


For Part (ii),  due to Lemma \ref{initial gap}, we know that $\mathcal L_{n,\lambda}(\widehat{\boldsymbol\beta},\,{\mathbf Z}_1^n)\leq \mathcal L_{n,\lambda}(\widehat{\boldsymbol\beta}^{\ell_1},\,{\mathbf Z}_1^n)\leq \mathcal L_{n,\lambda}({\boldsymbol\beta}^*_{\varepsilon_A},\,{\mathbf Z}_1^n)+\lambda\vert\boldsymbol\beta^*_{\varepsilon_{A}}\vert$ with probability one. Therefore, we may apply the results from Part (i)  for $\Gamma=\lambda\vert\boldsymbol\beta^*_{\varepsilon_{A}}\vert$. Thus $\frac{\Gamma}{\sigma }\leq \frac{\lambda\vert \boldsymbol\beta^*_{\varepsilon_{A}}\vert}{\sigma }\leq \frac{ \Vert \boldsymbol\beta^*_{\varepsilon_{A}}\Vert_\infty\cdot s\cdot \sqrt{\frac{8\sigma }{c\cdot a\cdot n^{2\varrho}}[\ln(n^{\rho}p)+\widetilde\zeta]}}{\sigma }$. Combining this inequality with  the assumption of \eqref{sample initial requirement} 
 (which implies that $n > c_5\cdot a^{-1}\cdot [\ln(n^{\varrho}p)+\widetilde\zeta]\cdot s^{\max\{1,\frac{1}{2-4\varrho},\,\frac{1}{2\varrho}\}}\cdot \left(\max\left\{1,\,\Vert \boldsymbol\beta^*_{\varepsilon_{A}}\Vert_\infty\right\}\right)^{\max\{\frac{1}{2-4\varrho},\,\frac{1}{2\varrho}\}}$)
 as well as  the assumption of $\sigma \geq 1$, we then know  that $\frac{\Gamma}{\sigma } \leq \Vert \boldsymbol\beta^*_{\varepsilon_{A}}\Vert_\infty\cdot s\cdot \sqrt{\frac{8}{c\sigma \cdot a\cdot n^{2\varrho}}[\ln(n^{\varrho}p)+\widetilde\zeta]}\leq c_6\cdot \sqrt{\frac{\Vert \boldsymbol\beta^*_{\varepsilon_{A}}\Vert_\infty\cdot s}{a^{1-2\varrho} }[\ln(n^{\varrho}p)+\widetilde\zeta]^{1-2\varrho}}$. Therefore, 
\begin{multline}\left(\frac{\Gamma+\varepsilon_{A}}{\sigma }\right)^{\frac{1}{1-2\varrho}}\leq \left( c_6\cdot  \sqrt{\frac{\Vert \boldsymbol\beta^*_{\varepsilon_{A}}\Vert_\infty\cdot s}{a^{1-2\varrho} }[\ln(n^{\varrho}p)+\widetilde\zeta]^{1-2\varrho}}+\frac{\varepsilon_{A}}{\sigma }\right)^{\frac{1}{1-2\varrho}}
\\\leq c_7\cdot \max\left\{1,\,(\Vert \boldsymbol\beta^*_{\varepsilon_{A}}\Vert_\infty)^{\frac{1}{2-4\varrho}}\right\}\cdot s^{\frac{1}{2-4\varrho}}\sqrt{a^{-1}\cdot [\ln(n^{\varrho}p)+\widetilde\zeta]}+c_7\left(\frac{\varepsilon_{A} }{\sigma  }\right)^{\frac{1}{1-2\varrho}}.
\end{multline}
Recall that $a<1$ and observe that $[\ln(n^{\varrho}p)+\widetilde\zeta]\geq 1$. We then have that, if $n$ satisfies \eqref{sample initial requirement}, then
\begin{align}
n >  &c_7\cdot\left(\frac{\varepsilon_{A}}{\sigma }\right)^{\frac{1}{1-2\varrho}}+ c_7\cdot a^{-1 }\cdot [\ln(n^{\varrho}p)+\widetilde\zeta]\cdot s^{\max\{1,\frac{1}{2-4\varrho},\,\frac{1}{2\varrho}\}}\left(\max\left\{1,\,\Vert \boldsymbol\beta^*_{\varepsilon_{A}}\Vert_\infty\right\}\right)^{\max\{\frac{1}{2-4\varrho},\,\frac{1}{2\varrho}\}}\nonumber
\\\geq &c_8\cdot \max\left\{1,\,(\Vert \boldsymbol\beta^*_{\varepsilon_{A}}\Vert_\infty)^{\frac{1}{2-4\varrho}}\right\}\cdot s^{\frac{1}{2-4\varrho}}\sqrt{a^{-1}\cdot [\ln(n^{1/3}p)+\widetilde\zeta]}+c_8\left(\frac{\varepsilon_{A} }{\sigma  }\right)^{\frac{1}{1-2\varrho}}+ c_8s\cdot\left(\ln(n^{\varrho}p)+\widetilde \zeta\right)\nonumber
\\\geq&C_1\cdot \left[\left(\frac{\Gamma+\varepsilon_{A}}{\sigma }\right)^{\frac{1}{1-2\varrho}}+s\cdot\left(\ln(n^{\varrho}p)+\widetilde \zeta\right)\right],\nonumber
\end{align}
 
Therefore, \eqref{part 1 result here in proof} above implies that 
$$
\mathbb L(\widehat{\boldsymbol\beta})-L_g^* 
\leq c_9\cdot \sigma \cdot\left(\frac{s\zeta_1(\epsilon)}{n^{2\varrho}}+\sqrt{\frac{s\zeta_1(\epsilon)}{n}}+\frac{1}{n^{\varrho}}+\frac{1}{n^{1-2\varrho}}+\frac{1}{n^{(1-\varrho)/2}}\right)
+c_9\cdot\sqrt{\frac{\sigma (\lambda\vert \boldsymbol\beta^*_{\varepsilon_{A}}\vert+\varepsilon_{A})}{n^{1-2\varrho}}}+\lambda\vert \boldsymbol\beta^*_{\varepsilon_{A}}\vert+\varepsilon_{A},
$$
  with probability at least 
$
1-2  (p+1) \exp(-\widetilde cn)-6\exp\left(-2cn^{4\varrho-1}\right)
$. The above bound  can be further simplified by noticing that $a<1$,  $s>1$, $0<\varrho<\frac{1}{2}$, $\sigma \geq 1$, $p\geq 1$, $\left[\ln(n^{\varrho}p)+\widetilde \zeta\right]\geq 1$ and $\sqrt{\frac{s\zeta_1(\epsilon)}{n}}\leq \frac{s\cdot\sqrt{\zeta_1(\epsilon)}}{n^{\frac{1-\varrho}{2}}}$. As a result,
$\mathbb L(\widehat{\boldsymbol\beta})-L_g^*
\leq c_{10}\cdot \sigma \cdot\left[\frac{s\cdot \left(\ln(n^{\varrho}p)+\widetilde \zeta\right)}{n^{2\varrho}}+\frac{1}{n^{\varrho}}+\frac{1}{n^{1-2\varrho}}\right]
+c_{10}\cdot \frac{s\cdot\max\left\{1,\,\Vert\boldsymbol\beta^*_{\varepsilon_{A}} \Vert_\infty\right\}\cdot\sigma ^{3/4}}{\min\left\{a^{1/2}n^\varrho,\,a^{1/4}n^{\frac{1-\varrho}{2}}\right\}} \left[\ln(n^{\varrho}p)+\widetilde \zeta\right]^{1/2}+c_{10}\cdot\sqrt{\frac{\sigma \varepsilon_{A}}{n^{1-2\varrho}}}+\varepsilon_{A}$, which immediately leads to the desired result in Part (ii). 
\hfill \ensuremath{\Box}
\end{proof}

\subsection{Proof of results for nonsmooth HDSL}\label{proof of theorem 2}
\begin{proof}{Proof of Theorem \ref{nonsmooth case}.} 
To show Part (a), we invoke Theorem \ref{lipschitz} and obtain that $ f_\mu(\boldsymbol\beta, \mathbf A(Z_i)):=\max_{\mathbf u\in\mathbb U}\,\left\{\mathbf u^\top\mathbf A(Z_i)\boldsymbol\beta- \frac{1}{2n^{\delta}} \Vert \mathbf u-\mathbf u_0\Vert^2\right\}$ is continuously differentiable with   Lipchitz continuous  gradient, and the corresponding Lipschitz constant is $ \frac{1}{n^{-\delta}}\Vert \mathbf A(Z_i)\Vert^2_{1,2}$, with $ \frac{1}{n^{-\delta}}\Vert \mathbf A(Z_i)\Vert^2_{1,2}\leq n^{\delta}\cdot U_{A}$, a.s. Therefore, it holds that, for all $j=1,...,p$,  the partial derivative, $\frac{\partial \widetilde{\mathcal L}_{n,\delta}(\widetilde{\boldsymbol\beta},\, \mathbf Z_1^n)}{\partial\beta_j}$,  is well-defined for all $\widetilde{\boldsymbol\beta}\in\Re^p$ and  Lipschitz continuous   for almost every $Z_i\in\mathcal W$.  Further noticing that $\frac{1}{n^{-\delta}}\Vert \mathbf A(Z_i)\Vert^2_{1,2}+U_{f_1}\leq U_{f_1}+n^{\delta}U_A$ with probability one, we have the desired result in Part (a).

To show Part (b), we  denote by $c_1,c_2,...$   potentially different universal constants throughout this proof. Let ${\boldsymbol\beta}^*_{\varepsilon_{A}'}$ be the sparse vector as in Assumption \ref{Assumption A-sparsity original} (where $\varepsilon_A$ is now denoted by  $\varepsilon_A:=\varepsilon_A'$ in this theorem) and $\widetilde{\mathcal L}_{n,\delta}$  as in \eqref{smoothed nonsmooth problem} (where  $\mathbb L(\,\cdot\,)$ in the statement of the assumption is replaced by $\mathbb E\left[n^{-1}\sum_{i=1}^nL_{ns}(\,\cdot\,,Z_i)\right]$). We claim that
\begin{align}
\mathbb E[\widetilde{\mathcal L}_{n,\delta}({\boldsymbol\beta}^*_{\varepsilon_{A}'},\mathbf Z_1^n)]-\inf_{\boldsymbol\beta}~\mathbb E[\widetilde{\mathcal L}_{n,\delta}({\boldsymbol\beta},\mathbf Z_1^n)]\leq \frac{D^2}{2n^{\delta}}+\varepsilon_{A}'.
\end{align}
To see this, one may observe that, under Assumption \ref{Assumption A-sparsity original},
\begin{align}
&\mathbb E\left[\frac{1}{n}\sum_{i=1}^nf_1({\boldsymbol\beta}^*_{\varepsilon_{A}'},Z_i)+\sum_{i=1}^n\frac{1}{n}\max_{\mathbf u\in \mathbb U}\left\{\mathbf u^\top\mathbf A(Z_i) {\boldsymbol\beta}^*_{\varepsilon_{A}'}-\phi(\mathbf u,\,Z_i)- \frac{1}{2n^{\delta}} \Vert \mathbf u-\mathbf u_0\Vert^2\right\}\right]-\inf_{\boldsymbol\beta}~\mathbb E[\widetilde{\mathcal L}_{n,\delta}(\boldsymbol\beta,\mathbf Z_1^n)]\nonumber
\\
=&\,\mathbb E\left[\frac{1}{n}\sum_{i=1}^nf_1({\boldsymbol\beta}^*_{\varepsilon_{A}'},Z_i)+\sum_{i=1}^n\frac{1}{n}\max_{\mathbf u\in \mathbb U}\left\{\mathbf u^\top\mathbf A(Z_i) {\boldsymbol\beta}^*_{\varepsilon_{A}'}-\phi(\mathbf u,\,Z_i)- \frac{1}{2n^{\delta}} \Vert \mathbf u-\mathbf u_0\Vert^2\right\}\right]\nonumber
\\&\qquad\qquad\qquad-\inf_{\boldsymbol\beta}~\mathbb E\left[\frac{1}{n}\sum_{i=1}^nf_1(\boldsymbol\beta,Z_i)+\sum_{i=1}^n\frac{1}{n}\max_{\mathbf u\in \mathbb U}\left\{\mathbf u^\top\mathbf A(Z_i) \boldsymbol\beta-\phi(\mathbf u,\,Z_i) - \frac{1}{2n^{\delta}} \Vert \mathbf u-\mathbf u_0\Vert^2\right\}\right]\nonumber
\\\leq &\mathbb E\left[\frac{1}{n}\sum_{i=1}^nf_1({\boldsymbol\beta}^*_{\varepsilon_{A}'},Z_i)+\sum_{i=1}^n\frac{1}{n}\max_{\mathbf u\in \mathbb U} \,\left\{\mathbf u^\top\mathbf A(Z_i) {\boldsymbol\beta}^*_{\varepsilon_{A}'} -\phi(\mathbf u,\,Z_i)\right\} \right]\nonumber
\\&\qquad\qquad\qquad\,-\inf_{\boldsymbol\beta}~\mathbb E\left[\frac{1}{n}\sum_{i=1}^nf_1(\boldsymbol\beta,Z_i)+\sum_{i=1}^n\frac{1}{n}\max_{\mathbf u\in \mathbb U}\,\left\{\mathbf u^\top\mathbf A(Z_i) \boldsymbol\beta-\phi(\mathbf u,\,Z_i)\right\} \right]+ \frac{D^2}{2n^{\delta}}  \nonumber
\\\leq &\,\inf_{\boldsymbol\beta\,}\,\left\{\,\mathbb E\left[\frac{1}{n}\sum_{i=1}^nf_1(\boldsymbol\beta,Z_i)+\sum_{i=1}^n\frac{1}{n}\max_{\mathbf u\in \mathbb U} \,\left\{\mathbf u^\top\mathbf A(Z_i) \boldsymbol\beta-\phi(\mathbf u,\,Z_i)\right\} \right]\right\}+\varepsilon_{A}'\nonumber
\\&\qquad\qquad\qquad-\inf_{\boldsymbol\beta}~\mathbb E\left[\frac{1}{n}\sum_{i=1}^nf_1(\boldsymbol\beta,Z_i)+\sum_{i=1}^n\frac{1}{n}\max_{\mathbf u\in \mathbb U}\,\left\{\mathbf u^\top\mathbf A(Z_i) \boldsymbol\beta-\phi(\mathbf u,\,Z_i)\right\} \right]+ \frac{D^2}{2n^{\delta}} \nonumber
\leq  \frac{D^2}{2n^{\delta}}+\varepsilon_{A}', \nonumber
\end{align}
where the last inequality is due to the definition of $\boldsymbol\beta^*_{\varepsilon'_A}$. 

Now  we consider the hypothetical population-level learning problem of   $\inf_{\boldsymbol\beta}~\mathbb E[\widetilde{\mathcal L}_{n,\delta}(\boldsymbol\beta,\mathbf Z_1^n)]$. The foregoing derivation indicates that this hypothetical problem also satisfies Assumption  \ref{Assumption A-sparsity original} with   $\varepsilon_A:=\left(\frac{D^2}{2n^{\delta}}+{\varepsilon_{A}'}\right)$ and  a sparsity level $s$.  We thus may analyze this hypothetical problem by employing Proposition \ref{second theorem}, where we let ${\mathcal L}_{n}$, ${\varepsilon_{A}}$, $\delta$, $\varrho$, and $a^{-1}$ in the original definition  to be ${\mathcal L}_{n}:=\widetilde{\mathcal L}_{n,\delta}$, $\varepsilon_A:=\frac{D^2}{2n^{\delta}}+{\varepsilon_{A}'}$, $\delta:=\frac{1}{4}$, $\varrho:=\frac{3}{8}$, and $a^{-1}:=2(U_{f_1}+n^{\delta}U_A)$, respectively,
 to bound the excess risk. To that end, we  first verify the satisfaction of \eqref{sample initial requirement} by the assumption of  \eqref{sample initial requirement nonsmooth}.  In other words, we need to ensure that
\begin{align}
 n >& c_1\cdot\left(\frac{\frac{D^2}{2n^{1/4}}+{\varepsilon_{A}'}}{\sigma }\right)^{\frac{1}{1-2\times \frac{3}{8}}}+ c_1\cdot \left(U_{f_1}+n^{1/4}U_A\right)\cdot [\ln(n^{\frac{3}{8}}p)+\widetilde\zeta]\cdot s^{\max\left\{1,\frac{1}{2-4\times \frac{3}{8}},\,\frac{1}{2\times \frac{3}{8}}\right\}}\left(\max\{1,\,\Vert \boldsymbol\beta_{\varepsilon_A'}^*\Vert_\infty\}\right)^{\max\left\{\frac{1}{2-4\times\frac{3}{8}},\,\frac{1}{2\times \frac{3}{8}}\right\}}\nonumber
 \\=&c_1\cdot\left(\frac{\frac{D^2}{2n^{1/4}}+{\varepsilon_{A}'}}{\sigma }\right)^{4}+ c_1\cdot \left(U_{f_1}+n^{1/4}U_A\right)\cdot [\ln(n^{\frac{3}{8}}p)+\widetilde\zeta]\cdot s^2\cdot (\max\{1,\,\Vert \boldsymbol\beta_{\varepsilon_A'}^*\Vert_\infty\})^2\nonumber
\end{align}
 From \eqref{sample initial requirement nonsmooth} (which implies that $\frac{16n\sigma^2}{D^4}>1$),
 \begin{align}
 \frac{n}{2}>&\,c_2\cdot\left(\frac{D^4}{2\sigma ^2}+\frac{(\varepsilon_{A}')^4}{\sigma^4}\right)=\,c_2\cdot\left(\frac{D^8}{32n\sigma^4}\cdot\frac{16n\sigma^2}{D^4}+\frac{(\varepsilon_{A}')^4}{\sigma^4}\right)\Longrightarrow n> \,c_3\cdot\left(\frac{D^2}{2n^{1/4}\sigma } +\frac{\varepsilon_{A}'}{\sigma}\right)^4.
 \end{align}
 Similarly,  \eqref{sample initial requirement nonsmooth} also implies that
$n>C_5\cdot(U_{f_1}+U_A)^{4/3}\cdot [\ln(n^{\frac{3}{8}}p)+\widetilde\zeta]^{4/3}\cdot s^{8/3}(\max\{1,\,\Vert \boldsymbol\beta_{\varepsilon_A'}^*\Vert_\infty\})^{8/3}\Longrightarrow \frac{n^{\frac{4}{3}}}{2}>\frac{c_3}{2}\cdot(U_{f_1}+U_A)^{4/3}n^{1/3}\cdot [\ln(n^{\frac{3}{8}}p)+\widetilde\zeta]^{4/3}\cdot s^{8/3}(\max\{1,\,\Vert \boldsymbol\beta_{\varepsilon_A'}^*\Vert_\infty\})^{8/3}\nonumber
\Longrightarrow n>c_4 \cdot(U_{f_1}+U_A)n^{1/4}\cdot [\ln(n^{\frac{3}{8}}p)+\widetilde\zeta]\cdot s^{2}(\max\{1,\,\Vert \boldsymbol\beta_{\varepsilon_A'}^*\Vert_\infty\})^{2}\geq c_5\cdot \left(U_{f_1}+n^{1/4}U_A\right)\cdot [\ln(n^{\frac{3}{8}}p)+\widetilde\zeta]\cdot s^2(\max\{1,\,\Vert \boldsymbol\beta_{\varepsilon_A'}^*\Vert_\infty\})^2$. Therefore, if \eqref{sample initial requirement nonsmooth} holds then $n>   c_6\cdot\left(\frac{\varepsilon_{A}}{\sigma}\right)^{1/(1-2\varrho)}+c_6\cdot a^{-1}\cdot [\ln(n^{\varrho}p)+\widetilde\zeta]\cdot s^{\max\left\{1,\frac{1}{2-4\varrho},\,\frac{1}{2\varrho}\right\}}(\max\{1,\,\Vert \boldsymbol\beta_{\varepsilon_A'}^*\Vert_\infty\})^{\max\left\{\frac{1}{2-4\varrho},\,\frac{1}{2\varrho}\right\}}$, which then means that \eqref{sample initial requirement} is verified.

We may now invoke Proposition \ref{second theorem}  with   ${\varepsilon_{A}}:=\frac{D^2}{2n^{\delta}}+{\varepsilon_{A}'}$ and $a:=\left[2(U_{f_1}+n^{\delta}U_A)\right]^{-1}$, respectively, in bounding  $\mathbb E[\widetilde{\mathcal L}_{n,\delta}(\widehat{\boldsymbol\beta},\mathbf Z_1^n)]-\inf_{\boldsymbol\beta}~\mathbb E[\widetilde{\mathcal L}_{n,\delta}({\boldsymbol\beta},\mathbf Z_1^n)]$. This proposition immediately leads to
\begin{multline}
\mathbb E\left[\widetilde{\mathcal L}_{n,\delta}(\widehat{\boldsymbol\beta},\mathbf Z_1^n)\right]-\inf_{\boldsymbol\beta}~\mathbb E\left[\widetilde{\mathcal L}_{n,\delta}(\boldsymbol\beta,\mathbf Z_1^n)\right]
\leq c_7\cdot\left[\frac{s\left(\ln(n^{\varrho}p)+\widetilde \zeta\right)}{n^{2\varrho}}+ \frac{1}{n^{\varrho}}+\frac{1}{n^{1-2\varrho}}\right]
\cdot \sigma 
\\+c_7\cdot \frac{s \cdot \max\left\{1,\,\Vert \boldsymbol\beta_{\varepsilon_A'}^*\Vert_\infty\right\}\cdot\sigma ^{3/4}}{\min\left\{[2(U_{f_1}+n^{\delta}U_A)]^{-\frac{1}{2}}n^\varrho,\,[2(U_{f_1}+n^{\delta}U_A)]^{-\frac{1}{4}}n^{\frac{1-\varrho}{2}}\right\}} \sqrt{\ln(n^{\varrho}p)+\widetilde \zeta}
\\+c_7\cdot\sqrt{\frac{\sigma \frac{D^2}{2n^{\delta}}+\sigma\cdot\varepsilon_{A}'}{n^{1-2\varrho}}}+\frac{D^2}{2n^{\delta}}+\varepsilon_{A}',
\label{bound 12 new}
\end{multline}
with probability at least $
1-2 (p+1) \exp(-\widetilde cn)-6\exp\left(-2cn^{4\varrho-1}\right)
$ for a universal constant $\widetilde c>0$.

Further notice that $\mathbb E[\widetilde{\mathcal L}_{n,\delta}({\boldsymbol\beta},\mathbf Z_1^n)]\leq \mathbb E[{\mathcal L}_{n}({\boldsymbol\beta},\mathbf Z_1^n)]\leq \mathbb E[\widetilde{\mathcal L}_{n,\delta}({\boldsymbol\beta},\mathbf Z_1^n)]+\frac{D^2}{2n^{\delta}}$ for any $\boldsymbol\beta\in\Re^p$ and $2(U_{f_1}+n^{1/4}U_A)\leq 2(U_{f_1}+U_A)\cdot n^{1/4}$. Also recall that  $\delta=\frac{1}{4}$, and $\varrho:=\frac{3}{8}$. We then can obtain from the above
\begin{multline}
\mathbb E[ {\mathcal L}_{n}(\widehat{\boldsymbol\beta},\mathbf Z_1^n)]-\inf_{\boldsymbol\beta}~\mathbb E[{\mathcal L}_{n}({\boldsymbol\beta},\mathbf Z_1^n)]
\leq c_8\cdot\left[\frac{s\left(\ln(n^{\varrho}p)+\widetilde \zeta\right)}{n^{3/4}}+\frac{1}{n^{1/4}}\right]
\cdot \sigma  
\\+c_8\cdot \frac{s\cdot \max\left\{1,\,\Vert \boldsymbol\beta_{\varepsilon_A'}^*\Vert_\infty\right\}\cdot \sigma ^{3/4}}{\min\left\{[2(U_{f_1}+U_A)]^{-\frac{1}{2}}n^{1/4},\,[2(U_{f_1}+U_A)]^{-\frac{1}{4}}n^{5/16}\right\}} \sqrt{\ln(n^{\frac{3}{8}}p)+\widetilde \zeta}
\\+c_8\cdot\sqrt{\frac{\sigma D^2}{n^{1/2}}}+\frac{D^2}{2n^{1/4}}+c_8\cdot\frac{\sqrt{\sigma\varepsilon_{A}'}}{n^{1/8}}+\varepsilon_{A}',
\label{bound 12 new 2}
\end{multline}
with probability $
1-2 (p+1) \exp(-\widetilde cn)-6\exp\left(-2cn^{1/2}\right)$ for a universal constant $\widetilde c>0$. The desired result  is then immediately implied from the above  after  some simplification under  $U_{f_1}\geq 1$ and $\sigma\geq 1$.\hfill \ensuremath{\Box}
\end{proof}

\subsection{Proof of generalizability for regularized NNs}\label{sec: proof regularized NNs}

\subsubsection{Proof of generalizability of regularized NNs in a generic case}\label{sec: proof of generalized NN in general}

\begin{proof}{Proof of Theorem \ref{general result theorem regularized NN}.}
We  denote by $c_1,c_2,...$   potentially different universal constants throughout this proof.

We  invoke Part (i) of Proposition \ref{comprehensive proposition} to show the desired result. To that end, we are to verify that all the conditions for the proposition are met.  We first recall that A-sparsity (as in Assumption \ref{Assumption A-sparsity}) holds as per \eqref{test inequality final result representability}
 with $L_g^*=0$,  ~~ $\varepsilon_{A}:=  \frac{1}{2}v^{-1} \ln n \cdot\Omega({s_A})+\frac{1}{\sqrt{n}}$, ~~$s:=s_A$,~~  and $R:=\frac{1}{2}v^{-1} \ln n\cdot R_{\Omega}$.

Secondly, because $\min\left\{\ln 2,\,\mathcal F\left(y\cdot F_{NN}(\mathbf x,\boldsymbol\beta)\right)\right\}\in(0,\,\ln2]$ with probability 1, we know that
$
\Vert\min\left\{\ln 2,\,\mathcal F\left(y\cdot F_{NN}(\mathbf x,\boldsymbol\beta)\right)\right\}\Vert_{\psi_1}\leq 1
$, which means that Assumption \ref{sub exponential condition} holds with $\sigma=1$. To see this, observe that $\sqrt{\min\left\{\ln 2,\,\mathcal F\left(y\cdot F_{NN}(\mathbf x,\boldsymbol\beta)\right)\right\}}\in\left(0,\,\sqrt{\ln2}\right]$, w.p.1. Thus, as per \cite{vershynin2018a} (Example 2.5.8.(c) therein), it holds that $\left\Vert \sqrt{\min\left\{\ln 2,\,\mathcal F\left(y\cdot F_{NN}(\mathbf x,\boldsymbol\beta)\right)\right\}}\right\Vert_{\psi_2}\leq \frac{1}{\sqrt{\ln 2}}\cdot \underset{\mathbf x,y,\boldsymbol\beta}{\text{ess\,sup}}\left\vert \sqrt{\min\left\{\ln 2,\,\mathcal F\left(y\cdot F_{NN}(\mathbf x,\boldsymbol\beta)\right)\right\}}\right\vert=1$. By the property that $\Vert X\Vert^2_{\psi_2}=\Vert X^2\Vert_{\psi_1}$, we thus know that 
\begin{align}\Vert\min\left\{\ln 2,\,\mathcal F\left(y\cdot F_{NN}(\mathbf x,\boldsymbol\beta)\right)\right\}\Vert_{\psi_1}\leq 1.\label{subexponential norm}
\end{align} 
 Notice that $\left\vert\frac{\partial^2 \mathcal F\left(y\cdot F_{NN}(\mathbf x,\boldsymbol\beta)\right)}{\partial \beta_j^2}\right\vert= \left\vert y^2\cdot \left[\frac{\partial^2 \mathcal F(z)}{\partial z^2}\right]_{z=y\cdot F_{NN}(\mathbf x,\boldsymbol\beta)}\cdot   \left[\frac{\partial F_{NN}(\mathbf x,\boldsymbol\beta)}{\partial \beta_j}\right]^2 +y\cdot \left[\frac{\partial \mathcal F(z)}{\partial z}\right]_{z=y\cdot F_{NN}(\mathbf x,\boldsymbol\beta)}\cdot \frac{\partial^2 F_{NN}(\mathbf x,\boldsymbol\beta)}{\partial^2\beta_j} \right\vert$. Because $\vert \frac{\partial^2 \mathcal F(z)}{\partial z^2}\vert\leq 1$, $\vert \frac{\partial \mathcal F(z)}{\partial z}\vert\leq 1$, and $\Vert \boldsymbol\beta\Vert\leq \sqrt{p}\cdot \frac{1}{2}\cdot R_{\Omega}\cdot v^{-1}\cdot \ln n$ for all $\boldsymbol\beta:\,\Vert\boldsymbol\beta\Vert_\infty\leq R_{\Omega}v^{-1}\ln n^{1/2}$, by Assumption \ref{new bound results 2.5}, we have $\left\vert\frac{\partial^2 \mathcal F\left(y\cdot F_{NN}(\mathbf x,\boldsymbol\beta)\right)}{\partial \beta_j^2}\right\vert\leq   2\exp\left\{2 \mathcal U_{NN} \cdot {\mathcal D}\cdot \ln\left(\mathcal U_{NN}\cdot R_{\Omega} \cdot {p}\cdot \frac{1}{2}\cdot v^{-1}\cdot \ln n+\mathcal U_{NN}\right)\right\}$. Because $a<\frac{1}{2}\cdot \exp\left\{- 2\mathcal U_{NN} \cdot {\mathcal D}\cdot \ln\left[p\cdot v^{-1}\cdot  \mathcal U_{NN} \cdot R_{\Omega}\cdot\ln n +\mathcal U_{NN}\right]\right\}$, the S$^3$ONC solution satisfies that $\widehat\beta_j\notin(0,\,a\lambda)$ for all $j$ with probability 1, as per Proposition \ref{Proposition 1}. 

Thirdly, we now show that $\mathcal F(y\cdot F_{NN}(\mathbf x,\cdot))$ obeys the Lipschitz-like condition as a special case to Assumption \ref{Lipschitz condition}. By Assumption \ref{new bound results 2.5},  we have $\left\Vert\nabla_{\boldsymbol\beta} \mathcal F\left(y\cdot F_{NN}(\mathbf x,\boldsymbol\beta)\right)\right\Vert=\left\Vert\left[y\cdot \frac{\partial \mathcal F(z)}{\partial z}\right]_{z=y\cdot F_{NN}(\mathbf x,\boldsymbol\beta)}\cdot  {\nabla_{\boldsymbol\beta} F_{NN}(\mathbf x,\boldsymbol\beta)} \right\Vert\leq \exp\left[\mathcal U_{NN}\cdot {\mathcal D}\cdot \ln\left(\mathcal U_{NN}\cdot \Vert \boldsymbol\beta\Vert +\mathcal U_{NN}\right)\right]\leq \exp\left[\mathcal U_{NN}\cdot {\mathcal D}\cdot \ln\left(p\cdot v^{-1}\cdot R_{\Omega}\cdot \mathcal U_{NN}\cdot \ln n+\mathcal U_{NN}\right)\right]$, which indicates that $\left\vert\mathcal F\left( y\cdot F_{NN}(\mathbf x,\boldsymbol\beta_1)\right)-\mathcal F\left( y\cdot F_{NN}(\mathbf x,\boldsymbol\beta_2)\right)\right\vert\leq  \exp\left[\mathcal U_{NN}\cdot {\mathcal D}\cdot \ln\left(p n\cdot v^{-1}\cdot R_{\Omega}\cdot \mathcal U_{NN}+\mathcal U_{NN}\right)\right]\cdot \Vert\boldsymbol\beta_1-\boldsymbol\beta_2\Vert \leq \frac{n}{\ln n}\exp\left[\mathcal U_{NN}\cdot {\mathcal D}\cdot \ln\left(p n\cdot v^{-1}\cdot R_{\Omega}\cdot \mathcal U_{NN}+\mathcal U_{NN}\right)\right]\cdot \Vert\boldsymbol\beta_1-\boldsymbol\beta_2\Vert$, for all $\boldsymbol\beta_1,\,\boldsymbol\beta_2\in \left[-\frac{1}{2}\cdot R_{\Omega}\cdot v^{-1}\cdot \ln n,~~\frac{1}{2}\cdot R_{\Omega}\cdot v^{-1}\cdot \ln n\right]^p$ and almost every $\mathbf x\in\mathcal X$.  Consequently, Assumption \ref{Lipschitz condition} holds with $\sigma_L=0$, $R:=\frac{1}{2}v^{-1} \ln n\cdot R_{\Omega}$, and $\mathcal C_\mu= \frac{n}{\ln n}\exp\left[\mathcal U_{NN}\cdot {\mathcal D}\cdot \ln\left(pn\cdot v^{-1}\cdot R_{\Omega}\cdot \mathcal U_{NN}+ \mathcal U_{NN}\right)\right]$. Thus, $\widetilde\zeta=\ln\left(3eR\cdot(\sigma_L+\mathcal  C_\mu)\right)=\ln\left( \frac{3}{2}eR_{\Omega}\cdot v^{-1}\cdot n\cdot \exp\left[\mathcal U_{NN}\cdot {\mathcal D}\cdot \ln\left(pn\cdot v^{-1}\cdot R_{\Omega}\cdot \mathcal U_{NN}+ \mathcal U_{NN}\right)\right]\right)=\ln(\frac{3}{2}eR_{\Omega} nv^{-1} )+ \mathcal U_{NN}\cdot {\mathcal D}\cdot \ln\left[\mathcal U_{NN} \cdot (1+ R_{\Omega}pnv^{-1}) \right]$.

So far, we have verified that all the conditions for Proposition \ref{comprehensive proposition} holds. 
Invoking this proposition with $\varrho=1/3$, we thus have, for any $\Gamma\geq 0$ and some    universal constant $c_2>0$, if
\begin{multline}
n > c_2\cdot \left[\left( {\Gamma+v^{-1}\cdot\Omega({s_A})\cdot \ln n  }\right)^{3}+{s_A}\cdot {\mathcal D}\cdot \mathcal U_{NN}\cdot \ln\left(\mathcal U_{NN}\cdot (n pR_{\Omega}v^{-1}+1)\right) \right]\nonumber
\\\Longrightarrow\nonumber
n > c_3\cdot \left[\vphantom{\left( \frac{\Gamma+n^{-1/2}+v^{-1}\cdot\Omega({s_A})\cdot \ln n^{1/2}  }{c_1}\right)^{3}}{s_A}\cdot\left(\ln(n^{1/3}p)+\ln\left(\frac{3}{2}eR_{\Omega} nv^{-1}\right)+ \mathcal U_{NN}\cdot {\mathcal D}\cdot \ln\left[\mathcal U_{NN} \cdot (1+ R_{\Omega}pnv^{-1}) \right]\right)\right.
\\\left.+\left( \frac{\Gamma+n^{-1/2}+v^{-1}\cdot\Omega({s_A})\cdot \ln n^{1/2}  }{1}\right)^{3}\right]\Longrightarrow \textnormal{Eq.\,}\eqref{sample initial requirement 2 ori test new}
\end{multline}
 and   $\mathcal L_{n,\lambda}(\widehat{\boldsymbol\beta},\,\mathbf Z_1^n)\leq \mathcal L_{n,\lambda}({\boldsymbol\beta}^{*}_{\varepsilon_{A}},\,\mathbf Z_1^n)+\Gamma$ almost surely, then we obtain the below by invoking  Proposition \ref{comprehensive proposition} (Part (i)) with $\varrho=1/3$ after some simplification:
 \begin{multline}
\mathbb E\left[\min\left\{\ln 2,\,\mathcal F\left(y\cdot F_{NN}(\mathbf x,\widehat{\boldsymbol\beta})\right)\right\}\right]-L_g^*
\\
\leq c_4\cdot\left(\frac{{s_A}\cdot {\mathcal D}\cdot \mathcal U_{NN}\cdot \ln\left(\mathcal U_{NN}\cdot e\cdot (1+n pR_{\Omega}v^{-1})\right)}{n^{2/3}}+\sqrt{\frac{{s_A}\cdot {\mathcal D}\cdot \mathcal U_{NN}\cdot \ln\left(\mathcal U_{NN}\cdot e\cdot (1+n pR_{\Omega}v^{-1})\right)}{n}}+ \frac{1}{n^{1/3}}\right)  \\+c_4\cdot\sqrt{\frac{\Gamma+v^{-1}\cdot\Omega({s_A})\cdot \ln n+\frac{1}{\sqrt{n}}}{n^{1/3}}}+\Gamma+v^{-1}\cdot\Omega({s_A})\cdot \ln n+\frac{1}{\sqrt{n}},\label{bound global 2}
\end{multline}
with probability at least 
$
1-2 (p+1) \exp(- n/c_4)-6\exp\left(-2cn^{1/3}\right).
$
%
Further noticing that $ \mathbb 1\left(t<0\right) \leq 2\cdot  \min\left\{\ln 2,\,\mathcal F\left(t\right)\right\} $ for all $t\in\Re$, we then have $\mathbb E\left[\mathbb 1(y\cdot F\left(\mathbf x,\widehat{\boldsymbol\beta})<0\right)\right]\leq 2\cdot \mathbb E\left[\min\left\{\ln 2,\,\mathcal F\left(y\cdot F_{NN}(\mathbf x,\widehat{\boldsymbol\beta})\right)\right\}\right]$, almost surely. This combined with \eqref{bound global 2}   immediately leads to the desired result.
\hfill \ensuremath{\Box}\end{proof}

\subsubsection{Proof of generalizability of a flexible set of NN architectures}\label{Proof ReLU arch result}
\begin{proof}{Proof of Corollary \ref{ReLU arch result}.}
Let $c_1,\,c_2$,... be universal constants. Because the output layer involves no nonlinear transformation, Assumption \ref{mild condition invariance} holds. Observe that, when $b_{l,{\mathcal D}}=0$, for $l=1,..., \mathcal D-1$, and  $\mathbf W_{l-1,l}=\mathbf 0$ and $\mathbf b_{l-1,l}=\mathbf 0$, for all  $l=2,...,{\mathcal D}-1$,  the NN defined as in \eqref{zero NN equality}-\eqref{second NN equality} can be reduced to $F_{NN}(\mathbf x,\boldsymbol\beta)=\sum_{l=1}^{{\mathcal D}-1}\left[\boldsymbol w_{l,{\mathcal D}}^\top\Psi\left( \mathbf W_{0,l} \mathbf x +\mathbf b_{0,1}\right)\right]$, which is essentially an NN with one hidden layer.
We therefore may  invoke Theorem 2.1 of  \cite{mhaskar1996a}, which is restated as Theorem \ref{useful theorem from M} in this paper for completeness. It establishes the representation error of a single-hidden-layer NN in approximating $g\in \mathbb F_{d,r}$ under Assumption \ref{activation function condition here}. As an immediate result of that theorem, if there are $\widetilde N$-many (active) hidden neurons in that single-hidden-layer NN, captured by $\widetilde{\boldsymbol w}^\top \Psi\left( \widetilde{\mathbf W} \mathbf x +\widetilde{\mathbf b}\right)$ for fitting parameters $\widetilde{\boldsymbol w}\in\Re^{{\widetilde N}}$, $\widetilde{\mathbf W}\in \Re^{{\widetilde N}\times d}$, and $\widetilde{\mathbf b}\in\Re^{\widetilde N}$, then the model misspecification error $\Omega(\widetilde N)$ is at most $\mathcal C_{NN}\cdot {\widetilde N}^{-r/d}$, where   $\mathcal C_{NN}>0$ is a quantity that depends only on $d$ and $r$; more formally,
\[\underset{\substack{\widetilde{\boldsymbol w}\in\Re^{1\times {\widetilde N}},\,\widetilde{\mathbf W}\in \Re^{{\widetilde N}\times d}
\\\widetilde{\mathbf b}\in\Re^{\widetilde N}}}{\inf}~\underset{\mathbf x\in[-1,1]^{d}}{\text{ess\,sup}}\left\vert \left[\widetilde{\boldsymbol w}^\top \Psi\left( \widetilde{\mathbf W} \mathbf x +\widetilde{\mathbf b}\right)\right] -g(\mathbf x)\right\vert\leq \mathcal C_{NN}\cdot {\widetilde N}^{-r/d}.\]

  Meanwhile, the total number of fitting parameters of this single-hidden-layer NN is $(d+2)\cdot \widetilde N$. Observing that this single-hidden-layer NN is a subnetwork of $F_{NN}(\mathbf x,\boldsymbol\beta)$ if $\widetilde N\leq  K\cdot  {\mathcal D}$, we obtain that
\begin{multline}
\inf_{\Vert\boldsymbol\beta\Vert_{0}\leq \widetilde {N}\cdot (d+2)}\mathbb E\left[ \left\vert g(\mathbf x)-F_{NN}(\mathbf x,\,\boldsymbol\beta)\right\vert\right]\leq\underset{\boldsymbol\beta:\,\Vert \boldsymbol\beta\Vert_0\leq (d+2)\cdot{\widetilde N}}{\inf}~\underset{\mathbf x\in[-1,1]^{d}}{\text{ess\,sup}}~ \vert g(\mathbf x)-F_{NN}(\mathbf x,\,\boldsymbol\beta)\vert
\\
\leq \underset{\substack{\widetilde{\boldsymbol w}\in\Re^{{\widetilde N}},\,\widetilde{\mathbf W}\in \Re^{{\widetilde N}\times d}
\\\widetilde{\mathbf b}\in\Re^{\widetilde N}}}{\inf}~\underset{\mathbf x\in[-1,1]^{d}}{\text{ess\,sup}}\left\vert  \widetilde{\boldsymbol w}^\top \Psi\left( \widetilde{\mathbf W} \mathbf x +\widetilde{\mathbf b}\right)  -g(\mathbf x)\right\vert \leq \mathcal C_{NN}\cdot ({\widetilde N})^{-r/d},
\end{multline}
 for any positive integers $\widetilde N:\,\widetilde N\leq K\cdot \mathcal D$. 

We now invoke Theorem \ref{general result theorem regularized NN} with   ${s_A}:=(d+2)\cdot\widetilde N$, where we let $\widetilde N= \min\{K\cdot {\mathcal D},\,n^{1/3}\}$, and 
$\Omega((d+2)\widetilde N):=\mathcal C_{NN}\cdot   ({\widetilde N})^{-r/d}\leq \mathcal C_{NN}\cdot \ \max\{n^{-\frac{r}{3d}},\,(K\cdot {\mathcal D})^{-r/d}\}\leq \mathcal C_{NN}.$
To satisfy \eqref{sample initial requirement 3}, it suffices to stipulate   both $n>c_1\cdot \left(\mathcal C_{NN}\cdot v^{-1}\ln n\right)^3+c_1\cdot(\Gamma+1)^3$ and
\begin{multline}n>\min\left\{c_2\cdot \left[ (d+2)\cdot {\mathcal D}\cdot \mathcal U_{NN}\cdot \ln \left(\mathcal U_{NN}\cdot\left(npR_\Omega v^{-1}+1\right)\right)\right]^{3/2},\right.
\\\left.\,c_2\cdot \left[ (d+2)\cdot  K\cdot {\mathcal D}^2\cdot \mathcal U_{NN}\cdot \ln \left(\mathcal U_{NN}\cdot\left(npR_\Omega v^{-1}+1\right)\right)\right]\vphantom{a^{3^3}}\right\},
\end{multline} 
  which are simultaneously satisfied by \eqref{sample initial requirement 3 bound here}.
Then, the desired result is   implied by Theorem \ref{general result theorem regularized NN}.
\hfill \ensuremath{\Box}\end{proof} 

\subsubsection{Proof of suboptimality-independent generalizability of NN}\label{proof of local solution results}
\begin{proof}{Proof of Theorem \ref{local solution results}.} We  first show Part (b). We  denote by $c_0,\,c_1,\,c_2,...$  potentially different universal constants throughout this proof. The general idea is (i) to first show that  Algorithm 1 always generates a sparse solution and that, with the initialization via Algorithm 2, the  suboptimality gap is well controlled, and (ii) then, to invoke Proposition \ref{test proposition important}, which provides generalization error bounds for sparse solutions with a small   suboptimality gap.  Accordingly, this proof is divided into three steps, with the analysis for (i) provided in Steps 1 and 2, and the details for (ii) provided in Step 3.

 Our proof relies on the analysis of the following hypothetical formulation:
$\inf_{\boldsymbol\beta\in\Re^p}\,\frac{1}{n}\sum_{i=1}^n\min\left\{\ln 2,\,\mathcal F\left(y_i\cdot F_{NN}(\mathbf x_i,{\boldsymbol\beta})\right)\right\}+\sum_{j=1}^pP_\lambda(\vert \beta_j\vert).$ Meanwhile, because of the termination criterion in \eqref{termination criterion Algorithm 1} (where  $\widetilde{f}(\,\cdot\,):=n^{-1}\sum_{i=1}^n\mathcal F\left(y_i F_{NN}(\mathbf x_i,\,\cdot\,)\right)$), we have, for all $k=1,...,k^*(\widehat{\boldsymbol\beta}^{initial},\mathbf X,\mathbf y)$, \begin{align}
n^{-1}\sum_{i=1}^n\mathcal F\left(y_i F_{NN}(\mathbf x_i,\,{\boldsymbol\beta}^{k}\,)\right)+\sum_{j=1}^pP_\lambda(\vert\beta_j^k\vert)\leq n^{-1}\sum_{i=1}^n\mathcal F\left(y_i F_{NN}(\mathbf x_i,\,\widehat{\boldsymbol\beta}^{initial}\,)\right)+\sum_{j=1}^pP_\lambda(\vert\widehat\beta_j^{initial}\vert)-k\frac{\gamma^2_{opt}}{2\mathcal M}.\label{descent algorithm inequality ensured}
\end{align}
 
{\bf Step 1.} For the above hypothetical problem, this step verifies that the conditions required by Proposition \ref{comprehensive proposition} are satisfied in the case where $k^*(\widehat{\boldsymbol\beta}^{initial},\mathbf X,\mathbf y)\geq 1$. We  divide this step into five sub-steps as below.
 
 {\it Step 1.1.} We  first verify Assumption \ref{Assumption A-sparsity}. Because of \eqref{sample size initial requirement 1}, it holds  that $K^*:=\left\lceil 10 n^{1/3}\cdot (\ln n)^{5/3}\right\rceil\geq  d\ln(dK^*)= d\ln ( \lceil10 n^{1/3}\cdot (\ln n)^{5/3}\rceil d)$. 
Thus, as a direct implication of Lemma \ref{useful here for local solution NN},  
\begin{multline}
\mathbb P\left[\sup_{(\mathbf x,\mathbf y)\in supp(\mathbb D)}\left\vert\frac{y\cdot\ln n}{v}\frac{1}{{K^*}}\sum_{k=1}^{K^*} C_g(\xi_k)\cdot \max\{0,\,\mathbf x^\top\xi_k\}-\frac{y\cdot\ln n}{v}g(\mathbf x) \right\vert\leq c_1 \cdot\frac{\ln n}{v}\cdot \sqrt{\frac{d\ln\left(d{K^*}\right)}{{K^*}}}\right]
\\
\geq  1-  2\exp\left(-d\ln\left(d{K^*}\right)\right)-\exp(-d\cdot {K^*}).\label{probability bound bernstein new to show here}
\end{multline}
where $\xi$   follow the same definition as in Lemma \ref{useful here for local solution NN} and $K^*$ is defined as in Algorithm 1.
Observe that, as per Assumption  \ref{Gaussian assumption NN}, 
$y\cdot g(\mathbf x)=y\cdot \mathbb{E}_{\xi}\left[C_g(\xi)\cdot \max\{0,\,\xi^\top\mathbf x\}\right]\geq v
\Longleftrightarrow \frac{\ln n}{v}y\cdot g(\mathbf x)\geq \ln n,$ 
 for all $(\mathbf x,y)\in supp(\mathbb D)$. Also observe that the first and second derivatives of $\mathcal F$ are calculated as $\mathcal F'(z)=-\frac{\exp(-z)}{1+\exp(-z)}$ and $\mathcal F''(z)=\frac{\exp(z)}{(1+\exp(z))^2}=\frac{\exp(z)}{1+\exp(2z)+2\exp(z)}$. Thus  $\mathcal F'$ is 0.5-Lipschitz continuous and hence a well-known inequality  yields that
$
\mathcal F(x_1)-\mathcal F(x_2)\leq   \mathcal F'(x_2)\cdot (x_1-x_2)+0.5/2\cdot (x_1-x_2)^2.$ In view of $\vert \mathcal F'(z)\vert=\frac{\exp(-z)}{1+\exp(-z)}\leq  \frac{1/n}{1+1/n} \leq \frac{1}{n}$ for all $z\geq \ln n$, we then obtain that $\vert \mathcal F'(v^{-1}\cdot y\cdot g(\mathbf x)\cdot \ln n)\vert \leq \frac{1}{n}$. This, combined with \eqref{probability bound bernstein new to show here}, yields that \begin{multline}
\sup_{(\mathbf x,\,y)\in supp(\mathbb D)}\left\{\mathcal F\left(\frac{y\cdot\ln n}{{K^*}v}\sum_{k=1}^{K^*}C_g(\xi_k)\cdot \max\left\{0,\,\xi_k^\top\mathbf x\right\}\right)  - \mathcal F\left(\frac{y\cdot\ln n}{v}g(\mathbf x)\right) \right\}
\\\leq c_2\cdot \frac{1}{n}\cdot \ln n\cdot   \sqrt{\frac{d\ln\left(d\cdot {K^*} \right)}{{K^*} \cdot v^2}}+c_2\cdot (\ln n)^2\frac{d\ln\left(d\cdot {K^*} \right)}{{K^*} \cdot v^2},\label{important inequality derived Generalizability}
\end{multline}
with probability 
$1-  2\exp\left(-d\ln\left(d {K^*}   \right)\right)-\exp(-d\cdot {K^*} )$.
Observe that $\frac{\ln n}{v}\frac{1}{{K^*}}\sum_{k=1}^{K^*} C_g(\xi_k)\cdot \max\left\{0,\,\xi_k^\top\mathbf x\right\}$ is representable by $F_{NN}(\mathbf x,\,\boldsymbol\beta)$ for some $\boldsymbol\beta:\,\Vert\boldsymbol\beta\Vert_0\leq K^*\cdot (d+1),\, \Vert\boldsymbol\beta\Vert_\infty\leq n$. To see this, we can assign the fitting parameters in \eqref{zero NN equality 2}-\eqref{first NN equality 3} to be  the following: (i) Let
${\boldsymbol w}_{1,\mathcal D}:= (\widetilde{\boldsymbol w}_{1,\mathcal D}^\top,\,\underbrace{0,\,...,\,0}_{\text{$(K-K^*)$-many 0's}} )^\top\in\Re^K$ and ${\mathbf W}_{0,1}:=\left[\begin{matrix}\widetilde{\mathbf W}^\top_{0,1},~\mathbf 0_{d\times (K-K^*)}
\end{matrix}\right]^\top\in\Re^{K\times d}$,
where $\widetilde{\boldsymbol w}_{1,\mathcal D}=(\frac{y\cdot\ln n}{{K^*}v}\cdot C_g(\xi_k):\,k=1,...,{K^*})$, $\widetilde{\mathbf W}_{0,1}=(\xi_k^\top:\,k=1,...,{K^*})$, and $\mathbf 0_{  d\times (K-K^*)}$ is a $d$-by-$(K-K^*)$ all-zero matrix.  (ii)  Let   the rest of the fitting parameters to be zero.  With the foregoing assignment of values,  no more than ${K^*}\cdot (d+1)$-many of the fitting parameters are nonzero.  Furthermore,
$\mathbb P\left[\max_{k\in\{1,...,{K^*}\}}\{\Vert \xi_k\Vert_\infty\}\leq n\right]\geq 1-d{K^*}\cdot \exp(-\frac{n^2}{2})$, by  the fact that each entry of $\xi_k$ is an i.i.d. standard Gaussian random variable. Meanwhile,  $\Vert\widetilde{\boldsymbol w}_{1,\mathcal D}\Vert_\infty\leq v^{-1}\ln n\leq n$ (because \eqref{sample size initial requirement 1} implies that $n\geq \frac{\ln n}{v}$). 

Consequently, \eqref{important inequality derived Generalizability} implies that
$
\,\min_{\substack{\boldsymbol\beta:\, \,\Vert\boldsymbol\beta\Vert_0\leq (d+1){K^*},\,
\\\Vert\boldsymbol\beta\Vert_\infty\leq n}}\,\mathbb E[\mathcal F(y\cdot F_{NN}(\mathbf x,\boldsymbol\beta))] -\mathbb E[\mathcal F(v^{-1}y\cdot g(\mathbf x)\cdot \ln n)]\nonumber
\leq\, \sup_{(\mathbf x,\,y)\in supp(\mathbb D)}\{\mathcal F(\frac{y\cdot\ln n}{{K^*}v}\sum_{k=1}^{K^*}C_g(\xi_k)\cdot \max\{0,\,\xi_k^\top\mathbf x\})  - \mathcal F(\frac{y\cdot\ln n}{v}g(\mathbf x)) \}\nonumber
\leq \,  c_2\cdot \frac{1}{n}\cdot \ln n\cdot   \sqrt{\frac{d\ln(d\cdot {K^*} )}{{K^*} \cdot v^2}}+c_2\cdot (\ln n)^2\frac{d\ln(d\cdot {K^*} )}{{K^*} \cdot v^2},$
with probability at least $1-  2\exp\left(-d\ln\left(d{K^*}\right)\right)-\exp(-d\cdot {K^*})-d{K^*}\cdot \exp(-\frac{n^2}{2})$.  Furthermore, because Assumption \ref{Gaussian assumption NN} and the definition of $\mathcal F$ (which is a decreasing function) imply that 
\begin{align}\mathcal F\left(v^{-1}y\cdot g(\mathbf x)\cdot \ln n\right)\leq \mathcal F\left(\ln n\right)=\ln\left(1+\exp(-\ln n)\right)\leq 1/n\leq \ln 2.\label{inequality as per definition of F}
\end{align}
for all $(\mathbf x,y)\in supp(\mathbb D)$, we may continue from the above to obtain that
$
\min_{\substack{\boldsymbol\beta:\, \,\Vert\boldsymbol\beta\Vert_0\leq (d+1){K^*},\,
\\\Vert\boldsymbol\beta\Vert_\infty\leq n}}\,\mathbb E\left[\min\left\{\ln 2,\,\mathcal F\left(y\cdot F_{NN}(\mathbf x,\boldsymbol\beta)\right)\right\}\right]-0
\leq   c_2\cdot \frac{1}{n}\cdot \ln n\cdot   \sqrt{\frac{d\ln\left(d\cdot {K^*} \right)}{{K^*} \cdot v^2}}+c_2\cdot (\ln n)^2\frac{d\ln\left(d\cdot {K^*} \right)}{{K^*} \cdot v^2}+\frac{1}{n},$
with probability at least $1-  2\exp\left(-d\ln\left(d{K^*}\right)\right)-\exp(-d\cdot {K^*})-d{K^*}\cdot \exp(-\frac{n^2}{2})$.
Because $\mathcal F(t)> 0$ for all $t\in\Re$, we thus know that A-sparsity as in Assumption \ref{Assumption A-sparsity} (while we let $\mathbb L(\cdot)$, $L_g^*$, $s$, $R$, and $\varepsilon_A$ from that definition to be $\mathbb L(\cdot):=\mathbb E\left[\min\left\{\ln 2,\,\mathcal F\left(y\cdot F_{NN}(\mathbf x,\,\cdot\,)\right)\right\}\right]$, $L_g^*:=0$, $s:=(d+1)\cdot K^*$, $R:=n$, and $\varepsilon_A:=c_2\cdot \frac{1}{n}\cdot \ln n\cdot   \sqrt{\frac{d\ln\left(d\cdot {K^*} \right)}{{K^*} \cdot v^2}}+c_2\cdot (\ln n)^2\frac{d\ln\left(d\cdot {K^*} \right)}{{K^*} \cdot v^2}+\frac{1}{n}$, respectively) holds with probability at least $1-  2\exp\left(-d\ln\left(d{K^*}\right)\right)-\exp(-d\cdot {K^*})-d{K^*}\cdot \exp(-\frac{n^2}{2})$. 
This completes Step 1.1.

 {\it Step 1.2.} Because  $\min\left\{\ln 2,\,\mathcal F\left(y_i\cdot F_{NN}(\mathbf x_i,\widehat{\boldsymbol\beta})\right)\right\}\in(0,\,\ln2]$, we thus know that  $\left\Vert \min\left\{\ln 2,\,\mathcal F\left(y_i\cdot F_{NN}(\mathbf x_i,\widehat{\boldsymbol\beta})\right)\right\}\right\Vert_{\psi_1}\leq  1$ from the same argument as in deriving \eqref{subexponential norm}.  Therefore, Assumption \ref{sub exponential condition} holds with $\sigma=1$.

 {\it Step 1.3.} To verify Assumption \ref{Lipschitz condition}, we observe   that $\Vert \mathbf W_{l-1,l}\Vert\leq \Vert \mathbf W_{l-1,l}\Vert_F\leq K\cdot R$ for all  $\boldsymbol\beta=vec((\mathbf W_{l-1,l}:\,2\leq l\leq \mathcal D-1),(\mathbf b_{l-1,l}:\,2\leq l\leq \mathcal D-1),\,\boldsymbol w_{\mathcal D-1,\mathcal D},\,\boldsymbol w_{1,\mathcal D},\,b_{\mathcal D-1,\mathcal D},\mathbf W_{0,1},\mathbf b_{0,1}):\,\Vert \boldsymbol\beta\Vert_\infty\leq R$, $\mathbf x:\,\Vert\mathbf x\Vert=1$ and $2\leq l\leq \mathcal D-1$ (because $\mathbf W_{l-1,l}$ has no more than $K^2$-many entries and the  absolute value of each entry has an upper bound of $R$).  Likewise, it also holds that $\Vert\mathbf W_{0,1} \Vert\leq \Vert\mathbf W_{0,1} \Vert_F\leq\sqrt{d\cdot K}\cdot R\leq KR$ (where the last inequality is due to  $K\geq d$)
 and $\Vert\mathbf b_{l-1,l}\Vert\leq \sqrt{K}\cdot R$ for all $l=1,...,\mathcal D$. Therefore, by \eqref{zero NN equality 2}-\eqref{first NN equality 3},
$\Vert f_{NN,l}(\mathbf x,\boldsymbol\beta)\Vert\leq \left[\prod_{l'=1}^{l} \Vert \mathbf W_{l'-1,l'}\Vert\right]\cdot\Vert\mathbf x\Vert+\sum_{\ell=2}^{l}\left[\prod_{l'=\ell}^{l} \Vert \mathbf W_{l'-1,l'}\Vert\right]\cdot\Vert\mathbf b_{\ell-2,\ell-1}\Vert + \Vert\mathbf b_{l-1,l}\Vert\leq ({K}\cdot R)^{l}+\sum_{\ell=2}^{l+1}\left({K}\cdot R\right)^{l-\ell+1}\cdot {K}\cdot R\leq ({K}\cdot R)^{l}+\frac{({K}R)^{l}}{1-({K}R)^{-1}}$. Since ${K}\geq 2$ and $R\geq 1$ we have  $\Vert f_{NN,l}(\mathbf x,\boldsymbol\beta)\Vert\leq 3\cdot ({K}\cdot R)^{l}$ for all $l:\,2\leq l\leq \mathcal D-1$. 

Based on the above, one may further verify that   
$\vert n^{-1}\sum_{i=1}^n\min\{\ln 2,\,\mathcal F(y_i\cdot F_{NN}(\mathbf x_i,\widetilde{\boldsymbol\beta}_1))\}-n^{-1}\sum_{i=1}^n\min\{\ln 2,\,\mathcal F(y_i\cdot F_{NN}(\mathbf x_i,\widetilde{\boldsymbol\beta}_2))\}\vert \leq 3\sqrt{p}\cdot (K\cdot R)^{\mathcal D}\cdot \Vert \widetilde{\boldsymbol\beta}_1-\widetilde{\boldsymbol\beta}_2\Vert,$
 for any $\widetilde{\boldsymbol\beta}_1,\,\widetilde{\boldsymbol\beta}_2\,\in\,\left\{{\boldsymbol\beta}:\,\Vert {\boldsymbol\beta}\Vert_\infty\leq R\right\}$.   
To see this, consider the case where  $\widetilde{\boldsymbol\beta}_2=\widetilde{\boldsymbol\beta}_1+e_j\cdot \delta$ for any $\delta\in\Re$  such that $\widetilde{\boldsymbol\beta}_1,\,\widetilde{\boldsymbol\beta}_2\in \{\boldsymbol\beta:\Vert{\boldsymbol\beta}\Vert_\infty\leq R\}$, it holds that
$\vert n^{-1}\sum_{i=1}^n\min\{\ln 2,\,\mathcal F(y_i\cdot F_{NN}(\mathbf x_i,\widetilde{\boldsymbol\beta}_1))\}-n^{-1}\sum_{i=1}^n\min\{\ln 2,\,\mathcal F(y_i\cdot F_{NN}(\mathbf x_i,\widetilde{\boldsymbol\beta}_2))\}\vert 
\leq  n^{-1}\sum_{i=1}^n \vert\min\{\ln 2,\,\mathcal F(y_i\cdot F_{NN}(\mathbf x_i,\widetilde{\boldsymbol\beta}_1))\}- \min\{\ln 2,\,\mathcal F(y_i\cdot F_{NN}(\mathbf x_i,\widetilde{\boldsymbol\beta}_2))\}\vert\nonumber
 \leq  n^{-1}\sum_{i=1}^n \vert  \mathcal F(y_i\cdot F_{NN}(\mathbf x_i,\widetilde{\boldsymbol\beta}_1)) - \mathcal F(y_i\cdot F_{NN}(\mathbf x_i,\widetilde{\boldsymbol\beta}_2))\vert.
$
Recall that $\vert\mathcal F'(z)\vert\leq 1$  for all $z\in\Re$ (from which we obtain that $\mathcal F(z)$ is 1-Lipscthiz continuous). Together with the fact that $y_i\in\{-1,1\}$ for all $i$, the above implies that
$\left\vert n^{-1}\sum_{i=1}^n\min\left\{\ln 2,\,\mathcal F\left(y_i\cdot F_{NN}(\mathbf x_i,\widetilde{\boldsymbol\beta}_1)\right)\right\}-n^{-1}\sum_{i=1}^n\min\left\{\ln 2,\,\mathcal F\left(y_i\cdot F_{NN}(\mathbf x_i,\widetilde{\boldsymbol\beta}_2)\right)\right\}\right\vert\nonumber
   \leq  n^{-1}\sum_{i=1}^n \left\vert   y_i\cdot F_{NN}(\mathbf x_i,\widetilde{\boldsymbol\beta}_1)  - y_i\cdot F_{NN}(\mathbf x_i,\widetilde{\boldsymbol\beta}_2) \right\vert\leq n^{-1}\sum_{i=1}^n \left\vert    F_{NN}(\mathbf x_i,\widetilde{\boldsymbol\beta}_1)  -  F_{NN}(\mathbf x_i,\widetilde{\boldsymbol\beta}_2) \right\vert.$ Recall that $\widetilde{\boldsymbol\beta}_2=\widetilde{\boldsymbol\beta}_1+e_j\cdot \delta$.
Let the $j$th fitting parameter be  the weight for the connection  between the $\iota_1$th neuron in Layer $(l-1)$ and the $\iota_2$th neuron in  Layer  $l$ for any $l:\,2\leq l\leq \mathcal D-1$. Then, \eqref{zero NN equality 2}-\eqref{first NN equality 3} and $\Vert f_{NN,l}(\mathbf x,\boldsymbol\beta)\Vert\leq 3\cdot ({K}\cdot R)^{l}$ lead to 
$
\left\vert    F_{NN}(\mathbf x_i,\widetilde{\boldsymbol\beta}_1)  -  F_{NN}(\mathbf x_i,\widetilde{\boldsymbol\beta}_2) \right\vert \leq \Vert \boldsymbol w_{\mathcal D-1,\mathcal D}\Vert\cdot \left(\prod_{\ell=l+1}^{\mathcal D-1}\Vert\mathbf W_{\ell-1,\ell}\Vert\right)\cdot  \delta\cdot \Vert f_{NN,l-1}(\mathbf x,\widetilde{\boldsymbol\beta}_1)\Vert\leq 3(KR)^{\mathcal D}\cdot \delta.$
We may generalize the above argument  to all the dimensions of $\boldsymbol\beta$.  Consequently, if $\widetilde{\boldsymbol\beta}_2=\widetilde{\boldsymbol\beta}_1+\sum_{j=1}^pe_j\cdot \delta_j$ for any $\{\delta_j\}\subset\Re:\widetilde{\boldsymbol\beta}_1,\,\widetilde{\boldsymbol\beta}_2\in \{\boldsymbol\beta:\Vert{\boldsymbol\beta}\Vert_\infty\leq R\}$, then
$\vert n^{-1}\sum_{i=1}^n\min\{\ln 2,\,\mathcal F(y_i\cdot F_{NN}(\mathbf x_i,\widetilde{\boldsymbol\beta}_1))\}-n^{-1}\sum_{i=1}^n\min\{\ln 2,\,\mathcal F(y_i\cdot F_{NN}(\mathbf x_i,\widetilde{\boldsymbol\beta}_2))\}\vert
 \leq  3(K\cdot R)^{\mathcal D}\sum_{j=1}^p\cdot  \vert\delta_j \vert
\leq 3\sqrt{p}\cdot (K\cdot R)^{\mathcal D}\cdot \sqrt{ \sum_{j=1}^{p}\vert\delta_j\vert^2}\leq 3\sqrt{p}\cdot (K\cdot R)^{\mathcal D}\cdot \Vert \widetilde{\boldsymbol\beta}_1-\widetilde{\boldsymbol\beta}_2\Vert.$
Thus, Assumption \ref{Lipschitz condition} holds with $\sigma_L=0$ and $\mathcal C_\mu =  3\sqrt{p}\cdot (K\cdot R)^{\mathcal D}$.

 {\it Step 1.4.} It is evident from the same argument as in proving Part (d) of Theorem \ref{Algorithm complexity theorem}  that $\widehat{\boldsymbol\beta}=(\widehat{\beta}_j)$, where we let $\widehat{\boldsymbol\beta}:=\boldsymbol\beta^{k^*(\widehat{\boldsymbol\beta}^{initial},\mathbf X,\mathbf y)}$, satisfies that   $\vert \widehat\beta_j\vert\notin(0,\,a\lambda)$ for all $j=1,...,p$, if $k^*(\widehat{\boldsymbol\beta}^{initial},\mathbf X,\mathbf y)\geq 1$. 

 {\it Step 1.5.} This sub-step is to derive an estimate on the suboptimality gap $\Gamma$ for the initial solution generated through  Algorithm 2.   As per Lemma \ref{useful here for local solution NN}, because $K^*\geq   d\cdot \ln(d\cdot K^*)$ and $\widetilde{\mathbf W}_{0,1}^{initial}=\left(\left(\mathbf w_{0,1,k}^{initial}\right)^\top:\,k=1,...,K^*\right)$ has i.i.d. standard normal entries (and thus $\mathbf w_{0,1,k}^{initial}$ follows the same distribution as both $\xi$ and $\xi_k$) it  holds that
\[
\sup_{(\mathbf x,\mathbf y)\in supp(\mathbb D)}\left\vert\frac{y\cdot\ln n}{v}\left(\frac{1}{{K^*}}\sum_{k=1}^{K^*} C_g(\mathbf w_{0,1,k}^{initial}) \max\{0,\,\mathbf x^\top\mathbf w_{0,1,k}^{initial}\}-g(\mathbf x)\right)\right\vert\vphantom{\sqrt{\frac{d\ln\left(d{K^*}\right)}{{K^*}}}}
\leq c_3 \cdot\frac{\ln n}{v}\cdot \sqrt{\frac{d\ln\left(d{K^*}\right)}{{K^*}}},\]
with probability at least $1-  2\exp\left(-d\ln\left(d{K^*}\right)\right)-\exp(-d\cdot {K^*})$.
Following the same argument as in deriving \eqref{important inequality derived Generalizability}, we  obtain
$
 n^{-1}\sum_{i=1}^n\mathcal F\left(\frac{y_i\ln n}{{K^*}v}\sum_{k=1}^{K^*}C_g(\mathbf w_{0,1,k}^{initial})\cdot \max\left\{0,\,\left(\mathbf w_{0,1,k}^{initial}\right)^\top\mathbf x_i\right\}\right) -n^{-1}\sum_{i=1}^n\mathcal F\left(\frac{y_i\cdot\ln n}{v} g(\mathbf x_i)\right) 
\leq c_4\cdot \frac{1}{n}\cdot \ln n\cdot   \sqrt{\frac{d\ln\left(d\cdot {K^*} \right)}{{K^*} \cdot v^2}}+c_4\cdot (\ln n)^2\frac{d\ln\left(d\cdot {K^*} \right)}{{K^*} \cdot v^2}$
with probability at least $1-  2\exp\left(-d\ln\left(d {K^*}   \right)\right)-\exp(-d\cdot {K^*} )$. As an immediate result,
 \begin{align}
 & \frac{1}{n}\sum_{i=1}^n\mathcal F\left(y_i\cdot F^{sub}_{NN}\left(\mathbf x_i,({\mathbf W}^{initial}_{0,1},\, {\boldsymbol w}^{initial}_{1,L})\right)\right)-n^{-1}\sum_{i=1}^n\mathcal F\left(\frac{y_i\cdot\ln n}{v} g(\mathbf x_i)\right) \nonumber
\\\leq & n^{-1}\sum_{i=1}^n\mathcal F\left(\frac{y_i\ln n}{{K^*}v}\sum_{k=1}^{K^*}C_g(\mathbf w_{0,l,k}^{initial})\cdot \max\left\{0,\,\left(\mathbf w_{0,l,k}^{initial}\right)^\top\mathbf x_i\right\}\right) -n^{-1}\sum_{i=1}^n\mathcal F\left(\frac{y_i\cdot\ln n}{v} g(\mathbf x_i)\right) \nonumber
\\\leq &c_4\cdot \frac{1}{n}\cdot \ln n\cdot   \sqrt{\frac{d\ln\left(d\cdot {K^*} \right)}{{K^*} \cdot v^2}}+c_4\cdot (\ln n)^2\frac{d\ln\left(d\cdot {K^*} \right)}{{K^*} \cdot v^2}.\label{important inequality derived Generalizability 2}
\end{align}
with probability at least $1-  2\exp\left(-d\ln\left(d {K^*}   \right)\right)-\exp(-d\cdot {K^*} )$.
 Further recall that $F^{sub}_{NN}\left(\,\cdot\,,({\mathbf W}^{initial}_{0,1},\, {\boldsymbol w}^{initial}_{1,L})\right) = F_{NN}(\,\cdot\,,\widehat{\boldsymbol\beta}^{initial})$. We thus have (combined with \eqref{inequality as per definition of F})
 $n^{-1}\sum_{i=1}^n\mathcal F\left(y_i\cdot F_{NN}(\mathbf x_i,\widehat{\boldsymbol\beta}^{initial})\right)=\frac{1}{n}\sum_{i=1}^n\mathcal F\left(y_i\cdot F^{sub}_{NN}\left(\mathbf x_i,({\mathbf W}^{initial}_{0,1},\, {\boldsymbol w}^{initial}_{1,L})\right)\right)
\leq  \frac{1}{n} +c_4\cdot \frac{1}{n}\cdot \ln n\cdot   \sqrt{\frac{d\ln\left(d\cdot {K^*} \right)}{{K^*}\cdot  v^2}}+c_4\cdot (\ln n)^2\frac{d\ln\left(d\cdot {K^*} \right)}{{K^*}\cdot  v^2},$
with probability at least $1- 2 \exp\left(-d\ln\left(d{K^*}\right)\right)-\exp(-d\cdot {K^*})$.
Because  $P_\lambda(\cdot)\leq\frac{a\lambda^2}{2}$ and $\Vert \widehat{\boldsymbol\beta}^{initial}\Vert_0\leq (d+1)\cdot K^*$, we further obtain
\begin{multline}
n^{-1}\sum_{i=1}^n\mathcal F\left(y_i\cdot F_{NN}(\mathbf x_i,\widehat{\boldsymbol\beta}^{initial})\right)+\sum_{j=1}^pP_\lambda(\vert\widehat\beta_j^{initial}\vert)\leq   \frac{1}{n} +c_3\cdot \frac{1}{n}\cdot \ln n\cdot   \sqrt{\frac{d\ln\left(d\cdot {K^*} \right)}{{K^*}\cdot  v^2}}
\\+c_3\cdot (\ln n)^2\frac{d\ln\left(d\cdot {K^*} \right)}{{K^*}\cdot  v^2} +{K^*}\cdot (d+1)\cdot \frac{a\lambda^2}{2},\label{suboptimality gap}
\end{multline}
with probability at least $1- 2 \exp\left(-d\ln\left(d{K^*}\right)\right)-\exp(-d\cdot {K^*})$.  Because of \eqref{descent algorithm inequality ensured}, we have $n^{-1}\sum_{i=1}^n\mathcal F\left(y_i\cdot F_{NN}(\mathbf x_i,\widehat{\boldsymbol\beta})\right)+\sum_{j=1}^pP_\lambda(\widehat\beta_j)\leq n^{-1}\sum_{i=1}^n\mathcal F\left(y_i\cdot F_{NN}(\mathbf x_i,\widehat{\boldsymbol\beta}^{initial})\right)+\sum_{j=1}^pP_\lambda(\widehat\beta_j^{initial})$. It thus holds that
$
n^{-1}\sum_{i=1}^n\min\left\{\ln 2,\,\mathcal F\left(y_i\cdot F_{NN}(\mathbf x_i,\widehat{\boldsymbol\beta})\right)\right\}+\sum_{j=1}^pP_\lambda(\vert\widehat\beta_j\vert)\leq   \frac{1}{n} +c_4\cdot \frac{1}{n}\cdot \ln n\cdot   \sqrt{\frac{d\ln\left(d\cdot {K^*} \right)}{{K^*}\cdot  v^2}}
+c_4\cdot (\ln n)^2\frac{d\ln\left(d\cdot {K^*} \right)}{{K^*}\cdot  v^2} +{K^*}\cdot (d+1)\cdot \frac{a\lambda^2}{2},$
with probability at least $1- 2 \exp\left(-d\ln\left(d{K^*}\right)\right)-\exp(-d\cdot {K^*})$. 
Further observing that $\inf_t\,\mathcal F(t)=0$ and $\inf_t\,P_\lambda(\vert t\vert)= 0$, we then have that $n^{-1}\sum_{i=1}^n\min\left\{\ln 2,\,\mathcal F\left(y_i\cdot F_{NN}(\mathbf x_i,\widehat{\boldsymbol\beta})\right)\right\}+\sum_{j=1}^pP_\lambda(\vert\widehat\beta_j \vert)\leq \inf_{\boldsymbol\beta}\left[n^{-1}\sum_{i=1}^n\min\left\{\ln 2,\,\mathcal F\left(y_i\cdot F_{NN}(\mathbf x_i, {\boldsymbol\beta})\right)\right\}+\sum_{j=1}^pP_\lambda(\vert \beta_j \vert)\right]+\Gamma$ with $\Gamma:=\frac{1}{n} +c_3\cdot \frac{1}{n}\cdot \ln n\cdot   \sqrt{\frac{d\ln\left(d\cdot {K^*} \right)}{{K^*}\cdot  v^2}}
+c_3\cdot (\ln n)^2\frac{d\ln\left(d\cdot {K^*} \right)}{{K^*}\cdot  v^2} +{K^*}\cdot (d+1)\cdot \frac{a\lambda^2}{2}$ with probability at least $1- 2 \exp\left(-d\ln\left(d{K^*}\right)\right)-\exp(-d\cdot {K^*})$.

{\bf Step 2.} 
 In this step, we are to derive an upper bound on $\Vert\widehat{\boldsymbol\beta}\Vert_0$. To that end, we differentiate the cases of $k^*(\widehat{\boldsymbol\beta}^{initial},\mathbf X,\mathbf y)=0$ and $k^*(\widehat{\boldsymbol\beta}^{initial},\mathbf X,\mathbf y)\geq 1$. 
 
{\it Case 2.1.} We first consider the case of $k^*(\widehat{\boldsymbol\beta}^{initial},\mathbf X,\mathbf y)=0$; that is, $\widehat{\boldsymbol\beta}=
\widehat{\boldsymbol\beta}^{initial}$. In such a case, recall that $K^*:=\left\lceil 10 n^{1/3}\cdot (\ln n)^{5/3}\right\rceil$.  By Algorithm 2, it is evident that 
$
\Vert\widehat{\boldsymbol\beta}\Vert_0\leq K^*\cdot (d+1)= \left\lceil 10 n^{1/3}\cdot (\ln n)^{5/3}\right\rceil\cdot (d+1).
$

{\it Case 2.2.}  In the next, we   consider the case where $k^*(\widehat{\boldsymbol\beta}^{initial},\mathbf X,\mathbf y)\geq 1$.
To that end, we may  invoke Proposition \ref{comprehensive proposition} to bound $\Vert\widehat{\boldsymbol\beta}\Vert_0$. According to Step 1, with probability at least $1-4  \exp\left(-d\ln\left(d{K^*}\right)\right)-2\exp(-d\cdot {K^*})-d{K^*}\cdot \exp(-\frac{n^2}{2})$, all the assumptions required by Proposition \ref{comprehensive proposition} are satisfied with the following configurations: $\mathcal L_{n,\lambda}(\widehat{\boldsymbol\beta},\mathbf Z_1^n):=\frac{1}{n}\sum_{i=1}^n\min\left\{\ln 2,\,\mathcal F\left(y_i\cdot F_{NN}(\mathbf x_i,\widehat{\boldsymbol\beta})\right)\right\}+\sum_{j=1}^pP_\lambda(\vert\widehat\beta_j \vert)$ and
\begin{align}
\begin{split}
&\sigma= 1,~~~~\mathcal C_\mu =  3\sqrt{p}\cdot (K\cdot R)^{\mathcal D},~~~~\sigma_L=0,~~~~\varrho:=1/3; 
\\&\lambda:=\sqrt{\frac{8\sigma }{ c\cdot  a\cdot n^{2/3}}[\ln(n^{1/3}p)+\widetilde\zeta]},~~\text{with}~~~ \widetilde\zeta:=\ln\left(9 e  R   \sqrt{p} \cdot ({K}  R)^{\mathcal D}\right)=\ln(9eR{\mathcal D}\sqrt{p} )+{\mathcal D}\ln({K}R); 
\\&{s}:={K^*}(d+1),
~~~~~L_g^*:=0,~~~~~ \varepsilon_{A}:=c_2\cdot \frac{1}{n}\cdot \ln n\cdot   \sqrt{\frac{d\ln\left(d\cdot {K^*} \right)}{{K^*} \cdot v^2}}+c_2\cdot (\ln n)^2\frac{d\ln\left(d\cdot {K^*} \right)}{{K^*} \cdot v^2}+\frac{1}{n};
\\
&\Gamma:=\frac{1}{{n}} +c_4\cdot \frac{1}{n}\cdot \ln n\cdot   \sqrt{\frac{d\ln\left(d\cdot {K^*} \right)}{{K^*}\cdot  v^2}}
+c_4\cdot (\ln n)^2\frac{d\ln\left(d\cdot {K^*} \right)}{{K^*}\cdot  v^2} +{K^*}\cdot (d+1)\cdot \frac{a\lambda^2}{2}.
\end{split}\label{test new to verify epsilon_a}
\end{align}
To satisfy \eqref{sample initial requirement 2 ori} as required by Proposition \ref{comprehensive proposition}, it suffices to stipulate \eqref{sample size initial requirement 1}. To see this, observe that $\frac{\Gamma+\varepsilon_A}{\sigma}\leq \frac{2}{{n}} +(c_2+c_4)\cdot \frac{1}{n}\cdot \ln n\cdot   \sqrt{\frac{d\ln\left(d\cdot {K^*} \right)}{{K^*}\cdot  v^2}}
+(c_2+c_4)\cdot (\ln n)^2\frac{d\ln\left(d\cdot {K^*} \right)}{{K^*}\cdot  v^2} +{K^*}\cdot (d+1)\cdot \frac{a\lambda^2}{2}\leq \frac{2}{{n}} +(c_2+c_4)\cdot \frac{1}{n}\cdot \ln n\cdot   \sqrt{\frac{d\ln\left(d\cdot {\lceil10n^{1/3}\cdot (\ln n)^{5/3}\rceil} \right)}{{\lceil10n^{1/3}\cdot (\ln n)^{5/3}\rceil}\cdot  v^2}}
+(c_2+c_4)\cdot (\ln n)^2\frac{d\ln\left(d\cdot {\lceil10n^{1/3}\cdot (\ln n)^{5/3}\rceil} \right)}{{\lceil10n^{1/3}\cdot (\ln n)^{5/3}\rceil}\cdot  v^2} +{\lceil10n^{1/3}\cdot (\ln n)^{5/3}\rceil}\cdot (d+1)\cdot \frac{8\sigma}{cn^{2/3}}\cdot [\ln(n^{1/3}p)+\widetilde\xi]<(\frac{n}{2})^{1/3}$ under \eqref{sample size initial requirement 1}. Meanwhile,  it holds that $s(\ln(np)+\widetilde \xi)\leq c_5\cdot d\cdot n^{1/3}\ln(9eR \mathcal D p)\cdot(\ln(n))^{7/3}+c_5\cdot dn^{1/3}\mathcal D\ln(KR)\cdot (\ln(n))^{5/3}<\frac{n}{2}$ under  \eqref{sample size initial requirement 1}. In view of the above (and also $\mathcal D\leq p$ and $K\leq p$), \eqref{sample initial requirement 2 ori} is satisfied. We may now invoke Part (i) in Proposition \ref{comprehensive proposition}, which implies that
$
\Vert \widehat{\boldsymbol\beta}\Vert_0
\leq\, \left\lceil\frac{2cn^{1/3}}{\ln(n^{\varrho}p)+\widetilde \zeta}+\frac{2cn^{2/3}}{\sigma \left(\ln(n^{\varrho}p)+\widetilde \zeta\right)}\cdot\left(\Gamma+\varepsilon_{A}+\frac{2}{n^{\varrho}}\right)+8s\right\rceil\nonumber
 \leq \, \frac{2cn^{1/3}+2cn^{2/3}\cdot \left(\frac{c_6}{{n^{1/3}}}+\frac{c_6}{{n}} +  \frac{c_6}{n}\cdot \ln n\cdot   \sqrt{\frac{d\ln\left(d\cdot {K^*} \right)}{{K^*}\cdot  v^2}}
 +c_6\cdot (\ln n)^2\frac{d\ln\left(d\cdot {K^*} \right)}{{K^*}\cdot  v^2} \right)}{\ln(n^{1/3}p)+\ln(9eR{\mathcal D}\sqrt{p} )
 +{\mathcal D}\ln({K}R)}+c_6\cdot s 
 =:p_{NN},$
with probability at least $1-c_6\cdot  p\cdot \exp(-n^{1/3}/c_6)-c_6\cdot(d\cdot n)^{-d/3}-c_6\cdot n^{1/3}(\ln n)^{5/3} d \exp(-n^2/c_6)\geq 1-c_7\cdot  p\cdot \exp(-n^{1/3}/c_7)-c_7\cdot(d\cdot n)^{-d/3}-c_7\cdot n^{1/3} d \exp(-n^2/c_7)$.

Combining the above two cases, we thus know that,  for all $k^*(\widehat{\boldsymbol\beta}^{initial},\mathbf X,\mathbf y)\geq 0$,
\begin{multline}
\mathbb P\left[\Vert \widehat{\boldsymbol\beta}\Vert_0=\Vert  {\boldsymbol\beta}^{k^*(\widehat{\boldsymbol\beta}^{initial},\mathbf X,\mathbf y)}\Vert_0\leq\max\left\{p_{NN},\,\left\lceil 10 n^{1/3}\cdot (\ln n)^{5/3}\right\rceil  (d+1)\right\}\right]\\\geq 1-c_7\cdot  p\cdot \exp(-n^{1/3}/c_7)-c_7\cdot(d\cdot n)^{-d/3}-c_7\cdot n^{1/3} d \exp(-n^2/c_7).\label{inequality for sparsity level here}
\end{multline}

{\bf Step 3.} This step  employs results from Step 2 and Proposition \ref{test proposition important} to show the desired generalizability of $\widehat{\boldsymbol\beta}$. 
By \eqref{descent algorithm inequality ensured}  (where we let $k=k^*(\widehat{\boldsymbol\beta}^{initial},\mathbf X,\mathbf y)$ and $\widehat{\boldsymbol\beta}={\boldsymbol\beta}^{k^*(\widehat{\boldsymbol\beta}^{initial},\mathbf X,\mathbf y)}$) and \eqref{suboptimality gap} (as well as $P_\lambda(t)\geq 0$ for any $t\geq 0$) together, we obtain that
$
n^{-1}\sum_{i=1}^n\min\left\{\ln 2,\,\mathcal F\left(y_i F_{NN}(\mathbf x_i,\,\widehat{\boldsymbol\beta}\,)\right)\right\} \leq \frac{1}{n} +c_3\cdot \frac{1}{n}\cdot \ln n\cdot   \sqrt{\frac{d\ln\left(d\cdot {K^*} \right)}{{K^*}\cdot  v^2}}
+c_3\cdot (\ln n)^2\frac{d\ln\left(d\cdot {K^*} \right)}{{K^*}\cdot  v^2} +{K^*}\cdot (d+1)\cdot \frac{a\lambda^2}{2}- \frac{\gamma^2_{opt}}{2\mathcal M}\cdot k^*(\widehat{\boldsymbol\beta}^{initial},\mathbf X,\mathbf y),
$
with probability $1- 2 \exp\left(-d\ln\left(d{K^*}\right)\right)-\exp(-d\cdot {K^*})$. Recall that $K^*:={\lceil10n^{1/3}\cdot (\ln n)^{5/3}\rceil}$.  Invoking \eqref{inequality for sparsity level here} in Step 2 and Proposition \ref{test proposition important} with the same $\sigma$, $\sigma_L$, $\widetilde\zeta$, and $\mathcal C_\mu$ as in \eqref{test new to verify epsilon_a},
 we  have
$
\mathbb E[\min\{\ln 2,\,\mathcal F(y_i F_{NN}(\mathbf x_i,\,\widehat{\boldsymbol\beta}\,))\}]-n^{-1}\sum_{i=1}^n\min\{\ln 2,\,\mathcal F(y_i F_{NN}(\mathbf x_i,\,\widehat{\boldsymbol\beta}\,))\}
\leq {\frac{1}{\sqrt{n}}}\sqrt{2 \cdot c^{-1}\cdot\max\{p_{NN},\,\lceil 10 n^{1/3}\cdot (\ln n)^{5/3}\rceil  (d+1)\}}\cdot\sqrt{[\ln(n^{1/3}p)+\widetilde\xi]} 
+\frac{2\sigma }{n} \cdot c^{-1}\cdot \max\{p_{NN},\,\lceil 10 n^{1/3}\cdot (\ln n)^{5/3}\rceil  (d+1)\}[\ln(n^{1/3}p)+\widetilde\xi]+\frac{1}{n^{1/3}},
$
with probability at least $1-c_7\cdot  p\cdot \exp(-n^{1/3}/c_7)-c_7\cdot(d\cdot n)^{-d/3}-c_7\cdot n^{1/3} d \exp(-n^2/c_7)-2\exp\left(- \max\left\{p_{NN},\,\left\lceil 10 n^{1/3}\cdot (\ln n)^{5/3}\right\rceil \right\}\left[\ln(n^{1/3}p)+\widetilde\xi\right])\right)-2\exp(-\widetilde cn)$. Combining the above, we then have
\begin{multline}
\mathbb E\left[\min\left\{\ln 2,\,\mathcal F\left(y_i F_{NN}(\mathbf x_i,\,\widehat{\boldsymbol\beta}\,)\right)\right\}\right]\leq  {\frac{1}{\sqrt{n}}}\sqrt{\frac{2 \max\left\{p_{NN},\,\left\lceil 10 n^{1/3}\cdot (\ln n)^{5/3}\right\rceil  (d+1)\right\}}{c}\cdot \left[\ln(n^{1/3}p)+\widetilde\xi\right]} 
\\+\frac{\sigma }{n}\cdot\frac{2 \max\left\{p_{NN},\,\left\lceil 10 n^{1/3}\cdot (\ln n)^{5/3}\right\rceil  (d+1)\right\}}{c}\left[\ln(n^{1/3}p)+\widetilde\xi\right]+\frac{1}{n^{1/3}}+\frac{1}{n} +c_3\cdot \frac{1}{n}\cdot \ln n\cdot   \sqrt{\frac{d\ln\left(d\cdot {K^*} \right)}{{K^*}\cdot  v^2}}
\\+c_3\cdot (\ln n)^2\frac{d\ln\left(d\cdot {K^*} \right)}{{K^*}\cdot  v^2} +{K^*}\cdot (d+1)\cdot \frac{a\lambda^2}{2}-k^*(\widehat{\boldsymbol\beta}^{initial},\mathbf X,\mathbf y)\cdot \frac{\gamma^2_{opt}}{2\mathcal M},
\end{multline}
with probability at least $1-c_7\cdot  p\cdot \exp(-n^{1/3}/c_7)-c_7\cdot(d\cdot n)^{-d/3}-c_7\cdot n^{1/3} d \exp(-n^2/c_7)-2\exp\left(- \max\left\{p_{NN},\,\left\lceil 10 n^{1/3}\cdot (\ln n)^{5/3}\right\rceil \right\}\left[\ln(n^{1/3}p)+\widetilde\xi\right])\right)-2\exp(-\widetilde cn)- 2 \exp\left(-d\ln\left(d{K^*}\right)\right)-\exp(-d\cdot {K^*})$.
Observing that $d\leq p$, $\mathcal D\leq p$, and $K\leq p$, we obtain after some reorganization   that
$
\mathbb E\left[\min\left\{\ln 2,\,\mathcal F\left(y\cdot F_{NN}(\mathbf x,\widehat{\boldsymbol\beta})\right)\right\}\right]\leq c_7\frac{d \cdot {\mathcal D}}{n^{1/3}v}\cdot \left[\left(\ln n\right)^{4/3}\cdot \ln(pR)\right]-k^*(\widehat{\boldsymbol\beta}^{initial},\mathbf X,\mathbf y)\cdot \frac{\gamma_{opt}^2}{2\mathcal M},
$
with probability $1-c_7\cdot  p\cdot \exp(-n^{1/3}/c_7)-c_7\cdot(d\cdot n)^{-d/3}-c_7\cdot n^{1/3} d \exp(-n^2/c_7)$. Finally, because $2\min\left\{\ln 2,\,\mathcal F\left(z\right)\right\}\geq \mathbb 1\{z<0\}$, $d\leq p$ and $\mathcal D\leq p$, we thus have
$
\mathbb E\left[\mathbb 1\left(y\cdot F_{NN}(\mathbf x,\widehat{\boldsymbol\beta})<0\right)\right]\leq  c_8\cdot \frac{d \cdot {\mathcal D}}{n^{1/3}v^2}\cdot \left[\left(\ln n\right)^{4/3}\cdot \ln(pR)\right]-\gamma_{opt}^2\cdot \frac{k^*(\widehat{\boldsymbol\beta}^{initial},\mathbf X,\mathbf y)}{2\mathcal M},
$
with probability $1-c_8\cdot  p\cdot \exp(-n^{1/3}/c_8)-c_8\cdot n^{1/3} d \exp(-n^2/c_8)-c_8\cdot(d\cdot n)^{-d/3}$.  This then leads to Part (b) of the theorem. 

To show Part (a), suppose that $k^*({\boldsymbol\beta}^{0},\mathbf X,\mathbf y)\geq \left(\left\lceil 2\mathcal M\cdot \frac{\mathcal T_{n,\lambda}({\boldsymbol\beta}^{0})}{\gamma_{opt}^2}\right\rceil+1\right)$ for the sake of contradiction. Then  \eqref{descent algorithm inequality ensured} would imply that $\mathcal T_{n,\lambda}(\boldsymbol\beta^{k^*({\boldsymbol\beta}^{0},\mathbf X,\mathbf y)})=n^{-1}\sum_{i=1}^n\mathcal F\left(y_i F_{NN}(\mathbf x_i,\,{\boldsymbol\beta}^{k^*({\boldsymbol\beta}^{0},\mathbf X,\mathbf y)}\,)\right)+\sum_{j=1}^pP_\lambda(\vert\beta_j^{k^*({\boldsymbol\beta}^{0},\mathbf X,\mathbf y)}\vert)\leq n^{-1}\sum_{i=1}^n\mathcal F\left(y_i F_{NN}(\mathbf x_i,\,{\boldsymbol\beta}^0\,)\right)+\sum_{j=1}^pP_\lambda(\vert\beta_j^0\vert)-\frac{\gamma^2_{opt}}{2\mathcal M}\cdot k^*({\boldsymbol\beta}^{0},\mathbf X,\mathbf y)\leq  \mathcal T_{n,\lambda}({\boldsymbol\beta}^{0})-\left(\left\lceil 2\mathcal M\cdot \frac{\mathcal T_{n,\lambda}({\boldsymbol\beta}^{0})}{\gamma_{opt}^2}\right\rceil+1\right)\cdot \frac{\gamma^2_{opt}}{2\mathcal M}<0$. This contradicts with $\mathcal T_{n,\lambda}(\boldsymbol\beta^{k^*({\boldsymbol\beta}^{0},\mathbf X,\mathbf y)})\geq 0$ (since $\inf_u\mathcal F(u)\geq 0$ and $\inf_uP_\lambda(\vert u\vert)\geq 0$). 
\hfill \ensuremath{\Box}
\end{proof}

\subsection{Proof of Computational Complexity of Algorithm 1} \label{Algorithm complexity proof}
\begin{proof}{Proof of Theorem \ref{Algorithm complexity theorem}.} Note that $\mathcal M\geq \widetilde U_{L,2}$. The following is a useful inequality well-known for a function with Lipschitz gradient:
\begin{align}\widetilde{f}(\boldsymbol\beta_1)-\widetilde{f}(\boldsymbol\beta_2)\leq \left\langle \nabla \widetilde{f}(\boldsymbol\beta_2),\,\boldsymbol\beta_1-\boldsymbol\beta_2\right\rangle+\frac{{\mathcal M}}{2}\Vert \boldsymbol\beta_1-\boldsymbol\beta_2\Vert^2,~~~~\boldsymbol\beta_1,\,\boldsymbol\beta_2\in\Re^p.\label{well known inequality}\end{align}
The KKT conditions for \eqref{first subproblem} in Step 2 of Algorithm 1 yield that
\begin{align}
\nabla \widetilde{f}(\boldsymbol\beta^k)+{\mathcal M}(\boldsymbol\beta^{k+\frac{1}{2}}-\boldsymbol\beta^k)+\left( P_{\lambda}'(\vert\beta^k_j\vert)\cdot \varkappa(\beta_j^{k+\frac{1}{2}}):\,j=1,...,p\right)=\mathbf 0,\label{test new kkt condition first sub problem}
\end{align}
where $\varkappa(\beta_j^{k+\frac{1}{2}})\in \partial\vert \beta_j^{k+\frac{1}{2}}\vert$ and $\partial\vert \beta_j^{k+\frac{1}{2}}\vert$ is the subdifferential  of $\vert\,\cdot\,\vert$ at $\beta_j^{k+\frac{1}{2}}$. Combining \eqref{test new kkt condition first sub problem} with the objective function of Eq.\,\eqref{first subproblem} yields that 
\begin{align}
&\left\langle \nabla \widetilde{f}(\boldsymbol\beta^k),\,\boldsymbol\beta^{k+\frac{1}{2}}-\boldsymbol\beta^k \right\rangle+\frac{{\mathcal M}}{2}\Vert \boldsymbol\beta^{k+\frac{1}{2}}-\boldsymbol\beta^k \Vert^2+\sum_{j=1}^pP_\lambda'(\vert \beta_j^k\vert)\cdot\vert\beta_j^{k+\frac{1}{2}}\vert\nonumber
\\\textcolor{black}{=}&
\left\langle -{\mathcal M}(\boldsymbol\beta^{k+\frac{1}{2}}-\boldsymbol\beta^k)-\left( P_{\lambda}'(\vert\beta^k_j\vert)\cdot \varkappa(\beta_j^{k+\frac{1}{2}}):\,j=1,...,p\right),\,\boldsymbol\beta^{k+\frac{1}{2}}-\boldsymbol\beta^k \right\rangle+\frac{{\mathcal M}}{2}\Vert \boldsymbol\beta^{k+\frac{1}{2}}-\boldsymbol\beta^k \Vert^2+\sum_{j=1}^pP_\lambda'(\vert \beta_j^k\vert)\cdot\vert\beta_j^{k+\frac{1}{2}}\vert\nonumber
\\=&
\left\langle  -\left( P_{\lambda}'(\vert\beta^k_j\vert)\cdot \varkappa(\beta_j^{k+\frac{1}{2}}):\,j=1,...,p\right),\,\boldsymbol\beta^{k+\frac{1}{2}}-\boldsymbol\beta^k \right\rangle-\frac{{\mathcal M}}{2}\Vert \boldsymbol\beta^{k+\frac{1}{2}}-\boldsymbol\beta^k \Vert^2+\sum_{j=1}^pP_\lambda'(\vert \beta_j^k\vert)\cdot\vert\beta_j^{k+\frac{1}{2}}\vert.\nonumber
\end{align}
By the convexity of   $P_\lambda'(\vert \beta_j^k\vert)\cdot \vert  t \vert$ in $t$ for all $t\in\Re$ and all $j$, we may continue the above to have
\begin{align}
&\left\langle \nabla \widetilde{f}(\boldsymbol\beta^k),\,\boldsymbol\beta^{k+\frac{1}{2}}-\boldsymbol\beta^k \right\rangle+\frac{{\mathcal M}}{2}\Vert \boldsymbol\beta^{k+\frac{1}{2}}-\boldsymbol\beta^k \Vert^2+\sum_{j=1}^pP_\lambda'(\vert \beta_j^k\vert)\cdot\vert\beta_j^{k+\frac{1}{2}}\vert\nonumber
\\\leq &-\sum_{j=1}^pP_\lambda'(\vert \beta_j^k\vert)\cdot\vert\beta_j^{k+\frac{1}{2}}\vert+\sum_{j=1}^pP_\lambda'(\vert \beta_j^k\vert)\cdot\vert\beta_j^{k}\vert-\frac{{\mathcal M}}{2}\Vert \boldsymbol\beta^{k+\frac{1}{2}}-\boldsymbol\beta^k \Vert^2+\sum_{j=1}^pP_\lambda'(\vert \beta_j^k\vert)\cdot\vert\beta_j^{k+\frac{1}{2}}\vert\nonumber
\\= &\sum_{j=1}^pP_\lambda'(\vert \beta_j^k\vert)\cdot\vert\beta_j^{k}\vert-\frac{{\mathcal M}}{2}\Vert \boldsymbol\beta^{k+\frac{1}{2}}-\boldsymbol\beta^k \Vert^2.\nonumber
\end{align}
Invoking \eqref{well known inequality} with $\boldsymbol\beta_1:=\boldsymbol\beta^k$ and $\boldsymbol\beta_2:=\boldsymbol\beta^{k+\frac{1}{2}}$, we  obtain from the above that
\begin{align}\left(\widetilde{f}(\boldsymbol\beta^{k+\frac{1}{2}})+\sum_{j=1}^pP_\lambda'(\vert \beta_j^k\vert)\cdot\vert\beta_j^{k+\frac{1}{2}}\vert\right) -\left(\widetilde{f}(\boldsymbol\beta^{k})+\sum_{j=1}^pP_\lambda'(\vert \beta_j^k\vert)\cdot\vert\beta_j^{k}\vert\right)\leq -\frac{{\mathcal M}}{2}\Vert \boldsymbol\beta^{k+\frac{1}{2}}-\boldsymbol\beta^k \Vert^2.\nonumber
\end{align}
Since $P_\lambda(t)$ is concave in $t$ for all $t\geq 0$, we know that $P_\lambda'(\vert\beta_j^k\vert)\cdot(\vert\beta_j^{k+\frac{1}{2}}\vert-\vert\beta_j^{k}\vert)\geq P_{\lambda}(\vert\beta_j^{k+\frac{1}{2}}\vert)-P_{\lambda}(\vert\beta_j^{k}\vert)$. Therefore,
\begin{align}
 \left(\widetilde{f}(\boldsymbol\beta^{k+\frac{1}{2}})+\sum_{j=1}^pP_\lambda(\vert \beta_j^{k+\frac{1}{2}}\vert)\right) -\left(\widetilde{f}(\boldsymbol\beta^{k})+\sum_{j=1}^pP_\lambda(\vert \beta_j^k\vert)\right)\leq -\frac{{\mathcal M}}{2}\Vert \boldsymbol\beta^{k+\frac{1}{2}}-\boldsymbol\beta^k \Vert^2.\label{first problem result}
\end{align}
Consider the second subproblem \eqref{second subproblem} in Step 3 of Algorithm 1. Again, because of the inequality in \eqref{well known inequality}, it holds that $ \widetilde{f}(\boldsymbol\beta^{k+1})-\widetilde{f}(\boldsymbol\beta^{k+\frac{1}{2}})+\sum_{j=1}^pP_\lambda(\vert \beta_j^{k+1}\vert)\leq\left\langle \nabla \widetilde{f}(\boldsymbol\beta^{k+\frac{1}{2}}),\,\boldsymbol\beta^{k+1}-\boldsymbol\beta^{k+\frac{1}{2}} \right\rangle+\frac{{\mathcal M}}{2}\Vert \boldsymbol\beta^{k+1}-\boldsymbol\beta^{k+\frac{1}{2}} \Vert^2+\sum_{j=1}^pP_\lambda(\vert \beta_j^{k+1}\vert)\leq  \left\langle \nabla \widetilde{f}(\boldsymbol\beta^{k+\frac{1}{2}}),\,\boldsymbol\beta^{k+\frac{1}{2}}-\boldsymbol\beta^{k+\frac{1}{2}} \right\rangle+\frac{{\mathcal M}}{2}\Vert \boldsymbol\beta^{k+\frac{1}{2}}-\boldsymbol\beta^{k+\frac{1}{2}} \Vert^2+\sum_{j=1}^pP_\lambda(\vert \beta_j^{k+\frac{1}{2}}\vert)=\sum_{j=1}^pP_\lambda(\vert \beta_j^{k+\frac{1}{2}}\vert)$, where the last inequality is due to the fact that $\boldsymbol\beta^{k+1}$ is the minimizer to the subproblem in \eqref{second subproblem}. By some reorganization, we obtain 
$ \widetilde{f}(\boldsymbol\beta^{k+1})+\sum_{j=1}^pP_\lambda(\vert \beta_j^{k+1}\vert)\leq \widetilde{f}(\boldsymbol\beta^{k+\frac{1}{2}})+\sum_{j=1}^pP_\lambda(\vert \beta_j^{k+\frac{1}{2}}\vert)$. Combining this with \eqref{first problem result}, we have that \begin{align}\widetilde{f}(\boldsymbol\beta^{k+1})+\sum_{j=1}^pP_\lambda(\vert \beta_j^{k+1}\vert)\leq \widetilde{f}(\boldsymbol\beta^{k})+\sum_{j=1}^pP_\lambda(\vert \beta_j^{k}\vert)-\frac{{\mathcal M}}{2}\Vert \boldsymbol\beta^{k+\frac{1}{2}}-\boldsymbol\beta^k\Vert^2.  \label{to check and compare with solution difference}
\end{align}

Before the termination criterion in \eqref{termination criterion Algorithm 1} is met, it must hold that
\[
\widetilde{f}(\boldsymbol\beta^{k+1})+\sum_{j=1}^pP_\lambda(\vert \beta_j^{k+1}\vert)\leq \widetilde{f}(\boldsymbol\beta^{k})+\sum_{j=1}^pP_\lambda(\vert \beta_j^{k}\vert)-\frac{\gamma_{opt}^2}{{2\mathcal M}}.
\]
Invoking the above recursively, we have
\begin{align}
 \widetilde{f}(\boldsymbol\beta^{k})+\sum_{j=1}^pP_\lambda(\vert \beta_j^{k}\vert) \leq \widetilde{f}(\boldsymbol\beta^{0})+\sum_{j=1}^pP_\lambda(\vert \beta_j^{0}\vert) -k\cdot \frac{\gamma_{opt}^2}{{2\mathcal M}}.\label{label inequality descent}
\end{align}
Therefore,  there must exist some $k^*:\,k^*\leq \left\lfloor 2\mathcal M\cdot  \frac{\left(\widetilde{f}(\boldsymbol\beta^0)+\sum_{j=1}^k P_\lambda(\vert \beta_j^0\vert)\right)-\widetilde f_{\lambda}^*}{\gamma_{opt}^2}\right\rfloor+1$ such that $\widetilde{f}(\boldsymbol\beta^{k+1})+\sum_{j=1}^pP_\lambda(\vert \beta_j^{k+1}\vert)>\widetilde{f}(\boldsymbol\beta^{k})+\sum_{j=1}^pP_\lambda(\vert \beta_j^{k}\vert)-\frac{\gamma_{opt}^2}{{2\mathcal M}}$.  This is because, otherwise, Algorithm 1 would keep reducing the objective  value as per \eqref{label inequality descent}. Consequently,  
\begin{align}
&\widetilde{f}(\boldsymbol\beta^{k^*})+\sum_{j=1}^pP_\lambda(\vert \beta_j^{k^*}\vert)\nonumber
\leq  \left(\widetilde{f}(\boldsymbol\beta^{0})+\sum_{j=1}^pP_\lambda(\vert \beta_j^{0}\vert)\right)-\left(\left\lfloor2\mathcal M\cdot \frac{\left(\widetilde{f}(\boldsymbol\beta^0)+\sum_{j=1}^k P_\lambda(\vert \beta_j^0\vert)\right)-\widetilde f_{\lambda}^*}{\gamma_{opt}^2}\right\rfloor+1\right)\cdot\frac{\gamma_{opt}^2}{{2\mathcal M}}\nonumber<\widetilde f_{\lambda}^*,
\end{align}
which contradicts with the definition of $\widetilde f_{\lambda}^*$. This completes the proof for Part (a).

Suppose that $j:\,\vert\beta_j^{k}\vert\in(0,\,a\lambda)$ for some $j=1,...,p$ and $k\geq 1$. Because a global minimal solution to \eqref{second subproblem} must obey the second-order necessary conditions, which imply that $\left[\frac{\partial^2\left(\frac{1}{2}\left\langle \nabla \widetilde{f}(\boldsymbol\beta^{k+\frac{1}{2}}),\,\boldsymbol\beta-\boldsymbol\beta^{k+\frac{1}{2}} \right\rangle+\frac{{\mathcal M}}{2}\Vert \boldsymbol\beta-\boldsymbol\beta^{k+\frac{1}{2}} \Vert^2\right)}{\partial \beta_j^2}+\frac{\partial^2P_\lambda(\vert\beta_j\vert)}{\partial \beta_j^2}\right]_{\beta_j:=\beta_j^{k}}\geq 0$. This inequality can be simplified equivalently into  ${\mathcal M}-\frac{1}{a}\geq 0$, which, however, contradicts our assumption of $a<\frac{1}{{\mathcal M}}$. As a result, it must hold that $\vert \beta_j^{k}\vert \notin(0,\,a\lambda)$ for all $j=1,...,p$ for all $k\geq 1$. This  proves Part (d).

Let $k^*$ be the iteration count when  the algorithm  terminates with   $\widetilde{f}(\boldsymbol\beta^{k^*+1})+\sum_{j=1}^pP_\lambda(\vert \beta_j^{k^*+1}\vert)>\widetilde{f}(\boldsymbol\beta^{k^*})+\sum_{j=1}^pP_\lambda(\vert \beta_j^{k^*}\vert)-\frac{\gamma_{opt}^2}{2\mathcal M}$ being satisfied for the first time. This, combined with \eqref{to check and compare with solution difference} and the assumption that $\gamma_{opt}\leq a\lambda\mathcal M$, implies that
\begin{align}\Vert \boldsymbol\beta^{k^*+\frac{1}{2}}-\boldsymbol\beta^{k^*}\Vert  <   \frac{{\gamma_{opt}}}{{\mathcal M}}\leq a\lambda.\label{testNew terminate}
\end{align}
Combining this with \eqref{test new kkt condition first sub problem},  we have 
\begin{align}
{\gamma_{opt}}>{\mathcal M}\Vert \boldsymbol\beta^{{k^*}+\frac{1}{2}}-\boldsymbol\beta^{k^*}\Vert =\left\Vert\nabla \widetilde{f}(\boldsymbol\beta^{k^*}) +\left( P_{\lambda}(\vert\beta^{k^*}_j\vert)\cdot \varkappa(\beta_j^{k^*+\frac{1}{2}}):\,j=1,...,p\right)\right\Vert.\label{test kkt condition here 1}
\end{align}

Part (d) indicates that $\beta^{k}_j \neq 0\Longrightarrow \vert\beta_j^k\vert\geq a\lambda$ for all $k\geq 1$. In view of \eqref{testNew terminate}, we then know that   $\vert\beta^{k^*+\frac{1}{2}}_j-\beta_j^{k^*}\vert <a\lambda$ for all $j$. Hence, $\beta_j^{k^*+\frac{1}{2}}>0$ if $\beta_j^{k^*}>0$ and   $\partial(\vert\beta_j^{k^*+\frac{1}{2}}\vert)= \partial(\vert\beta_j^{k^*}\vert)=\{1\}$ for all $j:\,\beta^{k^*}_j> 0$. Likewise, it also holds that   $\partial(\vert\beta_j^{k^*+\frac{1}{2}}\vert)= \partial(\vert\beta_j^{k^*}\vert)=\{-1\}$ for all $j:\,\beta^{k^*}_j< 0$. Furthermore, we also observe that    $\varkappa(\vert\beta_j^{k^*+\frac{1}{2}}\vert)\in[-1,\,1]=\partial (\vert\beta_j^{k^*}\vert)$ for all $j:\,\beta_j^{k^*}=0$. In view of \eqref{test kkt condition here 1}, we have
$
{\gamma_{opt}}> \left\Vert\nabla \widetilde{f}(\boldsymbol\beta^{k^*}) +\left(P'_{\lambda}(\vert\beta_j^{k^*}\vert)\cdot \widetilde{\varkappa}_j,\,j=1,...,p\right)\right\Vert,$ for some $\widetilde{\boldsymbol\varkappa}:=(\widetilde{\varkappa}_j)$ such that $\widetilde{\varkappa}_j\in\partial(\vert \beta_j^{k^*}\vert)$ for all $j$. We have now proven  the   satisfaction of the approximate first-order conditions in \eqref{desired first order conditions}. Further, Part (d) implies that  $\{(k,\,p):\,\vert \beta^{k}_j\vert\in(0,\,a\lambda),\,k\geq 1,\,p=1,...,p\}=\emptyset$. Therefore, as part of the S$^3$ONC, the necessary condition of optimality  that $
\widetilde U_{L,\infty}+P''_\lambda(\vert\beta^{k}_j\vert)\geq 0$ for any $(k,\,p):\,\vert \beta^{k}_j\vert\in(0,\,a\lambda),\,k\geq 1,\,p=1,...,p$ is satisfied. We have thus proven Part (b).

Finally, invoking \eqref{label inequality descent}, we have the desired inequality of $\widetilde  {f}_\lambda(\boldsymbol\beta^{k^*})\leq \widetilde  {f}_\lambda(\boldsymbol\beta^{0})$, as claimed in Part (c).
\hfill \ensuremath{\Box}\end{proof}

\subsection{Useful Lemmata}
 
\begin{lemma}\label{new lemma result}
Suppose that Assumption \ref{Lipschitz condition} holds and that $\epsilon>0$ is an arbitrary scalar. 
\begin{itemize}
\item[(a)]   For some universal constant $\widetilde c>0$,  
\begin{multline}
\mathbb P\left[\left\vert \mathcal L_n(\boldsymbol\beta_1,\mathbf Z_1^n)-\mathcal L_n(\boldsymbol\beta_2,\mathbf Z_1^n)\right\vert\leq \left(2 \sigma_L+\mathcal  C_\mu\right)\cdot\epsilon,\,\forall\,(\boldsymbol\beta_1,\,\boldsymbol\beta_2)\in\Re^p: \,\right.
\\
\left.\Vert \boldsymbol\beta_1\Vert_\infty\leq R,\,\Vert \boldsymbol\beta_2\Vert_\infty\leq R,\,\Vert \boldsymbol\beta_1-\boldsymbol\beta_2\Vert\leq \epsilon \right]\geq 1-2\exp(-\widetilde c\cdot n)
\end{multline}
\item[(b)]$
\left\vert \mathbb E[ \mathcal L_n(\boldsymbol\beta_1,\mathbf Z_1^n)]-\mathbb E[\mathcal L_n(\boldsymbol\beta_2,\mathbf Z_1^n)]\right\vert\leq \mathcal  C_\mu\cdot \epsilon$, for all $(\boldsymbol\beta_1,\,\boldsymbol\beta_2)\in\Re^p: \,\Vert \boldsymbol\beta_1\Vert_\infty\leq R,\,\Vert \boldsymbol\beta_2\Vert_\infty\leq R,\,\Vert \boldsymbol\beta_1-\boldsymbol\beta_2\Vert\leq \epsilon$.
\end{itemize}
\end{lemma}
 \begin{proof}{Proof.}  
 This  lemma and its proof are  straightforward modifications from \cite{shapiro2014a}.
 To show Part (a), we invoke  a Bernstein-like inequality under Assumption \ref{Lipschitz condition}. Consequently,  for all $\boldsymbol\beta\in\Re^p:\,\Vert \boldsymbol\beta\Vert_\infty\leq R$ and some universal constant $\widetilde c>0$, it holds that $
\mathbb P\left[\left\vert \sum_{i=1}^n \frac{1}{n}\left\{\mathcal C(Z_i)-\mathbb E[\mathcal C(Z_i)] \right\}\right\vert>\sigma_L\left(\frac{t}{n}+\sqrt{\frac{t}{n}}\right)\right]\leq 2\exp\left(-\widetilde c t\right),~\forall t\geq 0.$
With $t:=n$ and $\mathbb E[\mathcal C(Z_i)]\leq \mathcal  C_\mu$ (due to Assumption \ref{Lipschitz condition}), we immediately have 
\begin{align}
\mathbb P\left[\sum_{i=1}^n\frac{\mathcal C(Z_i)}{n}\leq 2 \sigma_L+\mathcal  C_\mu\right]\leq 1-2\exp\left(-\widetilde c n\right).\label{prob lipschitz bound}
\end{align}

If we invoke Assumption \ref{Lipschitz condition}  given the event $\{\sum_{i=1}^n\frac{\mathcal C(Z_i)}{n}\leq 2 \sigma_L+\mathcal  C_\mu\}$, we have that for any $(\boldsymbol\beta_1,\,\boldsymbol\beta_2)\in\Re^p: \,\Vert \boldsymbol\beta_1\Vert_\infty\leq R,\,\Vert \boldsymbol\beta_2\Vert_\infty\leq R,\,\Vert \boldsymbol\beta_1-\boldsymbol\beta_2\Vert_\infty\leq \epsilon$,
\begin{align}
\left\Vert  \frac{1}{n}\sum_{i=1}^n L(\boldsymbol\beta_1,\, Z_i)-\frac{1}{n}\sum_{i=1}^n L(\boldsymbol\beta_2,\, Z_i)\right\Vert\nonumber
&\leq \frac{1}{n}\sum_{i=1}^n\Vert L(\boldsymbol\beta_1,\, Z_i)-L(\boldsymbol\beta_2,\, Z_i)\Vert\nonumber
\\&\leq \frac{1}{n}\sum_{i=1}^n \mathcal C(Z_i)\Vert \boldsymbol\beta_1-\boldsymbol\beta_2\Vert\leq (2\sigma_L+ \mathcal  C_\mu)\Vert \boldsymbol\beta_1-\boldsymbol\beta_2\Vert\leq (2\sigma_L+ \mathcal  C_\mu)\cdot\epsilon.\nonumber
\end{align} 
This, combined with \eqref{prob lipschitz bound}, yields the desired result in Part (a).
 
To show Part (b), by Assumption \ref{Lipschitz condition}, it holds that $\mathbb E\left[\left\vert \mathcal L_n(\boldsymbol\beta_1,\mathbf Z_1^n)-\mathcal L_n(\boldsymbol\beta_2,\mathbf Z_1^n)\vphantom{\frac{1}{1}}\right\vert\right]\leq \mathbb E\left[ \sum_{i=1}^n\frac{\mathcal C(Z_i)}{n}\Vert\boldsymbol\beta_1-\boldsymbol\beta_2\Vert\right]. $
Due to the convexity of the function $\vert\cdot\vert$, it therefore holds that
\begin{align}
\left\vert \vphantom{\frac{1}{1}}\mathbb E\left[\mathcal L_n(\boldsymbol\beta_1,\mathbf Z_1^n)\right]-\mathbb E\left[\mathcal L_n(\boldsymbol\beta_2,\mathbf Z_1^n)\right]\right\vert&\leq \mathbb E\left[ \sum_{i=1}^n\frac{\mathcal C(Z_i)}{n}\Vert\boldsymbol\beta_1-\boldsymbol\beta_2\Vert\right] =  \mathbb E\left[ \sum_{i=1}^n\frac{\mathcal C(Z_i)}{n}\right]\cdot\Vert\boldsymbol\beta_1-\boldsymbol\beta_2\Vert.\label{to combine lemma lipschitz 2}
\end{align}
Invoking Assumption \ref{Lipschitz condition} again, it holds that $\mathbb E\left[\sum_{i=1}^n\frac{\mathcal C(Z_i)}{n}\right]= \frac{\sum_{i=1}^n\mathbb E[\mathcal C(Z_i)]}{n}\leq \mathcal  C_\mu$.
This combined with \eqref{to combine lemma lipschitz 2} immediately leads to the desired result in Part (b).
  \hfill \ensuremath{\Box}\end{proof} 
 
\begin{lemma}\label{initial gap}
For any fixed  $\mathbf Z_1^n\in\mathcal W^n$, if  $\widehat{\boldsymbol\beta}^{\ell_1}$ is a finite optimal solution to the minimization problem  $\min_{\boldsymbol\beta}\mathcal L_{n}(\boldsymbol\beta,\, \mathbf Z_1^n)+\lambda\vert \boldsymbol\beta \vert$, then  $\mathcal L_{n,\lambda}(\widehat{\boldsymbol\beta}^{\ell_1},\mathbf Z_1^n)\leq  \mathcal L_{n,\lambda}(\boldsymbol\beta_{\varepsilon_{A}}^*,\mathbf Z_1^n)+\lambda \vert \boldsymbol\beta_{\varepsilon_{A}}^* \vert$.
\end{lemma}
 \begin{proof}{Proof.}  
Let ${\beta}_{{\varepsilon_{A}},j}^{*}$ be the $j$-th dimension of $\boldsymbol{\beta}_{{\varepsilon_{A}}}^{*}$. By the definition of $\widehat{\boldsymbol\beta}^{\ell_1}$, it holds that
\begin{align}
 \mathcal L_{n}(\widehat{\boldsymbol\beta}^{\ell_1},\mathbf Z_1^n)+\lambda\vert \widehat{\boldsymbol\beta}^{\ell_1} \vert\leq  \mathcal L_{n}(\boldsymbol\beta^{*}_{\varepsilon_{A}},\mathbf Z_1^n)+\lambda\vert \boldsymbol\beta^{*}_{\varepsilon_{A}} \vert.\label{small lemma first inequality}
\end{align}
Now consider that, for $\beta_j$ (an arbitrarily chosen entry of $\boldsymbol\beta$), it holds that
$
P_\lambda(\vert\beta_j\vert)=\int_0^{\vert\beta_j\vert}\frac{[a\lambda-\theta]_+}{a}d\theta\leq\int_0^{\vert\beta_j\vert}\frac{a\lambda}{a}d\theta=\lambda\vert\beta_j\vert.$
This combined with \eqref{small lemma first inequality} implies that
$ \mathcal L_{n}(\widehat{\boldsymbol\beta}^{\ell_1},\mathbf Z_1^n)+\sum_{j=1}^pP_\lambda(\vert \widehat{\beta}_j^{\ell_1}\vert) \nonumber
 \leq \,\mathcal L_{n}(\boldsymbol\beta^{*}_{\varepsilon_{A}},\mathbf Z_1^n)+\lambda\vert \boldsymbol\beta_{\varepsilon_{A}}^{*} \vert \nonumber
 \leq \,\mathcal L_{n}(\boldsymbol\beta^{*}_{\varepsilon_{A}},\mathbf Z_1^n)+\sum_{j=1}^pP_\lambda(\vert  {\beta}_{{\varepsilon_{A}},j}^{*}\vert)+\lambda\vert \boldsymbol\beta_{\varepsilon_{A}}^{*} \vert,$ which is as claimed.
  \hfill \ensuremath{\Box}\end{proof}

\begin{theorem}\label{lipschitz}
\citep{nesterov2005a}
For any  convex and compact set ${\mathcal Q}\subset \Re^{\widetilde m}$ for an integer $\widetilde m>0$.
Consider a function $f_\mu(\boldsymbol\beta,\mathbf A):=\max_{\mathbf u}\{\langle \mathbf A\boldsymbol\beta,\,\mathbf u\rangle-\phi(\mathbf u)-\frac{1}{2}\mu \Vert \mathbf u-\mathbf u_0\Vert^2:\, \mathbf u\in {\mathcal Q}\}$ for any $\mathbf A\in\Re^{\widetilde m\times p}$, convex and continuous function $\phi:\,{\mathcal Q}\rightarrow \Re$, and scalar $\mu>0$. This function is well-defined, continuously differentiable, and convex. Its gradient given as $\nabla f_\mu (\boldsymbol\beta,\mathbf A)=\mathbf A^{\top}\mathbf u^*_\mu(\boldsymbol\beta)$ is Lipschitz continuous with constant $L_\mu(\mathbf A)=\frac{1}{\mu}\Vert \mathbf A\Vert^2_{1,2}$, where $\mathbf u^*_\mu(\boldsymbol\beta)=\arg\max_{\mathbf u}\{\langle \mathbf A\boldsymbol\beta,\,\mathbf u\rangle-\phi(u)-\frac{1}{2}\mu \Vert \mathbf u-\mathbf u_0\Vert^2:\, \mathbf u\in {\mathcal Q}\}$.
\end{theorem}
\begin{proof}{Proof.}
See Theorem 1  by \cite{nesterov2005a}.
\hfill \ensuremath{\Box}\end{proof} 

\begin{theorem}\label{useful theorem from M}
Consider an arbitrary function $g\in \mathbb F_{d,r}$. Let $\Psi$ be an activation function that satisfies Assumption \ref{activation function condition here}.  There exist   $\widetilde{\mathbf W}\in \Re^{\widetilde N\times d}$, $\widetilde{\boldsymbol w}\in\Re^{\widetilde N}$, and  $\widetilde{\mathbf b}\in\Re^{\widetilde N}$ such that 
\[
 \underset{\mathbf x\in[-1,1]^{d}}{\textnormal{ess\,sup}}\left\vert \left[\widetilde{\boldsymbol w}^\top \Psi\left( \widetilde{\mathbf W} \mathbf x +\widetilde{\mathbf b}\right)\right] -g(\mathbf x)\right\vert\leq \mathcal C_{NN}\cdot {\widetilde N}^{-r/d}.\]
\end{theorem}
\begin{proof}{Proof.}
The desired result is an immediate implication of Theorem 2.1 by  \cite{mhaskar1996a}, where we set the quantities ``$p$'', ``$d$'', ``$s$'', ``$r$'', and ``$W_{r,s}^p$'' in \cite{mhaskar1996a} to be $\infty$, $1$, $d$, $r$, and $\mathbb F_{d,r}$, respectively, in this paper.\hfill \ensuremath{\Box}
\end{proof}

\begin{lemma}\label{useful here for local solution NN}
Suppose that Assumption \ref{Gaussian assumption NN} holds. Let $K^*$ be any integer such that $K^*\geq   d\cdot \ln(d\cdot K^*)$, let $\xi$ follow the $d$-variate standard normal distribution, and let $\xi_k$, $k=1,...,K^*$, be a sequence of i.i.d. random samples of $\xi$. Then, 
\begin{multline}
\mathbb P\left[\sup_{(\mathbf x,\mathbf y)\in supp(\mathbb D)}\left\vert\frac{y\cdot\ln n}{v}\frac{1}{{K^*}}\sum_{k=1}^{K^*} C_g(\xi_k)\cdot \max\{0,\,\mathbf x^\top\xi_k\}-\frac{y\cdot\ln n}{v}\cdot g(\mathbf x) \right\vert\leq c_6 \cdot\frac{\ln n}{v}\cdot \sqrt{\frac{d\ln\left(d{K^*}\right)}{{K^*}}}\right]
\\
\geq  1-  2\exp\left(-d\ln\left(d{\cdot K^*}\right)\right)-\exp(-d\cdot {K^*}).\label{probability bound bernstein new to show}
\end{multline}
\end{lemma}
\begin{proof}{Proof.}   Our proof below is divided into two steps, where we let $c_0,\,c_1,...$ be some  universal constants.

{\bf Step 1.} 
For a fixed $\mathbf x\in\mathcal X$, consider  a  random variable defined as $\mathcal G_{\mathbf x}(\xi):= C_g(\xi)\cdot\max\{0,\,\mathbf x^\top\xi\}$, where $\xi$ is a $d$-variate standard normal random vector (and thus its entries are i.i.d.).   Apparently, by Assumption \ref{Gaussian assumption NN}, $g(\mathbf x)=\mathbb E_{\xi}[\mathcal G_{\mathbf x}(\xi)]$, where $\mathbb E_{\xi}$ denotes the expectation over $\xi$. We show in step 1 that $\mathcal G_{\mathbf x}(\xi)-g(\mathbf x)$ is a subexponential random variable.

 Because $\Vert\mathbf x\Vert=1$ and $\xi$ has i.i.d. standard normal entries, $\xi^\top\mathbf x$ is a standard normal random variable (and thus it is subgaussian).  By the properties of a subgaussian random variable, $\Vert \xi^\top\mathbf x\Vert_{\psi_2}\leq c_0$ and   $\mathbb P[\vert   \xi^\top\mathbf x \vert\geq t]\leq 2\exp(-c_1\cdot t^2/c_0)$, for any $t\geq 0$. Therefore, $\mathbb P\left[\left\vert  \max\left\{0,\,\xi^\top\mathbf x\right\} \right\vert\geq t\right] \leq 2\exp\left(-c_1\cdot t^2/ c_0\right)$, for any $t\geq 0$. By the definition of the subgaussian norm, we know that $\left\Vert \max\{0,\,\xi^\top\mathbf x\}\right\Vert_{\psi_2}\leq c_2$. Because $\sup_{\xi'}\vert C_g(\xi')\vert\leq 1$ according to Assumption \ref{Gaussian assumption NN}, invoking Lemma 2.7.7 of \cite{vershynin2018a},  we   have $\Vert C_g(\xi)\cdot \max\{0,\,\xi^\top\mathbf x\}\Vert_{\psi_1}\leq \Vert C_g(\xi)\Vert_{\psi_2}\cdot \left\Vert \max\{0,\,\xi^\top\mathbf x\}\right\Vert_{\psi_2}\leq c_3$, which further leads to $\left\Vert C_g(\xi)\cdot \max\{0,\,\xi^\top\mathbf x\}-\mathbb E_\xi[C_g(\xi)\cdot \max\{0,\,\xi^\top\mathbf x\}]\right\Vert_{\psi_1}=\left\Vert \mathcal G_{\mathbf x}(\xi)-g(\mathbf x)\right\Vert_{\psi_1}\leq c_4$. Thus,   $\mathcal G_{\mathbf x}(\xi)-g(\mathbf x)$ is subexponential for a fixed $\mathbf x\in\mathcal X$, as desired in this step.

{\bf Step 2.} This step combines the result from Step 1 and the $\epsilon$-net argument to prove \eqref{probability bound bernstein new to show} as desired.  In doing so, for any $\epsilon\in(0,\,1]$, we construct a net of grids $\mathcal B_{\epsilon}$ such that, for any $\mathbf x\in\mathcal X$, there exists $\mathbf z\in\mathcal B_{\epsilon}$: $\Vert \mathbf x-\mathbf z\Vert\leq \frac{\epsilon}{(\sqrt{5}+1)\cdot\sqrt{d}}$. To that end, it suffices to involve  as many as $\vert \mathcal B_{\epsilon}\vert:=\left\lceil\frac{(\sqrt{5}+1)d}{\epsilon}\right\rceil^d\leq  \left[\frac{2(\sqrt{5}+1)d}{\epsilon}\right]^d$ grids. 

Consider the following two sets
\begin{align}
&\mathcal E^1:= \left\{\max_{\mathbf x\in \mathcal B_\epsilon} \left\vert \frac{1}{K^*}\sum_{k=1}^{K^*} \mathcal G_{\mathbf x}(\xi_k)-g(\mathbf x) \right\vert\leq c_5\cdot  \left(\frac{1}{K^*}t+\sqrt{\frac{t}{K^*}} \right)\right \};\qquad\text{and}\qquad\mathcal E^2:=\left\{\frac{1}{K^*}\sum_{k=1}^{K^*}\left\Vert\xi_k\right\Vert^2\leq 5d\right\}.\nonumber
\end{align}
Because $ \left\vert \frac{1}{{K^*}}\sum_{k=1}^{K^*}\max\left\{0,\,\xi_k^\top\mathbf x_1\right\}-\frac{1}{{K^*}}\sum_{k=1}^{K^*}\max\left\{0,\,\xi_k^\top\mathbf x_2\right\}\right\vert\leq \frac{1}{K^*} \sum_{k=1}^{K^*}\Vert\xi_k\Vert\cdot \Vert\mathbf x_1-\mathbf x_2\Vert\leq \sqrt{\frac{1}{{K^*}}\sum_{k=1}^{K^*}\left\Vert\xi_k\right\Vert^2}\cdot \Vert\mathbf x_1-\mathbf x_2\Vert$ for any $\mathbf x_1,\mathbf x_2\in\mathcal X$, we have
\begin{align}
\mathcal E^2\subseteq \left\{\left\vert \frac{1}{{K^*}}\sum_{k=1}^{K^*}\max\left\{0,\,\xi_k^\top\mathbf x_1\right\}-\frac{1}{{K^*}}\sum_{k=1}^{K^*}\max\left\{0,\,\xi_k^\top\mathbf x_2\right\}\right\vert\leq \sqrt{5d}\cdot\Vert\mathbf x_1-\mathbf x_2\Vert,\,\forall \,\mathbf x_1,\,\mathbf x_2\in\mathcal X\right\}.\nonumber
\end{align}
Further noticing that $\sup_{\xi}\vert C_g(\xi)\vert\leq 1$ as per  Assumption \ref{Gaussian assumption NN}, we then have
\begin{align}
\mathcal E^2\subseteq \left\{\sup_{\xi}\vert C_g(\xi)\vert\cdot \left\vert \frac{1}{{K^*}}\sum_{k=1}^{K^*}\max\left\{0,\,\xi_k^\top\mathbf x_1\right\}-\frac{1}{{K^*}}\sum_{k=1}^{K^*}\max\left\{0,\,\xi_k^\top\mathbf x_2\right\}\right\vert\leq \sqrt{5d}\cdot\Vert\mathbf x_1-\mathbf x_2\Vert,\,\forall \,\mathbf x_1,\,\mathbf x_2\in\mathcal X\right\}.\label{set event relationship}
\end{align}

We may then continue with the $\epsilon$-net argument to obtain that, given the event $\mathcal E^1\cap\mathcal E^2$, for any $\mathbf x\in\mathcal X$, there exists $\mathbf z\in\mathcal B_{\epsilon}:\Vert\mathbf z-\mathbf x\Vert\leq \frac{\epsilon}{(\sqrt{5}+1)\cdot\sqrt{d}}$ such that 
\begin{align}
&  \left\vert\frac{1}{{K^*}}\sum_{k=1}^{K^*} \mathcal G_{\mathbf x}(\xi_k)-\mathbb E_{\xi}\left[\mathcal G_{\mathbf x}(\xi)\right]\right\vert\nonumber
\\\leq  &  \left\vert\frac{1}{{K^*}}\sum_{k=1}^{K^*} \mathcal G_{\mathbf x}(\xi_k)-\frac{1}{{K^*}}\sum_{k=1}^{K^*} \mathcal G_{\mathbf z}(\xi_k)\right\vert+\left\vert\frac{1}{{K^*}}\sum_{k=1}^{K^*} \mathcal G_{\mathbf z}(\xi_k)-g(\mathbf z)\right\vert +\left\vert\mathbb E_{\xi}\left[\mathcal G_{\mathbf x}(\xi)\right]-g(\mathbf z)\right\vert.\nonumber
\\{\leq}  &  \left\vert\frac{1}{{K^*}}\sum_{k=1}^{K^*} \mathcal G_{\mathbf x}(\xi_k)-\frac{1}{{K^*}}\sum_{k=1}^{K^*} \mathcal G_{\mathbf z}(\xi_k)\right\vert+\left\vert\frac{1}{{K^*}}\sum_{k=1}^{K^*} \mathcal G_{\mathbf z}(\xi_k)-g(\mathbf z)\right\vert +\sup_{\xi}\left\vert C_g(\xi)\right\vert\cdot \mathbb E_{\xi}\left[\left\vert \max\{0,\,\xi^\top\mathbf x\}- \max\{0,\,\xi^\top\mathbf z\} \right\vert\right].\nonumber
\\\leq  &  \sup_{\xi'}\left\vert C_g(\xi')\right\vert\cdot\left\vert\frac{1}{{K^*}}\sum_{k=1}^{K^*}  \max\{0,\,\xi_k^\top\mathbf x\}-\frac{1}{{K^*}}\sum_{k=1}^{K^*}   \max\{0,\,\xi_k^\top\mathbf z\}\right\vert+\left\vert\frac{1}{{K^*}}\sum_{k=1}^{K^*} \mathcal G_{\mathbf z}(\xi_k)-g(\mathbf z)\right\vert\nonumber
\\&\qquad\qquad\qquad\qquad+\sup_{\xi'}\left\vert C_g(\xi')\right\vert\cdot \mathbb E_{\xi}\left[\Vert\xi\Vert\right]\cdot \Vert \mathbf x-\mathbf z\Vert.\nonumber
\\{\leq}  & \sqrt{5d}\Vert\mathbf z-\mathbf x\Vert+\left\vert\frac{1}{{K^*}}\sum_{k=1}^{K^*} \mathcal G_{\mathbf z}(\xi_k)-g(\mathbf z)\right\vert+\sqrt{d}\Vert\mathbf z-\mathbf x\Vert\label{to work out detail here}
\end{align}
Here \eqref{to work out detail here} is due to \eqref{set event relationship} and the observation that $\left(\mathbb E_{\xi}\left[\left\Vert  \xi\right\Vert\right]\right)^2\leq \mathbb E_{\xi}\left[\left\Vert  \xi \right\Vert^2\right]=d$, where the latter is based on the fact that  $\Vert  \xi  \Vert^2$ follows the $\chi^2$ distribution with  the degree of freedom being $d$.  We may then continue to obtain that, given $\mathcal E_1\cap\mathcal E_2$, it holds that
$
\left\vert\frac{1}{{K^*}}\sum_{k=1}^{K^*} \mathcal G_{\mathbf x}(\xi_k)-\mathbb E_{\xi}\left[\mathcal G_{\mathbf x}(\xi)\right]\right\vert\nonumber
\leq c_5\cdot \left(\frac{t}{{K^*}}+\sqrt{\frac{t}{{K^*}}}\right)+(\sqrt{5}+1)\sqrt{d}\Vert\mathbf z-\mathbf x\Vert\leq c_5\cdot \left(\frac{t}{{K^*}}+\sqrt{\frac{t}{{K^*}}}\right)+\epsilon.$

We   now establish the probability for $\mathcal E^1\cap\mathcal E^2$. As an immediate implication of Step 1,   a  Bernstein-like inequality holds, for any fixed $\mathbf x\in\mathcal B_{\epsilon}$, as below:
\begin{align}
\mathbb P\left[\left\vert\frac{1}{{K^*}}\sum_{k=1}^{K^*} \mathcal G_{\mathbf x}(\xi_k)-g(\mathbf x)\right\vert\geq c_5\cdot \left(\frac{t}{{K^*}}+\sqrt{\frac{t}{{K^*}}}\right)\right]\leq 2\exp(-t).\label{subexponential result NN local}
\end{align}  Together with $\vert \mathcal B_{\epsilon}\vert:=\left[\frac{2\cdot(\sqrt{5}+1)d}{\epsilon}\right]^d$, the above inequality implies that
\begin{align}
\mathbb P[\mathcal E^1]=\mathbb P\left[\max_{\mathbf x\in \mathcal B_\epsilon}\left\vert\frac{\sum_{k=1}^{K^*} \mathcal G_{\mathbf x}(\xi_k)}{{K^*}}-g(\mathbf x)\right\vert\leq c_5\cdot \left(\frac{t}{{K^*}}+\sqrt{\frac{t}{{K^*}}}\right)\right]
\geq  1-\left[\frac{2\cdot(\sqrt{5}+1)d}{\epsilon}\right]^d\exp(-t).\label{bernstein test here}
\end{align}

In establishing the probability of $\mathcal E^2$, we observe that  $\xi_k$ follows the $d$-variate standard Gaussian distribution. Thus,  $\sum_{k=1}^{K^*}\Vert\xi_k\Vert^2$ is a $\chi^2$-distribution, whose degree of freedom is $d\cdot K^*$. A well-known tail bound for the $\chi^2$-distribution yields that $\mathbb P\left[\sum_{k=1}^{K^*}\Vert\xi_k\Vert^2\leq dK^*\cdot \left(1+2\sqrt{t}+2t\right)\right]\geq 1-\exp(-dtK^*)$. This further  implies that $\mathbb P[\mathcal E^2]=\mathbb P\left[\frac{1}{{{K^*}}}\sum_{k=1}^{K^*}\Vert\xi_k\Vert^2\leq 5d \right] \geq 1-\exp(-d\cdot {K^*})$.   Thus, combining the above by invoking the union bound and De Morgan's law, for any $\epsilon>0$, we have that    $\mathbb P[\mathcal E^1\cap\mathcal E^2]\geq 1-\left[\frac{2(\sqrt{5}+1)d }{\epsilon}\right]^d\cdot\exp(-t)-\exp(-d\cdot {K^*})$.
Therefore, for any $\epsilon>0$,
\begin{align}
&\mathbb P\left[\sup_{\mathbf x:\,\Vert \mathbf x\Vert=1}\left\vert\frac{1}{{K^*}}\sum_{k=1}^{K^*} \mathcal G_{\mathbf x}(\xi_k)-g(\mathbf x)]\right\vert\leq c_5\cdot \left(\frac{t}{{K^*}}+\sqrt{\frac{t}{{K^*}}}\right)+\epsilon\right]\nonumber
\geq \,1-\left[\frac{2\cdot(\sqrt{5}+1)d}{\epsilon}\right]^d\cdot\exp(-t)-\exp(-d\cdot {K^*}) \nonumber
\\
=\,&1-  \exp\left(-t+d\ln\left[\frac{2\cdot(\sqrt{5}+1)d}{\epsilon}\right]\right)-\exp(-d\cdot {K^*}).\label{eq 62}
\end{align}
We may as well let $\epsilon = 1/{K^*}$ and $t=2d\ln\left[\frac{2\cdot(\sqrt{5}+1)d}{\epsilon}\right]=2d\ln\left(2(\sqrt{5}+1)d\cdot {K^*}\right)$. Consequently (and in view of the assumption that $K^*\geq  d\ln (dK^*)$), \eqref{eq 62} is reduced to 
\begin{multline}
\mathbb P\left[\sup_{\mathbf x:\,\Vert \mathbf x\Vert=1}\left\vert\frac{1}{{K^*}}\sum_{k=1}^{K^*} \mathcal G_{\mathbf x}(\xi_k)-\mathbb E_{\xi}\left[\mathcal G_{\mathbf x}(\xi)\right]\right\vert\leq c_6\cdot \sqrt{\frac{d\ln\left(d{K^*}\right)}{{K^*}}}\right]
\\\geq  1-  2\exp\left(-d\ln\left(2\cdot(\sqrt{5}+1)d{K^*}\right)\right)-\exp(-d\cdot {K^*})
\\\geq  1-  2\exp\left(-d\ln\left(d{K^*}\right)\right)-\exp(-d\cdot {K^*}),\label{probability bound bernstein}
\end{multline}
which (combined with $y\in\{-1,\,1\}$) further leads to 
\begin{multline}
\mathbb P\left[\sup_{(\mathbf x,\mathbf y)\in supp(\mathbb D)}\left\vert\frac{y\cdot\ln n}{v}\frac{1}{{K^*}}\sum_{k=1}^{K^*} \mathcal G_{\mathbf x}(\xi_k)-\frac{y\cdot\ln n}{v} \mathbb E_{\xi}\left[\mathcal G_{\mathbf x}(\xi)\right]\right\vert\leq c_6 \cdot\frac{\ln n}{v}\cdot \sqrt{\frac{d\ln\left(d{K^*}\right)}{{K^*}}}\right]
\\\geq  1-  2\exp\left(-d\ln\left(d{K^*}\right)\right)-\exp(-d\cdot {K^*}),\label{probability bound bernstein new 2}
\end{multline}
which is the desired result. \hfill \ensuremath{\Box}
\end{proof}

\vskip 0.2in
 \bibliographystyle{abbrvnat}
 \bibliography{ref_OR}

\end{document}